\documentclass[12pt,letter]{amsart}
\usepackage[
    top=3cm, bottom=3cm, inner=3.1cm, outer=3.1cm,
    marginparwidth=2cm 
    ]{geometry}

\usepackage{amssymb,latexsym, amsmath, amsxtra, bbm, bm}
\usepackage[dvips]{graphics}
\usepackage{xypic}
\usepackage{verbatim}
\usepackage[abs]{overpic}
\usepackage{hyperref}
\usepackage{mathtools}

\allowdisplaybreaks[3]

\theoremstyle{plain}
        \newtheorem{theorem}{Theorem}[section]
        \newtheorem*{theorem*}{Theorem}
        \newtheorem*{conj*}{Conjecture}
        \newtheorem{lemma}[theorem]{Lemma}
        \newtheorem{prop}[theorem]{Proposition}
        \newtheorem{cor}[theorem]{Corollary}

        \newtheorem{thmx}{Theorem}

\theoremstyle{definition}
        \newtheorem{definition}[theorem]{Definition}
        \newtheorem{rem}[theorem]{Remark}
         \newtheorem{rems}[theorem]{Remarks}
         
         \newtheorem*{assumptions}{The Assumptions}

\theoremstyle{remark}

\numberwithin{equation}{section}
\numberwithin{theorem}{section}
\numberwithin{table}{section}
\numberwithin{figure}{section}

\providecommand{\defn}[1]{\emph{#1}}

\renewcommand{\leq}{\leqslant}

\renewcommand{\geq}{\geqslant}

\newcommand{\diam}  {\operatorname{diam}}

\newcommand{\inte}  {\operatorname{inte}}
\newcommand{\inter}  {\operatorname{int}}

\newcommand{\id} {\operatorname{id}}

\newcommand{\card} {\operatorname{card}}
\newcommand{\supp}{\operatorname{supp}}
\newcommand{\R}{\mathbb{R}}
\newcommand{\B}{\mathbb{B}}    
\newcommand{\C}{\mathbb{C}}      

\newcommand{\N}{\mathbb{N}}      
\newcommand{\Z}{\mathbb{Z}}      

\newcommand{\D}{\mathbb{D}}     
 
\renewcommand{\H}{\mathbb{H}}    
\providecommand{\abs}[1]{\lvert#1\rvert}
\newcommand{\Abs}[1]{\left\lvert#1\right\rvert}
\providecommand{\Absbig}[1]{\bigl\lvert#1\bigr\rvert}
\providecommand{\Absbigg}[1]{\biggl\lvert#1\biggr\rvert}
\providecommand{\AbsBig}[1]{\Bigl\lvert#1\Bigr\rvert}
\providecommand{\AbsBigg}[1]{\Biggl\lvert#1\Biggr\rvert}
\providecommand{\norm}[1]{\|#1\|}
\providecommand{\Norm}[1]{\left\|#1\right\|}
\providecommand{\Normbig}[1]{\bigl\|#1\bigr\|}
\providecommand{\Normbigg}[1]{\biggl\|#1\biggr\|}
\providecommand{\NormBig}[1]{\Bigl\|#1\Bigr\|}

\renewcommand{\:}{\colon}

\renewcommand{\b}{\mathfrak{b}}
\newcommand{\w}{\mathfrak{w}}
\renewcommand{\c}{\mathfrak{c}}

\newcommand{\I}{\mathbf{i}}

\newcommand{\crit}{\operatorname{crit}}
\newcommand{\post}{\operatorname{post}}

\newcommand{\LIP}{\operatorname{LIP}}

\newcommand{\CC}{\mathcal{C}}

 \newcommand{\DD}{\mathbf{D}}

\newcommand{\X} {\mathbf{X}}

\newcommand{\E} {\mathbf{E}}

\newcommand{\V} {\mathbf{V}}
 
\newcommand{\CCC}{C}

\newcommand{\PPP}{\mathcal{P}}

\newcommand{\MMM}{\mathcal{M}}

\newcommand{\Holder}[1] {\CCC^{0,#1}}

\newcommand{\Hseminorm}[2] {\Abs{#2}_{#1}}

\newcommand{\Hseminormbig}[2] {\Absbig{#2}_{#1}}

\newcommand{\Hseminormbigg}[2] {\Absbigg{#2}_{#1}}

\newcommand{\HseminormBig}[2] {\AbsBig{#2}_{#1}}

\newcommand{\Hnorm}[3] {\Norm{#2}_{\Holder{#1}{#3}}}

\newcommand{\Hnormbig}[3] {\Normbig{#2}_{\Holder{#1}{#3}}}

\newcommand{\NHnorm}[4] {\Norm{#3}^{[#2]}_{\Holder{#1}{#4}}}

\newcommand{\NHnormbig}[4] {\Normbig{#3}^{[#2]}_{\Holder{#1}{#4}}}

\newcommand{\NHnormBig}[4] {\NormBig{#3}^{[#2]}_{\Holder{#1}{#4}}}

\newcommand{\HseminormD}[2] {\Abs{#2}_{#1}}

\newcommand{\NHnormD}[3] {\Norm{#3}^{[#2]}_{#1}}

\newcommand{\vertiii}[1]{{\left\vert\kern-0.25ex\left\vert\kern-0.25ex\left\vert #1 
    \right\vert\kern-0.25ex\right\vert\kern-0.25ex\right\vert}}

\newcommand{\vertiiibig}[1]{{\bigl\vert\kern-0.25ex\bigl\vert\kern-0.25ex\bigl\vert #1 
    \bigr\vert\kern-0.25ex\bigr\vert\kern-0.25ex\bigr\vert}}
    
\newcommand{\vertiiiBig}[1]{{\Bigl\vert\kern-0.25ex\Bigl\vert\kern-0.25ex\Bigl\vert #1 
    \Bigr\vert\kern-0.25ex\Bigr\vert\kern-0.25ex\Bigr\vert}}

\newcommand{\OpHnormD}[2] {\vertiii{#2}_{#1}}
\newcommand{\OpHnormDbig}[2] {\vertiiibig{#2}_{#1}}

\newcommand{\NOpHnormD}[3] {\vertiii{#3}^{[#2]}_{#1}}
\newcommand{\NOpHnormDbig}[3] {\vertiiibig{#3}^{[#2]}_{#1}}

\newcommand{\RR}{\mathcal{L}}

\newcommand{\RRR}{\mathbbm{L}}

\newcommand{\MM}{\mathbbm{M}}

\newcommand{\wt}[1]{\widetilde{#1}}

\newcommand{\Arg}{\operatorname{Arg}}

\newcommand{\minus}{\scalebox{0.6}[0.6]{$-\!\!\;$}}

\newcommand{\circsmall}{\scalebox{0.6}[0.6]{$\circ$}}

\newcommand{\Orb}{\mathfrak{P}}

\newcommand{\ti}{\vartriangle}

\newcommand{\DS}{\mathfrak{D}}

\newcommand{\bigO}{\mathcal{O}}

\begin{document}
\title[Prime orbit theorems for expanding Thurston maps]{Prime orbit theorems for expanding Thurston maps:\\ Latt\`{e}s maps and split Ruelle operators}
\author{Zhiqiang~Li \and Tianyi~Zheng}

\address{Zhiqiang~Li, School of Mathematical Sciences \& Beijing International Center for Mathematical Research, Peking University, Beijing 100871, China}
\email{zli@math.pku.edu.cn}
\address{Tianyi~Zheng, Department of Mathematics, University of California, San Diego, San Diego, CA 92093--0112}
\email{tzheng2@math.ucsd.edu}



\begin{abstract}
We obtain an analog of the prime number theorem for a class of branched covering maps on the $2$-sphere $S^2$ called expanding Thurston maps, which are topological models of some non-uniformly expanding rational maps without any smoothness or holomorphicity assumption. More precisely, we show that the number of primitive periodic orbits, ordered by a weight on each point induced by a non-constant (eventually) positive real-valued H\"{o}lder continuous function on $S^2$ satisfying the $\alpha$-strong non-integrability condition, is asymptotically the same as the well-known logarithmic integral, with an exponential error bound. In particular, our results apply to postcritically-finite rational maps for which the Julia set is the whole Riemann sphere. Moreover, a stronger result is obtained for Latt\`{e}s maps. 
\end{abstract}

\maketitle

\tableofcontents

\section{Introduction}

Complex dynamics is a vibrant field of dynamical systems, focusing on the study of iterations of polynomials and rational maps on the Riemann sphere $\widehat{\C}$. It is closely connected, via \emph{Sullivan's dictionary} \cite{Su85, Su83}, to geometric group theory, mainly concerning the study of Kleinian groups.

In complex dynamics, the lack of uniform expansion of a rational map arises from critical points in the Julia set. Rational maps for which each critical point is preperiodic (i.e., eventually periodic) are called \emph{postcritically-finite rational maps} or \emph{rational Thurston maps}. One natural class of non-uniformly expanding rational maps are called \emph{topological Collet--Eckmann maps}, whose basic dynamical properties have been studied extensively (see for example, \cite{PRLS03, PRL07, PRL11, RLS14}). In this paper, we focus on a subclass of topological Collet--Eckmann maps for which each critical point is preperiodic and the Julia set is the whole Riemann sphere. Actually, the most general version of our results is established for topological models of these maps, called \emph{expanding Thurston maps}. Thurston maps were studied by W.~P.~Thurston in his celebrated characterization theorem of postcritically-finite rational maps among such topological models \cite{DH93}. Thurston maps and Thurson's theorem, sometimes known as a fundamental theorem of complex dynamics, are indispensable tools in the modern theory of complex dynamics. Expanding Thurston maps were studied extensively by M.~Bonk, D.~Meyer \cite{BM10, BM17} and P.~Ha\"issinsky, K.~M.~Pilgrim \cite{HP09}.

The investigations of the growth rate of the number of periodic orbits (e.g.\ closed geodesics) have been a recurring theme in dynamics and geometry. 

Inspired by the seminal works of F.~Naud \cite{Na05} and H.~Oh, D.~Winter \cite{OW17} on the growth rate of periodic orbits, known as Prime Orbit Theorems, for hyperbolic (uniformly expanding) polynomials and rational maps, we establish in this paper the first Prime Orbit Theorems (to the best of our knowledge) with exponential error bounds in a non-uniformly expanding setting in complex dynamics. On the other side of Sullivan's dictionary, see related works \cite{MMO14, OW16, OP19}. For an earlier work on dynamical zeta functions for a class of sub-hyperbolic quadratic polynomials, see V.~Baladi, Y.~Jiang, and H.~H.~Rugh \cite{BJR02}. See also the related work of S.~Waddington \cite{Wad97} on strictly preperiodic points of hyperbolic rational maps and the recent work of M.~Pollicott and M.~Urba\'{n}ski \cite{PoU21} on periodic pairs and preimage points of many hyperbolic and parabolic systems.

Given a map  $f\: X \rightarrow X$ on a metric space $(X,d)$ and a function $\phi\: S^2 \rightarrow \R$,  we define the weighted length $l_{f,\phi} (\tau)$ of a primitive periodic orbit 
\begin{equation*}
\tau \coloneqq \bigl\{x, \, f(x), \,  \cdots, \,  f^{n-1}(x) \bigr\} \in \Orb(f)
\end{equation*}
as
\begin{equation}  \label{eqDefComplexLength}
l_{f,\phi} (\tau) \coloneqq \phi(x) + \phi(f(x)) + \cdots + \phi \bigl( f^{n-1}(x) \bigr).
\end{equation}
We denote by
\begin{equation}   \label{eqDefPiT}
\pi_{f,\phi}(T) \coloneqq \card \{ \tau \in \Orb(f)   :  l_{f,\phi}( \tau )  \leq T \}, \qquad   T>0,
\end{equation}
the number of primitive periodic orbits with weighted lengths up to $T$. Here $\Orb(f)$ denotes the set of all primitive periodic orbits of $f$ (see Section~\ref{sctNotation}). 

Note that the Prime Orbit Theorems in \cite{Na05, OW17} are established for the \emph{geometric potential} $\phi= \log \abs{f'}$. For hyperbolic rational maps, the Lipschitz continuity of the geometric potential plays a crucial role in \cite{Na05, OW17}. In our non-uniform expanding setting, critical points destroy the continuity of $\log\abs{f'}$. So we are left with two options to develop our theory, namely, considering 
\begin{enumerate}
\smallskip
\item[(a)] H\"older continuous $\phi$ or

\smallskip
\item[(b)] the geometric potential $\log\abs{f'}$.
\end{enumerate}
Despite the lack of H\"older continuity of $\log\abs{f'}$ in our setting, its value is closely related to the size of pull-backs of sets under backward iterations of the map $f$. This fact enables an investigation of the Prime Orbit Theorem in case (b), which will be studied in an upcoming series of separate works starting with \cite{LRL}. 

The current paper is the second of a series of three papers (together with \cite{LZ24a,LZ23c}) focusing on case (a), in which the incompatibility of H\"older continuity of $\phi$ and non-uniform expansion of $f$ calls for a close investigation of metric geometries associated to $f$.

Latt\`{e}s maps are rational Thurston maps with parabolic orbifolds (see Definition~\ref{defLattesMap}). They form a well-known class of rational maps. We first formulate our theorem for Latt\`{e}s maps. 

\begin{thmx}[Prime Orbit Theorem for Latt\`{e}s maps]   \label{thmLattesPOT}
Let $f\: \widehat{\C} \rightarrow \widehat{\C}$ be a Latt\`{e}s map on the Riemann sphere $\widehat{\C}$. Let $\phi \: \widehat{\C} \rightarrow \R$ be eventually positive and continuously differentiable. Then there exists a unique positive number $s_0>0$ with $P(f,-s_0 \phi) = 0$ and there exists $N_f\in\N$ depending only on $f$ such that the following statements are equivalent:
\begin{enumerate}
\smallskip
\item[(i)] $\phi$ is not cohomologous to a constant in the space $\CCC\bigl(\widehat{\C} \bigr)$ of real-valued continuous functions on $\widehat{\C}$.

\smallskip
\item[(ii)] For each $n\in \N$ with $n\geq N_f$, we have 
\begin{equation*}
\pi_{F,\Phi}(T) \sim \operatorname{Li}\bigl( e^{s_0 T} \bigr) \text{ as }  T \to + \infty,
\end{equation*}
where  $F\coloneqq f^n$ and $\Phi\coloneqq \sum_{i=0}^{n-1} \phi \circ f^i$.

\smallskip
\item[(iii)] For each $n\in \N$ with $n\geq N_f$, there exists a constant $\delta \in (0, s_0)$ such that
\begin{equation*}
\pi_{F,\Phi}(T) = \operatorname{Li}\bigl( e^{s_0 T} \bigr)  + \bigO \bigl( e^{(s_0 - \delta)T} \bigr) \text{ as } T \to + \infty,
\end{equation*}
where  $F\coloneqq f^n$ and $\Phi\coloneqq \sum_{i=0}^{n-1} \phi \circ f^i$.
\end{enumerate} 
Here $P(f,\cdot)$ denotes the topological pressure, and $\operatorname{Li} (y) \coloneqq \int_2^y\! \frac{1}{\log u} \,\mathrm{d} u$, $y>0$, is the \defn{Eulerian logarithmic integral function}.
\end{thmx}

See Definitions~\ref{defCohomologous} and~\ref{defEventuallyPositive} for the definitions of co-homology and eventually positive functions, respectively.

The implication (i)$\implies$(iii) relies crucially on some local properties of the metric geometry of Latt\`{e}s maps, and is not expected to hold (by the authors) in general. To establish the exponential error bound similar to that in (iii) for a class of more general rational Thurston maps, we impose a condition on the potential called \emph{$\alpha$-strong non-integrability condition} (Definition~\ref{defStrongNonIntegrability}), which turns out to be generic. The genericity of this condition will be the main theme of the third and last paper \cite{LZ23c} of the current series. An analog of this condition in the context of Anosov flows was first proposed by D.~Dolgopyat in his seminal work \cite{Do98}.

The following theorem is an immediate consequence of a more general result in Theorem~\ref{thmPrimeOrbitTheorem}.

\begin{thmx}[Prime Orbit Theorems for rational expanding Thurston maps]  \label{thmPrimeOrbitTheorem_rational}
Let $f\: \widehat{\C} \rightarrow \widehat{\C}$ be a postcritically-finite rational map without periodic critical points. Let $\sigma$ be the chordal metric on the Riemann sphere $\widehat{\C}$, and $\phi \: \widehat{\C} \rightarrow \R$ be an eventually positive real-valued H\"{o}lder continuous function. Then there exists a unique positive number $s_0>0$ with topological pressure $P(f,-s_0 \phi) = 0$ and there exists $N_f\in\N$ depending only on $f$ such that for each $n\in \N$ with $n\geq N_f$, the following statement holds for $F\coloneqq f^n$ and $\Phi\coloneqq \sum_{i=0}^{n-1} \phi \circ f^i$:

\begin{enumerate}
\smallskip
\item[(i)] $\pi_{F,\Phi}(T) \sim \operatorname{Li}\bigl( e^{s_0 T} \bigr)$ as $T \to + \infty$
if and only if $\phi$ is not cohomologous to a constant in $\CCC\bigl(\widehat{\C} \bigr)$.

\smallskip
\item[(ii)] Assume that $\phi$ satisfies the strong non-integrability condition (with respect to $f$ and a visual metric). Then there exists $\delta \in (0, s_0)$ such that
\begin{equation*}
	\pi_{F,\Phi}(T) = \operatorname{Li}\bigl( e^{s_0 T} \bigr)  + \bigO \bigl( e^{(s_0 - \delta)T} \bigr) \text{ as } T \to + \infty.
\end{equation*}
\end{enumerate}
\end{thmx}

\begin{rem}
	In  the case where $\phi$ is cohomologous to a constant in $\CCC( S^2 )$, similar results as the ones in Theorem~\ref{thmPrimeOrbitTheorem_rational}~(i) and Theorem~\ref{thmPrimeOrbitTheorem}~(i) below also hold. See Theorems~B and~C of the first paper \cite{LZ24a} of this series.
\end{rem}

Our strategy to overcome the obstacles presented by the incompatibility of the non-uniform expansion of our rational maps and the H\"older continuity of the weight $\phi$ (e.g.\ (a)~the set of $\alpha$-H\"older continuous functions is not invariant under the Ruelle operator $\RR_\phi$, for each $\alpha\in(0,1]$; (b) the weakening of the regularity of the temporal distance compared to that of the potential)  is to investigate the metric geometry of various natural metrics associated to the dynamics such as visual metrics, the canonical orbifold metric, and the chordal metric. Such considerations lead us beyond conformal, or even smooth, dynamical settings and into the realm of topological dynamical systems. More precisely, we will work in the abstract setting of branched covering maps on the topological $2$-sphere $S^2$ (see Subsection~\ref{subsctThurstonMap}) without any smoothness assumptions. A \emph{Thurston map} is a postcritically-finite branched covering map on $S^2$. Thurston maps can be considered as topological models of the corresponding rational maps. 

Via Sullivan's dictionary, the counterpart of Thurston's theorem \cite{DH93} in the geometric group theory is Cannon's Conjecture \cite{Ca94}. This conjecture predicts that a Gromov hyperbolic group $G$ whose boundary at infinity $\partial_\infty G$ is a topological $2$-sphere admits a geometric action on the hyperbolic $3$-space. Gromov hyperbolic groups can be considered as metric-topological systems generalizing the conformal systems in the context of the geometric group theory, namely, convex-cocompact Kleinian groups. Inspired by Sullivan's dictionary and their interest in Cannon's Conjecture, M.~Bonk and D.~Meyer, along with others, studied a subclass of Thurston maps by imposing some additional condition of expansion. Roughly speaking, we say that a Thurston map is \emph{expanding} if for any two points $x, \, y\in S^2$, their preimages under iterations of the map get closer and closer. For each expanding Thurston map, we can equip the $2$-sphere $S^2$ with a natural class of metrics called \emph{visual metrics}. As the name suggests, these metrics are constructed in a similar fashion as the visual metrics on the boundary $\partial_\infty G$ of a Gromov hyperbolic group $G$. See Subsection~\ref{subsctThurstonMap} for a more detailed discussion on these notions. Various ergodic properties, including thermodynamic formalism, on which the current paper crucially relies, have been studied by the first-named author in \cite{Li17} (see also \cite{Li15, Li16, Li18}). Generalization of results in \cite{Li17} to the more general branched covering maps studied by P.~Ha\"issinsky, K.~M.~Pilgrim \cite{HP09} has drawn significant interest recently \cite{HRL19, DPTUZ19, LZheH23}. We believe that our ideas introduced in this paper can be used to establish Prime Orbit Theorems in their setting.

M.~Bonk, D.~Meyer \cite{BM10, BM17} and P.~Ha\"issinsky, K.~M.~Pilgrim \cite{HP09} proved that an expanding Thurston map is conjugate to a rational map if and only if the sphere $(S^2,d)$ equipped with a visual metric $d$ is quasisymmetrically equivalent to the Riemann sphere $\widehat\C$ equipped with the chordal metric. The quasisymmetry cannot be promoted to Lipschitz equivalence due to the non-uniform expansion of Thurston maps. There exist expanding Thurston maps not conjugate to rational Thurston maps (e.g.\ ones with periodic critical points). Our theorems below apply to all expanding Thurston maps, which form the most general setting in this series of papers.

\begin{thmx}[Prime Orbit Theorems for expanding Thurston maps]  \label{thmPrimeOrbitTheorem}
Let $f\: S^2 \rightarrow S^2$ be an expanding Thurston map, and $d$ be a visual metric on $S^2$ for $f$. Let $\phi \in \Holder{\alpha}(S^2,d)$ be an eventually positive real-valued H\"{o}lder continuous function with an exponent $\alpha\in (0,1]$. Denote by $s_0$ the unique positive number with topological pressure $P(f,-s_0 \phi) = 0$. Then there exists $N_f\in\N$ depending only on $f$ such that for each $n\in \N$ with $n\geq N_f$, the following statements hold for $F\coloneqq f^n$ and $\Phi\coloneqq \sum_{i=0}^{n-1} \phi \circ f^i$:

\begin{enumerate}
\smallskip
\item[(i)] $\pi_{F,\Phi}(T) \sim \operatorname{Li}\bigl( e^{s_0 T} \bigr)$ as $T \to + \infty$ if and only if $\phi$ is not cohomologous to a constant in the space $\CCC( S^2 )$ of real-valued continuous functions on $S^2$.

\smallskip
\item[(ii)] Assume that $\phi$ satisfies the $\alpha$-strong non-integrability condition. Then there exists a constant $\delta \in (0, s_0)$ such that
\begin{equation*}
	\pi_{F,\Phi}(T) = \operatorname{Li}\bigl( e^{s_0 T} \bigr)  + \bigO \bigl( e^{(s_0 - \delta)T} \bigr)  \text{ as }  T \to + \infty	.
\end{equation*}
\end{enumerate}
Here $\operatorname{Li}(\cdot)$ is the Eulerian logarithmic integral function defined in Theorem~\ref{thmLattesPOT}.
\end{thmx}

Note that $\lim_{y\to+\infty}   \operatorname{Li}(y) / ( y / \log y )  = 1$, thus we also get $\pi_{F,\Phi}(T) \sim   e^{s_0 T} \big/ ( s_0 T)$ as $T\to + \infty$. 

We remark that our proofs can be modified to derive equidistribution of holonomies similar to the corresponding result in \cite{OW17}, but we choose to omit them in order to emphasize our new ideas and to limit the length of this paper.

In view of Remark~\ref{rmChordalVisualQSEquiv}, Theorem~\ref{thmPrimeOrbitTheorem_rational} is an immediate consequence of Theorem~\ref{thmPrimeOrbitTheorem}.

\begin{rem}    \label{rmNf}
The integer $N_f$ can be chosen as the minimum of $N(f,\wt{\CC})$ from Lemma~\ref{lmCexistsL} over all Jordan curves $\wt{\CC}$ with $\post f \subseteq \wt{\CC} \subseteq S^2$, in which case $N_f = 1$ if there exists a Jordan curve $\CC\subseteq S^2$ satisfying $f(\CC)\subseteq \CC$, $\post f\subseteq \CC$, and no $1$-tile in $\X^1(f,\CC)$ joins opposite sides of $\CC$ (see Definition~\ref{defJoinOppositeSides}). The same number $N_f$ is used in other results in this paper. We also remark that many properties of expanding Thurston maps $f$ can be established for $f$ after being verified first for $f^n$ for all $n\geq N_f$. However, some of the finer properties established for iterates of $f$ still remain open for the map $f$ itself; see for example, \cite{Me13, Me12}.
\end{rem}


Note that due to the lack of algebraic structure of expanding Thurston maps, even the fact that there are only countably many periodic points is not apparent from the definition (see \cite{Li16}). Without any algebraic, differential, or conformal structures, the main tools we rely on are from the interplay between the metric properties of various natural metrics and the combinatorial information on the iterated preimages of certain Jordan curves $\CC$ on $S^2$ (see Subsection~\ref{subsctThurstonMap}).

By well-known arguments of M.~Pollicott and R.~Sharp inspired from number theory \cite{PS98}, the counting result in Theorem~\ref{thmPrimeOrbitTheorem} follows from some quantitative information on the holomorphic extension of certain dynamical zeta function $\zeta_{F,\,\minus \Phi}$ defined as formal infinite products over periodic orbits. We briefly recall dynamical zeta functions and define the dynamical Dirichlet series in our context below. See Subsection~\ref{subsctDynZetaFn} for a more detailed discussion.

Let $f\: S^2\rightarrow S^2$ be an expanding Thurston map and $\psi\in\CCC(S^2,\C)$ be a complex-valued continuous function on $S^2$. We denote by the formal infinite product
\begin{equation*}  
\zeta_{f,\,\minus\psi} (s) \coloneqq 
\exp \Biggl( \sum_{n=1}^{+\infty} \frac{1}{n} \sum_{ x = f^n(x) } e^{-s S_n \psi(x)} \Biggr), \qquad s\in\C,
\end{equation*}
the \defn{dynamical zeta function} for the map $f$ and the \emph{potential} $\psi$. Here we write $S_n \psi (x)  \coloneqq \sum_{j=0}^{n-1} \psi(f^j(x))$ as defined in (\ref{eqDefSnPt}). We remark that $\zeta_{f,\,\minus\psi}$ is the Ruelle zeta function for the suspension flow over $f$ with roof function $\psi$ if $\psi$ is positive. We define the \emph{dynamical Dirichlet series} associated to $f$ and $\psi$ as the formal infinite product
\begin{equation*}  
\DS_{f,\,\minus\psi,\, \deg_f} (s) \coloneqq \exp \Biggl( \sum_{n=1}^{+\infty} \frac{1}{n} \sum_{x = f^n(x)} e^{-s S_n \psi(x)} \deg_{f^n}(x) \Biggr), \qquad s\in\C.
\end{equation*}
Here $\deg_{f^n}$ is the \emph{local degree} of $f^n$ at $x\in S^2$.

Note that if $f\: S^2 \rightarrow S^2$ is an expanding Thurston map, then so is $f^n$ for each $n\in\N$. 

Recall that a function is holomorphic on a set $A\subseteq \C$ if it is holomorphic on an open set containing $A$.

\begin{thmx}[Holomorphic extensions of dynamical Dirichlet series and zeta functions for expanding Thurston maps]  \label{thmZetaAnalExt_InvC}
Let $f\: S^2 \rightarrow S^2$ be an expanding Thurston map, and $d$ be a visual metric on $S^2$ for $f$. Fix $\alpha\in(0,1]$. Let $\phi \in \Holder{\alpha}(S^2,d)$ be an eventually positive real-valued H\"{o}lder continuous function that is not cohomologous to a constant in $\CCC( S^2 )$. Denote by $s_0$ the unique positive number with topological pressure $P(f,-s_0 \phi) = 0$.  
Then there exists $N_f\in\N$ depending only on $f$ such that for each $n\in \N$ with $n\geq N_f$, the following statements hold for $F\coloneqq f^n$ and $\Phi\coloneqq \sum_{i=0}^{n-1} \phi \circ f^i$:

\begin{enumerate}
\smallskip
\item[(i)] Both the dynamical zeta function $\zeta_{F,\,\minus \Phi} (s)$ and the dynamical Dirichlet series $\DS_{F,\,\minus \Phi,\,\deg_F} (s)$ converge on $\{s\in\C  :  \Re(s) > s_0 \}$ and extend to non-vanishing holomorphic functions on $\{s\in\C  :  \Re(s) \geq s_0\}$ except for the simple pole at $s=s_0$.

\smallskip
\item[(ii)] Assume in addition that $\phi$ satisfies the $\alpha$-strong non-integrability condition. Then there exists a constant $\epsilon_0 \in (0, s_0)$ such that both $\zeta_{F,\,\minus \Phi} (s)$ and $\DS_{F,\,\minus \Phi,\,\deg_F} (s)$ converge on $\{s\in\C  :  \Re(s) > s_0 \}$ and extend to non-vanishing holomorphic functions on $\{s\in\C  :  \Re(s) \geq s_0 - \epsilon_0 \}$ except for the simple pole at $s=s_0$. Moreover, for each $\epsilon >0$, there exist constants $C_\epsilon >0$, $a_\epsilon \in (0, \epsilon_0]$, and $b_\epsilon\geq  2 s_0+1$ such that
\begin{align}
     \exp\left( - C_\epsilon \abs{\Im(s)}^{2+\epsilon} \right) 
&\leq \abs{\zeta_{ F,\,\minus\Phi} (s) }
\leq \exp\left(   C_\epsilon \abs{\Im(s)}^{2+\epsilon} \right),  \label{eqZetaBound}\\
     \exp\left( - C_\epsilon \abs{\Im(s)}^{2+\epsilon} \right) 
&\leq \Absbig{ \DS_{ F,\,\minus\Phi,\,\deg_F} (s) }
\leq \exp\left(   C_\epsilon \abs{\Im(s)}^{2+\epsilon} \right) \label{eqWeightedZetaBound}
\end{align}
for all $s\in\C$ with $\abs{\Re(s) - s_0} < a_\epsilon$ and $\abs{\Im(s)}  \geq b_\epsilon$.
\end{enumerate}
\end{thmx}

In order to get information about $\zeta_{F,\,\minus \Phi}$, we need to investigate the zeta function $\zeta_{\sigma_{A_{\ti}},\,\minus \phi\circsmall\pi_{\ti}}$ of a symbolic model $\sigma_{A_{\ti}} \: \Sigma_{A_{\ti}}^+\rightarrow \Sigma_{A_{\ti}}^+$ of $F$.

\begin{thmx}[Holomorphic extensions of the symbolic zeta functions]  \label{thmZetaAnalExt_SFT}
Let $f\: S^2 \rightarrow S^2$ be an expanding Thurston map with a Jordan curve $\CC\subseteq S^2$ satisfying $f(\CC)\subseteq \CC$, $\post f\subseteq \CC$, and no $1$-tile in $\X^1(f,\CC)$ joins opposite sides of $\CC$. Let $d$ be a visual metric on $S^2$ for $f$. Fix $\alpha\in(0,1]$. Let $\phi \in \Holder{\alpha}(S^2,d)$ be an eventually positive real-valued H\"{o}lder continuous function that is not cohomologous to a constant in $\CCC( S^2 )$. Denote by $s_0$ the unique positive number with $P(f,-s_0 \phi) = 0$. Let $\bigl(\Sigma_{A_{\ti}}^+,\sigma_{A_{\ti}}\bigr)$ be the one-sided subshift of finite type associated to $f$ and $\CC$ defined in Proposition~\ref{propTileSFT}, and let $\pi_{\ti}\: \Sigma_{A_{\ti}}^+\rightarrow S^2$ be the factor map as defined in (\ref{eqDefTileSFTFactorMap}).

Then the dynamical zeta function $\zeta_{\sigma_{A_{\ti}},\,\minus \phi\circsmall\pi_{\ti}} (s)$ converges on the open half-plane $\{s\in\C  :  \Re(s) > s_0 \}$, and the following statements hold:
\begin{enumerate}
\smallskip
\item[(i)] The function $\zeta_{\sigma_{A_{\ti}},\,\minus \phi\circsmall\pi_{\ti}} (s)$ extends to a non-vanishing holomorphic function on the closed half-plane $\{s\in\C  :  \Re(s) \geq s_0 \}$ except for the simple pole at $s=s_0$. 

\smallskip
\item[(ii)] Assume in addition that $\phi$ satisfies the $\alpha$-strong non-integrability condition. Then there exists a constant $\wt\epsilon_0 \in (0, s_0)$ such that $\zeta_{\sigma_{A_{\ti}},\,\minus \phi\circsmall\pi_{\ti}} (s)$  extends to a non-vanishing holomorphic function on the closed half-plane $\{s\in\C  :  \Re(s) \geq s_0 - \wt\epsilon_0 \}$ except for the simple pole at $s=s_0$. Moreover, for each $\epsilon >0$, there exist constants $\wt{C}_\epsilon >0$, $\wt{a}_\epsilon \in (0,s_0)$, and $\wt{b}_\epsilon \geq 2 s_0 + 1$ such that
\begin{equation}   \label{eqZetaBound_SFT}
     \exp\bigl( - \wt{C}_\epsilon \abs{\Im(s)}^{2+\epsilon} \bigr) 
\leq \Absbig{ \zeta_{\sigma_{A_{\ti}},\,\minus \phi\circsmall\pi_{\ti}} (s) }
\leq \exp\bigl(   \wt{C}_\epsilon \abs{\Im(s)}^{2+\epsilon} \bigr) 
\end{equation}
for all $s\in\C$ with $\abs{\Re(s) - s_0} < \wt{a}_\epsilon$ and $\abs{\Im(s)}  \geq \wt{b}_\epsilon$.
\end{enumerate}
\end{thmx}

We adapt D.~Dolgopyat's cancellation argument developed in his landmark work \cite{Do98} (building in part of work of Chernov \cite{Ch98})  and arguments of M.~Pollicott and R.~Sharp \cite{PS98} to establish a symbolic version of Theorem~\ref{thmZetaAnalExt_InvC} as stated in Theorem~\ref{thmZetaAnalExt_SFT}. The difficulties in adapting D.~Dolgopyat's machinery in our metric-topological setting are purely technical, but overcoming these difficulties in any context is the heart of the matter (see for example, \cite{Na05, OW17} as well as works on the decay of correlation and counting in \cite{Liv04, AGY06, OW16, OW17, BDL18}, etc.) We use the H\"older norm in the cancellation argument instead of the $\CCC^1$-norm used in \cite{Na05, OW17}. Another major technical difficulty comes from the fact that $S^2$ is connected and the usual Ruelle operator does not apply to characteristic functions on proper subsets of $S^2$, which is essential in Ruelle's estimate (see (\ref{eqTelescoping}) in Proposition~\ref{propTelescoping}). Our approach is to adjust the definition of the Ruelle operator and to introduce what we call the \emph{split Ruelle operator} (see Section~\ref{sctSplitRuelleOp}). Such an approach should be useful in establishing Prime Orbit Theorems in other contexts.

\smallskip

We will now give a brief description of the structure of this paper.

After fixing some notation in Section~\ref{sctNotation}, we give a review of basic definitions and results in Section~\ref{sctPreliminaries}. In Section~\ref{sctAssumptions}, we state the assumptions on some of the objects in this paper, which we are going to repeatedly refer to later as \emph{the Assumptions}. 
In Section~\ref{sctSplitRuelleOp}, we define the \emph{split Ruelle operator} $\RRR_{\minus s\phi}$ and study its properties including spectral gap.
Section~\ref{sctPreDolgopyat} contains arguments to bound the dynamical zeta function $\zeta_{\sigma_{A_{\ti}},\,\minus\phi\circsmall\pi_{\ti}}$ with the bounds of the operator norm of $\RRR_{\minus s\phi}$.
We provide a proof of Theorem~\ref{thmZetaAnalExt_InvC} in Subsection~\ref{subsctDynOnC_Reduction} to deduce the holomorphic extension of $\DS_{F,\,\minus \Phi,\,\deg_F}$ from that of $\zeta_{\sigma_{A_{\ti}},\,\minus\Phi\circsmall\pi_{\ti}}$, and ultimately to deduce  the holomorphic extension of $\zeta_{ F,\,\minus\Phi}$ from that of $\DS_{F,\,\minus \Phi,\,\deg_F}$.
In Section~\ref{sctDolgopyat}, we adapt the arguments of D.~Dolgopyat \cite{Do98} in our metric-topological setting aiming to prove Theorem~\ref{thmL2Shrinking} at the end of this section, consequently establishing Theorems~\ref{thmZetaAnalExt_InvC},~\ref{thmZetaAnalExt_SFT}, and~\ref{thmPrimeOrbitTheorem}.
Section~\ref{sctLattes} focuses on Latt\`{e}s maps (recalled in Definition~\ref{defLattesMap}). We include the proof of Theorem~\ref{thmLattesPOT} at the end of this section.

\subsection*{Acknowledgments} 
The first-named author is grateful to the Institute for Computational and Experimental Research in Mathematics (ICERM) at Brown University for the hospitality during his stay from February to May 2016, where he learned about this area of research while participating in the Semester Program ``Dimension and Dynamics'' as a research postdoctoral fellow. The authors want to express their gratitude to Mark~Pollicott for his introductory discussion on dynamical zeta functions and Prime Orbit Theorems in ICERM and many conversations since then, and to Hee~Oh for her beautiful talks on her work and helpful comments, while the first-named author also wants to thank Dennis~Sullivan for his encouragement to initiating this project, and Jianyu~Chen and  Polina~Vytnova for interesting discussions on related areas of research. Part of this work was done during the first-named author's stay at the Institute for Mathematical Sciences (IMS) at Stony Brook University as a postdoctoral fellow. He wants to thank IMS and his postdoctoral advisor Mikhail~Yu.~Lyubich for the great support and hospitality. The authors also would like to thank Gengrui Zhang for carefully reading the manuscript and pointing out several typos, and the anonymous referees for valuable suggestions to improve the paper. Z.~Li was partially supported by NSF DMS-1301602, DMS-1600519, NSFC Nos.~12101017, 12090010, 12090015, and BJNSF No.~1214021.

\section{Notation} \label{sctNotation}
Let $\C$ be the complex plane and $\widehat{\C}$ be the Riemann sphere. For each complex number $z\in \C$, we denote by $\Re(z)$ the real part of $z$, and by $\Im(z)$ the imaginary part of $z$. We denote by $\D$ the open unit disk $\D \coloneqq \{z\in\C  :  \abs{z}<1 \}$ on the complex plane $\C$. For each $a\in\R$, we denote by $\H_a$ the open (right) half-plane $\H_a\coloneqq \{ z\in\C  :  \Re(z) > a\}$ on $\C$, and by $\overline{\H}_a$ the closed (right) half-plane $\overline{\H}_a \coloneqq \{ z\in\C  :  \Re(z) \geq  a\}$. We follow the convention that $\N \coloneqq \{1, \, 2, \, 3, \, \dots\}$, $\N_0 \coloneqq \{0\} \cup \N$, and $\widehat{\N} \coloneqq \N\cup \{+\infty\}$, with the order relations $<$, $\leq$, $>$, $\geq$ defined in the obvious way. For $x\in\R$, we define $\lfloor x\rfloor$ as the greatest integer $\leq x$, and $\lceil x \rceil$ the smallest integer $\geq x$. As usual, the symbol $\log$ denotes the logarithm to the base $e$, and $\log_c$ the logarithm to the base $c$ for $c>0$. The symbol $\I$ stands for the imaginary unit in the complex plane $\C$. For each $z\in\C\setminus \{0\}$, we denote by $\Arg(z)$ the principle argument of $z$, i.e., the unique real number in $(-\pi,\pi]$ with the property that $\abs{z} e^{\I \Arg(z)} = z$. The cardinality of a set $A$ is denoted by $\card{A}$. 

Consider real-valued functions $u$, $v$, and $w$ on $(0,+\infty)$. We write $u(T) \sim v(T)$ as $T\to +\infty$ if $\lim_{T\to+\infty} \frac{u(T)}{v(T)} = 1$, and write $u(T) = v(T) + \bigO ( w(T) )$ as $T\to +\infty$ if $\limsup_{T\to+\infty} \Absbig{ \frac{u(T) - v(T) }{w(T)} } < +\infty$.

Let $g\: X\rightarrow Y$ be a map between two sets $X$ and $Y$. We denote the restriction of $g$ to a subset $Z$ of $X$ by $g|_Z$. Consider a map $f\: X\rightarrow X$ on a set $X$. We write $f^n$ for the $n$-th iterate of $f$, and $f^{-n} \coloneqq \left( f^n \right)^{-1}$, for $n\in\N$. We set $f^0 \coloneqq \id_X$, where the identity map $\id_X\: X\rightarrow X$ sends each $x\in X$ to $x$ itself. For each $n\in\N$, we denote by
\begin{equation} \label{eqDefSetPeriodicPts}
P_{n,f} \coloneqq \bigl\{ x\in X  :  f^n(x)=x, \,  f^k(x)\neq x,  \, k\in\{1, \, 2, \, \dots, \, n-1\} \bigr\}
\end{equation}
the \defn{set of periodic points of $f$ with periodic $n$}, and by
\begin{equation}  \label{eqDefSetPeriodicOrbits} 
\Orb(n,f) \coloneqq \bigl\{ \bigl\{f^i(x)   :  i\in \{0, \, 1, \, \dots, \, n-1\} \bigr\}  :  x\in P_{n,f}  \bigr\}
\end{equation}
the \defn{set of primitive periodic orbits of $f$ with period $n$}. The set of all primitive periodic orbits of $f$ is denoted by
\begin{equation} \label{eqDefSetAllPeriodicOrbits} 
\Orb(f) \coloneqq \bigcup_{n=1}^{+\infty}  \Orb(n,f).
\end{equation}

Given a complex-valued function $\varphi\: X\rightarrow \C$, we write
\begin{equation}    \label{eqDefSnPt}
S_n \varphi (x)  = S_{n}^f \varphi (x)  \coloneqq \sum_{j=0}^{n-1} \varphi(f^j(x)) 
\end{equation}
for $x\in X$ and $n\in\N_0$. The superscript $f$ is often omitted when the map $f$ is clear from the context. Note that when $n=0$, by definition, we always have $S_0 \varphi = 0$.

Let $(X,d)$ be a metric space. For subsets $A,B\subseteq X$, we set $d(A,B) \coloneqq \inf \{d(x,y) :  x\in A,\,y\in B\}$, and $d(A,x)=d(x,A) \coloneqq d(A,\{x\})$ for $x\in X$. For each subset $Y\subseteq X$, we denote the diameter of $Y$ by $\diam_d(Y) \coloneqq \sup\{d(x,y) : x, \, y\in Y\}$, the interior of $Y$ by $\inter Y$, and the characteristic function of $Y$ by $\mathbbm{1}_Y$, which maps each $x\in Y$ to $1\in\R$ and vanishes otherwise. We use the convention that $\mathbbm{1}=\mathbbm{1}_X$ when the space $X$ is clear from the context. For each $r>0$ and each $x\in X$, we denote the open (resp.\ closed) ball of radius $r$ centered at $x$ by $B_d(x, r)$ (resp.\ $\overline{B_d}(x,r)$). 

We set $\CCC(X)$ (resp.\ $B(X)$) to be the space of continuous (resp.\ bounded Borel) functions from $X$ to $\R$, $\MMM(X)$ the set of finite signed Borel measures, and $\PPP(X)$ the set of Borel probability measures on $X$. We denote by $\CCC(X,\C)$ (resp.\ $B(X,\C)$) the space of continuous (resp.\ bounded Borel) functions from $X$ to $\C$. We adopt the convention that unless specifically referring to $\C$, we only consider real-valued functions. If we do not specify otherwise, we equip $\CCC(X)$ and $\CCC(X,\C)$ with the uniform norm $\Norm{\cdot}_{\CCC^0(X)}$. For a continuous map $g\: X \rightarrow X$, $\MMM(X,g)$ is the set of $g$-invariant Borel probability measures on $X$. 

The space of real-valued (resp.\ complex-valued) H\"{o}lder continuous functions with an exponent $\alpha\in (0,1]$ on a compact metric space $(X,d)$ is denoted by $\Holder{\alpha}(X,d)$ (resp.\ $\Holder{\alpha}((X,d),\C)$). For each $\psi\in\Holder{\alpha}((X,d),\C)$, we denote
\begin{equation}   \label{eqDef|.|alpha}
\Hseminorm{\alpha,\,(X,d)}{\psi} \coloneqq \sup \{ \abs{\psi(x)- \psi(y)} / d(x,y)^\alpha  :  x, \, y\in X, \,x\neq y \},
\end{equation}
and for $b\in\R\setminus\{0\}$, the \emph{normalized H\"{o}lder norm} of $\psi$ is defined as
\begin{equation}  \label{eqDefNormalizedHolderNorm}
\NHnorm{\alpha}{b}{\psi}{(X,d)} \coloneqq \abs{b}^{-1} \Hseminorm{\alpha,\,(X,d)}{\psi}  + \Norm{\psi}_{\CCC^0(X)},
\end{equation}
while the standard H\"{o}lder norm of $\psi$ is denoted by
\begin{equation}    \label{eqDefHolderNorm}
\Hnorm{\alpha}{\psi}{(X,d)} \coloneqq \NHnorm{\alpha}{1}{\psi}{(X,d)} .
\end{equation}

For a Lipschitz map $g\: (X,d)\rightarrow (X,d)$ on a metric space $(X,d)$, we denote the Lipschitz constant by
\begin{equation}   \label{eqDefLipConst}
\LIP_d(g) \coloneqq \sup  \{ d(g(x),g(y)) / d(x,y)  :  x, \, y\in X \text{ with } x\neq y \}.
\end{equation}

\section{Preliminaries}  \label{sctPreliminaries}

\subsection{Thermodynamic formalism}  \label{subsctThermodynFormalism}

We first review some basic concepts from dynamical systems. We refer the reader to \cite[Subsection~3.1]{LZ24a} for more details and references.

Let $(X,d)$ be a compact metric space and $g\:X\rightarrow X$ a continuous map. For each real-valued continuous function $\phi\in\CCC(X)$, the \defn{measure-theoretic pressure} $P_\mu(g,\phi)$ of $g$ for a $g$-invariant Borel probability measure $\mu$ and the potential $\phi$ is
\begin{equation}  \label{eqDefMeasTheoPressure}
P_\mu(g,\phi) \coloneqq  h_\mu (g) + \int \! \phi \,\mathrm{d}\mu.
\end{equation}
Here $h_\mu(g)$ denotes the usual measure-theoretic entropy of $g$ for $\mu$.

By the Variational Principle (see for example, \cite[Theorem~3.4.1]{PrU10}), we have that for each $\phi\in\CCC(X)$, the \defn{topological pressure} $P(g,\phi)$ of $g$ with respect to the potential $\phi$ satisfies
\begin{equation}  \label{eqVPPressure}
P(g,\phi)=\sup\{P_\mu(g,\phi) : \mu\in \MMM(X,g)\}.
\end{equation}
In particular, when $\phi$ is the constant function $0$, the \defn{topological entropy} $h_{\operatorname{top}}(g)$ of $g$ satisfies
\begin{equation}  \label{eqVPEntropy}
h_{\operatorname{top}}(g)=\sup\{h_{\mu}(g) : \mu\in \MMM(X,g)\}.
\end{equation}
A measure $\mu$ that attains the supremum in (\ref{eqVPPressure}) is called an \defn{equilibrium state} for the map $g$ and the potential $\phi$. A measure $\mu$ that attains the supremum in (\ref{eqVPEntropy}) is called a \defn{measure of maximal entropy} of $g$.

\begin{definition}   \label{defCohomologous}
Let $g\: X\rightarrow X$ be a continuous map on a metric space $(X,d)$. Let $\mathcal{K} \subseteq \CCC(X,\C)$ be a subspace of the space $\CCC(X,\C)$ of complex-valued continuous functions on $X$. Two functions $\phi, \, \psi \in \CCC(X,\C)$ are said to be \defn{cohomologous (in $\mathcal{K}$)} if there exists $u\in\mathcal{K}$ such that $\phi-\psi = u\circ g - u$.
\end{definition}

\smallskip

One of the main tools in the study of the existence, uniqueness, and other properties of equilibrium states is the \emph{Ruelle operator}. We will postpone the discussion of the Ruelle operators of expanding Thurston maps to Subsection~\ref{subsctThurstonMap}.

\subsection{Thurston maps} \label{subsctThurstonMap}
In this subsection, we go over some key concepts and results on Thurston maps, and expanding Thurston maps in particular. For a more thorough treatment of the subject, we refer to \cite{BM17}.

Let $S^2$ denote an oriented topological $2$-sphere. A continuous map $f\:S^2\rightarrow S^2$ is called a \defn{branched covering map} on $S^2$ if for each point $x\in S^2$, there exists a positive integer $d\in \N$, open neighborhoods $U$ of $x$ and $V$ of $y=f(x)$, open neighborhoods $U'$ and $V'$ of $0$ in $\widehat{\C}$, and orientation-preserving homeomorphisms $\varphi\:U\rightarrow U'$ and $\eta\:V\rightarrow V'$ such that $\varphi(x)=0$, $\eta(y)=0$, and
\begin{equation*}
(\eta\circ f\circ\varphi^{-1})(z)=z^d
\end{equation*}
for each $z\in U'$. The positive integer $d$ above is called the \defn{local degree} of $f$ at $x$ and is denoted by $\deg_f (x)$. 


The \defn{degree} of $f$ is
\begin{equation}   \label{eqDeg=SumLocalDegree}
\deg f=\sum_{x\in f^{-1}(y)} \deg_f (x)
\end{equation}
for $y\in S^2$ and is independent of $y$. If $f\:S^2\rightarrow S^2$ and $g\:S^2\rightarrow S^2$ are two branched covering maps on $S^2$, then so is $f\circ g$, and
\begin{equation} \label{eqLocalDegreeProduct}
 \deg_{f\circsmall g}(x) = \deg_g(x)\deg_f(g(x)), \qquad \text{for each } x\in S^2,
\end{equation}   
and moreover, 
\begin{equation}  \label{eqDegreeProduct}
\deg(f\circ g) =  (\deg f)( \deg g).
\end{equation}

A point $x\in S^2$ is a \defn{critical point} of $f$ if $\deg_f(x) \geq 2$. The set of critical points of $f$ is denoted by $\crit f$. A point $y\in S^2$ is a \defn{postcritical point} of $f$ if $y = f^n(x)$ for some $x\in\crit f$ and $n\in\N$. The set of postcritical points of $f$ is denoted by $\post f$. Note that $\post f=\post f^n$ for all $n\in\N$.

\begin{definition} [Thurston maps] \label{defThurstonMap}
A Thurston map is a branched covering map $f\:S^2\rightarrow S^2$ on $S^2$ with $\deg f\geq 2$ and $\card(\post f)<+\infty$.
\end{definition}

We now recall the notation for cell decompositions of $S^2$ used in \cite{BM17} and \cite{Li17}. A \defn{cell of dimension $n$} in $S^2$, $n \in \{1, \, 2\}$, is a subset $c\subseteq S^2$ that is homeomorphic to the closed unit ball $\overline{\B^n}$ in $\R^n$. We define the \defn{boundary of $c$}, denoted by $\partial c$, to be the set of points corresponding to $\partial\B^n$ under such a homeomorphism between $c$ and $\overline{\B^n}$. The \defn{interior of $c$} is defined to be $\inte (c) = c \setminus \partial c$. For each point $x\in S^2$, the set $\{x\}$ is considered as a \defn{cell of dimension $0$} in $S^2$. For a cell $c$ of dimension $0$, we adopt the convention that $\partial c=\emptyset$ and $\inte (c) =c$. 

We record the following three definitions from \cite{BM17}.

\begin{definition}[Cell decompositions]\label{defcelldecomp}
Let $\DD$ be a collection of cells in $S^2$.  We say that $\DD$ is a \defn{cell decomposition of $S^2$} if the following conditions are satisfied:

\begin{itemize}

\smallskip
\item[(i)]
the union of all cells in $\DD$ is equal to $S^2$,

\smallskip
\item[(ii)] if $c\in \DD$, then $\partial c$ is a union of cells in $\DD$,

\smallskip
\item[(iii)] for $c_1, \, c_2 \in \DD$ with $c_1 \neq c_2$, we have $\inte (c_1) \cap \inte (c_2)= \emptyset$,  

\smallskip
\item[(iv)] every point in $S^2$ has a neighborhood that meets only finitely many cells in $\DD$.

\end{itemize}
\end{definition}

\begin{definition}[Refinements]\label{defrefine}
Let $\DD'$ and $\DD$ be two cell decompositions of $S^2$. We
say that $\DD'$ is a \defn{refinement} of $\DD$ if the following conditions are satisfied:
\begin{itemize}

\smallskip
\item[(i)] every cell $c\in \DD$ is the union of all cells $c'\in \DD'$ with $c'\subseteq c$,

\smallskip
\item[(ii)] for every cell $c'\in \DD'$ there exits a cell $c\in \DD$ with $c'\subseteq c$.

\end{itemize}
\end{definition}

\begin{definition}[Cellular maps and cellular Markov partitions]\label{defcellular}
Let $\DD'$ and $\DD$ be two cell decompositions of  $S^2$. We say that a continuous map $f \: S^2 \rightarrow S^2$ is \defn{cellular} for  $(\DD', \DD)$ if for every cell $c\in \DD'$, the restriction $f|_c$ of $f$ to $c$ is a homeomorphism of $c$ onto a cell in $\DD$. We say that $(\DD',\DD)$ is a \defn{cellular Markov partition} for $f$ if $f$ is cellular for $(\DD',\DD)$ and $\DD'$ is a refinement of $\DD$.
\end{definition}

Let $f\:S^2 \rightarrow S^2$ be a Thurston map, and $\CC\subseteq S^2$ be a Jordan curve containing $\post f$. Then the pair $f$ and $\CC$ induces natural cell decompositions $\DD^n(f,\CC)$ of $S^2$, for $n\in\N_0$, in the following way:

By the Jordan curve theorem, the set $S^2\setminus\CC$ has two connected components. We call the closure of one of them the \defn{white $0$-tile} for $(f,\CC)$, denoted by $X^0_\w$, and the closure of the other the \defn{black $0$-tile} for $(f,\CC)$, denoted by $X^0_\b$. The set of \defn{$0$-tiles} is $\X^0(f,\CC) \coloneqq \bigl\{ X_\b^0, \, X_\w^0 \bigr\}$. The set of \defn{$0$-vertices} is $\V^0(f,\CC) \coloneqq \post f$. We set $\overline\V^0(f,\CC) \coloneqq \{ \{x\}  :  x\in \V^0(f,\CC) \}$. The set of \defn{$0$-edges} $\E^0(f,\CC)$ is the set of the closures of the connected components of $\CC \setminus  \post f$. Then we get a cell decomposition 
\begin{equation*}
\DD^0(f,\CC) \coloneqq \X^0(f,\CC) \cup \E^0(f,\CC) \cup \overline\V^0(f,\CC)
\end{equation*}
of $S^2$ consisting of \emph{cells of level $0$}, or \defn{$0$-cells}.

We can recursively define unique cell decompositions $\DD^n(f,\CC)$, $n\in\N$, consisting of \defn{$n$-cells} such that $f$ is cellular for $(\DD^{n+1}(f,\CC),\DD^n(f,\CC))$. We refer to \cite[Lemma~5.12]{BM17} for more details. We denote by $\X^n(f,\CC)$ the set of $n$-cells of dimension 2, called \defn{$n$-tiles}; by $\E^n(f,\CC)$ the set of $n$-cells of dimension 1, called \defn{$n$-edges}; by $\overline\V^n(f,\CC)$ the set of $n$-cells of dimension 0; and by $\V^n(f,\CC)$ the set $\bigl\{x :  \{x\}\in \overline\V^n(f,\CC)\bigr\}$, called the set of \defn{$n$-vertices}. The \defn{$k$-skeleton}, for $k\in\{0, \, 1, \, 2\}$, of $\DD^n(f,\CC)$ is the union of all $n$-cells of dimension $k$ in this cell decomposition. 

We record Proposition~5.16 of \cite{BM17} here in order to summarize properties of the cell decompositions $\DD^n(f,\CC)$ defined above.

\begin{prop}[M.~Bonk \& D.~Meyer \cite{BM17}] \label{propCellDecomp}
Let $k, \, n\in \N_0$, let   $f\: S^2\rightarrow S^2$ be a Thurston map,  $\CC\subseteq S^2$ be a Jordan curve with $\post f \subseteq \CC$, and   $m=\card(\post f)$. 
 
\smallskip
\begin{itemize}

\smallskip
\item[(i)] The map  $f^k$ is cellular for $\bigl( \DD^{n+k}(f,\CC), \DD^n(f,\CC) \bigr)$. In particular, if  $c$ is any $(n+k)$-cell, then $f^k(c)$ is an $n$-cell, and $f^k|_c$ is a homeomorphism of $c$ onto $f^k(c)$.

\smallskip
\item[(ii)]  Let  $c$ be  an $n$-cell.  Then $f^{-k}(c)$ is equal to the union of all 
$(n+k)$-cells $c'$ with $f^k(c')=c$.

\smallskip
\item[(iii)] The $1$-skeleton of $\DD^n(f,\CC)$ is  equal to  $f^{-n}(\CC)$. The $0$-skeleton of $\DD^n(f,\CC)$ is the set $\V^n(f,\CC)=f^{-n}(\post f )$, and we have $\V^n(f,\CC) \subseteq \V^{n+k}(f,\CC)$. 

\smallskip
\item[(iv)] $\card(\X^n(f,\CC))=2(\deg f)^n$,  $\card(\E^n(f,\CC))=m(\deg f)^n$,  and $\card (\V^n(f,\CC)) \leq m (\deg f)^n$.

\smallskip
\item[(v)] The $n$-edges are precisely the closures of the connected components of $f^{-n}(\CC)\setminus f^{-n}(\post f )$. The $n$-tiles are precisely the closures of the connected components of $S^2\setminus f^{-n}(\CC)$.

\smallskip
\item[(vi)] Every $n$-tile  is an $m$-gon, i.e., the number of $n$-edges and the number of $n$-vertices contained in its boundary are equal to $m$.  

\smallskip
\item[(vii)] Let $F\coloneqq f^k$ be an iterate of $f$ with $k \in \N$. Then $\DD^n(F,\CC) = \DD^{nk}(f,\CC)$.
\end{itemize}
\end{prop}

We also note that for each $n$-edge $e\in\E^n(f,\CC)$, $n\in\N_0$, there exist exactly two $n$-tiles $X, \, X'\in\X^n(f,\CC)$ such that $X\cap X' = e$.

For $n\in \N_0$, we define the \defn{set of black $n$-tiles} as
\begin{equation*}
\X_\b^n(f,\CC) \coloneqq \bigl\{X\in\X^n(f,\CC)  :   f^n(X)=X_\b^0 \bigr\},
\end{equation*}
and the \defn{set of white $n$-tiles} as
\begin{equation*}
\X_\w^n(f,\CC) \coloneqq \bigl\{X\in\X^n(f,\CC)  :  f^n(X)=X_\w^0 \bigr\}.
\end{equation*}
It follows immediately from Proposition~\ref{propCellDecomp} that
\begin{equation}   \label{eqCardBlackNTiles}
\card ( \X_\b^n(f,\CC) ) = \card  (\X_\w^n(f,\CC) ) = (\deg f)^n 
\end{equation}
for each $n\in\N_0$.

From now on, if the map $f$ and the Jordan curve $\CC$ are clear from the context, we will sometimes omit $(f,\CC)$ in the notation above.

We denote, for each $x\in S^2$ and $n\in\Z$,
\begin{equation}  \label{defU^n}
U^n(x) \coloneqq \bigcup \{Y^n\in \X^n  :     \text{there exists } X^n\in\X^n  
                                        \text{ with } x\in X^n, \, X^n\cap Y^n \neq \emptyset  \}  
\end{equation}
if $n\geq 0$, and set $U^n(x) \coloneqq S^2$ otherwise. 

We can now recall a definition of expanding Thurston maps.

\begin{definition} [Expansion] \label{defExpanding}
A Thurston map $f\:S^2\rightarrow S^2$ is called \defn{expanding} if there exists a metric $d$ on $S^2$ that induces the standard topology on $S^2$ and a Jordan curve $\CC\subseteq S^2$ containing $\post f$ such that 
\begin{equation*}
\lim_{n\to+\infty}\max \{\diam_d(X)  :  X\in \X^n(f,\CC)\}=0.
\end{equation*}
\end{definition}

P.~Ha\"{\i}ssinsky and K.~M.~Pilgrim developed a notion of expansion in a more general context for finite branched coverings between topological spaces (see \cite[Sections~2.1 and~2.2]{HP09}). This applies to Thurston maps, and their notion of expansion is equivalent to our notion defined above in the context of Thurston maps (see \cite[Proposition~6.4]{BM17}). Our notion of expansion is not equivalent to classical notions such as forward-expansive maps or distance-expanding maps. One topological obstruction comes from the presence of critical points for (non-homeomorphic) branched covering maps on $S^2$.

For an expanding Thurston map $f$, we can fix a particular metric $d$ on $S^2$ called a \emph{visual metric for $f$}. For the existence and properties of such metrics, see \cite[Chapter~8]{BM17}. For a visual metric $d$ for $f$, there exists a unique constant $\Lambda > 1$ called the \emph{expansion factor} of $d$ (see \cite[Chapter~8]{BM17} for more details). One major advantage of a visual metric $d$ is that in $(S^2,d)$, we have good quantitative control over the sizes of the cells in the cell decompositions discussed above. We summarize several results of this type (\cite[Proposition~8.4, Lemmas~8.10,~8.11]{BM17}) in the lemma below.

\begin{lemma}[M.~Bonk \& D.~Meyer \cite{BM17}]   \label{lmCellBoundsBM}
Let $f\:S^2 \rightarrow S^2$ be an expanding Thurston map, and $\CC \subseteq S^2$ be a Jordan curve containing $\post f$. Let $d$ be a visual metric on $S^2$ for $f$ with expansion factor $\Lambda>1$. Then there exist constants $C\geq 1$, $K\geq 1$, and $n_0\in\N_0$ with the following properties:
\begin{enumerate}
\smallskip
\item[(i)] $d(\sigma,\tau) \geq C^{-1} \Lambda^{-n}$ whenever $\sigma$ and $\tau$ are disjoint $n$-cells for $n\in \N_0$.

\smallskip
\item[(ii)] $C^{-1} \Lambda^{-n} \leq \diam_d(\tau) \leq C\Lambda^{-n}$ for all $n$-edges and all $n$-tiles $\tau$ for $n\in\N_0$.

\smallskip
\item[(iii)] $B_d(x,K^{-1} \Lambda^{-n} ) \subseteq U^n(x) \subseteq B_d(x, K\Lambda^{-n})$ for $x\in S^2$ and $n\in\N_0$.

\smallskip
\item[(iv)] $U^{n+n_0} (x)\subseteq B_d(x,r) \subseteq U^{n-n_0}(x)$ where $n= \lceil -\log r / \log \Lambda \rceil$ for $r>0$ and $x\in S^2$.

\smallskip
\item[(v)] For every $n$-tile $X^n\in\X^n(f,\CC)$, $n\in\N_0$, there exists a point $p\in X^n$ such that $B_d(p,C^{-1}\Lambda^{-n}) \subseteq X^n \subseteq B_d(p,C\Lambda^{-n})$.
\end{enumerate}

Conversely, if $\wt{d}$ is a metric on $S^2$ satisfying conditions \textnormal{(i)} and \textnormal{(ii)} for some constant $C\geq 1$, then $\wt{d}$ is a visual metric with expansion factor $\Lambda>1$.
\end{lemma}

Recall that $U^n(x)$ is defined in (\ref{defU^n}).

In addition, we will need the fact that a visual metric $d$ induces the standard topology on $S^2$ (\cite[Proposition~8.3]{BM17}) and the fact that the metric space $(S^2,d)$ is linearly locally connected (\cite[Proposition~18.5]{BM17}). A metric space $(X,d)$ is \defn{linearly locally connected} if there exists a constant $L\geq 1$ such that the following conditions are satisfied:
\begin{enumerate}
\smallskip

\item  For all $z\in X$, $r > 0$, and $x, \, y\in B_d(z,r)$ with $x\neq y$, there exists a continuum $E\subseteq X$ with $x, \, y\in E$ and $E\subseteq B_d(z,rL)$.

\smallskip

\item For all $z\in X$, $r > 0$, and $x, \, y\in X \setminus B_d(z,r)$ with $x\neq y$, there exists a continuum $E\subseteq X$ with $x, \, y\in E$ and $E\subseteq X \setminus B_d(z,r/L)$.
\end{enumerate}
We call such a constant $L \geq 1$ a \defn{linear local connectivity constant of $d$}. 

In fact, visual metrics play a crucial role in connecting the dynamical arguments with geometric properties for rational expanding Thurston maps, especially Latt\`{e}s maps.

We first recall the following notions of equivalence between metric spaces.

\begin{definition}  \label{defQuasiSymmetry}
Consider two metric spaces $(X_1,d_1)$ and $(X_2,d_2)$. Let $g\: X_1 \rightarrow X_2$ be a homeomorphism. Then
\begin{enumerate}
\smallskip
\item[(i)] $g$ is \defn{bi-Lipschitz} if there exists a constant $C \geq 1$ such that for all $u,\, v \in X_1$,
\begin{equation*}
 C^{-1} d_1( u, v ) \leq d_2 ( g(u), g(v) ) \leq C d_1 ( u, v ).
\end{equation*}

\smallskip
\item[(ii)] $g$ is a \defn{snowflake homeomorphism} if there exist constants $\alpha >0$ and $C \geq 1$ such that for all $u,\, v \in X_1$,
\begin{equation*}
 C^{-1} d_1( u, v )^\alpha \leq d_2 ( g(u), g(v) ) \leq C d_1 ( u, v )^\alpha.
\end{equation*}

\smallskip
\item[(iii)] $g$ is a \defn{quasisymmetric homeomorphism} or a \defn{quasisymmetry} if there exists a homeomorphism $\eta \: [0, +\infty) \rightarrow [0, +\infty)$ such that for all $u,\,v,\,w \in X_1$,
\begin{equation*}
\frac{ d_2 ( g(u), g(v) ) }{ d_2 ( g(u), g(w) ) } 
\leq \eta \biggl(  \frac{ d_1 ( u, v ) }{ d_1 ( u, w ) }  \biggr).
\end{equation*}
\end{enumerate} 
Moreover, the metric spaces $(X_1,d_1)$ and $(X_2,d_2)$ are \emph{bi-Lipschitz}, \emph{snowflake}, or \defn{quasisymmetrically equivalent} if there exists a homeomorphism from $(X_1,d_1)$ to $(X_2,d_2)$ with the corresponding property.

When $X_1 = X_2 \eqqcolon X$, then we say the metrics $d_1$ and $d_2$ are \emph{bi-Lipschitz}, \emph{snowflake}, or \defn{quasisymmetrically equivalent} if the identity map from $(X,d_1)$ to $(X,d_2)$ has the corresponding property.
\end{definition}

\begin{rem}   \label{rmChordalVisualQSEquiv}
If $f\: \widehat\C \rightarrow \widehat\C$ is a rational expanding Thurston map (or equivalently, a postcritically-finite rational map without periodic critical points (see \cite[Proposition~2.3]{BM17})), then each visual metric is quasisymmetrically equivalent to the chordal metric on the Riemann sphere $\widehat\C$ (see \cite[Lemma~18.10]{BM17}). Here the chordal metric $\sigma$ on $\widehat\C$ is given by
$
\sigma(z,w) =\frac{2\abs{z-w}}{\sqrt{1+\abs{z}^2} \sqrt{1+\abs{w}^2}}
$
for $z, \, w\in\C$, and $\sigma(\infty,z)=\sigma(z,\infty)= \frac{2}{\sqrt{1+\abs{z}^2}}$ for $z\in \C$. We also note that quasisymmetric embeddings of bounded connected metric spaces are H\"{o}lder continuous (see \cite[Section~11.1 and Corollary~11.5]{He01}). Accordingly, the class of H\"{o}lder continuous functions on $\widehat\C$ equipped with the chordal metric and that on $S^2=\widehat\C$ equipped with any visual metric for $f$ are the same (up to a change of the H\"{o}lder exponent).
\end{rem}

A Jordan curve $\CC\subseteq S^2$ is \defn{$f$-invariant} if $f(\CC)\subseteq \CC$. We are interested in $f$-invariant Jordan curves that contain $\post f$, since for such a Jordan curve $\CC$, we get a cellular Markov partition $(\DD^1(f,\CC),\DD^0(f,\CC))$ for $f$. According to Example~15.11 in \cite{BM17}, such $f$-invariant Jordan curves containing $\post{f}$ need not exist. However, M.~Bonk and D.~Meyer \cite[Theorem~15.1]{BM17} proved that there exists an $f^n$-invariant Jordan curve $\CC$ containing $\post{f}$ for each sufficiently large $n$ depending on $f$. A slightly stronger version of this result was proved in \cite[Lemma~3.11]{Li16}, and we record it below.

\begin{lemma}[M.~Bonk \& D.~Meyer \cite{BM17}, Z.~Li \cite{Li16}]  \label{lmCexistsL}
Let $f\:S^2\rightarrow S^2$ be an expanding Thurston map, and $\wt{\CC}\subseteq S^2$ be a Jordan curve with $\post f\subseteq \wt{\CC}$. Then there exists an integer $N(f,\wt{\CC}) \in \N$ such that for each $n\geq N(f,\wt{\CC})$ there exists an $f^n$-invariant Jordan curve $\CC$ isotopic to $\wt{\CC}$ rel.\ $\post f$ such that no $n$-tile in $\X^n(f,\CC)$ joins opposite sides of $\CC$.
\end{lemma}

The phrase ``joining opposite sides'' has a specific meaning in our context. 

\begin{definition}[Joining opposite sides]  \label{defJoinOppositeSides} 
Fix a Thurston map $f$ with $\card(\post f) \geq 3$ and an $f$-invariant Jordan curve $\CC$ containing $\post f$.  A set $K\subseteq S^2$ \defn{joins opposite sides} of $\CC$ if $K$ meets two disjoint $0$-edges when $\card( \post f)\geq 4$, or $K$ meets  all  three $0$-edges when $\card(\post f)=3$. 
 \end{definition}
 
Note that $\card (\post f) \geq 3$ for each expanding Thurston map $f$ \cite[Lemma~6.1]{BM17}.

The following lemma proved in \cite[Lemma~3.13]{Li18} generalizes \cite[Lemma~15.25]{BM17}.

\begin{lemma}[M.~Bonk \& D.~Meyer \cite{BM17}, Z.~Li \cite{Li18}]   \label{lmMetricDistortion}
Let $f\:S^2 \rightarrow S^2$ be an expanding Thurston map, and $\CC \subseteq S^2$ be a Jordan curve that satisfies $\post f \subseteq \CC$ and $f^{n_\CC}(\CC)\subseteq\CC$ for some $n_\CC\in\N$. Let $d$ be a visual metric on $S^2$ for $f$ with expansion factor $\Lambda>1$. Then there exists a constant $C_0 > 1$, depending only on $f$, $d$, $\CC$, and $n_\CC$, with the following property:

If $k, \, n\in\N_0$, $X^{n+k}\in\X^{n+k}(f,\CC)$, and $x, \, y\in X^{n+k}$, then 
\begin{equation}   \label{eqMetricDistortion}
C_0^{-1} d(x,y) \leq \Lambda^{-n}  d(f^n(x),f^n(y)) \leq C_0 d(x,y).
\end{equation}
\end{lemma}

We summarize the existence, uniqueness, and some basic properties of equilibrium states for expanding Thurston maps in the following theorem.

\begin{theorem}[Z.~Li \cite{Li18}]   \label{thmEquilibriumState}
Let $f\:S^2 \rightarrow S^2$ be an expanding Thurston map and $d$ be a visual metric on $S^2$ for $f$. Let $\phi, \, \gamma\in \Holder{\alpha}(S^2,d)$ be real-valued H\"{o}lder continuous functions with an exponent $\alpha\in(0,1]$. Then the following statements hold:
\begin{enumerate}
\smallskip
\item[(i)] There exists a unique equilibrium state $\mu_\phi$ for the map $f$ and the potential $\phi$.

\smallskip
\item[(ii)] For each $t_0\in\R$, we have
$
\frac{\mathrm{d}}{\mathrm{d}t} P(f,\phi + t\gamma)|_{t=t_0}  = \int \!\gamma \,\mathrm{d}\mu_{\phi + t_0 \gamma}.
$

\smallskip
\item[(iii)] If $\CC \subseteq S^2$ is a Jordan curve containing $\post f$ with the property that $f^{n_\CC}(\CC)\subseteq \CC$ for some $n_\CC\in\N$, then
$
\mu_\phi \bigl( \, \bigcup_{i=0}^{+\infty}  f^{-i}(\CC)  \bigr)  = 0.
$
\end{enumerate}
\end{theorem}

Theorem~\ref{thmEquilibriumState}~(i) is part of \cite[Theorem~1.1]{Li18}. Theorem~\ref{thmEquilibriumState}~(ii) follows immediately from \cite[Theorem~6.13]{Li18} and the uniqueness of equilibrium states in Theorem~\ref{thmEquilibriumState}~(i). Theorem~\ref{thmEquilibriumState}~(iii) was established in \cite[Proposition~7.1]{Li18}.

The following distortion lemma serves as the cornerstone in the development of thermodynamic formalism for expanding Thurston maps in \cite{Li18} (see \cite[Lemmas~5.1 and~5.2]{Li18}).

\begin{lemma}[Z.~Li \cite{Li18}]    \label{lmSnPhiBound}
Let $f\:S^2 \rightarrow S^2$ be an expanding Thurston map and $\CC \subseteq S^2$ be a Jordan curve containing $\post f$ with the property that $f^{n_\CC}(\CC)\subseteq \CC$ for some $n_\CC\in\N$. Let $d$ be a visual metric on $S^2$ for $f$ with expansion factor $\Lambda>1$ and a linear local connectivity constant $L\geq 1$. Fix $\phi\in \Holder{\alpha}(S^2,d)$ with $\alpha\in(0,1]$. Then there exist  constants $C_1=C_1(f,\CC,d,\phi,\alpha)$ and $C_2=C_2(f,\CC,d,\phi,\alpha)  \geq 1$ depending only on $f$, $\CC$, $d$, $\phi$, and $\alpha$ such that
\begin{align} 
\Abs{S_n\phi(x)-S_n\phi(y)}  & \leq C_1 d(f^n(x),f^n(y))^\alpha, \label{eqSnPhiBound} \\
\frac{\sum_{z'\in f^{-n}(z)}  \deg_{f^n}(z')  \exp (S_n\phi(z'))}{\sum_{w'\in f^{-n}(w)}  \deg_{f^n}(w')  \exp (S_n\phi(w'))} 
                                                  & \leq \exp\left(4C_1 Ld(z,w)^\alpha\right) \leq C_2,\label{eqSigmaExpSnPhiBound}
\end{align}
for $n, \, m\in\N_0$ with $n\leq m $, $X^m\in\X^m(f,\CC)$, $x, \, y\in X^m$, and $z, \, w\in S^2$. We can choose
\begin{equation}   \label{eqC1C2}
C_1 \coloneqq \Hseminorm{\alpha,\, (S^2,d)}{\phi} C_0 (1-\Lambda^{-\alpha})^{-1}   \qquad \text{and} \qquad 
C_2 \coloneqq \exp\bigl(4C_1 L \bigl(\diam_d(S^2)\bigr)^\alpha \bigr) 
\end{equation}
where $C_0 > 1$ is the constant depending only on $f$, $\CC$, and $d$ from \cite[Lemma~3.13]{Li18}.
\end{lemma}

Recall that the main tool used in \cite{Li18} to develop the thermodynamic formalism for expanding Thurston maps is the Ruelle operator. We will need a complex version of the Ruelle operator in this paper discussed in \cite{Li17}. We summarize relevant definitions and facts about the Ruelle operator below and refer the reader to \cite[Section~3.3]{Li17} for a detailed discussion.

Let $f\: S^2\rightarrow S^2$ be an expanding Thurston map and $\phi\in \CCC(S^2,\C)$ be a complex-valued continuous function.  The \defn{Ruelle operator} $\RR_\phi$ (associated to $f$ and $\phi$) acting on $\CCC(S^2,\C)$ is defined as follows
\begin{equation}   \label{eqDefRuelleOp}
\RR_\phi(u)(x) \coloneqq \sum_{y\in f^{-1}(x)}  \deg_f(y) u(y) \exp(\phi(y)),
\end{equation}
for each $u\in \CCC(S^2,\C)$. Note that $\RR_\phi$ is a well-defined and continuous operator on $\CCC(S^2,\C)$. The Ruelle operator $\RR_\phi \: \CCC(S^2,\C) \rightarrow \CCC(S^2,\C)$ has an extension to the space of complex-valued bounded Borel functions $B(S^2,\C)$ (equipped with the uniform norm) given by (\ref{eqDefRuelleOp}) for each $u\in B(S^2,\C)$. 

We observe that if $\phi\in\CCC(S^2)$ is real-valued, then $\RR_\phi ( \CCC(S^2)) \subseteq \CCC(S^2)$ and $\RR_\phi ( B(S^2)) \subseteq B(S^2)$. The adjoint operator $\RR_\phi^*\: \CCC^*(S^2)\rightarrow \CCC^*(S^2)$ of $\RR_\phi$ acts on the dual space $\CCC^*(S^2)$ of the Banach space $\CCC(S^2)$. We identify $\CCC^*(S^2)$ with the space $\MMM(S^2)$ of finite signed Borel measures on $S^2$ by the Riesz representation theorem.

When $\phi\in\CCC(S^2)$ is real-valued, we denote 
\begin{equation}   \label{eqDefPhi-}
\overline\phi \coloneqq \phi - P(f,\phi).
\end{equation}

We record the following three technical results on the Ruelle operators in our context.

\begin{lemma}[Z.~Li \cite{Li18}]    \label{lmR1properties}
Let $f\:S^2 \rightarrow S^2$ be an expanding Thurston map and $\CC \subseteq S^2$ be a Jordan curve containing $\post f$ with the property that $f^{n_\CC}(\CC)\subseteq \CC$ for some $n_\CC\in\N$. Let $d$ be a visual metric on $S^2$ for $f$ with expansion factor $\Lambda>1$. Let $\phi\in \Holder{\alpha}(S^2,d)$ be a real-valued H\"{o}lder continuous function with an exponent $\alpha\in(0,1]$. Then there exists a constant $C_3=C_3(f, \CC, d, \phi, \alpha)$ depending only on $f$, $\CC$, $d$, $\phi$, and $\alpha$ such that for each $x, \, y\in S^2$ and each $n\in \N_0$ the following equations hold
\begin{align} 
 \RR_{\overline{\phi}}^n(\mathbbm{1})(x) / \RR_{\overline{\phi}}^n(\mathbbm{1})(y) & \leq \exp\left(4C_1 Ld(x,y)^\alpha\right) \leq C_2, \label{eqR1Quot} \\
 C_2^{-1} & \leq \RR_{\overline{\phi}}^n(\mathbbm{1})(x)  \leq C_2, \label{eqR1Bound} \\
 \Absbig{\RR_{\overline{\phi}}^n(\mathbbm{1})(x) - \RR_{\overline{\phi}}^n(\mathbbm{1})(y) }  
& \leq  C_2 ( \exp(4C_1 Ld(x,y)^\alpha ) - 1 )   \leq C_3 d(x,y)^\alpha,     \label{eqR1Diff}
\end{align}
where $C_1, C_2$ are the constants in Lemma~\ref{lmSnPhiBound} depending only on $f$, $\CC$, $d$, $\phi$, and $\alpha$.
\end{lemma}

Lemma~\ref{lmR1properties} was proved in \cite[Lemma~5.15]{Li18}. The next theorem is part of \cite[Theorem~5.16]{Li18}.

\begin{theorem}[Z.~Li \cite{Li18}]   \label{thmMuExist}
Let $f\:S^2 \rightarrow S^2$ be an expanding Thurston map and $\CC \subseteq S^2$ be a Jordan curve containing $\post f$ with the property that $f^{n_\CC}(\CC)\subseteq \CC$ for some $n_\CC\in\N$. Let $d$ be a visual metric on $S^2$ for $f$ with expansion factor $\Lambda>1$. Let $\phi\in \Holder{\alpha}(S^2,d)$ be a real-valued H\"{o}lder continuous function with an exponent $\alpha\in(0,1]$. Then the sequence $\bigl\{\frac{1}{n}\sum_{j=0}^{n-1} \RR_{\overline{\phi}}^j(\mathbbm{1})\bigr\}_{n\in\N}$ converges uniformly to a function $u_\phi \in \Holder{\alpha}(S^2,d)$, which satisfies
$\RR_{\overline{\phi}}(u_\phi)  = u_\phi$ and $C_2^{-1} \leq u_\phi(x) \leq C_2$ for each $x\in S^2$, 
where $C_2\geq 1$ is the constant from Lemma~\ref{lmSnPhiBound}.
\end{theorem}

Let $f\:S^2 \rightarrow S^2$ be an expanding Thurston map and $d$ be a visual metric on $S^2$ for $f$ with expansion factor $\Lambda>1$. Let $\phi\in \Holder{\alpha}(S^2,d)$ be a real-valued H\"{o}lder continuous function with an exponent $\alpha\in(0,1]$. Then we denote
\begin{equation}    \label{eqDefPhiWidetilde}
\wt{\phi} \coloneqq \phi - P(f,\phi) + \log u_\phi - \log ( u_\phi \circ f ),
\end{equation}
where $u_\phi$ is the continuous function given by Theorem~\ref{thmMuExist}.

\begin{lemma}     \label{lmUphiCountinuous}
Let $f\:S^2 \rightarrow S^2$ be an expanding Thurston map and $d$ be a visual metric on $S^2$ for $f$ with expansion factor $\Lambda>1$. Let $\phi\in \Holder{\alpha}(S^2,d)$ be a real-valued H\"{o}lder continuous function with an exponent $\alpha\in(0,1]$. We define a map $\tau \: \R\rightarrow \Holder{\alpha}(S^2,d)$ by setting $\tau(t) \coloneqq u_{t\phi}$. Then $\tau$ is continuous with respect to the uniform norm $\norm{\cdot}_{\CCC^0(S^2)}$ on $\Holder{\alpha}(S^2,d)$.
\end{lemma}

\begin{proof}
Fix an arbitrary bounded open interval $I \subseteq \R$. For each $n\in\N$, define $T_n \: I \rightarrow \CCC(S^2,d)$ by $T_n(t) \coloneqq \RR^n_{\overline{t\phi}}  ( \mathbbm{1}_{S^2} )$ for $t \in I$. Since $\overline{t\phi}  = t\phi - P(f, t\phi)$, by (\ref{eqDefRuelleOp}) and the continuity of the topological pressure (see for example, \cite[Theorem~3.6.1]{PrU10}), we know that $T_n$ is a continuous function with respect to the uniform norm $\norm{\cdot}_{\CCC^0(S^2)}$ on  $\CCC(S^2,d)$. Applying \cite[Theorem~6.8 and Corollary~6.10]{Li18}, we get that $T_n(t)$ converges to $\tau|_I (t)$ in the uniform norm on  $\CCC(S^2,d)$ uniformly in $t\in I$ as $n\to +\infty$. Hence, $\tau(t)$ is continuous on $I$. Recall $u_{t\phi} \in \Holder{\alpha}(S^2,d)$ (see Theorem~\ref{thmMuExist}). Therefore, $\tau(t)$ is continuous in $t\in\R$ with respect to the uniform norm on $\Holder{\alpha}(S^2,d)$.
\end{proof}

A measure $\mu\in\PPP(S^2)$ is an \defn{eigenmeasure} of $\RR_\phi^*$ if $\RR_\phi^* (\mu) = c\mu$ for some $c\in\R$. See \cite[Corollary~6.10]{Li18} for the uniqueness of such a measure. 

The potentials that satisfy the following property are of particular interest in the considerations of Prime Orbit Theorems and the analytic study of dynamical zeta functions.

\begin{definition}[Eventually positive functions]  \label{defEventuallyPositive}
Let $g\: X\rightarrow X$ be a map on a set $X$, and $\varphi\:X\rightarrow\C$ be a complex-valued function on $X$. Then $\varphi$ is \defn{eventually positive} if there exists $N\in\N$ such that $S_n\varphi(x)>0$ for each $x\in X$ and each $n\in\N$ with $n\geq N$.
\end{definition}

Theorem~\ref{thmEquilibriumState}~(ii) leads to the following corollary that we frequently use, often implicitly, throughout this paper. See \cite[Corollary~3.29]{LZ24a} for a proof.

\begin{cor}   \label{corS0unique}
Let $f\:S^2 \rightarrow S^2$ be an expanding Thurston map, and $d$ be a visual metric on $S^2$ for $f$. Let $\phi \in \Holder{\alpha}(S^2,d)$ be an eventually positive real-valued H\"{o}lder continuous function with an exponent $\alpha\in(0,1]$. Then the function $t\mapsto P(f,-t\phi)$, $t\in\R$, is strictly decreasing and there exists a unique number $s_0 \in\R$ such that $P(f,-s_0\phi)=0$. Moreover, $s_0>0$.
\end{cor}

\subsection{Subshifts of finite type}   \label{subsctSFT}

We give a brief review of the dynamics of one-sided subshifts of finite type in this subsection. We refer the reader to \cite{Ki98} for a beautiful introduction to symbolic dynamics. For a discussion on results on subshifts of finite type in our context, see \cite{PP90, Ba00}.

Let $S$ be a finite nonempty set, and $A \: S\times S \rightarrow \{0, \, 1\}$ be a matrix whose entries are either $0$ or $1$. We denote the \defn{set of admissible sequences defined by $A$} by
\begin{equation*}
\Sigma_A^+ \coloneqq \{ \{x_i\}_{i\in\N_0}  :  x_i \in S, \, A(x_i,x_{i+1})=1  \text{ for each } i\in\N_0\}.
\end{equation*}
Given $\theta\in(0,1)$, we equip the set $\Sigma_A^+$ with a metric $d_\theta$ given by $d_\theta(\{x_i\}_{i\in\N_0},\{y_i\}_{i\in\N_0})=\theta^N$ for $\{x_i\}_{i\in\N_0} \neq \{y_i\}_{i\in\N_0}$, where $N$ is the smallest integer with $x_N \neq y_N$. The topology on the metric space $\left(\Sigma_A^+,d_\theta \right)$ coincides with that induced from the product topology, and is therefore compact.

The \defn{left-shift operator} $\sigma_A \: \Sigma_A^+ \rightarrow \Sigma_A^+$ (defined by $A$) is given by
\begin{equation*}
\sigma_A ( \{x_i\}_{i\in\N_0} ) \coloneqq \{x_{i+1}\}_{i\in\N_0}  \qquad \text{for } \{x_i\}_{i\in\N_0} \in \Sigma_A^+.
\end{equation*}

The pair $\left(\Sigma_A^+, \sigma_A\right)$ is called the \defn{one-sided subshift of finite type} defined by $A$. The set $S$ is called the \defn{set of states} and the matrix $A\: S\times S \rightarrow \{0, \, 1\}$ is called the \defn{transition matrix}.

We say that a one-sided subshift of finite type $\bigl(\Sigma_A^+, \sigma_A \bigr)$ is \defn{topologically mixing} if there exists $N\in\N$ such that $A^n(x,y)>0$ for each $n\geq N$ and each pair of $x, \, y\in S$.

Let $X$ and $Y$ be topological spaces, and $f\:X\rightarrow X$ and $g\:Y\rightarrow Y$ be continuous maps. We say that the topological dynamical system $(X,f)$ is a \defn{factor} of the topological dynamical system $(Y,g)$ if there is a surjective continuous map $\pi\:Y\rightarrow X$ such that $\pi\circ g=f\circ\pi$. We call the map $\pi\: Y\rightarrow X$ a \emph{factor map}.

The following proposition is established in \cite[Proposition~3.31]{LZ24a}.

\begin{prop}   \label{propTileSFT}
Let $f\: S^2 \rightarrow S^2$ be an expanding Thurston map with a Jordan curve $\CC\subseteq S^2$ satisfying $f(\CC)\subseteq \CC$ and $\post f\subseteq \CC$. Let $d$ be a visual metric on $S^2$ for $f$ with expansion factor $\Lambda>1$. Fix $\theta\in(0,1)$. We set $S_{\ti} \coloneqq \X^1(f,\CC)$, and define a transition matrix $A_{\ti}\: S_{\ti}\times S_{\ti} \rightarrow \{0, \, 1\}$ by
\begin{equation*}
A_{\ti}(X,X') \coloneqq \begin{cases} 1 & \text{if } f(X)\supseteq X', \\ 0  & \text{otherwise}  \end{cases}
\end{equation*}
for $X, \, X'\in \X^1(f,\CC)$. Then $f$ is a factor of the one-sided subshift of finite type $\bigl(\Sigma_{A_{\ti}}^+, \sigma_{A_{\ti}}\bigr)$ defined by the transition matrix $A_{\ti}$ with a surjective and H\"{o}lder continuous factor map $\pi_{\ti}\: \Sigma_{A_{\ti}}^+ \rightarrow S^2$ given by 
\begin{equation}   \label{eqDefTileSFTFactorMap}
\pi_{\ti} \left( \{X_i\}_{i\in\N_0} \right) \coloneqq x,  \text{ where } \{x\} = \bigcap_{i \in \N_0} f^{-i} (X_i).
\end{equation}
Here $\Sigma_{A_{\ti}}^+$ is equipped with the metric $d_\theta$ defined in Subsection~\ref{subsctSFT}, and $S^2$ is equipped with the visual metric $d$.

Moreover, $\bigl(\Sigma_{A_{\ti}}^+, \sigma_{A_{\ti}} \bigr)$ is topologically mixing and $\pi_{\ti}$ is injective on $\pi_{\ti}^{-1} \bigl( S^2 \setminus \bigcup_{i\in\N_0} f^{-i}(\CC) \bigr)$.
\end{prop}

\begin{rem}
We can show that if $f$ has no periodic critical points, then $\pi$ is uniformly bounded-to-one (i.e., there exists $N\in\N_0$ depending only on $f$ such that $\card \left(\pi_{\ti}^{-1}(x)\right) \leq N$ for each $x\in S^2$); if $f$ has at least one periodic critical point, then $\pi_{\ti}$ is uncountable-to-one on a dense set. We will not use this fact in this paper.
\end{rem}

\subsection{Dynamical zeta functions and Dirichlet series}   \label{subsctDynZetaFn}

Let $g\: X\rightarrow X$ be a map on a topological space $X$. Let $\psi \: X \rightarrow \C$ be a complex-valued function on $X$. We write
\begin{equation}  \label{eqDefZn}
Z_{g,\,\minus\psi}^{(n)} (s) \coloneqq  \sum_{x\in P_{1,g^n}} e^{-s S_n \psi(x)}  , \qquad n\in\N \text{ and }  s\in\C.
\end{equation}
Recall that $P_{1,g^n}$ defined in (\ref{eqDefSetPeriodicPts}) is the set of fixed points of $g^n$, and $S_n\psi$ is defined in (\ref{eqDefSnPt}). We denote by the formal infinite product
\begin{equation}  \label{eqDefZetaFn}
\zeta_{g,\,\minus\psi} (s) \coloneqq \exp \Biggl( \sum_{n=1}^{+\infty} \frac{Z_{g,\,\minus\psi}^{(n)} (s)}{n}  \Biggr) 
 = \exp \biggl( \sum_{n=1}^{+\infty} \frac{1}{n} \sum_{x\in P_{1,g^n}} e^{-s S_n \psi(x)} \biggr), \qquad s\in\C,
\end{equation}
the \defn{dynamical zeta function} for the map $g$ and the potential $\psi$. More generally, we can define dynamical Dirichlet series as analogs of Dirichlet series in analytic number theory.

\begin{definition}   \label{defDynDirichletSeries}
Let $g\: X\rightarrow X$ be a map on a topological space $X$. Let $\psi \: X \rightarrow \C$ and $w\: X\rightarrow \C$ be complex-valued functions on $X$. We denote by the formal infinite product
\begin{equation}  \label{eqDefDynDirichletSeries}
\DS_{g,\,\minus\psi,\,w} (s) \coloneqq \exp \biggl( \sum_{n=1}^{+\infty} \frac{1}{n} \sum_{x\in P_{1,g^n}} e^{-s S_n \psi(x)} \prod_{i=0}^{n-1} w \bigl( g^i(x) \bigr) \biggr), \qquad s\in\C,
\end{equation}
the \defn{dynamical Dirichlet series} with coefficient $w$ for the map $g$ and the potential $\psi$.

\end{definition}

The following result is obtained in \cite[Proposition~5.5]{LZ24a}.

\begin{prop}   \label{propZetaFnConv_s0}
Let $f\: S^2 \rightarrow S^2$ be an expanding Thurston map with a Jordan curve $\CC\subseteq S^2$ satisfying $f(\CC) \subseteq \CC$ and $\post f \subseteq \CC$. Let $d$ be a visual metric on $S^2$ for $f$ with expansion factor $\Lambda>1$. Let $\phi \in \Holder{\alpha}(S^2,d)$ be an eventually positive real-valued H\"{o}lder continuous function with an exponent $\alpha\in (0,1]$. Denote by $s_0$ the unique positive number with $P(f,-s_0 \phi) = 0$. Let $\bigl(\Sigma_{A_{\ti}}^+,\sigma_{A_{\ti}}\bigr)$ be the one-sided subshift of finite type associated to $f$ and $\CC$ defined in Proposition~\ref{propTileSFT}, and let $\pi_{\ti}\: \Sigma_{A_{\ti}}^+\rightarrow S^2$ be the factor map defined in (\ref{eqDefTileSFTFactorMap}). Denote by $\deg_f(\cdot)$ the local degree of $f$. Then the following statements hold:
\begin{enumerate}
\smallskip
\item[(i)] $P(\sigma_{A_{\ti}}, \varphi\circ\pi_{\ti}) = P(f,\varphi)$ for each $\varphi \in \Holder{\alpha}(S^2,d)$ . In particular, for an arbitrary number $t\in\R$, we have $P(\sigma_{A_{\ti}}, - t \phi \circ \pi_{\ti}) = 0$ if and only if $t=s_0$.

\smallskip
\item[(ii)] All three infinite products $\zeta_{f,\,\minus \phi}$, $\zeta_{\sigma_{A_{\ti}},\,\minus\phi\circsmall\pi_{\ti}}$, and $\DS_{f,\,\minus\phi,\,\deg_f}$ converge uniformly and absolutely to respective non-vanishing continuous functions on $\overline{\H}_a=\{ s\in\C  :  \Re(s) \geq a \}$ that are holomorphic on $\H_a=\{ s\in\C  :  \Re(s) > a \}$, for each $a\in\R$ satisfies $a>s_0$.

\smallskip
\item[(iii)] For all $s\in \C$ with $\Re(s)>s_0$, we have
\begin{align}
\zeta_{f,\,\minus\phi} (s)           & = \prod_{\tau\in \Orb(f)} \biggl( 1- \exp\biggl( - s \sum_{y\in\tau} \phi(y) \biggr) \biggr)^{-1},  \label{eqZetaFnOrbitForm_ThurstonMap}\\
\DS_{f,\,\minus\phi,\,\deg_f}(s)  & = \prod_{\tau\in \Orb(f)} \biggl( 1- \exp\biggl( - s \sum_{y\in\tau} \phi(y) \biggr) \prod_{z\in\tau} \deg_f(z) \biggr)^{-1}, \label{eqZetaFnOrbitForm_ThurstonMapDegree}\\
\zeta_{\sigma_{A_{\ti}},\,\minus\phi\circsmall\pi_{\ti}} (s)    & = \prod_{\tau\in \Orb(\sigma_{A_{\ti}}) } \biggl( 1- \exp\biggl( - s \sum_{y\in\tau} \phi\circ\pi_{\ti}(y) \biggr) \biggr)^{-1}  \label{eqZetaFnOrbitForm_Subshift}.
\end{align}
\end{enumerate}
\end{prop}

Recall that $\Orb(g)$ denotes the set of all primitive periodic orbits of $g$ (see (\ref{eqDefSetAllPeriodicOrbits})). We recall that an infinite product of the form $\exp \sum a_i$, $a_i\in\C$, converges uniformly (resp.\ absolutely) if $\sum a_i$ converges uniformly (resp.\ absolutely).

\section{The Assumptions}      \label{sctAssumptions}
We state below the hypotheses under which we will develop our theory in most parts of this paper. We will repeatedly refer to such assumptions in the later sections. We emphasize again that not all of these assumptions are assumed in all the statements in this paper.

\begin{assumptions}
\quad

\begin{enumerate}

\smallskip

\item $f\:S^2 \rightarrow S^2$ is an expanding Thurston map.

\smallskip

\item $\CC\subseteq S^2$ is a Jordan curve containing $\post f$ with the property that there exists $n_\CC\in\N$ such that $f^{n_\CC} (\CC)\subseteq \CC$ and $f^m(\CC)\nsubseteq \CC$ for each $m\in\{1, \, 2, \, \dots, \, n_\CC-1\}$.

\smallskip

\item $d$ is a visual metric on $S^2$ for $f$ with expansion factor $\Lambda>1$ and a linear local connectivity constant $L\geq 1$.

\smallskip

\item $\alpha\in(0,1]$.

\smallskip

\item $\psi\in \Holder{\alpha}((S^2,d),\C)$ is a complex-valued H\"{o}lder continuous function with an exponent $\alpha$.

\smallskip

\item $\phi\in \Holder{\alpha}(S^2,d)$ is an eventually positive real-valued H\"{o}lder continuous function with an exponent $\alpha$, and $s_0\in\R$ is the unique positive real number satisfying $P(f, -s_0\phi)=0$.

\smallskip

\item $\mu_\phi$ is the unique equilibrium state for the  map $f$ and the potential $\phi$.

\end{enumerate}

\end{assumptions}

Note that the uniqueness of $s_0$ in~(6) is guaranteed by Corollary~\ref{corS0unique}. For a pair of $f$ in~(1) and $\phi$ in~(6), we will say that a quantity depends on $f$ and $\phi$ if it depends on $s_0$.

Observe that by Lemma~\ref{lmCexistsL}, for each $f$ in~(1), there exists at least one Jordan curve $\CC$ that satisfies~(2). Since for a fixed $f$, the number $n_\CC$ is uniquely determined by $\CC$ in~(2), in the remaining part of the paper, we will say that a quantity depends on $f$ and $\CC$ even if it also depends on $n_\CC$.

Recall that the expansion factor $\Lambda$ of a visual metric $d$ on $S^2$ for $f$ is uniquely determined by $d$ and $f$. We will say that a quantity depends on $f$ and $d$ if it depends on $\Lambda$.

Note that even though the value of $L$ is not uniquely determined by the metric $d$, in the remainder of this paper, for each visual metric $d$ on $S^2$ for $f$, we will fix a choice of linear local connectivity constant $L$. We will say that a quantity depends on the visual metric $d$ without mentioning the dependence on $L$, even though if we had not fixed a choice of $L$, it would have depended on $L$ as well.

In the discussion below, depending on the conditions we will need, we will sometimes say ``Let $f$, $\CC$, $d$, $\psi$, $\alpha$ satisfy the Assumptions.'', and sometimes say ``Let $f$ and $d$ satisfy the Assumptions.'', etc.

\section{Ruelle operators and split Ruelle operators}    \label{sctSplitRuelleOp}

In this section, we define appropriate variations of the Ruelle operator on the suitable function spaces in our context and establish some important inequalities that will be used later. More precisely, in Subsection~\ref{subsctSplitRuelleOp_Construction}, for an expanding Thurston map $f$ with some forward invariant Jordan curve $\CC\subseteq S^2$ and a complex-valued H\"{o}lder continuous function $\psi$, we ``split'' the Ruelle operator $\RR_{\psi} \: \CCC(S^2,\C) \rightarrow \CCC(S^2,\C)$ into pieces $\RR_{\psi,\c,E}^{(n)} \: \CCC(E,\C) \rightarrow C\bigl(X^0_\c,\C\bigr)$, for $\c\in\{\b, \, \w\}$, $n\in\N_0$, and a union $E\subseteq S^2$ of an arbitrary collection of $n$-tiles in the cell decomposition $\DD^n(f,\CC)$ of $S^2$ induced by $f$ and $\CC$. Such construction is crucial to the proof of Proposition~\ref{propTelescoping} where the images of characteristic functions supported on $n$-tiles under $\RR_{\psi,\c,E}^{(n)}$ are used to relate periodic points and preimage points of $f$. We then define the \emph{split Ruelle operators} $\RRR_\psi$ on the product space $\CCC\bigl(X^0_\b,\C\bigr) \times \CCC\bigl(X^0_\w, \C\bigr)$ by piecing together $\RR_{\psi,\c_1,\c_2}^{(1)} = \RR_{\psi,\c_1,X^0_{\c_2}}^{(1)}$, $\c_1, \, \c_2\in\{\b, \, \w\}$. Subsection~\ref{subsctSplitRuelleOp_BasicIneq} is devoted to establishing various inequalities, among them the \emph{basic inequalities} in Lemma~\ref{lmBasicIneq}, that are indispensable in the arguments in Section~\ref{sctDolgopyat}. In Subsection~\ref{subsctSplitRuelleOp_SpectralGap}, we verify the spectral gap for $\RRR_\psi$ that is essential in the proof of Theorem~\ref{thmOpHolderNorm}.

\subsection{Construction}   \label{subsctSplitRuelleOp_Construction}

\begin{lemma}   \label{lmExpToLiniear}
Let $f$, $\CC$, $d$, $\Lambda$, $\alpha$ satisfy the Assumptions. Fix a constant $T>0$. Then for all $n\in\N$, $X^n\in\X^n(f,\CC)$, $x, \, x'\in X^n$, and $\psi\in\Holder{\alpha}((S^2,d),\C)$ with $\Hseminorm{\alpha,\, (S^2,d)}{\Re(\psi)} \leq T$, we have
\begin{equation}  \label{eqExpToLinear}
\abs{1-\exp(S_n\psi(x)-S_n\psi(x'))} \leq C_{4} \Hseminorm{\alpha,\, (S^2,d)}{\psi} d(f^n(x),f^n(x'))^\alpha,
\end{equation}
where the constant
\begin{equation}     \label{eqDefC10}
C_{4} = C_{4}(f,\CC,d,\alpha,T) \coloneqq \frac{2 C_0 }{1-\Lambda^{-\alpha}} \exp\biggl(\frac{C_0 T }{1-\Lambda^{-\alpha}}\bigl(\diam_d(S^2)\bigr)^\alpha\biggr)  >1
\end{equation}
depends only on $f$, $\CC$, $d$, $\alpha$, and $T$. Here $C_0>1$ is the constant from Lemma~\ref{lmMetricDistortion} depending only on $f$, $\CC$, and $d$.
\end{lemma}

\begin{proof}
Fix $T>0$, $n\in\N$, $X^n\in\X^n(f,\CC)$, $x, \, x'\in X^n$, and $\psi\in\Holder{\alpha}((S^2,d),\C)$ with $\Hseminorm{\alpha,\, (S^2,d)}{\Re(\psi)} \leq T$. By Lemma~\ref{lmSnPhiBound}, for each $\phi\in\Holder{\alpha}(S^2,d)$,
\begin{equation}   \label{eqPflmExpToLinear1}
\abs{S_n \phi(x)-S_n \phi(x')}  \leq \frac{C_0\Hseminorm{\alpha,\, (S^2,d)}{\phi} }{1-\Lambda^{-\alpha}}  d(f^n(x),f^n(x') )^\alpha.
\end{equation}

Then by (\ref{eqPflmExpToLinear1}) and the fact that $\abs{1-e^y} \leq  \abs{y}e^{\abs{y}}$ and $\abs{1-e^{\I y}} \leq \abs{y}$ for $y\in\R$, we get
\begin{align*}
&              \AbsBig{1-e^{S_n\psi(x)-S_n\psi(x')}} \\
&\qquad \leq   \AbsBig{1- e^{S_n \Re(\psi)(x)-S_n\Re(\psi)(x')}}  + e^{S_n \Re(\psi)(x)-S_n\Re(\psi)(x')}  \AbsBig{1- e^{\I S_n \Im(\psi)(x) - \I S_n\Im(\psi)(x')}} \\
&\qquad \leq   \frac{C_0\Hseminorm{\alpha,\, (S^2,d)}{\Re(\psi)} }{1-\Lambda^{-\alpha}}  d(f^n(x),f^n(x'))^\alpha  \exp\biggl(\frac{ C_0 T}{1-\Lambda^{-\alpha}}\bigl(\diam_d(S^2)\bigr)^\alpha\biggr)  \\
&\qquad \quad +\exp\biggl(\frac{C_0 T }{1-\Lambda^{-\alpha}}\bigl(\diam_d(S^2)\bigr)^\alpha\biggr) \frac{C_0\Hseminorm{\alpha,\, (S^2,d)}{\Im(\psi)} }{1-\Lambda^{-\alpha}}  d(f^n(x),f^n(x'))^\alpha\\
&\qquad \leq   C_{4}  \Hseminorm{\alpha,\, (S^2,d)}{\psi} d(f^n(x),f^n(x'))^\alpha.
\end{align*}
Here the constant $C_{4} =C_{4} (f, \CC, d, \alpha, T )$ is defined in (\ref{eqDefC10}).
\end{proof}

Fix an expanding Thurston map $f\: S^2 \rightarrow S^2 $ with a Jordan curve $\CC\subseteq S^2$ satisfying $\post f\subseteq \CC$. Let $d$ be a visual metric for $f$ on $S^2$, and $\psi\in\Holder{\alpha}((S^2,d),\C)$ a complex-valued H\"{o}lder continuous function.

Let $n\in\N$, $\c \in\{\b, \, \w\}$, and $x \in\inte \bigl( X^0_\c \bigr)$, where $X^0_\b$ (resp.\ $X^0_\w$) is the black (resp.\ white) $0$-tile. If $E\subseteq S^2$ is a union of $n$-tiles in $\X^n(f,\CC)$, $u\in\CCC((E,d),\C)$ a complex-valued continuous function defined on $E$, and if we define a function $v\in B(S^2,\C)$ by
\begin{equation}   \label{eqExtendToS2}
v(y) \coloneqq  \begin{cases} u(y) & \text{if } y\in E, \\ 0  & \text{otherwise}, \end{cases}
\end{equation}
then by Proposition~\ref{propCellDecomp}~(i) and (ii), the Ruelle operator associated to $f$ and $\psi$ (recalled in (\ref{eqDefRuelleOp})) acting on $B(S^2,\C)$ can be written in the following form:
\begin{equation}  \label{eqDefLcInte}
\RR_\psi^n (v) (x) = \sum\limits_{\substack{X^n\in\X^n_\c\\ X^n\subseteq E}} u\bigl((f^n|_{X^n})^{-1}(x)\bigr) \exp \bigl(S_n\psi\bigl((f^n|_{X^n})^{-1}(x)\bigr)\bigr).
\end{equation}

Note that by default, a summation over an empty set is equal to $0$. We will always use this convention in this paper. Inspired by (\ref{eqDefLcInte}), we give the following definition.

\begin{definition}   \label{defSplitRuelle}
Let $f\: S^2 \rightarrow S^2 $ be an expanding Thurston map, $\CC\subseteq S^2$ a Jordan curve containing $\post f$, and $\psi\in \CCC(S^2,\C)$ a complex-valued continuous function. Let $n\in\N_0$, and $E\subseteq S^2$ be a union of $n$-tiles in $\X^n(f,\CC)$. We define a map $\RR_{\psi,\c,E}^{(n)} \: \CCC(E,\C) \rightarrow \CCC\bigl(X^0_\c,\C\bigr)$, for each $\c\in\{\b, \, \w\}$, by
\begin{equation}  \label{eqDefLc}
\RR_{\psi,\c,E}^{(n)} (u) (y) \coloneqq \sum\limits_{\substack{X^n\in\X^n_\c\\ X^n\subseteq E}} u\bigl((f^n|_{X^n})^{-1}(y)\bigr) \exp \bigl(S_n\psi\bigl((f^n|_{X^n})^{-1}(y)\bigr)\bigr),
\end{equation}
for each complex-valued continuous function $u\in\CCC(E,\C)$ defined on $E$, and each point $y \in X^0_\c$. When $E=X^0_{\c'}$ for some $\c'\in\{\b, \, \w\}$, we often write 
\begin{equation*}
\RR_{\psi,\c,\c'}^{(n)} \coloneqq \RR_{\psi,\c,X^0_{\c'}}^{(n)}.
\end{equation*}
\end{definition}

Note that 
\begin{equation*}
	\RR_{\psi,\c,E}^{(0)} (u) = \begin{cases} u & \text{if } X^0_\c \subseteq E, \\ 0  & \text{otherwise},  \end{cases} \qquad \text{ for } \c\in\{\b, \, \w\},
\end{equation*}
whenever the expression on the left-hand side of the equation makes sense.

\begin{lemma}  \label{lmLDiscontProperties}
Let $f$, $\CC$, $d$, $\alpha$ satisfy the Assumptions.  Let $\psi \in \CCC(S^2,\C)$ be a complex-valued continuous function. Fix numbers $n, \,  m\in\N_0$ and a union $E\subseteq S^2$ of an arbitrary collection of $n$-tiles in $\X^n(f,\CC)$ (i.e., $E= \bigcup \{  X^n  \in \X^n(f,\CC)   :    X^n\subseteq E\}$). Then for each $\c\in\{\b, \, \w\}$ and each $u\in \CCC(E,\C)$, we have
$\RR_{\psi, \c, E}^{(n)} (u) \in \CCC\bigl( X^0_\c,\C \bigr)$, and
\begin{equation}  \label{eqLDiscontProperties_Split}
\RR_{\psi,\c,E}^{(n+m)} (u)  = \sum_{\c'\in\{\b, \, \w\}}   \RR_{\psi,\c,\c'}^{(m)} \Bigl( \RR_{\psi,\c',E}^{(n)} (u)  \Bigr). 
\end{equation}
If, in addition, $\psi \in \Holder{\alpha}((S^2,d),\C)$ and $u\in\Holder{\alpha}((E,d),\C)$ are H\"{o}lder continuous, then
\begin{equation}    \label{eqLDiscontProperties_Holder}
\RR_{\psi, \c, E}^{(n)} (u) \in \Holder{\alpha} \bigl( \bigl( X^0_\c,d \bigr),\C \bigr).
\end{equation}
\end{lemma}

\begin{rem}    \label{rmSplitRuelleCoincide}
In the above context, $\RR_\psi^n(v) \in B(S^2,\C)$ may not be continuous on $S^2$ if $E\neq S^2$, where $v$ is defined in (\ref{eqExtendToS2}) extending $u$ to $S^2$. If $E=S^2$, then it follows immediately from  (\ref{eqDefLc}) that for each $\c\in\{\b, \, \w\}$,
$
  \RR_{\psi,\c,E}^{(n)} (u)   = \bigl( \RR_{\psi}^n (u) \bigr) \big|_{X^0_\c}.
$
Hence, by (\ref{eqLDiscontProperties_Holder}) and the linear local connectivity of $(S^2,d)$, it can be shown that
$
\RR_\psi^n  ( \Holder{\alpha}((S^2,d),\C)  ) \subseteq \Holder{\alpha}((S^2,d),\C).
$
We will not use this fact in this paper.
\end{rem}

\begin{proof}
Fix arbitrary $\c\in\{\b, \, \w\}$ and $u\in \CCC(E,\C)$.

The cases of Lemma~\ref{lmLDiscontProperties} when either $m=0$ or $n=0$ follow immediately from Definition~\ref{defSplitRuelle}. Thus, without loss of generality, we can assume $m, \, n\in\N$.

The continuity of $\RR_{\psi, \c, E}^{(n)} (u)$ follows trivially from (\ref{eqDefLc}) and Proposition~\ref{propCellDecomp}~(i).

By  (\ref{eqDefLc}), Proposition~\ref{propCellDecomp}~(i) and (ii), and the fact that $f^{-m}(x) \cap \CC \neq \emptyset$, we get
\begin{align*}
&     \sum_{\c'\in\{\b, \, \w\}}   \RR_{\psi,\c,\c'}^{(m)} \Bigl( \RR_{\psi,\c',E}^{(n)} (u)  \Bigr) (x)  \\
&\qquad=   \sum_{\c'\in\{\b, \, \w\}}  \sum_{ y\in f^{-m}(x) \cap X^0_{\c'} } e^{S_m\psi(y)}  \sum_{z\in f^{-n}(y) \cap E}  e^{S_n\psi(z)}   u(z)\\
&\qquad=   \sum_{  y\in f^{-m}(x)  }  \sum_{ z\in f^{-n}(y) \cap E} e^{S_m\psi(y) + S_n\psi(z)}   u(z) \\
&\qquad=   \sum_{ z\in f^{-(n+m)}(x) \cap E} e^{S_{n+m}\psi(z) }   u(z) \\
&\qquad=   \RR_{\psi,\c,E}^{(n+m)} (u)(x).
\end{align*}
Identity (\ref{eqLDiscontProperties_Split}) is now established by the continuity of two sides of the equation above.

\smallskip

Finally, to prove (\ref{eqLDiscontProperties_Holder}), we first fix two distinct points $x, \, x'\in X^0_\c$. We denote, for each $X^n\in\X^n_\c$, $y_{X^n} \coloneqq (f^n|_{X^n})^{-1}(x)$ and $y'_{X^n} \coloneqq (f^n|_{X^n})^{-1}(x')$.

By Lemmas~\ref{lmMetricDistortion},~\ref{lmSnPhiBound}, and~\ref{lmExpToLiniear}, we have
\begin{align*}
&    \Absbig{   \RR_{\psi,\c,E}^{(n)}(u) (x) -  \RR_{\psi,\c,E}^{(n)}(u) (x')  } \big/ d(x,x')^\alpha  \\
&\qquad\leq  \frac{1}{d(x,x')^\alpha} \sum\limits_{\substack{X^n\in\X^n_\c \\ X^n\subseteq E} }  \Absbig{  e^{S_n\psi(y_{X^n})} u(y_{X^n}) - e^{S_n\psi(y'_{X^n})} u(y'_{X^n}) } \\
&\qquad\leq  \frac{1}{d(x,x')^\alpha} \sum\limits_{\substack{X^n\in\X^n_\c \\ X^n\subseteq E} }     \Absbig{e^{S_n\psi(y_{X^n})}}  \abs{ u(y_{X^n}) - u(y'_{X^n})}  \\
&\qquad\qquad      + \frac{1}{d(x,x')^\alpha}   \sum\limits_{\substack{X^n\in\X^n_\c \\ X^n\subseteq E} }      \Absbig{e^{S_n\psi(y_{X^n})}  -  e^{S_n\psi(y'_{X^n})}}    \abs{  u(y'_{X^n})}  \\
&\qquad\leq  \frac{1}{d(x,x')^\alpha} \sum\limits_{\substack{X^n\in\X^n_\c \\ X^n\subseteq E} }   e^{S_n \Re(\psi)(y_{X^n})} \Hseminorm{\alpha,\, (E,d)}{u} d(y_{X^n},y'_{X^n})^\alpha \\
&\qquad\qquad      + \frac{1}{d(x,x')^\alpha} \sum\limits_{\substack{X^n\in\X^n_\c \\ X^n\subseteq E} } \Absbig{1- e^{ S_n\psi(y_{X^n}) - S_n\psi(y'_{X^n}) } }  e^{S_n \Re(\psi)(y'_{X^n})}  \abs{  u(y'_{X^n})}  \\
&\qquad\leq  \Hseminorm{\alpha,\, (E,d)}{u} C_0^\alpha  \Lambda^{ - \alpha n}   \sum_{X^n\in\X^n} e^{S_n\Re(\psi)(y_{X^n})} 
      + C_{4} \Hseminorm{\alpha,\, (S^2,d)}{\psi}   \sum\limits_{\substack{X^n\in\X^n_\c \\ X^n\subseteq E} }  e^{S_n\Re(\psi)(y'_{X^n})} \abs{u(y'_{X^n})} \\
&\qquad\leq   C_0  \Lambda^{ - \alpha n }  \Hseminorm{\alpha,\, (E,d)}{u}  \Normbig{\RR_{\Re(\psi)}^n (\mathbbm{1}_{S^2}) }_{\CCC^0(S^2)}  
      + C_{4} \Hseminorm{\alpha,\, (S^2,d)}{\psi}  \NormBig{\RR_{\Re(\psi),\c,E}^{(n)} (\abs{u}) }_{\CCC^0(S^2)}  ,      
\end{align*}
where $C_0 > 1$ is the constant depending only on $f$, $\CC$, and $d$ from Lemma~\ref{lmMetricDistortion}, and $C_{4}>1$ is the constant depending only on $f$, $\CC$, $d$, $\alpha$, and $\psi$ from Lemma~\ref{lmExpToLiniear}. Therefore, (\ref{eqLDiscontProperties_Holder}) holds.
\end{proof}

Now that we have ``split'' the Ruelle operator $\RR_{\psi} \: \CCC(S^2,\C) \rightarrow \CCC(S^2,\C)$ into pieces and studied some basic properties of the pieces, we are ready to define the \emph{split Ruelle operators} $\RRR_\psi$ on the product space $\CCC\bigl(X^0_\b,\C\bigr) \times \CCC\bigl(X^0_\w, \C\bigr)$ by piecing together the pieces.

\begin{definition}[Split Ruelle operators]   \label{defSplitRuelleOnProductSpace}
Let $f\: S^2 \rightarrow S^2$ be an expanding Thurston map with a Jordan curve $\CC\subseteq S^2$ satisfying $f(\CC)\subseteq\CC$ and $\post f\subseteq \CC$. Let $d$ be a visual metric for $f$ on $S^2$, and $\psi\in\Holder{\alpha}((S^2,d),\C)$ a complex-valued H\"{o}lder continuous function with an exponent $\alpha\in(0,1]$. Let $X^0_\b, \, X^0_\w\in\X^0(f,\CC)$ be the black $0$-tile and the while $0$-tile, respectively. The \defn{split Ruelle operator} 
$
\RRR_\psi \: \CCC\bigl(X^0_\b,\C\bigr) \times \CCC\bigl(X^0_\w, \C\bigr) \rightarrow \CCC \bigl(X^0_\b, \C\bigr) \times \CCC \bigl(X^0_\w,\C\bigr)
$
on the product space $\CCC \bigl(X^0_\b, \C\bigr) \times \CCC \bigl(X^0_\w, \C\bigr)$ is given by
\begin{equation*}    \label{eqDefSplitRuelleOnProductSpace}
\RRR_\psi(u_\b,u_\w) \coloneqq \Bigl(   \RR_{\psi, \b, \b}^{(1)} (u_\b) + \RR_{\psi, \b, \w}^{(1)} (u_\w)   ,
                              \RR_{\psi, \w, \b}^{(1)} (u_\b) + \RR_{\psi, \w, \w}^{(1)} (u_\w)           \Bigr)
\end{equation*}
for $u_\b \in \CCC\bigl(X^0_\b, \C\bigr)$ and $u_\w \in \CCC \bigl(X^0_\w, \C\bigr)$.
\end{definition}

Note that by Lemma~\ref{lmLDiscontProperties}, the operator $\RRR_\psi$ is well-defined. Moreover, by (\ref{eqLDiscontProperties_Holder}) in Lemma~\ref{lmLDiscontProperties}, we have
\begin{equation}    \label{eqSplitRuelleRestrictHolder}
 \RRR_\psi \bigl(  \Holder{\alpha}\bigl(\bigl(X^0_\b,d\bigr),\C\bigr) \times \Holder{\alpha}\bigl(\bigl(X^0_\w,d\bigr),\C\bigr) \bigr) 
  \subseteq        \Holder{\alpha}\bigl(\bigl(X^0_\b,d\bigr),\C\bigr) \times \Holder{\alpha}\bigl(\bigl(X^0_\w,d\bigr),\C\bigr).
\end{equation}

Note that it follows immediately from Definition~\ref{defSplitRuelle} that $\RRR_\psi$ is a linear operator on the Banach space $\Holder{\alpha}\bigl(\bigl(X^0_\b,d\bigr),\C\bigr) \times \Holder{\alpha}\bigl(\bigl(X^0_\w,d\bigr),\C\bigr)$ equipped with a norm given by 
\begin{equation*}\
	\norm{(u_\b,u_\w)} \coloneqq \max\bigl\{  \NHnorm{\alpha}{b}{ u_\b }{(X^0_\b,d)}, \, \NHnorm{\alpha}{b}{ u_\w }{(X^0_\w,d)} \bigr\},
\end{equation*}
for each $b\in\R \setminus \{0\}$. See (\ref{eqDefNormalizedHolderNorm}) for the definition of the normalized H\"{o}lder norm $\NHnorm{\alpha}{b}{u}{(E,d)}$.

For each $\c\in\{\b, \, \w\}$, we define the projection $\pi_\c \: \CCC(X^0_\b,\C) \times \CCC(X^0_\w,\C) \rightarrow \CCC(X^0_\c,\C)$ by
\begin{equation}   \label{eqDefProjections}
\pi_\c(u_\b,u_\w) \coloneqq u_\c, \qquad \mbox{ for } (u_\b,u_\w)\in \CCC(X^0_\b,\C) \times \CCC(X^0_\w,\C).
\end{equation}

\begin{definition}   \label{defOpHNormDiscont}
Let $f\: S^2 \rightarrow S^2$ be an expanding Thurston map with a Jordan curve $\CC\subseteq S^2$ satisfying $f(\CC)\subseteq\CC$ and $\post f\subseteq \CC$. Let $d$ be a visual metric for $f$ on $S^2$, and $\psi\in\Holder{\alpha}((S^2,d),\C)$ a complex-valued H\"{o}lder continuous function with an exponent $\alpha\in(0,1]$. For all $n\in\N_0$ and $b\in \R\setminus\{0\}$, we write the operator norm
\begin{align}  \label{eqDefNormalizedOpHNormSplitRuelle}
 \NOpHnormD{\alpha}{b}{\RRR_\psi^n}   
&\coloneqq   \sup \biggl\{     \NHnormbig{\alpha}{b}{  \pi_\c \bigl( \RRR_\psi^n (u_\b, u_\w)\bigr) }{(X^0_\c,d)}  : \notag  \\
                   & \qquad \qquad \qquad   \begin{array}{ll}  \c\in\{\b, \, \w\}, \, u_\b\in\Holder{\alpha} ((X^0_\b,d),\C), \, u_\w\in\Holder{\alpha}((X^0_\w,d),\C) \\ 
                     \text{with } \NHnorm{\alpha}{b}{ u_\b}{(X^0_\b,d)} \leq 1 \text{ and } \NHnorm{\alpha}{b}{ u_\w}{(X^0_\w,d)} \leq 1\end{array} \biggr\}.   
\end{align}
We write
$\OpHnormDbig{\alpha}{\RRR_\psi^n}  \coloneqq  \NOpHnormDbig{\alpha}{1}{\RRR_\psi^n}$.
\end{definition}

\begin{lemma}   \label{lmSplitRuelleCoordinateFormula}
Let $f$, $\CC$, $d$, $\alpha$, $\psi$ satisfy the Assumptions. We assume, in addition, that $f(\CC) \subseteq \CC$. Let $X^0_\b, \, X^0_\w\in\X^0(f,\CC)$ be the black $0$-tile and the while $0$-tile, respectively. Then for all $n\in\N_0$, $u_\b \in \CCC \bigl(X^0_\b, \C\bigr)$, and $u_\w \in \CCC \bigl(X^0_\w, \C\bigr)$, 
\begin{equation}    \label{eqSplitRuelleCoordinateFormula}
\RRR_\psi^n(u_\b,u_\w) = \Bigl(   \RR_{\psi, \b, \b}^{(n)} (u_\b) + \RR_{\psi, \b, \w}^{(n)} (u_\w)   ,
                                \RR_{\psi, \w, \b}^{(n)} (u_\b) + \RR_{\psi, \w, \w}^{(n)} (u_\w)           \Bigr).
\end{equation}
Consequently,
\begin{align}   \label{eqSplitRuelleOpNormQuot}
    \NOpHnormD{\alpha}{b}{\RRR_\psi^n}  
 =  \sup \Biggl\{ & \frac{   \NHnormbig{\alpha}{b}{   \RR_{\psi, \c, \b}^{(n)} (u_\b)  +  \RR_{\psi, \c, \w}^{(n)} (u_\w)    }{(X^0_\c,d)}      }{  \max\bigl\{     \NHnorm{\alpha}{b}{u_\b}{(X^0_\b,d)}, \, \NHnorm{\alpha}{b}{u_\w}{(X^0_\w,d)}  \bigr\}   }    : \notag\\
                &\qquad  \begin{array}{ll}  \c\in\{\b, \, \w\}, \, u_\b\in\Holder{\alpha}((X^0_\b,d),\C), \, u_\w\in\Holder{\alpha}((X^0_\w,d),\C) \\ \text{with } \norm{u_\b}_{\CCC^0(X^0_\b)} \norm{u_\w}_{\CCC^0(X^0_\w)} \neq 0 \end{array} \Biggr\}.
\end{align}
\end{lemma}

\begin{proof}
We prove (\ref{eqSplitRuelleCoordinateFormula}) by induction. The case where $n=0$ and the case where $n=1$ both hold by definition. Assume now (\ref{eqSplitRuelleCoordinateFormula}) holds when $n=m$ for some $m\in\N$. Then by (\ref{eqLDiscontProperties_Split}) in Lemma~\ref{lmLDiscontProperties}, for each $\c\in\{\b, \, \w\}$, we have
\begin{align*}
        \pi_\c \bigl( \RRR_\psi^{m+1}(u_\b,u_\w)  \bigr) 
&=   \pi_\c \Bigl(  \RRR_\psi \Bigl(   \RR_{\psi, \b, \b}^{(m)} (u_\b) + \RR_{\psi, \b, \w}^{(m)} (u_\w)   ,
                                                             \RR_{\psi, \w, \b}^{(m)} (u_\b) + \RR_{\psi, \w, \w}^{(m)} (u_\w)           \Bigr)    \Bigr)\\
&=   \sum_{\c'\in\{\b, \, \w\}}  \RR_{\psi, \c, \c'}^{(1)} \Bigl(   \RR_{\psi, \c', \b}^{(m)} (u_\b) + \RR_{\psi, \c', \w}^{(m)} (u_\w) \Bigr) \\
&=   \sum_{\c'\in\{\b, \, \w\}}  \RR_{\psi, \c, \c'}^{(1)}  \Bigl(\RR_{\psi, \c', \b}^{(m)} (u_\b) \Bigr)
         +  \sum_{\c'\in\{\b, \, \w\}}  \RR_{\psi, \c, \c'}^{(1)}  \Bigl(\RR_{\psi, \c', \w}^{(m)} (u_\w) \Bigr)  \\  
&=        \RR_{\psi, \c, \b}^{(m+1)} (u_\b) + \RR_{\psi, \c, \w}^{(m+1)} (u_\w) ,                               
\end{align*}
for $u_\b \in \CCC \bigl(X^0_\b, \C\bigr)$ and $u_\w \in \CCC \bigl(X^0_\w, \C\bigr)$. This completes the inductive step, establishing (\ref{eqSplitRuelleCoordinateFormula}).

Identity (\ref{eqSplitRuelleOpNormQuot}) follows immediately from Definition~\ref{defOpHNormDiscont} and (\ref{eqSplitRuelleCoordinateFormula}).
\end{proof}

\subsection{Basic inequalities}   \label{subsctSplitRuelleOp_BasicIneq}

Let $f\: S^2 \rightarrow S^2$ be an expanding Thurston map, and $d$ be a visual metric on $S^2$ for $f$ with expansion factor $\Lambda>1$. Let $\psi\in \Holder{\alpha}((S^2,d),\C)$ be a complex-valued H\"{o}lder continuous function with an exponent $\alpha\in(0,1]$. We define
\begin{equation}   \label{eqDefWideTildePsiComplex}
\wt{\psi} \coloneqq \wt{\Re(\psi)} +  \I \Im(\psi) = \psi - P(f,\Re(\psi)) + \log u_{\Re(\psi)} - \log \bigl(u_{\Re(\psi)} \circ f \bigr),
\end{equation}
where $u_{\Re(\psi)}$ is the continuous function given by Theorem~\ref{thmMuExist} with $\phi \coloneqq \Re(\psi)$. Then for each $u\in \CCC(S^2,\C)$ and each $x\in S^2$, we have
\begin{align}   \label{eqRuelleTilde_NoTilde}
     \RR_{\wt{\psi}} (u)(x)
& =  \sum_{y\in f^{-1}(x)} \deg_f(y)u(y) e^{ \psi(y) - P(f,\Re(\psi)) + \log u_{\Re(\psi)}(y) - \log  ( u_{\Re(\psi)} (f(y)) ) }   \notag \\
& =  \frac{\exp(-P(f,\Re(\psi))}{u_{\Re(\psi)} (x) }  \sum_{y\in f^{-1}(x)} \deg_f(y)u(y) u_{\Re(\psi)}(y) \exp ( \psi(y))\\
& =  \frac{\exp(-P(f,\Re(\psi))}{u_{\Re(\psi)} (x) }  \RR_{ \psi } \bigl(u_{\Re(\psi)} u \bigr)(x). \notag
\end{align}
Given a Jordan curve $\CC\subseteq S^2$ with $\post f\subseteq\CC$, then for each $n\in\N_0$, each union $E$ of $n$-tiles in $\X^n(f,\CC)$, each $v\in\CCC(E,\C)$, each $\c\in\{\b, \, \w\}$, and each $z\in X^0_\c$,
\begin{align}   \label{eqSplitRuelleTilde_NoTilde}
   \RR_{\wt{\psi},\c,E}^{(n)} (v)(z)  
&=   \sum\limits_{\substack{X^n\in\X^n_\c(f,\CC)\\ X^n\subseteq E}} \Bigl( v  e^{ S_n  ( \psi - P(f,\Re(\psi)) + \log u_{\Re(\psi)}  - \log  ( u_{\Re(\psi)}\circ f )  )} \Bigr) \bigl(  (f^n|_{X^n} )^{-1}(z) \bigr)  \notag \\
&=   \frac{\exp(-n P(f,\Re(\psi))}{u_{\Re(\psi)} (z) }  \sum\limits_{\substack{X^n\in\X^n_\c(f,\CC)\\ X^n\subseteq E}} \bigl(v  u_{\Re(\psi)}  \exp (S_n \psi )\bigr)   \bigl(  (f^n|_{X^n} )^{-1}(z) \bigr) \\
&=   \frac{\exp(-n P(f,\Re(\psi))}{u_{\Re(\psi)} (z) }  \RR_{ \psi,\c,E }^{(n)} \bigl(u_{\Re(\psi)} v \bigr)(z). \notag
\end{align}

\begin{definition}[Cones]  \label{defCone}
Let $f\: S^2 \rightarrow S^2$ be an expanding Thurston map, and $d$ be a visual metric on $S^2$ for $f$ with expansion factor $\Lambda>1$. Fix a constant $\alpha\in(0,1]$. For each subset $E\subseteq S^2$ and each constant $B\in \R$ with $B>0$, we define the \defn{$B$-cone inside $\Holder{\alpha}(E,d)$} as 
\begin{align*}  
 K_B(E,d)   \coloneqq   \bigl\{ u\in \Holder{\alpha} (E,d)    : \text{for all } &x, \, y\in E, \, u(x)> 0 \text{ and }\\
 																						&\qquad \abs{u(x)-u(y)} \leq B (u(x)+u(y)) d(x,y)^\alpha \bigr\}.  
\end{align*}
\end{definition}

It is essential to define the $B$-cones inside $\Holder{\alpha}(E,d)$ in the form above in order to establish the following lemma, which will be used in the proof of Proposition~\ref{propDolgopyatOperator}.

\begin{lemma}   \label{lmConePower}
Let $(X,d)$ be a metric space and $\alpha\in (0,1]$. Then for each $B>0$ and each $u\in K_B(X,d)$, we have $u^2 \in K_{2B}(X,d)$.
\end{lemma}

\begin{proof}
Fix arbitrary $B>0$ and $u\in K_B(X,d)$. For any $x, \, y\in X$,
\begin{align*}
      \Absbig{u^2(x)-u^2(y)} 
      &=  \abs{u(x)+u(y)}\abs{u(x)-u(y)} \\
      &\leq B\abs{u(x)+u(y)}^2 d(x,y)^\alpha \\
      &\leq   2B \bigl(u^2(x)+u^2(y)\bigr) d(x,y)^\alpha.
\end{align*}
Therefore, $u^2\in K_{2B}(X,d)$.
\end{proof}

\begin{lemma}   \label{lmRtildeNorm=1}
Let $f$, $d$, $\alpha$, $\psi$ satisfy the Assumptions. Let $\phi\in \Holder{\alpha}(S^2,d)$ be a real-valued H\"{o}lder continuous function with an exponent $\alpha$. Then the operator norm of $\RR_{\wt\phi}$ acting on $\CCC(S^2)$ is given by $\Normbig{\RR_{\wt\phi}}_{C^0(S^2)}=1$. In addition, $\RR_{\wt\phi}(\mathbbm{1}_{S^2})=\mathbbm{1}_{S^2}$.

Moreover, consider a Jordan curve $\CC\subseteq S^2$ satisfying $\post f\subseteq \CC$. Assume in addition that $f(\CC)\subseteq\CC$. Then for all $n\in\N_0$, $\c, \, \c' \in\{\b, \, \w\}$, $u_\b\in\CCC(X^0_\b,\C)$, and $u_\w\in\CCC(X^0_\w,\C)$, we have
\begin{align} 
      \Norm{  \RR_{\wt{\psi}, \c, \c'}^{(n)}  (u_{\c'})  }_{C^0(X^0_\c)} 
& \leq \norm{u_{\c'}}_{C^0(X^0_{\c'})} \qquad  \text{ and}  \label{eqRDiscontTildeSupDecreasing} \\
      \Norm{  \RR_{\wt{\psi}, \c, \b}^{(n)}  (u_\b) + \RR_{\wt{\psi}, \c, \w} ^{(n)} (u_\w)  }_{C^0(X^0_\c)}  
& \leq \max\bigl\{\norm{u_\b}_{C^0(X^0_\b)}, \, \norm{u_\w}_{C^0(X^0_\w)}\bigr\}.  \label{eqSplitRuelleTildeSupDecreasing}
\end{align}
\end{lemma}

\begin{proof}
The fact that $\Normbig{\RR_{\wt\phi}}_{C^0(S^2)}=1$ and $\RR_{\wt\phi}(\mathbbm{1}_{S^2})=\mathbbm{1}_{S^2}$ is established in \cite[Lemma~5.25]{Li17}.

To prove (\ref{eqSplitRuelleTildeSupDecreasing}), we first fix arbitrary $n\in\N_0$, $\c\in\{\b, \, \w\}$, $u_\b\in\CCC(X^0_\b)$, and $u_\w\in\CCC(X^0_\w)$. Denote $M \coloneqq \max\bigl\{\norm{u_\b}_{C^0(X^0_\b)},\,\norm{u_\w}_{C^0(X^0_\w)}\bigr\}$. Then by Definition~\ref{defSplitRuelle}, (\ref{eqDefWideTildePsiComplex}), and the fact that $\RR_{\wt{\Re(\psi)}}(\mathbbm{1}_{S^2})=\mathbbm{1}_{S^2}$, for each $y\in\inte(X^0_\c)$,
\begin{align*}
              \Norm{  \RR_{\wt{\psi}, \c, \b}^{(n)}  (u_\b) + \RR_{\wt{\psi}, \c, \w} ^{(n)} (u_\w)  }_{C^0(X^0_\c)}  
  &\leq  M  \sum_{X^n\in\X^n_\c} \Absbig{  \exp \bigl(S_n \wt{\psi} \bigl((f^n|_{X^n})^{-1}(y)\bigr)\bigr)     }  \\
  &=       M \RR_{\wt{\Re(\psi)}}^n (\mathbbm{1}_{S^2})  (y)
  =          M.
\end{align*}
This establishes (\ref{eqSplitRuelleTildeSupDecreasing}). Finally, (\ref{eqRDiscontTildeSupDecreasing}) follows immediately from (\ref{eqSplitRuelleTildeSupDecreasing}) and Definition~\ref{defSplitRuelle} by setting one of the functions $u_\b$ and $u_\w$ to be $0$.
\end{proof}

\begin{lemma}   \label{lmBound_aPhi}
Let $f$, $\CC$, $d$, $L$, $\alpha$, $\Lambda$ satisfy the Assumptions. Then there exist constants $C_{5}>1$ and $C_{6}>0$ depending only on $f$, $\CC$, $d$, and $\alpha$ such that the following statement holds:

\smallskip 

For all $K, \, M, \, T, \, a\in\R$ with $K>0$, $M>0$, $T>0$, and $\abs{a}\leq T$, and all real-valued H\"{o}lder continuous function $\phi\in\Holder{\alpha}(S^2,d)$ with $\Hseminorm{\alpha,\, (S^2,d)}{\phi}  \leq K$ and $\Norm{\phi}_{\CCC^0(S^2)} \leq M$, we have
\begin{align}
\Normbig{\wt{a\phi}}_{\CCC^0(S^2)}                 & \leq C_{5} (K+M)T + \abs{\log(\deg f)},  \label{eqBound_aPhi_Sup} \\
\Hseminormbig{\alpha,\, (S^2,d)}{\wt{a\phi}}  &\leq C_{5} KT   e^{C_{6}KT}, \label{eqBound_aPhi_Hseminorm} \\
\Hnorm{\alpha}{u_{a\phi}}{(S^2,d)}                    &\leq \bigl( 4 TK C_0 (1-\Lambda^{-\alpha})^{-1}  L +1 \bigr) e^{2C_{7}}, \label{eqBound_Uaphi_Hnorm} \\
\exp(-C_{7})                                                          &\leq u_{a\phi}(x) \leq \exp(C_{7})  \label{eqBound_Uaphi_UpperLower}
\end{align}
for $x\in S^2$, where the constant $C_0>1$ depending only on $f$, $d$, and $\CC$ is from Lemma~\ref{lmMetricDistortion}, and the constant
\begin{equation}     \label{eqConst_lmBound_aPhi1}
C_{7} = C_{7}(f,\CC,d,\alpha,T,K) \coloneqq 4 TKC_0 ( 1-\Lambda^{-\alpha} )^{-1} L \bigl(\diam_d(S^2)\bigr)^\alpha  >  0
\end{equation}
depends only on $f$, $\CC$, $d$, $\alpha$, $T$, and $K$.
\end{lemma}

\begin{proof}
Fix $K$, $M$, $T$, $a$, $\phi$ that satisfy the conditions in this lemma.

Recall $\wt{a\phi} = a\phi - P(f,a\phi) + \log u_{a\phi} - \log(u_{a\phi} \circ f)$, where the function $u_{a\phi}$ is defined as $u_{\phi}$ in Theorem~\ref{thmMuExist}.

By Theorem~\ref{thmMuExist} and (\ref{eqC1C2}) in Lemma~\ref{lmSnPhiBound}, we immediately get (\ref{eqBound_Uaphi_UpperLower}).

By Lemma~3.25 in \cite{LZ24a}, (\ref{eqDeg=SumLocalDegree}), and (\ref{eqLocalDegreeProduct}), for each $x\in S^2$,
\begin{align*}
P(f,a\phi) & = \lim_{n\to+\infty} \frac{1}{n} \log \sum_{y\in f^{-n}(x)} \deg_{f^n}(y) e^{aS_n\phi(y)} 
            \leq  \lim_{n\to+\infty} \frac{1}{n} \log \sum_{y\in f^{-n}(x)} \deg_{f^n}(y) e^{nTM}  \\
          & = TM + \lim_{n\to+\infty}\frac{1}{n}  \log \sum_{y\in f^{-n}(x)} \deg_{f^n}(y) 
            = TM + \log(\deg f).
\end{align*}
Similarly, $P(f,a\phi)\geq -TM + \log(\deg f)$. So $\abs{P(f,a\phi)}  \leq TM + \abs{\log(\deg f)}$.

Thus, by combining the above with (\ref{eqBound_Uaphi_UpperLower}) and (\ref{eqConst_lmBound_aPhi1}), we get 
\begin{equation}  \label{eqPflmBound_aPhi3}
\Normbig{\wt{a\phi}}_{\CCC^0(S^2)} 
\leq TM +  TM + \abs{\log(\deg f)} + 2C_{7} 
\leq C_{8} T(K+M) + \abs{\log(\deg f)},
\end{equation}
where $C_{8} \coloneqq 2 + 8 \frac{C_0}{1-\Lambda^{-\alpha}} L \bigl(\diam_d(S^2)\bigr)^\alpha$ is a constant depending only on $f$, $\CC$, $d$, and $\alpha$.

Note $f$ is Lipschitz with respect to $d$ (see \cite[Lemma~3.12]{Li18}). Thus, by (\ref{eqBound_Uaphi_UpperLower}) and the fact that $\abs{\log t_1 - \log t_2} \leq \frac{\abs{t_1-t_2}}{\min\{t_1, \, t_2\}}$ for all $t_1,t_2>0$, we get 
\begin{align}  \label{eqPflmBound_aPhi2}
\Hseminormbig{\alpha,\, (S^2,d)}{\wt{a\phi}}  
&\leq  \Hseminorm{\alpha,\, (S^2,d)}{a\phi}  + \Hseminorm{\alpha,\, (S^2,d)}{\log u_{a\phi}}  +\Hseminorm{\alpha,\, (S^2,d)}{\log(u_{a\phi}\circ f)} \notag \\
&\leq  TK + e^{C_{7}} (1+ \LIP_d(f)  ) \Hseminorm{\alpha,\, (S^2,d)}{ u_{a\phi}}   .                                            
\end{align}
Here $\LIP_d(f)$ denotes the Lipschitz constant of $f$ with respect to the visual metric $d$ (see (\ref{eqDefLipConst})). 

By Theorem~\ref{thmMuExist}, (\ref{eqR1Diff}) in Lemma~\ref{lmR1properties}, (\ref{eqC1C2}) in Lemma~\ref{lmSnPhiBound}, (\ref{eqConst_lmBound_aPhi1}), and the fact that $\abs{1-e^{-t}} \leq t$ for $t>0$, we get 
\begin{align*}
        \abs{u_{a\phi}(x) - u_{a\phi}(y) }  
&=     \Absbigg{ \lim_{n\to+\infty} \frac{1}{n} \sum_{j=0}^{n-1}  \Bigl( \RR_{\overline{a\phi}}^j \bigl(\mathbbm{1}_{S^2}\bigr)(x) - \RR_{\overline{a\phi}}^j \bigl(\mathbbm{1}_{S^2}\bigr)(y)  \Bigr)}\\
&\leq  \limsup_{n\to+\infty} \frac{1}{n} \sum_{j=0}^{n-1} \AbsBig{ \RR_{\overline{a\phi}}^j \bigl(\mathbbm{1}_{S^2}\bigr)(x) - \RR_{\overline{a\phi}}^j \bigl(\mathbbm{1}_{S^2}\bigr)(y)  } \\
&\leq  e^{2C_{7}} \Bigl( 1 - \exp\Bigl( - 4 \frac{TKC_0}{1-\Lambda^{-\alpha}} L d(x,y)^\alpha   \Bigr)  \Bigr)  \\
&\leq  e^{2C_{7}}  \frac{4TKC_0}{1-\Lambda^{-\alpha}} L d(x,y)^\alpha,
\end{align*}
for all $x, \, y\in S^2$. So 
\begin{equation}    \label{eqBound_Uaphi_Hseminorm}
\Hseminorm{\alpha,\, (S^2,d)}{u_{a\phi}} \leq 4\frac{TKC_0}{1-\Lambda^{-\alpha}} L e^{2C_{7}} . 
\end{equation}

Thus, by (\ref{eqPflmBound_aPhi2}), (\ref{eqBound_Uaphi_Hseminorm}), and (\ref{eqConst_lmBound_aPhi1}), we get
\begin{equation*}
\Hseminormbig{\alpha,\, (S^2,d)}{\wt{a\phi}} \leq TK  C_{5} e^{C_{6}TK} ,
\end{equation*}
where the constants
$
C_{5}  \coloneqq \max\bigl\{ C_{8}, \, 1+  4C_0 (1-\Lambda^{-\alpha})^{-1}  L (1+ \LIP_d(f) ) \bigr\}
$
and 
$
C_{6} \coloneqq 12 C_0 (1-\Lambda^{-\alpha} )^{-1} L (\diam_d(S^2))^\alpha
$
depend only on $f$, $\CC$, $d$, and $\alpha$. Since $C_{5}\geq C_{8}$, (\ref{eqBound_aPhi_Sup}) follows from (\ref{eqPflmBound_aPhi3}). 

Finally, (\ref{eqBound_Uaphi_Hnorm}) follows from (\ref{eqBound_Uaphi_UpperLower}) and (\ref{eqBound_Uaphi_Hseminorm}).
\end{proof}

\begin{lemma}[Basic inequalities]   \label{lmBasicIneq}
Let $f$, $\CC$, $d$, $\alpha$, $\phi$, $s_0$ satisfy the Assumptions. Then there exists a constant $A_0=A_0\bigl(f,\CC,d,\Hseminorm{\alpha,\, (S^2,d)}{\phi},\alpha\bigr)\geq 2C_0>2$ depending only on $f$, $\CC$, $d$, $\Hseminorm{\alpha,\, (S^2,d)}{\phi}$, and $\alpha$ such that $A_0$ increases as $\Hseminorm{\alpha,\, (S^2,d)}{\phi}$ increases, and that for all $\c\in\{\b, \, \w\}$, $x, \, x'\in X^0$, $n\in\N$, union $E\subseteq S^2$ of $n$-tiles in $\X^n(f,\CC)$, $B\in\R$ with $B>0$, and $a, \, b\in\R$ with $\abs{a}\leq 2 s_0$ and $\abs{b} \in\{0\}\cup [1,+\infty)$, the following statements hold:

\begin{enumerate}
\smallskip
\item[(i)] For each $u\in K_B(E,d)$, we have
\begin{equation}   \label{eqBasicIneqR}
\frac{\Absbig{\RR_{\wt{a\phi},\c,E}^{(n)} (u)(x) - \RR_{\wt{a\phi},\c,E}^{(n)} (u)(x')}}
     {  \RR_{\wt{a\phi},\c,E}^{(n)} (u)(x)  + \RR_{\wt{a\phi},\c,E}^{(n)} (u)(x') }  
\leq A_0\biggl( \frac{B}{\Lambda^{\alpha n}} + \frac{\Hseminormbig{\alpha,\, (E,d)}{\wt{a\phi}}}{1-\Lambda^{-\alpha}}  \biggr) d(x,x')^\alpha.
\end{equation}

\smallskip
\item[(ii)] Denote $s \coloneqq a+\I b$. Fix an arbitrary $v\in \Holder{\alpha}((E,d),\C)$. Then
\begin{equation}  \label{eqBasicIneqN1}
       \AbsBig{\RR_{\wt{s\phi}, \c, E}^{(n)} (v)(x) - \RR_{\wt{s\phi}, \c, E}^{(n)} (v)(x')} 
\leq    \biggl( C_0  \frac{ \Hseminorm{\alpha,\, (E,d)}{v} }{\Lambda^{\alpha n}}  + A_0 \max\{1, \, \abs{b}\} \RR_{\wt{a\phi}, \c, E}^{(n)} (\abs{v})(x)    \biggr) d(x,x')^\alpha,
\end{equation}
where $C_0>1$ is the constant from Lemma~\ref{lmMetricDistortion} depending only on $f$, $d$, and $\CC$.

If, in addition, there exists a non-negative real-valued H\"{o}lder continuous function $h\in\Holder{\alpha}(E,d)$ such that 
\begin{equation*}
\abs{v(y)-v(y')} \leq B (h(y)+h(y')) d(y,y')^\alpha
\end{equation*}
whenever $y, \, y'\in E$, then
\begin{align}  \label{eqBasicIneqC}
     & \AbsBig{\RR_{\wt{s\phi}, \c, E}^{(n)} (v)(x) - \RR_{\wt{s\phi}, \c, E}^{(n)} (v)(x')}   \\
&\qquad \leq   A_0 \biggl(  \frac{B}{\Lambda^{\alpha n}}  \Bigl( \RR_{\wt{a\phi}, \c, E}^{(n)} (h)(x) + \RR_{\wt{a\phi}, \c, E}^{(n)} (h)(x')  \Bigr) + \max\{1, \, \abs{b}\} \RR_{\wt{a\phi}, \c, E}^{(n)} (\abs{v})(x)    \biggr) d(x,x')^\alpha.\notag
\end{align}
\end{enumerate}
\end{lemma}

\begin{proof}
Fix $\c$, $n$, $E$, $B$, $a$, and $b$ as in the statement of Lemma~\ref{lmBasicIneq}.

\smallskip

(i) Note that by Lemma~\ref{lmBound_aPhi},
\begin{equation}  \label{eqT0Bound} 
\sup\bigl\{ \Hseminormbig{\alpha,\, (S^2,d)}{\wt{\tau\phi}}  :  \tau \in\R, \abs{ \tau }\leq 2 s_0 + 1  \bigr\} \leq T_0,
\end{equation}
where the constant
\begin{equation}  \label{eqDefT0}
T_0 = T_0\bigl(f,\CC,d,\Hseminorm{\alpha,\, (S^2,d)}{\phi},\alpha\bigr) \coloneqq (2 s_0 + 1) C_{5} \Hseminorm{\alpha,\, (S^2,d)}{\phi} \exp\bigl((2 s_0 + 1) C_{6} \Hseminorm{\alpha,\, (S^2,d)}{\phi}\bigr)>0
\end{equation}
depends only on $f$, $\CC$, $d$, $\Hseminorm{\alpha,\, (S^2,d)}{\phi}$, and $\alpha$. Here $C_{5}>1$ and $C_{6}>0$  are the constants from Lemma~\ref{lmBound_aPhi} depending only on $f$, $\CC$, $d$, and $\alpha$.

Fix $u\in K_B(E,d)$ and $x, \, x'\in X^0_\c$. For each $X^n\in\X^n_\c$, denote $y_{X^n}  \coloneqq (f^n|_{X^n})^{-1}(x)$ and $y'_{X^n} \coloneqq (f^n|_{X^n})^{-1}(x')$.

Then by Definition~\ref{defCone},
\begin{align*}
&      \AbsBig{\RR_{\wt{a\phi},\c,E}^{(n)} (u)(x)  -  \RR_{\wt{a\phi},\c,E}^{(n)} (u)(x')} \\
&\quad\leq   \sum\limits_{\substack{X^n\in\X^n_\c \\ X^n\subseteq E}}  \Abs{u(y_{X^n}) e^{S_n\wt{a\phi}(y_{X^n})} - u(y'_{X^n}) e^{S_n\wt{a\phi}(y'_{X^n})} } \\
&\quad\leq   \sum\limits_{\substack{X^n\in\X^n_\c \\ X^n\subseteq E}}  \Bigl(  \Abs{u(y_{X^n}) -  u(y'_{X^n}) }  e^{S_n\wt{a\phi}(y'_{X^n})}  +  u(y_{X^n}) \AbsBig{e^{S_n\wt{a\phi}(y_{X^n})}  - e^{S_n\wt{a\phi}(y'_{X^n})} }  \Bigr)     \\
&\quad\leq   \sum\limits_{\substack{X^n\in\X^n_\c \\ X^n\subseteq E}} 
            B \Bigl( u(y_{X^n}) e^{S_n\wt{a\phi}(y_{X^n})}    e^{\Absbig{S_n\wt{a\phi}(y'_{X^n}) - S_n\wt{a\phi}(y_{X^n})}}
                    +u(y'_{X^n})e^{S_n\wt{a\phi}(y'_{X^n})}  \Bigr)    d(y_{X^n}, y'_{X^n})^\alpha    \\
&\quad \qquad     +   \sum\limits_{\substack{X^n\in\X^n_\c \\ X^n\subseteq E}}  
             u(y_{X^n}) \AbsBig{1- e^{S_n\wt{a\phi}(y'_{X^n}) - S_n\wt{a\phi}(y_{X^n})} }   e^{S_n\wt{a\phi}(y_{X^n})}  .
\end{align*}
Combining the above with Lemmas~\ref{lmSnPhiBound},~\ref{lmMetricDistortion},~\ref{lmExpToLiniear}, (\ref{eqT0Bound}), and (\ref{eqDefT0}), we get
\begin{align*}
&              \frac{\Absbig{\RR_{\wt{a\phi},\c,E}^{(n)} (u)(x)  -  \RR_{\wt{a\phi},\c,E}^{(n)} (u)(x')}}
                                        {\RR_{\wt{a\phi},\c,E}^{(n)} (u)(x)  +  \RR_{\wt{a\phi},\c,E}^{(n)} (u)(x')  }  \\
&\qquad\leq   B\exp \biggl( \frac{\Hseminormbig{\alpha,\, (S^2,d)}{\wt{a\phi}}  C_0 ( \diam_d(S^2) )^\alpha}{1-\Lambda^{-\alpha}}  \biggr) \frac{d(x,x')^\alpha C_0^\alpha}{\Lambda^{\alpha n}}      + C_{4} \Hseminormbig{\alpha,\, (S^2,d)}{\wt{a\phi}}  d(x,x')^\alpha \\
&\qquad\leq   A_1 \biggl( \frac{B}{\Lambda^{\alpha n}}  + \frac{ \Hseminormbig{\alpha,\, (S^2,d)}{\wt{a\phi}} }{1-\Lambda^{-\alpha}} \biggr) d(x,x')^\alpha,
\end{align*}
where 
\begin{equation}   \label{eqPflmBasicIneq_DefC10}   
C_{4}= C_{4} (f,\CC,d,\alpha,T_0) = \frac{2 C_0 }{1-\Lambda^{-\alpha}} \exp\biggl(\frac{C_0 T_0 }{1-\Lambda^{-\alpha}}\bigl(\diam_d(S^2)\bigr)^\alpha\biggr)
\end{equation}
is the constant from Lemma~\ref{lmExpToLiniear}, and
\begin{equation}    \label{eqDefA1}   
A_1 \coloneqq (1-\Lambda^{-\alpha}) C_{4}(f,\CC,d,\alpha,T_0).
\end{equation}
Both of these constants only depend on $f$, $\CC$, $d$, $\Hseminorm{\alpha,\, (S^2,d)}{\phi}$ and $\alpha$.

Define a constant
\begin{equation}   \label{eqDefA0}
A_0= A_0\bigl(f,\CC,d,\Hseminorm{\alpha,\, (S^2,d)}{\phi},\alpha\bigr) \coloneqq \frac{(1+2T_0) A_1}{1-\Lambda^{-\alpha}}    = (1+2T_0)C_{4}\bigl(f,\CC,d,\alpha,T_0\bigr)>2
\end{equation}
depending only on $f$, $\CC$, $d$, $\Hseminorm{\alpha,\, (S^2,d)}{\phi}$, and $\alpha$. By (\ref{eqDefA0}), (\ref{eqDefT0}), and (\ref{eqPflmBasicIneq_DefC10}), we see that $A_0$ increases as $\Hseminorm{\alpha,\, (S^2,d)}{\phi}$ increases. Now (\ref{eqBasicIneqR}) follows from the fact that $A_0\geq A_1$.

\smallskip

(ii) Fix $x, \, x'\in X^0_\c$. For each $X^n\in\X^n_\c$, denote $y_{X^n} \coloneqq (f^n|_{X^n})^{-1}(x)$ and $y'_{X^n} \coloneqq (f^n|_{X^n})^{-1}(x')$.

Note that by (\ref{eqDefPhiWidetilde}) and (\ref{eqT0Bound}), we have
\begin{equation}  \label{eqPflmBasicIneqC1}
\Hseminormbig{\alpha,\, (S^2,d)}{\wt{s\phi}}  \leq  \Hseminormbig{\alpha,\, (S^2,d)}{\wt{a\phi}} + \Hseminorm{\alpha,\, (S^2,d)}{b\phi} \leq T_0+ \abs{b} \Hseminorm{\alpha,\, (S^2,d)}{\phi} \leq 2T_0 \max\{1, \, \abs{b}\},
\end{equation}
since $T_0 \geq \Hseminorm{\alpha,\, (S^2,d)}{\phi}$ by (\ref{eqDefT0}) and the fact that $C_{5}>1$ from Lemma~\ref{lmBound_aPhi}.

Note that
\begin{align}    \label{eqPflmBasicIneqC2}
&       \AbsBig{\RR_{\wt{s\phi},\c,E}^{(n)} (v)(x) - \RR_{\wt{s\phi},\c,E}^{(n)} (v)(x')}   \\
&\qquad \leq    \sum\limits_{\substack{X^n\in\X^n_\c \\ X^n\subseteq E}} 
             \Abs{ v(y_{X^n})  e^{S_n\wt{s\phi}(y_{X^n})}- v(y'_{X^n}) e^{S_n\wt{s\phi}(y'_{X^n})} } \notag  \\
&\qquad \leq   \sum\limits_{\substack{X^n\in\X^n_\c \\ X^n\subseteq E}}   
             \Bigl(  \Abs{v(y_{X^n}) -  v(y'_{X^n}) }  \AbsBig{ e^{S_n\wt{s\phi}(y'_{X^n})}}  
                   + \abs{v(y_{X^n})} \AbsBig{e^{S_n\wt{s\phi}(y_{X^n})}  - e^{S_n\wt{s\phi}(y'_{X^n})}  } \Bigr).  \notag
\end{align}

We bound the two terms in the last summation above separately.

By Lemmas~\ref{lmSnPhiBound},~\ref{lmExpToLiniear}, (\ref{eqDefA1}), and (\ref{eqPflmBasicIneqC1}), 
\begin{align}   \label{eqPflmBasicIneqC3}
&            \sum\limits_{\substack{X^n\in\X^n_\c \\ X^n\subseteq E}} 
                     \abs{v(y_{X^n})} \Abs{e^{S_n\wt{s\phi}(y_{X^n})}  - e^{S_n\wt{s\phi}(y'_{X^n})}  }    \notag \\
&\qquad  =  \sum\limits_{\substack{X^n\in\X^n_\c \\ X^n\subseteq E}} 
                     \abs{v(y_{X^n})}  \AbsBig{1- e^{S_n\wt{s\phi}(y'_{X^n}) - S_n\wt{s\phi}(y_{X^n})} }  
                      e^{S_n\wt{a\phi}(y_{X^n})}    \notag \\
&\qquad \leq C_{4}(f,\CC,d,\alpha,T_0) \Hseminormbig{\alpha,\, (S^2,d)}{\wt{s\phi}}  d(x,x')^\alpha  \RR_{\wt{a\phi},\c,E}^{(n)} (\abs{v})(x)  \\
&\qquad \leq  A_1  \frac{   2T_0 \max\{1, \, \abs{b}\} \RR_{\wt{a\phi},\c,E}^{(n)} (\abs{v})(x)   }{   1-\Lambda^{-\alpha} }  d(x,x')^\alpha  \notag\\
&\qquad =     A_0  \max\{1, \, \abs{b}\} \RR_{\wt{a\phi}, \c, E}^{(n)} (\abs{v})(x)   d(x,x')^\alpha, \notag
\end{align}
where the last inequality follows from (\ref{eqDefA0}).

By (\ref{eqDefWideTildePsiComplex}), Lemma~\ref{lmMetricDistortion}, and (\ref{eqRDiscontTildeSupDecreasing}) in Lemma~\ref{lmRtildeNorm=1},
\begin{align}   \label{eqPflmBasicIneqC4}
&         \sum\limits_{\substack{X^n\in\X^n_\c \\ X^n\subseteq E}}    
            \Abs{v(y_{X^n}) -  v(y'_{X^n}) }  \AbsBig{ e^{S_n\wt{s\phi}(y'_{X^n})}}  \notag\\
&\qquad \leq    \sum\limits_{\substack{X^n\in\X^n_\c \\ X^n\subseteq E}} 
             \HseminormD{\alpha,\, (E,d)}{v}    d(y_{X^n}, y'_{X^n})^\alpha  e^{S_n\wt{a\phi}(y'_{X^n})}      \\
&\qquad \leq      \HseminormD{\alpha,\, (E,d)}{v}    \frac{d(x,x')^\alpha C_0^\alpha}{\Lambda^{\alpha n}}   \sum_{X^n\in\X^n_\c} e^{S_n\wt{a\phi}(y'_{X^n})}  \notag\\
&\qquad \leq    C_0    \HseminormD{\alpha,\, (E,d)}{v}     \Lambda^{-\alpha n}     d(x,x')^\alpha.  \notag
\end{align}

Thus, (\ref{eqBasicIneqN1}) follows from (\ref{eqPflmBasicIneqC2}), (\ref{eqPflmBasicIneqC3}), and (\ref{eqPflmBasicIneqC4}).

If, in addition, there exists a non-negative real-valued H\"{o}lder continuous function $h\in\Holder{\alpha}(E,d)$ such that 
$
\abs{v(y)-v(y')} \leq B (h(y)+h(y')) d(y,y')^\alpha
$
when $y, \, y'\in E$, then by Lemmas~\ref{lmSnPhiBound},~\ref{lmMetricDistortion}, (\ref{eqT0Bound}), (\ref{eqDefA1}), and (\ref{eqPflmBasicIneq_DefC10}),
\begin{align*}   \label{eqPflmBasicIneqC5}
&             \sum\limits_{\substack{X^n\in\X^n_\c \\ X^n\subseteq E}}   
                  \Abs{v(y_{X^n}) -  v(y'_{X^n}) }   \AbsBig{ e^{S_n\wt{s\phi}(y'_{X^n})}}  \\
&\quad \leq  \sum\limits_{\substack{X^n\in\X^n_\c \\ X^n\subseteq E}} 
                 B \Bigl( h(y_{X^n})  e^{S_n\wt{a\phi}(y_{X^n})}   e^{\Absbig{S_n\wt{a\phi}(y'_{X^n}) - S_n\wt{a\phi}(y_{X^n})}} 
                         +h(y'_{X^n}) e^{S_n\wt{a\phi}(y'_{X^n})}  \Bigr)  d(y_{X^n}, y'_{X^n})^\alpha \notag \\
&\quad \leq  B\exp \Biggl( \frac{ \Hseminormbig{\alpha,\, (S^2,d)}{\wt{a\phi}}  C_0 \bigl(\diam_d(S^2)\bigr)^\alpha  }
                                 {  1-\Lambda^{-\alpha}  }  \Biggr) \frac{  d(x,x')^\alpha C_0^\alpha  }{  \Lambda^{\alpha n}  }   
                     \Bigl( \RR_{\wt{a\phi},\c,E}^{(n)} (h)(x) + \RR_{\wt{a\phi},\c,E}^{(n)} (h)(x')  \Bigr)  \notag\\
&\quad \leq   A_1  B \Lambda^{ - \alpha n }  \Bigl( \RR_{\wt{a\phi},\c,E}^{(n)} (h)(x) + \RR_{\wt{a\phi},\c,E}^{(n)} (h)(x')  \Bigr) d(x,x')^\alpha.   \notag
\end{align*}
Therefore, (\ref{eqBasicIneqC}) follows from (\ref{eqPflmBasicIneqC2}), (\ref{eqPflmBasicIneqC3}), the last inequality, and the fact that $A_0\geq A_1$ from (\ref{eqDefA0}).
\end{proof}

\subsection{Spectral gap}   \label{subsctSplitRuelleOp_SpectralGap}

Let $(X,d)$ be a metric space. A function $h\: [0,+\infty)\rightarrow [0,+\infty)$ is an \defn{abstract modulus of continuity} if it is continuous at $0$, non-decreasing, and $h(0)=0$. Given any constant $\tau \in [0,+\infty]$, and any abstract modulus of continuity $g$, we define the subclass $\CCC_g^\tau((X, d),\C)$ of $\CCC(X,\C)$ as
\begin{align*}
&\CCC_g^\tau((X,d),\C)\\
&\qquad \coloneqq  \bigl\{u\in\CCC(X,\C)  :   \Norm{u}_{\CCC^0(X)}\leq \tau \text{ and for all }x, \, y\in X, \, \Abs{u(x)-u(y)}\leq g(d(x,y))  \bigr\}.
\end{align*}
We denote $\CCC_g^\tau(X,d) \coloneqq \CCC_g^\tau((X,d),\C) \cap \CCC(X)$.

Assume now that $(X,d)$ is compact. Then by the Arzel\`a--Ascoli Theorem, each $\CCC_g^\tau((X,d),\C)$ (resp.\ $\CCC_g^\tau(X,d)$) is precompact in $\CCC(X,\C)$ (resp.\ $\CCC(X)$) equipped with the uniform norm. It is easy to see that each $\CCC_g^\tau((X,d),\C)$ (resp.\ $\CCC_g^\tau(X,d)$) is actually compact. On the other hand, for $u\in\CCC(X,\C)$, we can define an abstract modulus of continuity by
\begin{equation}    \label{eqAbsModContForU}
g(t) \coloneqq \sup \{\abs{u(x)-u(y)}  :  x, \, y \in X, d(x,y) \leq t\}
\end{equation}
for $t\in [0,+\infty)$, so that $u\in \CCC_g^\iota ((X,d),\C)$, where $\iota \coloneqq \Norm{u}_\infty$.

The following lemma is easy to check (see also \cite[Lemma~5.24]{Li17}).

\begin{lemma}  \label{lmChbChbSubsetChB}
Let $(X,d)$ be a metric space. For each pair of constants $\tau_1, \, \tau_2 \geq 0$, each pair of abstract moduli of continuity $g_1, \, g_2$, and each real number $c > 0$, we have
\begin{align*}
\bigl\{u_1 u_2  :  u_1 \in \CCC_{g_1}^{\tau_1} ((X,d),\C),  \, u_2 \in \CCC_{g_2}^{\tau_2} ((X,d),\C) \bigr\}  & \subseteq \CCC_{\tau_1 g_2 + \tau_2 g_1}^{\tau_1 \tau_2} ((X,d),\C) \quad \text{ and} \\
\bigl\{ 1 / u  :   u\in \CCC_{g_1}^{\tau_1} ((X,d),\C),  \, u(x) \geq c \text{ for each } x\in X\bigr\}                    & \subseteq \CCC_{c^{-2} g_1}^{c^{-1}}((X,d),\C).
\end{align*}
\end{lemma}

The following corollary follows immediately from Lemma~\ref{lmChbChbSubsetChB}. We leave the proof to the reader.

\begin{cor}    \label{corProductInverseHolderNorm}
Let $(X,d)$ be a metric space, and $\alpha \in (0,1]$ a constant. Then for all H\"{o}lder continuous functions $u, \, v \in \Holder{\alpha}((X,d),\C)$, we have $u,  \, v \in \Holder{\alpha}((X,d),\C)$ with
\begin{equation*}
 \Hnorm{\alpha}{u v}{(X,d)} \leq \Hnorm{\alpha}{u }{(X,d)}   \Hnorm{\alpha}{v}{(X,d)},
\end{equation*}
and if, in addition, $\abs{u(x)} \geq c$ for each $x\in X$, for some constant $c>0$, then $1/u \in \Holder{\alpha}((X,d),\C)$ with
$\Hnorm{\alpha}{ 1 / u }{(X,d)} \leq c^{-1} +  c^{-2} \Hnorm{\alpha}{u}{(X,d)}$.
\end{cor}

\begin{lemma}   \label{lmRDiscontTildeDual}
Let $f$, $\CC$, $d$, $\alpha$ satisfy the Assumptions. Assume in addition that $f(\CC)\subseteq\CC$. Let $\phi\in \Holder{\alpha}(S^2,d)$ be a real-valued H\"{o}lder continuous function with an exponent $\alpha$, and $\mu_\phi$ denote the unique equilibrium state for $f$ and $\phi$. Fix arbitrary $\c\in\{\b, \, \w\}$ and $u\in\CCC(X^0)$. Then for each $n\in\N$,
\begin{equation*}
    \int_{X^0_\c}\! u\,\mathrm{d}\mu_\phi 
=   \sum_{ \c'\in\{\b, \, \w\} }\int_{X^0_{\c'}} \!  \RR_{\wt{\phi}, \c', \c}^{(n)} (u) \,\mathrm{d}\mu_\phi.
\end{equation*}
\end{lemma}

\begin{proof}
We define a function $v\in B(S^2)$ by setting $v(x)  = u(x)$ if  $x\in \inte \bigl(X^0_\c \bigr)$ and $v(x)  = 0$ otherwise. We choose a pointwise increasing sequence of continuous non-negative functions $\tau_i\in \CCC(S^2)$, $i\in\N$, such that  $\lim_{i\to+\infty} \tau_i(x) = \mathbbm{1}_{\inte (X^0_\c)}$ for all $x\in S^2$. Then $\{v\tau_i\}_{i\in\N}$ is a bounded sequence of continuous functions on $S^2$, convergent pointwise to $v$. 

Fix $n\in\N$. Since $\mu_\phi (\CC) = 0$ by \cite[Proposition~5.39]{Li17}, then by (\ref{eqDefLc}), Proposition~\ref{propCellDecomp}~(i) and (ii), and the Dominated Convergence Theorem, we get
\begin{align*}
&      \sum_{ \c'\in\{\b, \, \w\} } \int_{X^0_{\c'}} \!  \RR_{\wt{\phi}, \c', \c}^{(n)} (u) \,\mathrm{d}\mu_\phi   \\
&\qquad= \sum_{ \c'\in\{\b, \, \w\} } \int_{X^0_{\c'}} \! \sum\limits_{\substack{X^n\in \X^n_{\c'} \\ X^n\subseteq X^0_\c}  } \Bigl(e^{S_n\wt{\phi}} u\Bigr) \bigl((f^n|_{X^n})^{-1} (x)\bigr) \,\mathrm{d}\mu_\phi(x) \\ 
&\qquad = \sum_{ \c'\in\{\b, \, \w\} } \lim_{i\to+\infty}   \int_{\inte( X^0_{\c'} )} \! \sum\limits_{\substack{X^n\in \X^n_{\c'} \\ X^n\subseteq X^0_\c}  } \Bigl(e^{S_n\wt{\phi}} v\tau_i \Bigr) \bigl((f^n|_{X^n})^{-1} (x)\bigr) \,\mathrm{d}\mu_\phi(x) \\
&\qquad= \sum_{ \c'\in\{\b, \, \w\} } \lim_{i\to+\infty}   \int_{\inte( X^0_{\c'}) } \! \RR_{\wt{\phi}}^n (v \tau_i)(x) \,\mathrm{d}\mu_\phi(x)    \\
&\qquad=  \lim_{i\to+\infty}   \int_{S^2} \! \RR_{\wt{\phi}}^n (v \tau_i)  \,\mathrm{d}\mu_\phi \\
&\qquad=  \lim_{i\to+\infty}   \int_{S^2} \! v \tau_i \,\mathrm{d} \bigl(\RR^*_{\wt{\phi}}\bigr)^n (\mu_\phi)  \\
&\qquad=  \lim_{i\to+\infty}   \int_{S^2} \! v \tau_i \,\mathrm{d}   \mu_\phi   \\
&\qquad=     \int_{S^2} \! v   \,\mathrm{d}   \mu_\phi  \\
&\qquad=     \int_{X^0_\c} \! u   \,\mathrm{d}   \mu_\phi . \qedhere
\end{align*}
\end{proof}

\begin{lemma}    \label{lmLSplitTildeSupBound}
Let $f$, $\CC$, $d$ satisfy the Assumptions. Assume in addition that $f(\CC)\subseteq\CC$. Fix an abstract modulus of continuity $g$. Then for each $\alpha \in (0,1]$, each $K\in(0,+\infty)$, and each $\delta_1\in(0,+\infty)$, there exist constants $\delta_2\in (0,+\infty)$ and $N\in\N$ with the following property:

For all $\c\in\{\b, \, \w\}$, $u_\b\in \CCC_g^{+\infty}(X^0_\b,d)$, $u_\w\in \CCC_g^{+\infty}(X^0_\w,d)$, and $\phi\in\Holder{\alpha}(S^2,d)$, if $\Hnorm{\alpha}{\phi}{(S^2,d)} \leq K$, $\max\bigl\{  \norm{u_\b}_{\CCC^0(X^0_\b)} , \, \norm{u_\w}_{\CCC^0(X^0_\w)}  \bigr\} \geq \delta_1$, and $\int_{X^0_\b}\! u_\b\,\mathrm{d}\mu_\phi + \int_{X^0_\w}\! u_\w\,\mathrm{d}\mu_\phi  = 0$ where $\mu_\phi$ denotes the unique equilibrium state for $f$ and $\phi$, then
\begin{equation*}
\NormBig{\RR_{\wt{\phi}, \c, \b}^{(N)}(u_\b)  + \RR_{\wt{\phi}, \c, \w}^{(N)}(u_\w)  }_{\CCC^0(X^0_\c)} \leq \max\bigl\{  \norm{u_\b}_{\CCC^0(X^0_\b)} , \, \norm{u_\w}_{\CCC^0(X^0_\w)}  \bigr\} - \delta_2.
\end{equation*}
\end{lemma}

\begin{proof}
Fix arbitrary constants $\alpha \in (0,1]$, $K\in (0,+\infty)$, and $\delta_1 \in (0,+\infty)$. Choose $\epsilon > 0$ small enough such that $g(\epsilon) < \frac{\delta_1}{2}$. Let $n_0\in\N$ be the smallest number such that $f^{n_0} \bigl( \inte\bigl( X^0_\b \bigr) \bigr)= S^2 = f^{n_0}  \bigl( \inte\bigl( X^0_\w \bigr) \bigr)$.

By Lemma~\ref{lmCellBoundsBM}~(iv), there exists a number $N\in \N$ depending only on $f$, $\CC$, $d$, $g$, and $\delta_1$ such that $N\geq 2n_0$ and for each $z\in S^2$, we have $U^{N-n_0}(z) \subseteq B_d(z,\epsilon)$ (see (\ref{defU^n})).

Fix arbitrary $\c\in\{\b, \, \w\}$, $\phi\in \Holder{\alpha}(S^2,d)$ with $\Hnorm{\alpha}{\phi}{(S^2,d)} \leq K$, and functions $u_\b\in \CCC_g^{+\infty}(X^0_\b,d)$ and $u_\w\in \CCC_g^{+\infty}(X^0_\w,d)$ with $\max\bigl\{  \norm{u_\b}_{\CCC^0(X^0_\b)} , \, \norm{u_\w}_{\CCC^0(X^0_\w)}  \bigr\} \geq \delta_1$ and $\int_{X^0_\b}\! u_\b\,\mathrm{d}\mu_\phi + \int_{X^0_\w}\! u_\w\,\mathrm{d}\mu_\phi  = 0$. Without loss of generality, we assume that $\int_{X^0_\b}\! u_\b\,\mathrm{d}\mu_\phi \leq 0$ and $\int_{X^0_\w}\! u_\w\,\mathrm{d}\mu_\phi \geq 0$. So we can choose points $y_1 \in X^0_\b$ and $y_2\in X^0_\w$ in such a way that $u_\b(y_1) \leq 0$ and $u_\w(y_2)\geq 0$.

We denote
\begin{equation*}
M   \coloneqq   \max\bigl\{  \norm{u_\b}_{\CCC^0(X^0_\b)} , \, \norm{u_\w}_{\CCC^0(X^0_\w)}  \bigr\}.
\end{equation*}

We fix a point $x\in X^0_\c$. Since $f^N \bigl( U^{N-n_0}(y_1) \cap X^0_\b \bigr) = S^2$, there exists $y\in f^{-N}(x)\cap X^0_\b$ such that $y\in U^{N-n_0}(y_1) \subseteq B_d(y_1,\epsilon)$. Since $M\geq \delta_1$,
$u_\b(y)\leq  u_\b(y_1) + g(\epsilon) < \frac{\delta_1}{2} \leq M - \frac{\delta_1}{2}$.
Choose $X^N_y \in \X^N_\c$ such that $y\in X^N_y \subseteq X^0_\b$. Denote $w_{X^N} \coloneqq (f^N|_{X^N})^{-1}(x)$ for each $X^N\in \X^N_\c$. So by Lemma~\ref{lmRtildeNorm=1}, we have
\begin{align*}
&            \RR_{\wt{\phi}, \c, \b}^{(N)} (u_\b) (x)  +  \RR_{\wt{\phi}, \c, \w}^{(N)} (u_\w) (x)      \\
&\qquad =    u_\b(y) e^{S_N\wt{\phi}(y)} 
              + \sum\limits_{\substack{X^N\in \X^N_\c \setminus\{X^N_y\} \\ X^N\subseteq X^0_\b}}  u_\b(w_{X^N}) e^{S_N\wt{\phi}(w_{X^N})}
              + \sum\limits_{\substack{X^N\in \X^N_\c                   \\ X^N\subseteq X^0_\w}}  u_\w(w_{X^N}) e^{S_N\wt{\phi}(w_{X^N})}      \\
&\qquad\leq \biggl( M - \frac{\delta_1}{2} \biggr) \exp\bigl(S_N\wt{\phi}(y)\bigr)  
              +  M  \sum_{ X^N\in \X^N_\c\setminus\{X^N_y\}}  \exp\bigl(S_N\wt{\phi}(w_{X^N})\bigr) \\
&\qquad =    M  \sum_{ X^N\in \X^N_\c  }  \exp\bigl(S_N\wt{\phi}(w_{X^N})\bigr)  -  \frac{\delta_1}{2}   \exp\bigl(S_N\wt{\phi}(y)\bigr)  \\
&\qquad =    M  -  \frac{\delta_1}{2}   \exp\bigl(S_N\wt{\phi}(y)\bigr) .
\end{align*}
Similarly, there exists $z\in f^{-N}(x)\cap X^0_\w$ such that $z\in U^{N-n_0}(y_2) \subseteq B_d(y_2, \epsilon)$ and
\begin{equation*}
\RR_{\wt{\phi}, \c, \b}^{(N)} (u_\b) (x)  +  \RR_{\wt{\phi}, \c, \w}^{(N)} (u_\w) (x) \geq - M  +  \frac{\delta_1}{2}   \exp\bigl(S_N\wt{\phi}(z)\bigr) .
\end{equation*}
Hence, we get
$       \NormBig{ \RR_{\wt{\phi}, \c, \b}^{(N)} (u_\b)  +  \RR_{\wt{\phi}, \c, \w}^{(N)} (u_\w)  }_{ \CCC^0(X^0_\c) }  
\leq  M -  \frac{\delta_1}{2} \inf \bigl\{  \exp\bigl(S_N\wt{\phi}(w)\bigr)   :  w\in S^2  \bigr\}$.

By (\ref{eqBound_aPhi_Sup}) in Lemma~\ref{lmBound_aPhi} with $T\coloneqq 1$, the definition of $M$ above, and (\ref{eqDefHolderNorm}), we have
\begin{equation*}
     \NormBig{\RR_{\wt{\phi}, \c, \b}^{(N)}(u_\b)  + \RR_{\wt{\phi}, \c, \w}^{(N)}(u_\w)  }_{\CCC^0(X_\c^0)}
 \leq \max\Bigl\{  \norm{u_\b}_{\CCC^0(X^0_\b)} , \, \norm{u_\w}_{\CCC^0(X^0_\w)}  \Bigr\} - \delta_2
\end{equation*}
with
$\delta_2   \coloneqq    \frac{\delta_1}{2} \exp ( - N (C_{5} K + \abs{\log(\deg f)}))$, where $C_{5}$ is the constant from Lemma~\ref{lmBound_aPhi} depending only on $f$, $\CC$, $d$, and $\alpha$. Therefore, the constant $\delta_2$ depends only on $f$, $\CC$, $d$, $\alpha$, $g$, $K$, and $\delta_1$.
\end{proof}

\begin{theorem}    \label{thmLDiscontTildeUniformConv}
Let $f\: S^2 \rightarrow S^2$ be an expanding Thurston map with a Jordan curve $\CC\subseteq S^2$ satisfying $f(\CC)\subseteq\CC$ and $\post f\subseteq \CC$. Let $d$ be a visual metric on $S^2$ for $f$ with expansion factor $\Lambda>1$ and $\alpha\in(0,1]$ be a constant. Let $H$, $H_\b$, and $H_\w$ be bounded subsets of $\Holder{\alpha}(S^2,d)$, $\Holder{\alpha} \bigl( X^0_\b,d \bigr)$, and $\Holder{\alpha} \bigl(X^0_\w,d \bigr)$, respectively (with respect to H\"older norms). Then for all $\c\in\{\b, \, \w\}$, $\phi\in H$, $u_\b\in H_\b$, and $u_\w\in H_\w$, we have
\begin{equation}  \label{eqLDiscontTildeUniformConv}
\lim_{n\to+\infty}  \NormBig{ \RR_{\wt{\phi}, \c, \b}^{(n)}  (\overline{u}_\b )    +       \RR_{\wt{\phi}, \c, \w}^{(n)} (\overline{u}_\w)}_{\CCC^0(X^0_\c)} = 0,
\end{equation}
where the pair of functions $\overline{u}_\b \in \Holder{\alpha} \bigl(X^0_\b,d \bigr)$ and $\overline{u}_\w \in \Holder{\alpha} \bigl(X^0_\w,d \bigr)$ are given by
\begin{equation*}
\overline{u}_\b  \coloneqq  u_\b - \int_{X^0_\b}\! u_\b\,\mathrm{d}\mu_\phi - \int_{X^0_\w}\! u_\w\,\mathrm{d}\mu_\phi   \mbox{ and } 
\overline{u}_\w \coloneqq  u_\w - \int_{X^0_\b}\! u_\b\,\mathrm{d}\mu_\phi - \int_{X^0_\w}\! u_\w\,\mathrm{d}\mu_\phi
\end{equation*}
with $\mu_\phi$ denoting the unique equilibrium state for $f$ and $\phi$.

Moreover, the convergence in (\ref{eqLDiscontTildeUniformConv}) is uniform in $\phi\in H$, $u_\b\in H_\b$, and $u_\w\in H_\w$.
\end{theorem}

\begin{proof}
Without loss of generality, we assume that $H\neq\emptyset$, $H_\b\neq\emptyset$, and $H_\w\neq\emptyset$. Define constants $K  \coloneqq  \sup \bigl\{ \Hnorm{\alpha}{\phi}{(S^2,d)}  :  \phi\in H\bigr\} \in [0,+\infty)$ and $K_\c  \coloneqq  \sup \bigl\{ \Hnorm{\alpha}{u_\c}{(X^0_\b,d)}  :  u_\c\in H_\c\bigr\}  \in [0,+\infty)$ for $\c\in\{\b, \, \w\}$. Define for each $n\in\N_0$,
\begin{equation*}
a_n  \coloneqq  \sup \Bigl\{  \NormBig{ \RR_{\wt{\phi}, \c, \b}^{(n)}  (\overline{u}_\b )    +       \RR_{\wt{\phi}, \c, \w}^{(n)} (\overline{u}_\w)}_{\CCC^0(X^0_\c)}   :   \c\in\{\b, \, \w\},\, \phi\in H, \, u_\b\in H_\b, \, u_\w\in H_\w \Bigr\}.
\end{equation*}
Note that by Definition~\ref{defSplitRuelle}, $a_0\leq 2K_\b+2K_\w <+\infty$.

By (\ref{eqLDiscontProperties_Split}) in Lemma~\ref{lmLDiscontProperties} and (\ref{eqSplitRuelleTildeSupDecreasing}) in Lemma~\ref{lmRtildeNorm=1}, for all $n\in\N_0$, $\phi\in H$, $\c\in\{\b, \, \w\}$, $v_\b\in\CCC(X^0_\b)$, and $v_\w\in\CCC(X^0_\w)$, we have
\begin{align*}
&   \NormBig{  \RR_{\wt{\phi}, \c, \b}^{(n+1)}  (v_\b) +  \RR_{\wt{\phi}, \c, \w}^{(n+1)}  (v_\w)   }_{ \CCC^0(X^0_\c,d)}   \\
&\qquad =   \Normbigg{   \sum_{\c'\in\{\b, \, \w\}} \  \RR_{\wt{\phi}, \c, \c'}^{(1)}  \Bigl(  \RR_{\wt{\phi}, \c', \b}^{(n)}  (v_\b) + \RR_{\wt{\phi}, \c', \w}^{(n)}  (v_\w)\Bigr)   }_{ \CCC^0(X^0_\c,d)}    \\
 &\qquad \leq\max\Bigl\{   \NormBig{   \RR_{\wt{\phi}, \c', \b}^{(n)}  (v_\b) + \RR_{\wt{\phi}, \c', \w}^{(n)}  (v_\w)   }_{ \CCC^0(X^0_{\c'},d)}   :  \c'\in\{\b, \, \w\}    \Bigr\}.
\end{align*}
So $\{a_n\}_{n\in\N_0}$ is a non-increasing sequence of non-negative real numbers. 

Suppose now that $\lim_{n\to+\infty} a_n   \eqqcolon   a_* >0$. By Lemma~\ref{lmRtildeNorm=1}, (\ref{eqBasicIneqN1}) in Lemma~\ref{lmBasicIneq} with $a \coloneqq 1$ and $b  \coloneqq 0$, (\ref{eqDefA0}), (\ref{eqDefT0}), and (\ref{eqDefC10}), we get that $\RR_{\wt{\phi},  \c,  \b}^{(n)} (\overline{u}_\b)  + \RR_{\wt{\phi},  \c,  \w}^{(n)} (\overline{u}_\w)  \in \CCC^{2(K_\b+K_\w)}_g(X^0_\c,d)$, for each $\c\in\{\b, \, \w\}$ and each pair of $u_\b\in H_\b$ and $u_\w\in H_\w$, with an abstract modulus of continuity $g$ given by $g(t)  \coloneqq  2 (C_0(K_\b+K_\w)+2(K_\b+K_\w)A)t^\alpha$, $t\in[0,+\infty)$, where the constant $A>1$ is given by
\begin{equation*}
A  \coloneqq (1+2T) \frac{ 2C_0 }{1-\Lambda^{-\alpha}} \exp\Bigl( \frac{C_0 T}{1-\Lambda^{-\alpha}} \bigl( \diam_d(S^2)\bigr)^\alpha   \Bigr),
\end{equation*}
and $T  \coloneqq  (2 s_0 + 1)  C_{5} K \exp ( 2 s_0 C_{6} K)$. Here the constant $C_0 > 1$ depending only on $f$, $d$, and $\CC$ comes from Lemma~\ref{lmMetricDistortion}, and $C_{5}>1$, $C_{6}>0$ are the constants from Lemma~\ref{lmBound_aPhi} depending only on $f$, $\CC$, $d$, and $\alpha$. So $g$ and $A$ both depend only on $f$, $\CC$, $d$, $\alpha$, $H$, $H_\b$, and $H_\w$. By Lemma~\ref{lmRDiscontTildeDual},
\begin{equation*}
\sum_{\c\in\{\b, \, \w\}}  \int_{X^0_\c} \!  \Bigl(  \RR_{\wt{\phi}, \c, \b}^{(n)}  (\overline{u}_\b) + \RR_{\wt{\phi}, \c, \w}^{(n)}  (\overline{u}_\w)  \Bigr) \,\mathrm{d}\mu_\phi 
=  \int_{X^0_\b} \!  \overline{u}_\b \,\mathrm{d}\mu_\phi + \int_{X^0_\w} \!  \overline{u}_\w \,\mathrm{d}\mu_\phi = 0.
\end{equation*}
By  (\ref{eqLDiscontProperties_Split}) in Lemma~\ref{lmLDiscontProperties}, (\ref{eqSplitRuelleTildeSupDecreasing}) in Lemma~\ref{lmRtildeNorm=1}, and applying Lemma~\ref{lmLSplitTildeSupBound} with $f$, $\CC$, $d$, $g$, $\alpha$, $K$, and $\delta_1 \coloneqq  \frac{a_*}{2}>0$, we find constants $N\in\N$ and $\delta_2>0$ such that
\begin{align*}
&           \NormBig{  \RR_{\wt{\phi}, \c, \b}^{(N+n)} (\overline{u}_\b)  + \RR_{\wt{\phi}, \c, \w}^{(N+n)} (\overline{u}_\w)      }_{\CCC^0(X^0_\c)}   \\
&\qquad   = \NormBig{ \sum_{\c'\in\{\b, \, \w\}} \RR_{\wt{\phi}, \c, \c'}^{(N)} \Bigl(  \RR_{\wt{\phi}, \c', \b}^{(n)} (\overline{u}_\b ) +   \RR_{\wt{\phi}, \c', \w}^{(n)} (\overline{u}_\w) \Bigr)       }_{\CCC^0(X^0_\c)} \\
&\qquad  \leq \max \Bigl\{  \NormBig{  \RR_{\wt{\phi}, \c', \b}^{(n)} (\overline{u}_\b ) +   \RR_{\wt{\phi}, \c', \w}^{(n)} (\overline{u}_\w)  }_{\CCC^0(X^0_{\c'})}  :  \c'\in\{\b, \, \w\}       \Bigr\}   - \delta_2 \\
&\qquad  \leq   a_n - \delta_2,
\end{align*}
for all $n\in\N_0$, $\c\in\{\b, \, \w\}$, $\phi\in H$, $u_\b\in H_\b$, and $u_\w\in H_\w$ satisfying
\begin{equation}   \label{eqPfthmLDiscontTildeUniformConv_Geq}
\max \Bigl\{  \NormBig{  \RR_{\wt{\phi}, \c', \b}^{(n)} (\overline{u}_\b ) +   \RR_{\wt{\phi}, \c', \w}^{(n)} (\overline{u}_\w)  }_{\CCC^0(X^0_{\c'})}  :  \c'\in\{\b, \, \w\}       \Bigr\}    \geq  a_* / 2 .
\end{equation}
Since $\lim_{n\to+\infty} a_n = a_*$, we can fix $m\geq 1$ large enough so that $a_m\leq a_* +\frac{\delta_2}{2}$. Then for each $\c\in\{\b, \, \w\}$, each $\phi\in H$, and each pair $u_\b\in H_\b$ and $u_\w\in H_\w$ satisfying (\ref{eqPfthmLDiscontTildeUniformConv_Geq}) with $n  \coloneqq  m$, we have 
\begin{equation*}
 \NormBig{  \RR_{\wt{\phi}, \c, \b}^{(N+m)} (\overline{u}_\b)  + \RR_{\wt{\phi}, \c, \w}^{(N+m)} (\overline{u}_\w)      }_{\CCC^0(X^0_\c)} 
\leq a_m - \delta_2 \leq a_* - 2^{-1} \delta_2 .
\end{equation*}
On the other hand, by (\ref{eqSplitRuelleTildeSupDecreasing}) in Lemma~\ref{lmRtildeNorm=1}, for all $\phi \in H$, $u_\b\in H_\b$, and $u_\w\in H_\w$ with 
$\max \bigl\{  \Normbig{  \RR_{\wt{\phi}, \c', \b}^{(m)} (\overline{u}_\b ) +   \RR_{\wt{\phi}, \c', \w}^{(m)} (\overline{u}_\w)  }_{\CCC^0(X^0_{\c'})}  :  \c'\in\{\b, \, \w\}      \bigr\}  <  a_* / 2$,
the following inequality
\begin{equation*}
 \NormBig{  \RR_{\wt{\phi}, \c, \b}^{(N+m)} (\overline{u}_\b)  + \RR_{\wt{\phi}, \c, \w}^{(N+m)} (\overline{u}_\w)      }_{\CCC^0(X^0_\c)} 
<   a_* / 2  
\end{equation*}
 holds for each $\c\in\{\b, \, \w\}$. Thus, $a_{N+m} \leq \max \bigl\{ a_*- \frac{\delta_2}{2}, \, \frac{a_*}{2} \bigr\}  < a_*$, contradicting the fact that $\{a_n\}_{n\in\N_0}$ is a non-increasing sequence and the assumption that $\lim_{n\to+\infty}  a_n = a_*>0$. This proves the uniform convergence in (\ref{eqLDiscontTildeUniformConv}).
\end{proof}

\begin{theorem}    \label{thmSpectralGap}
Let $f\: S^2 \rightarrow S^2$ be an expanding Thurston map with a Jordan curve $\CC\subseteq S^2$ satisfying $f(\CC)\subseteq\CC$ and $\post f\subseteq \CC$. Let $d$ be a visual metric on $S^2$ for $f$ with expansion factor $\Lambda>1$. Let $\alpha\in(0,1]$ be a constant and $H$ be a bounded subset of $\Holder{\alpha}(S^2,d)$ with respect to the H\"older norm. Then there exists a constant $\rho_1  \in(0,1)$ depending on $f$, $\CC$, $d$, $\alpha$, and $H$ such that the following property holds:

\smallskip

For all $\phi\in H$, $n\in\N_0$, $\c\in\{\b, \, \w\}$, $u_\b\in \Holder{\alpha} \bigl(X^0_\b,d \bigr)$, and $u_\w\in \Holder{\alpha} \bigl(X^0_\w,d \bigr)$, we have
\begin{equation}  \label{eqSpectralGap}
 \NormBig{  \sum_{\c' \in \{ \b, \, \w \} } \RR_{\wt{\phi}, \c, \c'}^{(n)} ( \overline{u}_{\c'}) }_{\CCC^0(X^0_\c)}   
 \leq 6 \rho_1^n \max_{ \c' \in \{ \b, \, \w \} } \Bigl\{    \Hnorm{\alpha}{u_{\c'} }{(X^0_{\c'},d)}    \Bigr\},
\end{equation}
where $\overline{u}_{\c'} \in \Holder{\alpha} \bigl(X^0_{\c'},d \bigr)$ for $c'\in\{ \b, \, \w \}$ are given by
$\overline{u}_{\c'}   \coloneqq   u_{\c'} - \int_{X^0_\b}\! u_\b\,\mathrm{d}\mu_\phi - \int_{X^0_\w}\! u_\w\,\mathrm{d}\mu_\phi$, 
with $\mu_\phi$ denoting the unique equilibrium state for $f$ and $\phi$. In particular,
\begin{equation}   \label{eqSpectralGapAll}
      \NormBig{  \sum_{\c' \in \{ \b, \, \w \} } \RR_{\wt{\phi}, \c, \c'}^{(n)} (u_{\c'}) }_{\CCC^0(X^0_\c)}   
\leq  \Absbigg{  \int_{X^0_\b}\! u_\b\,\mathrm{d}\mu_\phi + \int_{X^0_\w}\! u_\w\,\mathrm{d}\mu_\phi  } 
+  6 \rho_1^n \max_{ \c' \in \{ \b, \, \w \} } \Bigl\{    \Hnorm{\alpha}{u_{\c'} }{(X^0_{\c'},d)}    \Bigr\}.
\end{equation}
\end{theorem}

\begin{proof}
Without loss of generality, we assume that $H\neq \emptyset$. Define a constant 
\begin{equation}   \label{eqPfthmSpectralGap_K}
K   \coloneqq   \sup\bigl\{  \Hnorm{\alpha}{\phi}{(S^2,d)}   :  \phi\in H   \bigr\} \in [0,+\infty).
\end{equation}
Denote, for each $\c\in\{\b, \, \w\}$,
\begin{equation*}
H_\c    \coloneqq    \bigl\{  v_\c\in \Holder{\alpha} \bigl(X^0_\c,d \bigr)  :  \Hnorm{\alpha}{v_\c}{(X^0_\c,d)} \leq 3      \bigr\}.
\end{equation*}

Inequality~(\ref{eqSpectralGapAll}) follows immediately from (\ref{eqSpectralGap}), the triangle inequality, and the fact that $\RR_{\wt{\phi}, \c, \b}^{(1)} \bigl(\mathbbm{1}_{X^0_\b}\bigr) + \RR_{\wt{\phi}, \c, \w}^{(1)} \bigl(\mathbbm{1}_{X^0_\w}\bigr) = \mathbbm{1}_{X^0_\c}$ by (\ref{eqDefLc}) and Lemma~\ref{lmRtildeNorm=1}. So it suffices to establish (\ref{eqSpectralGap}).

We first consider the special case where $u_\b \in H_\b$ and $u_\w \in H_\w$.

By (\ref{eqBasicIneqN1}) in Lemma~\ref{lmBasicIneq} with $s\coloneqq 1$, (\ref{eqRDiscontTildeSupDecreasing}) in Lemma~\ref{lmRtildeNorm=1}, and (\ref{eqPfthmSpectralGap_K}), for all $j\in \N$, $\c\in\{\b, \, \w\}$, $\phi\in H$, $u_\b\in H_\b$, and $u_\w\in H_\w$, we have
\begin{align}   \label{eqPfthmSpectralGap_HolderBound}
              \Hseminormbigg{\alpha,\, (X^0_\c,d)}{   \sum_{\c'\in\{\b, \, \w\}}    \RR_{\wt{\phi}, \c, \c'}^{(j)}( \overline{u}_{\c'} )   }    
&\leq  \frac{C_0}{\Lambda^{\alpha j}}    \sum_{\c'\in\{\b, \, \w\}}  \Hseminorm{\alpha,\, (X^0_{\c'},d)}{ \overline{u}_{\c'} }  
                          +  A_0    \sum_{\c'\in\{\b, \, \w\}}   \NormBig{ \RR_{\wt{\phi}, \c, \c'}^{(j)}  (\abs{\overline{u}_{\c'} })}_{\CCC^0(X^0_\c)}   \notag\\
&\leq  \frac{6C_0}{\Lambda^{\alpha j}}  + A_0    \sum_{\c'\in\{\b, \, \w\}}    \Norm{\overline{u}_{\c'}}_{\CCC^0(X^0_{\c'})}   
                 \leq    C_{9}, 
\end{align}
where the constant $C_{9}$ is given by $C_{9} \coloneqq 6 C_0 + 12 A_0$, the constant $A_0 \coloneqq A_0 \bigl(f,\CC,d,K, \alpha  \bigr)>2$ defined in (\ref{eqDefA0}) from Lemma~\ref{lmBasicIneq} depends only on $f$, $\CC$, $d$, $H$, and $\alpha$,  and the constant $C_0>1$ from Lemma~\ref{lmMetricDistortion} depends only on $f$, $\CC$, and $d$. Thus, $C_{9}>1$ depends only on $f$, $\CC$, $d$, and $H$.

So by (\ref{eqLDiscontProperties_Split}) in Lemma~\ref{lmLDiscontProperties}, (\ref{eqBasicIneqN1}) in Lemma~\ref{lmBasicIneq} with $s \coloneqq 1$, (\ref{eqPfthmSpectralGap_HolderBound}), and (\ref{eqRDiscontTildeSupDecreasing}) in Lemma~\ref{lmRtildeNorm=1}, we get that for all $k\in\N$,
\begin{align}   \label{eqPfthmSpectralGap_JK}
&            \HseminormBig{\alpha,\, (X^0_\c,d)}{    \RR_{\wt{\phi}, \c, \b}^{(k+j)}( \overline{u}_\b )   +  \RR_{\wt{\phi}, \c, \w}^{(k+j)}( \overline{u}_\w ) }    \notag\\
&\qquad\leq  \sum_{\c'\in\{\b, \, \w\}}   \HseminormBig{ \alpha,\, (X^0_\c,d)}{     \RR_{\wt{\phi}, \c, \c'}^{(k)} \Bigl( \RR_{\wt{\phi}, \c', \b}^{(j)}( \overline{u}_\b )  +     \RR_{\wt{\phi}, \c', \w}^{(j)}( \overline{u}_\w ) \Bigr)  } \notag\\
&\qquad\leq  \sum_{\c'\in\{\b, \, \w\}}  \biggl(   \frac{C_0}{\Lambda^{\alpha k}}   \HseminormBig{ \alpha,\, (X^0_{\c'},d)}{    \RR_{\wt{\phi}, \c', \b}^{(j)}( \overline{u}_\b ) 	 +     \RR_{\wt{\phi}, \c', \w}^{(j)}( \overline{u}_\w )   }   \\
&\qquad\qquad\qquad\qquad +   A_0   \NormBig{     \RR_{\wt{\phi}, \c, \c'}^{(k)} \Bigl( \AbsBig{ \RR_{\wt{\phi}, \c', \b}^{(j)}( \overline{u}_\b )  +   \RR_{\wt{\phi}, \c', \w}^{(j)}( \overline{u}_\w ) }   \Bigr)  }_{\CCC^0(X^0_\c)} \biggr)    \notag\\
&\qquad\leq  \frac{2C_0 C_{9}}{\Lambda^{\alpha k}}  + A_0  \sum_{\c'\in\{\b, \, \w\}}  \NormBig{ \RR_{\wt{\phi}, \c', \b}^{(j)}( \overline{u}_\b )  +  \RR_{\wt{\phi}, \c',\w}^{(j)}( \overline{u}_\w )}_{\CCC^0(X^0_{\c'})}.  \notag
\end{align}

By Theorem~\ref{thmLDiscontTildeUniformConv}, we can choose $N_0\in \N$ with the property that 
\begin{equation}   \label{eqPfthmSpectralGap_N0}
\frac{2C_0 C_{9}}{\Lambda^{\alpha j}} \leq \frac18    \quad \mbox{ and } \quad
 ( 1+ A_0 )    \NormBig{ \RR_{\wt{\phi}, \c, \b}^{(j)}( \overline{u}_\b )  +   \RR_{\wt{\phi}, \c, \w}^{(j)}( \overline{u}_\w )}_{\CCC^0(X^0_\c)}    \leq \frac18,
\end{equation}
for all $j\in \N$ with $j\geq N_0$, $\c\in\{\b, \, \w\}$, $\phi\in H$, $u_\b\in H_\b$, and $u_\w\in H_\w$. We set $N_0\in\N$ to be the smallest integer with this property. So $N_0$ depends only on $f$, $\CC$, $d$, $\alpha$, and $H$.

For each $m\in \N$, each $\c\in\{\b, \, \w\}$, each $\phi\in H$, and each pair of functions $u_\b\in H_\b$ and $u_\w\in H_\w$, we denote
\begin{equation}    \label{eqPfthmSpectralGap_v_mc}
v_{m,\c}   \coloneqq    \RR_{\wt{\phi}, \c, \b}^{(2N_0 m)} (\overline{u}_\b)  + \RR_{\wt{\phi}, \c, \w}^{(2N_0 m)} (\overline{u}_\w).
\end{equation}
Then by (\ref{eqPfthmSpectralGap_JK}) and (\ref{eqPfthmSpectralGap_N0}), the function $v_{m,\c} \in \Holder{\alpha} \bigl(X^0_\c,d \bigr)$ satisfies
$\Hnorm{\alpha}{v_{m,\c}}{(X^0_\c,d)}   \leq 3 / 8$.
So $2 v_{m,\c} \in H_\c$. We also note that by Lemma~\ref{lmRDiscontTildeDual},
\begin{equation*}  
        \sum_{\c\in\{\b, \, \w\}}  \int_{X^0_\c} \! v_{m,\c} \,\mathrm{d}\mu_\phi 
 =  \sum_{\c\in\{\b, \, \w\}} \sum_{\c'\in\{\b, \, \w\}} \int_{X^0_\c} \! \RR_{\wt{\phi}, \c, \c'}^{(2N_0 m)} (\overline{u}_{\c'})   \,\mathrm{d}\mu_\phi 
 =  \sum_{\c'\in\{\b, \, \w\}}   \int_{X^0_{\c'}}  \! \overline{u}_{\c'} \,\mathrm{d}\mu_\phi  
 =  0.                                   
\end{equation*}

\smallskip

Next, we prove by induction that for each $m\in\N$, each $\phi\in H$, and each pair of functions $u_\b\in H_\b$ and $u_\w\in H_\w$, we have
\begin{equation}   \label{eqPfthmSpectralGap_Induction}
\max\bigl\{   \Hnorm{\alpha}{v_{m,\b}}{(X^0_\b,d)} ,  \, \Hnorm{\alpha}{v_{m,\w}}{(X^0_\w,d)}   \bigr\}  \leq 3 ( 1 / 2 )^m.
\end{equation}

We have already shown that (\ref{eqPfthmSpectralGap_Induction}) holds for $m=1$. 

Assume that (\ref{eqPfthmSpectralGap_Induction}) holds for $m=j$ for some $j\in\N$, then $2^j v_{j,\b} \in H_\b$ and $2^j v_{j,\w} \in H_\w$. By (\ref{eqLDiscontProperties_Split}) in Lemma~\ref{lmLDiscontProperties}, for each $\c\in\{\b, \, \w\}$, 
$2^j v_{j+1,\c}  = \RR_{\wt{\phi}, \c, \b}^{(2N_0)} (2^j v_{j,\b})  +  \RR_{\wt{\phi}, \c, \w}^{(2N_0)} (2^j v_{j,\w})$.
Thus, $\Hnorm{\alpha}{   2^j v_{j+1,\c}}{(X^0_\c,d)}  \leq 3 / 8 < 1 / 2$. So  $\Hnorm{\alpha}{  v_{j+1,\c}}{(X^0_\c,d)} \leq     ( 1 / 2 )^{j+1} < 3 ( 1 / 2 )^{j+1}$.

The induction is now complete.

\smallskip

Then by (\ref{eqLDiscontProperties_Split}) in Lemma~\ref{lmLDiscontProperties}, (\ref{eqSplitRuelleTildeSupDecreasing}) in Lemma~\ref{lmRtildeNorm=1}, (\ref{eqPfthmSpectralGap_v_mc}), and (\ref{eqPfthmSpectralGap_Induction}), the following holds for all $j\in \N$, $m\in\N_0$, $\c\in\{\b, \, \w\}$, $\phi\in H$, $u_\b\in H_\b$, and $u_\w\in H_\w$:
\begin{align*}
&           \NormBig{  \RR_{\wt{\phi}, \c, \b}^{(j+ 2N_0 m)} (\overline{u}_\b) +  \RR_{\wt{\phi}, \c, \w}^{(j+ 2N_0 m)} (\overline{u}_\w)  }_{\CCC^0(X^0_\c)}    \\
&\qquad =   \NormBig{ \sum_{\c'\in\{\b, \, \w\}} \RR_{\wt{\phi}, \c, \c'}^{(j)} \Bigl(\RR_{\wt{\phi}, \c', \b}^{(2N_0 m)} (\overline{u}_\b) +    \RR_{\wt{\phi}, \c', \w}^{(2N_0 m)} (\overline{u}_\w) \Bigr)  }_{\CCC^0(X^0_\c)}    \\
&\qquad\leq  \max \Bigl\{   \NormBig{    \RR_{\wt{\phi}, \c', \b}^{(2N_0 m)} (\overline{u}_\b) +    \RR_{\wt{\phi}, \c', \w}^{(2N_0 m)} (\overline{u}_\w)    }_{\CCC^0(X^0_{\c'})}    :   \c'\in\{\b, \, \w\}   \Bigr\}   \\
&\qquad\leq  3    ( 1 / 2 )^m.
\end{align*}
Hence, for each $n\in\N_0$,
\begin{equation}    \label{eqPfthmSpectralGap_Half}
 \NormBig{  \RR_{\wt{\phi}, \c, \b}^{(n)} (\overline{u}_\b) +  \RR_{\wt{\phi}, \c, \w}^{(n)} (\overline{u}_\w)  }_{\CCC^0(X^0_\c)}
\leq 3  ( 1 / 2 )^{ \lfloor \frac{n}{2N_0} \rfloor}
\leq 6 \rho_1^n,
\end{equation}
where the constant 
$
\rho_1 \coloneqq 2^{- 1 / ( 2N_0 ) }
$
depends only on $f$, $\CC$, $d$, $\alpha$, and $H$.

Finally, we consider the general case. For each pair of functions $w_\b \in \Holder{\alpha} \bigl(X^0_\b,d \bigr)$ and $w_\w \in \Holder{\alpha} \bigl(X^0_\w,d \bigr)$, we denote
$M  \coloneqq \max\bigl\{    \Hnorm{\alpha}{w_\b}{(X^0_\b,d)}, \, \Hnorm{\alpha}{w_\w}{(X^0_\w,d)}  \bigr\}$ and
$\overline{w}_{\c'} \coloneqq w_{\c'} - \int_{X^0_\b}\! w_\b\,\mathrm{d}\mu_\phi - \int_{X^0_\w}\! w_\w\,\mathrm{d}\mu_\phi$ for each $\c' \in \{ \b, \, \w \}$. Let $u_\b \coloneqq \frac{1}{M}w_\b$ and $u_\w \coloneqq \frac{1}{M}w_\w$. Then clearly $u_\b\in H_\b$, $u_\w\in H_\w$, $\overline{u}_\b = \frac{1}{M}\overline{w}_\b$, and $\overline{u}_\w = \frac{1}{M}\overline{w}_\w$. Therefore, by (\ref{eqPfthmSpectralGap_Half}), for each $n\in\N_0$, each $\phi\in H$, and each $\c\in\{\b, \, \w\}$,
\begin{equation*}
 \NormBig{ \RR_{\wt{\phi}, \c, \b}^{(n)}  \Bigl(\frac{\overline{w}_\b}{M}  \Bigr)    +       \RR_{\wt{\phi}, \c, \w}^{(n)} \Bigl(\frac{\overline{w}_\w}{M} \Bigr)}_{\CCC^0(X^0_\c)}
 \leq 6 \rho_1^n.
\end{equation*}
Now (\ref{eqSpectralGap}) follows. This completes the proof.
\end{proof}

\begin{rem}   \label{rmSpectralGap}
For $\phi\in \Holder{\alpha}(S^2, d)$, the existence of the spectral gap for the split Ruelle operator $\RRR_{\wt{\phi}}$  on 
$
\Holder{\alpha}\bigl(X^0_\b,d\bigr) \times \Holder{\alpha}\bigl(X^0_\w,d\bigr)
$
follows immediately from (\ref{eqSplitRuelleCoordinateFormula}) in Lemma~\ref{lmSplitRuelleCoordinateFormula}, Theorem~\ref{thmSpectralGap}, and Lemma~\ref{lmBasicIneq}~(ii).
\end{rem}

Finally, we establish the following lemma that will be used in Section~\ref{sctPreDolgopyat}.

\begin{lemma}    \label{lmSplitRuelleUnifBoundNormalizedNorm}
Let $f$, $\CC$, $d$, $\alpha$, $\phi$, $s_0$ satisfy the Assumptions. Assume in addition $f(\CC)\subseteq\CC$. Then for all  $n\in\N$ and $s\in\C$ satisfying $\abs{\Re(s)} \leq 2 s_0$ and $\abs{\Im(s)}\geq 1$, we have
\begin{equation}     \label{eqSplitRuelleUnifBoundNormalizedNorm}
\NOpHnormDbig{\alpha}{\Im(s)}{\RRR_{\wt{s\phi}}^n}  \leq 4 A_0,
\end{equation}
and more generally,
\begin{equation}   \label{eqSplitRuelleUnifBoundNormalizedNorm_m}
\NHnormBig{\alpha}{\Im(s)}{ \Bigl( \RR_{\wt{s\phi}, \c, \b}^{(n)} (u_\b) + \RR_{\wt{s\phi}, \c, \w}^{(n)} (u_\w)  \Bigr)^m }{(X^0_\c,d)}  \leq  (3m+1)A_0
\end{equation}  
for all $m\in\N$, $\c\in\{\b, \, \w\}$, $u_\b \in \Holder{\alpha} \bigl( \bigl(X^0_\b, d \bigr),\C \bigr)$, and $u_\w \in \Holder{\alpha} \bigl( \bigl(X^0_\w, d \bigr),\C \bigr)$ satisfying 
\begin{equation}  \label{eqSplitRuelleUnifBoundNormalizedNorm_u}
\NHnorm{\alpha}{\Im(s)}{u_\b}{(X^0_\b,d)} \leq 1 \quad \mbox{ and } \quad \NHnorm{\alpha}{\Im(s)}{u_\w}{(X^0_\w,d)}  \leq 1.
\end{equation}
Here $A_0=A_0\bigl(f,\CC,d,\Hseminorm{\alpha,\, (S^2,d)}{\phi},\alpha\bigr)\geq 2C_0>2$ is the constant from Lemma~\ref{lmBasicIneq} depending only on $f$, $\CC$, $d$, $\Hseminorm{\alpha,\, (S^2,d)}{\phi}$, and $\alpha$, and $C_0 > 1$ is the constant depending only on $f$, $\CC$, and $d$ from Lemma~\ref{lmMetricDistortion}.
\end{lemma}

\begin{proof}
Fix $n, \, m\in\N$, $\c\in\{\b, \, \w\}$, and $s=a+\I b$ with $a, \, b\in\R$ satisfying $\abs{a}\leq 2 s_0$ and $\abs{b}\geq 1$. Choose arbitrary $u_\b \in \Holder{\alpha} \bigl( \bigl(X^0_\b, d \bigr),\C \bigr)$ and $u_\w \in \Holder{\alpha} \bigl( \bigl(X^0_\w, d \bigr),\C \bigr)$ satisfying (\ref{eqSplitRuelleUnifBoundNormalizedNorm_u}). Denote $M \coloneqq \NormBig{ \RR_{\wt{s\phi}, \c, \b}^{(n)} (u_\b) + \RR_{\wt{s\phi}, \c, \w}^{(n)} (u_\w) }_{\CCC^0(X^0_\c)}$. By (\ref{eqSplitRuelleTildeSupDecreasing}) in Lemma~\ref{lmRtildeNorm=1}, $M\leq 1$. 

We then observe that for each $v\in \Holder{\alpha} ( (X,d_0 ),\C )$ on a compact metric space $(X,d_0)$, we have 
$\Hseminorm{\alpha,\, (X,d_0)}{v^m} \leq m \norm{v}_{\CCC^0(X)}^{m-1}\Hseminorm{\alpha,\, (X,d_0)}{v}$. Thus, we get from (\ref{eqBasicIneqN1}) in Lemma~\ref{lmBasicIneq}, (\ref{eqRDiscontTildeSupDecreasing}) in Lemma~\ref{lmRtildeNorm=1}, (\ref{eqSplitRuelleUnifBoundNormalizedNorm_u}), and the observation above that
\begin{align*}
&            \NHnormBig{\alpha}{b}{ \Bigl( \RR_{\wt{s\phi}, \c, \b}^{(n)} (u_\b) + \RR_{\wt{s\phi}, \c, \w}^{(n)} (u_\w) \Bigr)^m }{(X^0_\c,d)}   \\
&\qquad =     M^m + \frac{1}{\abs{b}} \HseminormBig{\alpha,\, (X^0_\c,d)}{ \Bigl( \RR_{\wt{s\phi}, \c, \b}^{(n)} (u_\b) + \RR_{\wt{s\phi}, \c, \w}^{(n)} (u_\w) \Bigr)^m}  \\  
&\qquad\leq  1 +  mM^{m-1} \abs{b}^{-1}  \HseminormBig{\alpha,\, (X^0_\c,d)}{   \RR_{\wt{s\phi}, \c, \b}^{(n)} (u_\b) + \RR_{\wt{s\phi}, \c, \w}^{(n)} (u_\w)  } \\      
&\qquad\leq  1 +   m    C_0 \Lambda^{  - \alpha n}    \sum_{\c'\in\{\b, \, \w\}}    \NHnorm{\alpha}{b}{u_{\c'}}{(X^0_{\c'},d)} 
                          + m A_0 \sum_{\c'\in\{\b, \, \w\}} \NormBig{\RR_{\wt{a\phi}, \c, \c'}^{(n)} (\abs{u_{\c'}})  }_{\CCC^0(X^0_\c)}      \\
&\qquad\leq  1 +   2mC_0 +  m A_0 \bigl( \norm{u_\b}_{\CCC^0(X^0_\b)}  +   \norm{u_\w}_{\CCC^0(X^0_\w)}  \bigr) \\
&\qquad\leq  (3m+1)A_0,
\end{align*}
where $C_0 > 1$ is the constant depending only on $f$, $\CC$, and $d$ from Lemma~\ref{lmMetricDistortion}, and the last inequality follows from the fact that $A_0\geq 2C_0>2$ (see Lemma~\ref{lmBasicIneq}). The inequality (\ref{eqSplitRuelleUnifBoundNormalizedNorm_m}) is now established, and (\ref{eqSplitRuelleUnifBoundNormalizedNorm}) follows from (\ref{eqSplitRuelleOpNormQuot}) in Lemma~\ref{lmSplitRuelleCoordinateFormula} and (\ref{eqSplitRuelleUnifBoundNormalizedNorm_m}).
\end{proof}

\section{Bound the zeta function with the operator norm}   \label{sctPreDolgopyat}

In this section, we bound the dynamical zeta function $\zeta_{\sigma_{A_{\ti}},\,\minus\phi\circsmall\pi_{\ti}}$ using some bounds of the operator norm of $\RRR_{\minus s\phi}$, for an expanding Thurston map $f$ with some forward invariant Jordan curve $\CC$ and an eventually positive real-valued H\"{o}lder continuous potential $\phi$.

Subsection~\ref{subsctRuelleEstimate} focuses on Proposition~\ref{propTelescoping}, which provides a bound of the dynamical zeta function $\zeta_{\sigma_{A_{\ti}},\,\minus\phi\circsmall\pi_{\ti}}$ for the symbolic system $\bigl(\Sigma_{A_{\ti}}^+, \sigma_{A_{\ti}}\bigr)$ associated to $f$ in terms of the operator norms of $\RRR_{\minus s\phi}^n$, $n\in\N$ and $s\in\C$ in some vertical strip with $\abs{\Im(s)}$ large enough. The idea of the proof originated from D.~Ruelle \cite{Ru90}. In Subsection~\ref{subsctOperatorNorm}, we establish in Theorem~\ref{thmOpHolderNorm} an exponential decay bound on the operator norm $\OpHnormD{\alpha}{\RRR_{\minus s\phi}^n}$ of $\RRR_{\minus s\phi}^n$, $n\in\N$, assuming the bound stated in Theorem~\ref{thmL2Shrinking}. Theorem~\ref{thmL2Shrinking} will be proved at the end of Subsection~\ref{subsctCancellation}. Combining the bounds in Proposition~\ref{propTelescoping} and Theorem~\ref{thmOpHolderNorm}, we give a proof of Theorem~\ref{thmZetaAnalExt_SFT} in Subsection~\ref{subsctProofthmZetaAnalExt_SFT}. Finally, in Subsection~\ref{subsctProofthmLogDerivative}, we deduce Theorem~\ref{thmLogDerivative} from Theorem~\ref{thmZetaAnalExt_InvC} following the ideas from \cite{PS98} using basic complex analysis.

\subsection{Ruelle's estimate}     \label{subsctRuelleEstimate}

\begin{prop}  \label{propTelescoping}
Let $f$, $\CC$, $d$, $\Lambda$, $\alpha$, $\phi$, $s_0$ satisfy the Assumptions. We assume, in addition, that $f(\CC) \subseteq \CC$ and no $1$-tile in $\X^1(f,\CC)$ joins opposite sides of $\CC$. Let $\bigl( \Sigma_{A_{\ti}}^+,\sigma_{A_{\ti}} \bigr)$ be the one-sided subshift of finite type associated to $f$ and $\CC$ defined in Proposition~\ref{propTileSFT}, and let $\pi_{\ti}\: \Sigma_{A_{\ti}}^+\rightarrow S^2$ be defined in (\ref{eqDefTileSFTFactorMap}). Then for each $\delta>0$ there exists a constant $D_\delta>0$ such that for all integers $n\geq 2$ and $k\in\N$, we have
\begin{equation} \label{eqSumHnormSplitRuelleOn1}
            \sum_{X^k\in\X^k(f,\CC)} \max_{\c\in\{\b, \, \w\}} \Hnormbig{\alpha}{\RR_{\minus s\phi, \c, X^k}^{(k)} (\mathbbm{1}_{X^k})}{(X^0_\c,d)}    
 \leq    D_\delta \abs{ \Im(s) } \Lambda^{-\alpha}  \exp(k(\delta + P(f,-\Re(s)\phi)))    
\end{equation}
and
\begin{align} \label{eqTelescoping}
&       \Absbigg{ Z_{\sigma_{A_{\ti}},\,\minus \phi \circsmall \pi_{\ti}}^{(n)} (s) - \sum_{\c\in\{\b, \, \w\}}\sum\limits_{\substack{X^1\in\X^1(f,\CC)\\X^1\subseteq X^0_\c}} \RR_{\minus s\phi,\c,X^1}^{(n)}(\mathbbm{1}_{X^1})(x_{X^1})   }  \\
&\qquad\leq  D_\delta \abs{\Im(s)} \sum_{m=2}^{n} \OpHnormDbig{\alpha}{\RRR_{\minus s\phi}^{n-m}} \Bigl(\frac{1}{\Lambda^\alpha} \exp(\delta + P(f,-\Re(s)\phi))  \Bigr)^m     \notag
\end{align}
for any choice of a point $x_{X^1} \in \inte (X^1)$ for each $X^1\in\X^1(f,\CC)$, and for all $s\in\C$ with $\abs{\Im(s)} \geq 2 s_0 +1$ and $\abs{\Re(s)-s_0} \leq  s_0 $, where $Z_{\sigma_{A_{\ti}},\,\minus \phi \circsmall \pi_{\ti}}^{(n)} (s)$ is defined in (\ref{eqDefZn}).
\end{prop}

\begin{proof}
Fix the integer $n\geq 2$.

We first choose $x_{X^n} \in X^n$ for each $n$-tile $X^n \in \X^n$ in the following way. If $X^n\subseteq f^n(X^n)$, then let $x_{X^n}$ be the unique point in $X^n\cap P_{1,f^n}$ (see \cite[Lemmas~4.1 and~4.2]{Li16}); otherwise $X^n$ must be a black $n$-tile contained in the white $0$-tile, or a white $n$-tile contained in the black $0$-tile, in which case we choose an arbitrary point $x_{X^n} \in \inte (X^n)$. Next, for each $i\in\N_0$ with $i\leq n-1$, and each $X^i\in\X^i$, we fix an arbitrary point $x_{X^i}\in \inte (X^i)$.

By (\ref{eqDefLc}) and our construction, we get that for all $s\in\C$, $\c\in\{\b, \, \w\}$, and $X^n\in\X^n$ with $X^n\subseteq X^0_\c$,
\begin{equation}   \label{eqPfpropTelescoping_1}
\RR_{\minus s\phi,\c,X^n}^{(n)} (\mathbbm{1}_{X^n})(x_{X^n}) =  \begin{cases} \exp(-sS_n\phi(x_{X^n})) & \text{if } X^n \subseteq f^n (X^n), \\ 0  & \text{otherwise}. \end{cases}
\end{equation}

It is easy to check that by (\ref{eqPfpropTelescoping_1}), the function $Z_{\sigma_{A_{\ti}},\,\minus \phi \circsmall \pi_{\ti}}^{(n)} (s)$ defined in (\ref{eqDefZn}) satisfies the following equality
\begin{equation}   \label{eqPfpropTelescoping_ZnRewrite}
Z_{\sigma_{A_{\ti}},\,\minus \phi \circsmall \pi_{\ti}}^{(n)} (s) = \sum_{\c\in\{\b, \, \w\}} \sum\limits_{\substack{X^n\in\X^n\\X^n \subseteq X^0_\c}}  \RR_{\minus s\phi,\c,X^n}^{(n)} (\mathbbm{1}_{X^n}) (x_{X^n}).
\end{equation}

Thus, by the triangle inequality, we get
\begin{align}   \label{eqPfpropTelescoping_SplitTileSum}
&   \Absbigg{  Z_{\sigma_{A_{\ti}},\,\minus \phi \circsmall \pi_{\ti}}^{(n)} (s) - \sum_{\c\in\{\b, \, \w\}} \sum\limits_{\substack{X^1\in\X^1\\X^1 \subseteq X^0_\c}} \RR_{\minus s\phi,\c,X^1}^{(n)} (\mathbbm{1}_{X^1}) (x_{X^1}) }  \notag \\
&\qquad\leq  \sum_{m=2}^{n}  \sum_{\c\in\{\b, \, \w\}} \Absbigg{ \sum\limits_{\substack{X^{m-1}\in\X^{m-1}\\X^{m-1} \subseteq X^0_\c}} \RR_{\minus s\phi,\c,X^{m-1}}^{(n)} (\mathbbm{1}_{X^{m-1}}) (x_{X^{m-1}})  \notag \\
&\qquad\qquad\qquad\qquad\qquad\qquad -  \sum\limits_{\substack{X^m\in\X^m\\X^m \subseteq X^0_\c}} \RR_{\minus s\phi,\c,X^m}^{(n)} (\mathbbm{1}_{X^m}) (x_{X^m})  } \\
&\qquad\leq  \sum_{m=2}^{n}  \sum_{\c\in\{\b, \, \w\}}  \sum\limits_{\substack{X^{m-1}\in\X^{m-1}\\X^{m-1} \subseteq X^0_\c}} 
    \Absbigg{  \RR_{\minus s\phi,\c,X^{m-1}}^{(n)} (\mathbbm{1}_{X^{m-1}}) (x_{X^{m-1}})  \notag  \\
&\qquad\qquad\qquad\qquad\qquad\qquad\quad -  \sum\limits_{\substack{X^m\in\X^m\\X^m \subseteq X^{m-1}}} \RR_{\minus s\phi,\c,X^m}^{(n)} (\mathbbm{1}_{X^m}) (x_{X^m}) }.  \notag
\end{align}

Note that for all $s\in\C$, $2\leq m\leq n$, $\c\in\{\b, \, \w\}$, and $X^{m-1} \in \X^{m-1}$ with $X^{m-1}\subseteq X^0_\c$, by (\ref{eqDefLc}),
\begin{align} \label{eqPfpropTelescoping_SplitTile}
    \RR_{\minus s\phi, \c, X^{m-1}}^{(n)} (\mathbbm{1}_{X^{m-1}}) (x_{X^{m-1}})  
&=  \sum\limits_{\substack{X^n\in\X^n_\c\\X^n \subseteq X^{m-1}}} \exp\left( - sS_n\phi \left( (f^n|_{X^n})^{-1}(x_{X^{m-1}}) \right)\right)   \notag \\
&=  \sum\limits_{\substack{X^m\in\X^m\\X^m \subseteq X^{m-1}}} \sum\limits_{\substack{X^n\in\X^n_\c\\X^n \subseteq X^m}}  
                                                                                         \exp\left( -sS_n\phi \left( (f^n|_{X^n})^{-1}(x_{X^{m-1}}) \right)\right)  \\
&=  \sum\limits_{\substack{X^m\in\X^m\\X^m \subseteq X^{m-1}}} \RR_{\minus s\phi, \c, X^m}^{(n)} (\mathbbm{1}_{X^m}) (x_{X^{m-1}}).  \notag
\end{align}

Hence, by (\ref{eqPfpropTelescoping_SplitTileSum}), (\ref{eqPfpropTelescoping_SplitTile}), and (\ref{eqLDiscontProperties_Holder}), we get
\begin{align*}
&                     \Absbigg{  Z_{\sigma_{A_{\ti}},\,\minus \phi \circsmall \pi_{\ti}}^{(n)} (s) 
                                         - \sum_{\c\in\{\b, \, \w\}} \sum\limits_{\substack{X^1\in\X^1\\X^1 \subseteq X^0_\c}} \RR_{\minus s\phi,\c,X^1}^{(n)} (\mathbbm{1}_{X^1}) (x_{X^1}) }\\
&\quad \leq \sum_{m=2}^{n}  \sum_{\c\in\{\b, \, \w\}}  \sum\limits_{\substack{X^{m-1}\in\X^{m-1}\\X^{m-1} \subseteq X^0_\c}}    \sum\limits_{\substack{X^m\in\X^m\\X^m \subseteq X^{m-1}}}   
                                       \Absbig{ \RR_{\minus s\phi,\c,X^m}^{(n)} (\mathbbm{1}_{X^m}) (x_{X^{m-1}})    -  \RR_{\minus s\phi,\c,X^m}^{(n)} (\mathbbm{1}_{X^m}) (x_{X^m}) }  \\
&\quad \leq  \sum_{m=2}^{n}  \sum_{\c\in\{\b, \, \w\}}  \sum\limits_{\substack{X^{m-1}\in\X^{m-1}\\X^{m-1} \subseteq X^0_\c}}    \sum\limits_{\substack{X^m\in\X^m\\X^m \subseteq X^{m-1}}}   
                                      \Hnormbig{\alpha}{\RR_{\minus s\phi,\c,X^m}^{(n)} (\mathbbm{1}_{X^m})}{(X^0_\c,d)}  d(x_{X^{m-1}},x_{X^m})^\alpha.
\end{align*}

Note that by (\ref{eqLDiscontProperties_Holder}),
$\RR_{\minus s\phi,\c,X^m}^{(m)} (\mathbbm{1}_{X^m}) \in \Holder{\alpha} \bigl( \bigl( X^0_\c,d \bigr), \C \bigr)$
for $s\in\C$, $m\in\N$, $\c\in\{\b, \, \w\}$, $X^m\in\X^m$, and that by Lemma~\ref{lmCellBoundsBM}~(ii),
$d(x_{X^{m-1}},x_{X^m}) \leq \diam_d(X^{m-1}) \leq C\Lambda^{-m+1}$.
Here $C\geq 1$ is the constant from Lemma~\ref{lmCellBoundsBM} depending only on $f$, $\CC$, and $d$. So by (\ref{eqLDiscontProperties_Split}) in Lemma~\ref{lmLDiscontProperties} and (\ref{eqSplitRuelleOpNormQuot}) in Lemma~\ref{lmSplitRuelleCoordinateFormula},
\begin{align}   \label{eqPfpropTelescoping_BoundZn-L}
&       \Absbigg{ Z_{\sigma_{A_{\ti}},\,\minus \phi \circsmall \pi_{\ti}}^{(n)} (s) 
                           - \sum_{\c\in\{\b, \, \w\}} \sum\limits_{\substack{X^1\in\X^1\\X^1 \subseteq X^0_\c}} \RR_{\minus s\phi,\c,X^1}^{(n)} (\mathbbm{1}_{X^1}) (x_{X^1}) } \notag\\
&\qquad  \leq   \sum_{m=2}^{n}   \sum_{ X^m\in\X^m }    \OpHnormD{\alpha}{\RRR_{\minus s\phi}^{n-m}}      
                               \Bigl(\max_{\c'\in\{\b, \, \w\}} \Hnormbig{\alpha}{\RR_{\minus s\phi,\c',X^m}^{(m)} (\mathbbm{1}_{X^m})}{(X^0_{\c'},d)}  \Bigr)       C^\alpha\Lambda^{\alpha(1-m)}  .   
\end{align}

\smallskip

We now give an upper bound for $\sum_{X^m\in\X^m}  \max_{\c'\in\{\b, \, \w\}} \Hnormbig{\alpha}{\RR_{\minus s\phi,\c',X^m}^{(m)} (\mathbbm{1}_{X^m})}{(X^0_{\c'},d)} $.

Fix an arbitrary point $y\in\CC\setminus \post f$.

Consider arbitrary $s\in\C$ with $\abs{\Re(s)-s_0}\leq s_0$, $m\in\N$, $X^m_\b\in\X^m_\b$, $X^m_\w\in\X^m_\w$, $X^m \in\X^m$, $x_\b, \, x'_\b \in X^0_\b$, $x_\w, \, x'_\w\in X^0_\w$, and $\c, \, \c'\in\{\b, \, \w\}$ with $\c\neq\c'$. By (\ref{eqDefLc}), Lemmas~\ref{lmMetricDistortion}, and~\ref{lmCellBoundsBM}~(ii), we have
\begin{equation}   \label{eqPfpropTelescoping_LDiscontSupNormBW}
\RR_{\minus s\phi, \c',X^m_\c}^{(m)} \bigl( \mathbbm{1}_{X^m_\c} \bigr)  (x_{\c'}) = 0
\end{equation}
and 
\begin{align}   \label{eqPfpropTelescoping_LDiscontSupNormBB}
&           \Abs{\RR_{\minus s\phi, \c', X^m_{\c'}}^{(m)}  (\mathbbm{1}_{X^m_{\c'}} ) (x_{\c'})}  \notag\\
&\qquad =   \Abs{\exp\bigl( -sS_m\phi \bigl(  (f^m|_{X^m_{\c'}} )^{-1}(x_{\c'})\bigr)  \bigr)}  \notag \\
&\qquad =   \exp\bigl( -\Re(s)S_m\phi \bigl(  (f^m|_{X^m_{\c'}} )^{-1}(y)\bigr)  \bigr)   \frac{  \exp  ( -\Re(s)S_m\phi   (  (f^m|_{X^m_{\c'}} )^{-1}(x_{\c'})   ) )} {\exp ( -\Re(s) S_m\phi  (  (f^m|_{X^m_{\c'}} )^{-1}(y) ) )  }  \\
&\qquad\leq \exp\bigl( -\Re(s)S_m\phi \bigl(  (f^m|_{X^m_{\c'}} )^{-1}(y)\bigr)  \bigr)   \exp\bigl( \Re(s) C_1 \left( \diam_d\bigl(X^0_{\c'}\right)\bigr)^\alpha  \bigr)  \notag\\
&\qquad\leq \exp\bigl( -\Re(s)S_m\phi \bigl(  (f^m|_{X^m_{\c'}} )^{-1}(y)\bigr)  \bigr)   \exp (\Re(s)C^\alpha C_1 ),   \notag
\end{align}
where $C_1>0$ is the constant from Lemma~\ref{lmSnPhiBound} depending only on $f$, $\CC$, $d$, $\phi$, and $\alpha$.

Hence, by (\ref{eqPfpropTelescoping_LDiscontSupNormBW}) and (\ref{eqPfpropTelescoping_LDiscontSupNormBB}), we get 
\begin{equation}   \label{eqPfpropTelescoping_LDiscontSupNorm}
      \Norm{ \RR_{\minus s\phi,\c',X^m }^{(m)} (\mathbbm{1}_{X^m }) }_{\CCC^0(X^0_{\c'})}   
\leq  \exp \bigl( -\Re(s)S_m\phi  \bigl(  (f^m|_{X^m} )^{-1}(y)  \bigr) \bigr)  \exp\left(\Re(s)C^\alpha C_1\right)  .
\end{equation}

By (\ref{eqDefLc}),
\begin{equation} \label{eqPfpropTelescoping_LDiscontHseminormBW}
\RR_{\minus s\phi, \c', X^m_{\c}}^{(m)} \bigl( \mathbbm{1}_{X^m_{\c}} \bigr) (x_{\c'}) - \RR_{\minus s\phi, \c', X^m_{\c} }^{(m)} \bigl(\mathbbm{1}_{X^m_{\c}} \bigr) (x'_{\c'}) = 0.
\end{equation}

By (\ref{eqDefLc}) and Lemma~\ref{lmExpToLiniear} with $T \coloneqq  2 s_0  \Hseminorm{\alpha,\, (S^2,d)}{\phi}$,
\begin{align*}
&                       \Absbig{1-   \RR_{\minus s\phi,\c',X^m_{\c'}}^{(m)}  (\mathbbm{1}_{X^m_{\c'}} ) (x_{\c'})  \big / \RR_{\minus s\phi,\c',X^m_{\c'}}^{(m)}  (\mathbbm{1}_{X^m_{\c'}} ) (x'_{\c'}) }   \\  
&\qquad  =    \Absbig{1-   \exp  \bigl( -s \bigl(S_m \phi  \bigl(  (f^m|_{X^m_{\c'}} )^{-1}(x_{\c'})  \bigr) \bigr) -    S_m   \phi  \bigl(  (f^m|_{X^m_{\c'}} )^{-1}(x'_{\c'})  \bigr) \bigr)    } \\
&\qquad \leq  C_{4} \Hseminorm{\alpha,\, (S^2,d)}{s\phi} d(x_{\c'},x'_{\c'})^\alpha \\
&\qquad  =     C_{4} \abs{s} \Hseminorm{\alpha,\, (S^2,d)}{\phi} d(x_{\c'},x'_{\c'})^\alpha,
\end{align*}
where the constant $C_{4}=C_{4}(f,\CC,d,\alpha,T)>1$ depends only on $f$, $\CC$, $d$, $\alpha$, and $\phi$ in our context.

Thus, by (\ref{eqPfpropTelescoping_LDiscontSupNormBB}),
\begin{align*}
&           \Absbig{ \RR_{\minus s\phi,\c',X^m_{\c'}}^{(m)}  (\mathbbm{1}_{X^m_{\c'}} ) (x_{\c'}) - \RR_{\minus s\phi,\c',X^m_{\c'}}^{(m)}  (\mathbbm{1}_{X^m_{\c'}} ) (x'_{\c'})  }  \\
&\qquad\leq  \Absbig{1-  \RR_{\minus s\phi,\c',X^m_{\c'}}^{(m)}  (\mathbbm{1}_{X^m_{\c'}} ) (x_{\c'}) \big/ \RR_{\minus s\phi,\c',X^m_{\c'}}^{(m)}  (\mathbbm{1}_{X^m_{\c'}} ) (x'_{\c'})  }
                         \Absbig{\RR_{\minus s\phi,\c',X^m_{\c'}}^{(m)}  (\mathbbm{1}_{X^m_{\c'}} ) (x'_{\c'}) }  \\
&\qquad\leq  4^{-1} C_{10} \abs{s} d(x_{\c'},x'_{\c'})^\alpha  \exp\bigl( -\Re(s)S_m\phi \bigl(  (f^m|_{X^m_{\c'}} )^{-1}(y)\bigr)  \bigr)   ,
\end{align*}
where we define the constant 
\begin{equation}  \label{eqConst_propTelescoping1}
C_{10} \coloneqq   \max \bigl\{ 2 , \, 4   C_{4}  \Hseminorm{\alpha,\, (S^2,d)}{\phi}  \bigr\}   \exp\left( 2 s_0  C^\alpha C_1\right) 
\end{equation}
depending only on $f$, $\CC$, $d$, $\alpha$, and $\phi$.

So we get
\begin{equation} \label{eqPfpropTelescoping_LDiscontHseminormBB}
      \Hseminormbig{\alpha,\, (X^0_{\c'},d )}{\RR_{\minus s\phi,\c',X^m_{\c'}}^{(m)} ( \mathbbm{1}_{X^m_{\c'}} )}  
 \leq  4^{-1} C_{10} \abs{s} \exp\bigl( -\Re(s)S_m\phi \bigl(  (f^m|_{X^m_{\c'}} )^{-1}(y)\bigr)  \bigr)   .  
\end{equation}

Thus, by (\ref{eqPfpropTelescoping_LDiscontHseminormBW}) and (\ref{eqPfpropTelescoping_LDiscontHseminormBB}), we have
\begin{equation} \label{eqPfpropTelescoping_LDiscontHseminorm}
      \Hseminormbig{\alpha,\, (X^0_{\c'},d )}{\RR_{\minus s\phi,\c',X^m}^{(m)}  ( \mathbbm{1}_{X^m} )}  
 \leq  4^{-1} C_{10} \abs{s} \exp\bigl( -\Re(s)S_m\phi \bigl(  (f^m|_{X^m} )^{-1}(y)\bigr)  \bigr) . 
\end{equation}

Hence, by (\ref{eqPfpropTelescoping_LDiscontSupNorm}) and (\ref{eqPfpropTelescoping_LDiscontHseminorm}), for all $m\in\N$, $X^m\in\X^m$, $s\in\C$, and $\c'\in\{\b, \, \w\}$ satisfying 
$\abs{\Im(s)} \geq 2 s_0 + 1  \text{ and } \abs{\Re(s)-s_0} \leq  s_0$,
we have
\begin{equation}  \label{eqPfpropTelescoping_LDiscontHNorm}
\Hnormbig{\alpha}{\RR_{\minus s\phi,\c',X^m}^{(m)} (\mathbbm{1}_{X^m})}{(X^0_{\c'},d)}    \leq C_{10} \abs{\Im(s)} \exp\bigl(-\Re(s) S_m\phi \bigl(  (f^m|_{X^m} )^{-1}(y)\bigr)  \bigr).
\end{equation}

So by (\ref{eqPfpropTelescoping_LDiscontHNorm}) and the fact that $y\in \CC$, we get
\begin{align}   \label{eqPfpropTelescoping_SumHnormD=L}
&         \sum_{X^m\in\X^m}   \max_{\c'\in\{\b, \, \w\}} \Hnormbig{\alpha}{\RR_{\minus s\phi,\c',X^m}^{(m)} (\mathbbm{1}_{X^m})}{(X^0_{\c'},d)}  \notag \\
&\qquad\leq   C_{10} \abs{\Im(s)}   \sum_{X^m\in\X^m} \exp\bigl( -\Re(s)S_m\phi \bigl(  (f^m|_{X^m} )^{-1}(y)\bigr)  \bigr)  \\
&\qquad=   2C_{10}  \abs{\Im(s)}  \RR_{\minus \Re(s)\phi}^m  ( \mathbbm{1}_{S^2}  ) (y)   \notag.
\end{align}

We construct a sequence of continuous functions $p_m \: \R\rightarrow\R$, $m\in\N$, as
\begin{equation}   \label{eqDefp_m}
p_m(a) \coloneqq   \bigl( \RR_{\minus a\phi}^m  ( \mathbbm{1}_{S^2}  ) (y) \bigr)^{ 1/m }.
\end{equation}
By Lemma~3.25 in \cite{LZ24a}, the function $a \mapsto p_m(a) - e^{P(f, -a\phi)}$ converges to $0$ as $m$ tends to $+\infty$, uniformly in $a\in [0, 2s_0]$. Recall that $a\mapsto P(f, - a\phi)$ is continuous in $a\in\R$ (see for example, \cite[Theorem~3.6.1]{PrU10}). Thus, by (\ref{eqPfpropTelescoping_SumHnormD=L}), there exists a constant $C_{11}>0$ depending only on $f$, $\CC$, $d$, $\alpha$, $\phi$, and $\delta$ such that for all $m\in\N$ and $s\in\C$ with $\abs{\Im(s)}\geq 2 s_0 +1$ and $\abs{\Re(s)-s_0}\leq s_0$, 
\begin{align}
&                       \sum_{X^m\in\X^m} \max_{\c'\in\{\b, \, \w\}} \Hnormbig{\alpha}{\RR_{\minus s\phi,\c',X^m}^{(m)} (\mathbbm{1}_{X^m})}{(X^0_{\c'},d)}    \label{eqPfpropTelescoping_SumHnormSplitRuelleOn1} \\
&\qquad  \leq   2C_{10} \abs{\Im(s)}  (p_m(\Re(s)))^m \leq   C_{11} \abs{\Im(s)}  e^{m(\delta + P(f,-\Re(s)\phi))} .   \notag 
\end{align}

Combining (\ref{eqPfpropTelescoping_BoundZn-L}) with the above inequality, we get for all $s\in\C$ with $\abs{\Im(s)}\geq 2 s_0 +1$ and $\abs{\Re(s)-s_0}\leq  s_0$, 
\begin{align*}
&                    \Absbigg{  Z_{\sigma_{A_{\ti}},\,\minus \phi \circsmall \pi_{\ti}}^{(n)} (s)
                                         - \sum_{\c\in\{\b, \, \w\}} \sum\limits_{\substack{X^1\in\X^1\\X^1 \subseteq X^0_\c}} \RR_{\minus s\phi,\c, X^1}^{(n)} (\mathbbm{1}_{X^1}) (x_{X^1}) }\\
&\qquad \leq   D_\delta \abs{\Im(s)} \sum_{m=2}^{n} \OpHnormD{\alpha}{\RRR_{\minus s\phi}^{n-m}} \Bigl(\frac{1}{\Lambda^\alpha} \exp(\delta + P(f, -\Re(s)\phi))  \Bigr)^m,    \notag
\end{align*}
where $D_\delta \coloneqq C^\alpha C_{11} \Lambda^\alpha > C_{11} >0$ is a constant depending only on $f$, $\CC$, $d$, $\phi$, $\alpha$, and $\delta$.

Inequality (\ref{eqSumHnormSplitRuelleOn1}) now follows from (\ref{eqPfpropTelescoping_SumHnormSplitRuelleOn1}) and $D_\delta  \Lambda^{-\alpha} \geq C_{11}$.
\end{proof}

\subsection{Operator norm}    \label{subsctOperatorNorm}

The following theorem is one of the main estimates we need to prove in this paper. 

\begin{theorem}  \label{thmL2Shrinking}
Let $f\: S^2 \rightarrow S^2$ be an expanding Thurston map with a Jordan curve $\CC\subseteq S^2$ satisfying $f(\CC) \subseteq \CC$ and $\post f\subseteq \CC$. Let $d$ be a visual metric on $S^2$ for $f$ with expansion factor $\Lambda>1$, and $\phi\in \Holder{\alpha}(S^2,d)$ be an eventually positive real-valued H\"{o}lder continuous function with an exponent $\alpha\in(0,1]$ that satisfies the $\alpha$-strong non-integrability condition. Let $s_0\in\R$ be the unique positive real number satisfying $P(f, -s_0\phi)=0$. 

Then there exist constants $\iota \in \N$, $a_0 \in (0, s_0 ]$, $b_0 \in [ 2 s_0 +1,+\infty)$, and $\rho \in (0,1)$ such that for each $\c\in\{\b, \, \w\}$, each $n\in\N$, each $s\in\C$ with $\abs{\Re(s)-s_0} \leq a_0$ and $\abs{\Im(s)} \geq b_0$, and each pair of functions $u_\b \in \Holder{\alpha}\bigl(\bigl(X^0_\b,d\bigr),\C\bigr)$ and $u_\w \in \Holder{\alpha}\bigl(\bigl(X^0_\w,d\bigr),\C\bigr)$ satisfying $\NHnorm{\alpha}{\Im(s)}{u_\b}{(X^0_\b,d)} \leq 1$ and $\NHnorm{\alpha}{\Im(s)}{u_\w}{(X^0_\w,d)} \leq 1$, we have 
\begin{equation}
\int_{X^0_\c} \! \AbsBig{ \RR_{\wt{\minus s\phi}, \c, \b}^{(n \iota)} (u_\b) + \RR_{\wt{\minus s\phi}, \c, \w}^{(n\iota)} (u_\w)  }^2   \,\mathrm{d}\mu_{\minus s_0\phi} \leq \rho^n.
\end{equation}
Here $\mu_{\minus s_0\phi}$ denotes the unique equilibrium state for the map $f$ and the potential $-s_0\phi$.
\end{theorem}

We will prove the above theorem at the end of Section~\ref{sctDolgopyat}. Assuming Theorem~\ref{thmL2Shrinking}, we can establish the following theorem.

\begin{theorem}   \label{thmOpHolderNorm}
Let $f\: S^2 \rightarrow S^2$ be an expanding Thurston map with a Jordan curve $\CC\subseteq S^2$ satisfying $f(\CC) \subseteq \CC$ and $\post f\subseteq \CC$. Let $d$ be a visual metric on $S^2$ for $f$ with expansion factor $\Lambda>1$, and $\phi\in \Holder{\alpha}(S^2,d)$ be an eventually positive real-valued H\"{o}lder continuous function with an exponent $\alpha\in(0,1]$ that satisfies the $\alpha$-strong non-integrability condition. Let $s_0\in\R$ be the unique positive real number satisfying $P(f,-s_0\phi)=0$. 

Then there exists a constant $D' = D'(f,\CC,d,\alpha,\phi) >0$ such that for each $\epsilon>0$, there exist constants $\delta_\epsilon \in (0, s_0)$, $\wt{b}_\epsilon\geq 2s_0 + 1$, and $\rho_\epsilon \in (0,1)$ with the following property:

For each $n\in\N$ and all $s\in\C$ satisfying $\abs{\Re(s) - s_0} < \delta_\epsilon$ and $\abs{\Im(s)} \geq \wt{b}_\epsilon$, we have
\begin{equation}  \label{eqOpHolderNorm}
\OpHnormD{\alpha}{\RRR_{\minus s\phi}^n} \leq D'  \abs{\Im(s)}^{1+\epsilon} \rho_\epsilon^n.
\end{equation}
\end{theorem}

\begin{proof}
Fix an arbitrary number $\epsilon>0$. Let $\iota \in \N$, $a_0 \in (0, s_0 ]$, $b_0 \in [ 2 s_0 +1,+\infty)$, and $\rho \in (0,1)$ be the constants from Theorem~\ref{thmL2Shrinking} depending only on $f$, $\CC$, $d$, $\alpha$, and $\phi$.

We choose $\iota_0\in\N$ to be the smallest integer satisfying $\frac{1}{2\iota_0} < \epsilon$, $\iota_0\geq 2$, and $\frac{\iota_0}{\iota}\in\N$.

Denote 
\begin{equation} \label{eqPfthmOpHolderNorm_gamma}
\gamma \coloneqq -\log \max\bigl\{ \rho^{ \iota_0 / ( 2\iota ) },  \, \rho_1^{ 1 / 2 }, \, \Lambda^{-\alpha} \bigr\} > 0,
\end{equation}
where $\rho_1 \coloneqq \rho_1 \bigl(f,\CC,d,\alpha,  H \bigr)\in (0,1)$, with $H \coloneqq \bigl\{ \wt{- t\phi}  :  t \in \R, \, \abs{t-s_0}\leq a_0 \bigr\}$ a bounded subset of $\Holder{\alpha}(S^2,d)$, is the constant from Theorem~\ref{thmSpectralGap} depending only on $f$, $\CC$, $d$, and $\alpha$ in our context here.

We define
\begin{align}   
\rho_\epsilon    &\coloneqq e^{ - \gamma / (32\iota_0) } \in (0,1),\label{eqPfthmOpHolderNorm_rho_eps}  \\
\wt{b}_\epsilon &\coloneqq \max \bigl\{ e^{\iota_0\gamma}, \,  (21 A_0^2 )^{2\iota_0},  \,  2s_0 + 1 \bigr\} > e.  \label{eqDefwtB_eps}
\end{align}
Here $A_0=A_0\bigl(f,\CC,d,\Hseminorm{\alpha,\, (S^2,d)}{\phi}, \alpha\bigr)>2$ is the constant from Lemma~\ref{lmBasicIneq} depending only on $f$, $\CC$, $d$, $\Hseminorm{\alpha,\, (S^2,d)}{\phi}$, and $\alpha$.

Moreover, note that by (\ref{eqDefPhiWidetilde}),
$ \Normbig{\wt{ - a\phi} -  \wt{ - s_0\phi} }_{\CCC^0(S^2)}    
\leq  \abs{a-s_0} \norm{\phi}_{\CCC^0(S^2)} + \abs{P(f,-a\phi) - P(f,-s_0\phi)}    + 2 \norm{ \log u_{\minus a\phi} - \log u_{\minus s_0\phi}}_{\CCC^0(S^2)}$.
Since the function $t\mapsto P(f,t\phi)$ is continuous (see for example, \cite[Theorem~3.6.1]{PrU10}), $P(f, - s_0 \phi)=0$, and the map $t\mapsto u_{t\phi}$ is continuous on $\Holder{\alpha}(S^2,d)$ equipped with the uniform norm $\norm{\cdot}_{\CCC^0(S^2)}$ by Lemma~\ref{lmUphiCountinuous}, we can choose $\delta_\epsilon \in (0,  a_0 )$ sufficiently small so that if $a\in [s_0-\delta_\epsilon, s_0 + \delta_\epsilon]$, then
\begin{equation}  \label{eqPfthmOpHolderNorm_PressureBound}
\abs{P(f,-a\phi)} \leq - \log \rho_\epsilon  \quad \text{and} \quad
\Normbig{\wt{ - a\phi} -  \wt{ - s_0\phi} }_{\CCC^0(S^2)} \leq \log \min\bigl\{ \rho^{ - 1 / (2\iota) } , \,  \rho_1^{ - 1/2 }  \bigr\}.
\end{equation}

Fix an arbitrary number $s=a+\I b\in\C$ with $a, \, b\in\R$ satisfying $\abs{a-s_0}\leq \delta_\epsilon$ and $\abs{b}\geq \wt{b}_\epsilon$, and fix an arbitrary pair of complex-valued H\"{o}lder continuous functions $u_\b \in \Holder{\alpha} \bigl( \bigl(X^0_\b,d \bigr),\C \bigr)$ and $u_\w \in \Holder{\alpha} \bigl( \bigl(X^0_\w,d \bigr),\C \bigr)$ satisfying $\NHnorm{\alpha}{b}{u_\b}{(X^0_\b,d)} \leq 1$ and $\NHnorm{\alpha}{b}{u_\w}{(X^0_\w,d)} \leq 1$. 

We denote by $m\in\N$ the smallest integer satisfying
\begin{equation}   \label{eqPfthmOpHolderNorm_m}
m \iota_0 \gamma 
\geq 2\log\abs{b} \geq 0.
\end{equation}
Then $m\geq 2$ by (\ref{eqDefwtB_eps}).

We first note that by (\ref{eqDefLc}), the Cauchy--Schwarz inequality, Lemma~\ref{lmRtildeNorm=1}, (\ref{eqSpectralGapAll}) in Theorem~\ref{thmSpectralGap}, Theorem~\ref{thmL2Shrinking},  (\ref{eqSplitRuelleUnifBoundNormalizedNorm_m}) in Lemma~\ref{lmSplitRuelleUnifBoundNormalizedNorm}, and (\ref{eqPfthmOpHolderNorm_PressureBound}), and by denoting $L_{\c'} \coloneqq \Absbig{ \RR_{\wt{\minus s\phi}, \c', \b}^{(m\iota_0)} (u_\b)  +  \RR_{\wt{\minus s\phi}, \c', \w}^{(m\iota_0)} (u_\w)}$, we have for each $\c\in\{\b, \, \w\}$ and each $x\in X^0_\c$,
\begin{align*}
&      \biggl(  \sum_{\c'\in\{\b, \, \w\}}  \RR_{\wt{\minus a\phi}, \c, \c'}^{(m\iota_0)}
         \Bigl( \Absbig{ \RR_{\wt{\minus s\phi}, \c', \b}^{(m\iota_0)} (u_\b)  +  \RR_{\wt{\minus s\phi}, \c', \w}^{(m\iota_0)} (u_\w)}  \Bigr)(x)   \biggr)^2\\
&\qquad=        \biggl(  \sum_{\c'\in\{\b, \, \w\}}  \sum\limits_{\substack{ X \in \X^{m\iota_0}_\c \\ X \in X^0_{\c'} } }  
                                       \Bigl( e^{ \frac{1}{2} S_{m\iota_0}  \wt{ - a\phi}  +  \frac{1}{2} S_{m\iota_0} ( \wt{ - a\phi} -  \wt{ - s_0\phi} ) }
                                                  \cdot  e^{ \frac{1}{2} S_{m\iota_0} \wt{ - s_0\phi}  } L_{\c'}    \Bigr)  \Bigl( \bigl( f^{m\iota_0}|_{X} \bigr)^{-1} (x) \Bigr)  \biggr)^2              \\              
&\qquad\leq  \biggl(   \sum_{\c'\in\{\b, \, \w\}}  \RR_{\wt{\minus a\phi}, \c, \c'}^{(m\iota_0)} 
         \Bigl(  e^{ m\iota_0\norm{\wt{ - a\phi} -  \wt{ - s_0\phi} }_{\CCC^0(S^2)}  }    \Bigr)(x)   \biggr)
         \biggl(   \sum_{\c'\in\{\b, \, \w\}}  \RR_{\wt{\minus s_0\phi}, \c, \c'}^{(m\iota_0)} \bigl( L_{\c'}^2\bigr)(x)   \biggr)       \\
&\qquad\leq    e^{ m\iota_0\norm{\wt{ - a\phi} -  \wt{ - s_0\phi} }_{\CCC^0(S^2)}  }
                           \biggl(  6 \rho_1^{m\iota_0} \max_{\c'\in\{\b, \, \w\}} \Hnormbig{\alpha}{ L_{\c'}^2 }{(X^0_{\c'}, d)} +\sum_{\c'\in\{\b, \, \w\}}  \int_{X^0_{\c'}} L_{\c'}^2  \,\mathrm{d}\mu_{\minus s_0 \phi} \biggr) \\
&\qquad\leq   42 A_0 \rho_1^{  m\iota_0 / 2 } \abs{b} + 2 \rho^{ m \iota_0   / ( 2 \iota ) } .
\end{align*}
Combining the above with (\ref{eqPfthmOpHolderNorm_m}), (\ref{eqPfthmOpHolderNorm_gamma}), and the fact that $\iota_0\geq 2$ and $A_0>2$, we get for each $\c\in\{\b, \, \w\}$,
\begin{align}    \label{eqPfthmOpHolderNorm_CommonSup}
&               \Normbigg{  \sum_{\c'\in\{\b, \, \w\}}  \RR_{\wt{\minus a\phi}, \c, \c'}^{(m\iota_0)}
                     \Bigl( \AbsBig{ \RR_{\wt{\minus s\phi}, \c', \b}^{(m\iota_0)} (u_\b) 
                                             +\RR_{\wt{\minus s\phi}, \c', \w}^{(m\iota_0)} (u_\w)}\Bigr) }_{\CCC^0(X^0_\c)} \\
&\qquad  \leq  \bigl( 42 A_0 \abs{b}^{-2+1} + 2\abs{b}^{- 2 / \iota_0 }  \bigr)^{ 1 / 2 } 
                  \leq 7A_0 \abs{b}^{ - 1 / \iota_0 }.                      \notag
\end{align}

Thus, by (\ref{eqDefProjections}), (\ref{eqSplitRuelleCoordinateFormula}), (\ref{eqLDiscontProperties_Split}), and (\ref{eqPfthmOpHolderNorm_CommonSup}), we get that for each $\c\in\{\b, \, \w\}$,
\begin{align}    \label{eqPfthmOpHolderNorm_Sup}
             \NormBig{  \pi_\c \Bigl(  \RRR_{\wt{\minus s\phi}}^{2m\iota_0}  (u_\b,u_\w)  \Bigr) }_{\CCC^0(X^0_\c)}  
&  =    \Normbigg{\sum_{\c'\in\{\b, \, \w\}}  \RR_{\wt{\minus s\phi},  \c, \c'}^{(m\iota_0)}
         \Bigl(  \RR_{\wt{\minus s\phi}, \c', \b}^{(m\iota_0)} (u_\b) 
                  +\RR_{\wt{\minus s\phi}, \c', \w}^{(m\iota_0)} (u_\w) \Bigr)}_{\CCC^0(X^0_\c)}  \notag \\  
& \leq \Normbigg{\sum_{\c'\in\{\b, \, \w\}}  \RR_{\wt{\minus a\phi}, \c, \c'}^{(m\iota_0)}
         \Bigl( \AbsBig{ \RR_{\wt{\minus s\phi}, \c', \b}^{(m\iota_0)} (u_\b) 
                                 +\RR_{\wt{\minus s\phi}, \c', \w}^{(m\iota_0)} (u_\w)}\Bigr)}_{\CCC^0(X^0_\c)} \\
& \leq  7A_0 \abs{b}^{- 1 / \iota_0}.   \notag
\end{align}

By (\ref{eqDefProjections}), (\ref{eqSplitRuelleCoordinateFormula}), (\ref{eqBasicIneqN1}) in Lemma~\ref{lmBasicIneq}, Lemma~\ref{lmSplitRuelleUnifBoundNormalizedNorm}, (\ref{eqPfthmOpHolderNorm_CommonSup}),  (\ref{eqPfthmOpHolderNorm_m}), and (\ref{eqPfthmOpHolderNorm_gamma}), we have for each $\c\in\{\b, \, \w\}$,
\begin{align}   \label{eqPfthmOpHolderNorm_HolderSeminorm}
&            \frac{1}{\abs{b}}  \HseminormBig{\alpha,\, (X^0_\c,d)}{  \pi_\c  \Bigl( \RRR_{\wt{\minus s\phi}}^{2m\iota_0}  (u_\b,u_\w)  \Bigr) }   \notag\\
&\qquad =     \frac{1}{\abs{b}}  \Hseminormbigg{\alpha,\, (X^0_\c,d)}{ \sum_{\c'\in\{\b, \, \w\}}  \RR_{\wt{\minus s\phi}, \c, \c'}^{(m\iota_0)}
             \Bigl( \RR_{\wt{\minus s\phi}, \c', \b}^{(m\iota_0)}(u_\b)
                     +\RR_{\wt{\minus s\phi}, \c', \w}^{(m\iota_0)}(u_\w)   \Bigr)  }  \notag \\
&\qquad\leq  \sum_{\c'\in\{\b, \, \w\}}   \frac{C_0}{\Lambda^{\alpha m \iota_0}} \NHnormBig{\alpha}{b}{\RR_{\wt{\minus s\phi}, \c', \b}^{(m \iota_0)} (u_\b)
                                                                                                                                                                                            +\RR_{\wt{\minus s\phi}, \c', \w}^{(m \iota_0)} (u_\w) }{(X^0_{\c'},d)}   \notag\\
&\qquad\quad   + \sum_{\c'\in\{\b, \, \w\}}    A_0\NormBig{ \RR_{\wt{\minus a\phi}, \c, \c'}^{(m \iota_0)} 
                            \Bigl(   \AbsBig{ \RR_{\wt{\minus s\phi}, \c', \b}^{(m \iota_0)} (u_\b)
                                                      +\RR_{\wt{\minus s\phi}, \c', \w}^{(m \iota_0)} (u_\w)}    \Bigr)}_{\CCC^0(X^0_\c)}   \\
&\qquad\leq  8 A_0  C_0 \Lambda^{ - \alpha m \iota_0}   + A_0 \bigl(7A_0 \abs{b}^{- 1 / \iota_0} \bigr)   \notag\\
&\qquad\leq   7A_0^2 \abs{b}^{-2} + 7 A_0^2 \abs{b}^{- 1 / \iota_0} \notag\\
&\qquad\leq  14A_0^2 \abs{b}^{- 1 / \iota_0}, \notag
\end{align}
where $C_0 > 1$ is the constant depending only on $f$, $\CC$, and $d$ from Lemma~\ref{lmMetricDistortion}, and $A_0 \geq 2C_0 >2$ (see Lemma~\ref{lmBasicIneq}).

Hence, for each $n\in\N$, by choosing $k\in\N_0$ and $r\in\{0, \, 1, \, \dots, \, 2 m \iota_0 - 1\}$ with $n=2 m \iota_0 k +r$, we get from (\ref{eqPfthmOpHolderNorm_Sup}), (\ref{eqPfthmOpHolderNorm_HolderSeminorm}), Definition~\ref{defOpHNormDiscont}, and (\ref{eqSplitRuelleUnifBoundNormalizedNorm}) in Lemma~\ref{lmSplitRuelleUnifBoundNormalizedNorm} that since $\abs{b} \geq \wt{b}_\epsilon$ and $m\geq 2$,
\begin{align}    \label{eqSplitRuelleTildeOpHnormBound}
      \OpHnormDbig{\alpha}{ \RRR_{\wt{\minus s\phi}}^n}  
&\leq    \abs{b} \NOpHnormDbig{\alpha}{b}{  \RRR_{\wt{\minus s\phi}}^{2 m \iota_0 k + r} } \notag \\
&\leq     \abs{b} \Bigl( \NOpHnormDbig{\alpha}{b}{  \RRR_{\wt{\minus s\phi}}^{2 m \iota_0} } \Bigr)^k 
           \, \NOpHnormDbig{\alpha}{b}{  \RRR_{\wt{\minus s\phi}}^{r} }  \notag \\
&\leq   4A_0 \abs{b} \bigl(  7A_0 \abs{b}^{- 1 / \iota_0} +  14A_0^2 \abs{b}^{- 1 / \iota_0} \bigr)^k  \notag \\
&\leq      4A_0  \abs{b}^{1-\frac{k}{2\iota_0}}  \\
&\leq   4A_0  \abs{b}^{1+ \frac{1}{2\iota_0} - \frac{2 m \iota_0 k + r}{2\iota_0}\frac{1}{2 m \iota_0} }  \notag \\
&\leq      4A_0  \abs{b}^{1+ \frac{1}{2\iota_0}} \abs{b}^{-\frac{n}{4 m\iota_0^2} } \notag \\
&\leq   4A_0 \abs{b}^{1+\epsilon}  e^{  - \frac{ n \log\abs{b} } { 8(m-1) \iota_0^2 } }  \notag \\
&\leq      4A_0 \abs{b}^{1+\epsilon} \rho_\epsilon^{2n},  \notag
\end{align}
where the last inequality follows from (\ref{eqPfthmOpHolderNorm_rho_eps}) and the fact that $m$ is the smallest integer satisfying (\ref{eqPfthmOpHolderNorm_m}).

\smallskip

We now turn the upper bound for $\OpHnormDbig{\alpha}{ \RRR_{\wt{\minus s\phi}}^n}$ in (\ref{eqSplitRuelleTildeOpHnormBound}) into a bound for $\OpHnormD{\alpha}{ \RRR_{ \minus s\phi }^n}$.

By (\ref{eqSplitRuelleOpNormQuot}), (\ref{eqSplitRuelleTilde_NoTilde}), (\ref{eqBound_Uaphi_UpperLower}) in Lemma~\ref{lmBound_aPhi}, and Corollary~\ref{corProductInverseHolderNorm}, we get
\begin{align*}  
 &            e^{ - n P(f, - a\phi) } \OpHnormD{\alpha}{\RRR_{\minus s \phi}^n}  \\
 &\qquad=   e^{ - n P(f, - a\phi) }   \sup \Biggl\{  \frac{   \Hnormbig{\alpha}{   \sum_{\c\in\{\b, \, \w\}} \RR_{ \minus s \phi, \c', \c}^{(n)} (v_\c)     }{(X^0_{\c'},d)}      }
                                           {  \max\bigl\{     \Hnorm{\alpha}{v_\c}{(X^0_\c,d)}  :  \c\in\{\b, \, \w\}  \bigr\}}   \Biggr\}  \\
&\qquad\leq   \Hnorm{\alpha}{u_{\minus a\phi}}{(S^2,d)}                          
            \sup \Biggl\{  \frac{   \Hnormbig{\alpha}{    \sum_{\c\in\{\b, \, \w\}}  \RR_{ \wt{\minus s \phi}, \c', \c}^{(n)} \bigl( v_\c / u_{\minus a\phi}  \bigr)   }{(X^0_{\c'},d)}      }
                                           {  \max\bigl\{     \Hnorm{\alpha}{v_\c}{(X^0_\c,d)}  :  \c\in\{\b, \, \w\}  \bigr\}}   \Biggr\}  \\
&\qquad\leq   \Hnorm{\alpha}{u_{\minus a\phi}}{(S^2,d)}     \OpHnormDbig{\alpha}{\RRR_{\wt{\minus s \phi}}^n}  
            \sup \Biggl\{  \frac{ \max\bigl\{     \Hnorm{\alpha}{  v_\c / u_{\minus a\phi}   }{(X^0_\c,d)}  :  \c\in\{\b, \, \w\}  \bigr\}   }
                                            { \max\bigl\{     \Hnorm{\alpha}{v_\c}{(X^0_\c,d)}  :  \c\in\{\b, \, \w\}  \bigr\}   }   \Biggr\}  \\    
&\qquad\leq  \Hnorm{\alpha}{u_{\minus a\phi}}{(S^2,d)}     \OpHnormDbig{\alpha}{\RRR_{\wt{\minus s \phi}}^n}  
                                                      \Hnorm{\alpha}{   1  / u_{\minus a\phi} }{(S^2,d)}    \\                                              
&\qquad\leq   \Hnorm{\alpha}{u_{\minus a\phi}}{(S^2,d)}        \OpHnormDbig{\alpha}{\RRR_{\wt{\minus s \phi}}^n}
             e^{ 2 C_{7} }    \bigl(   1+ \Hnorm{\alpha}{u_{\minus a\phi}}{(S^2,d)} \bigr)                  \\
&\qquad\leq   \OpHnormDbig{\alpha}{\RRR_{\wt{\minus s \phi}}^n}       e^{2 C_{7}}    
            \bigl(   1+ \Hnorm{\alpha}{u_{\minus a\phi}}{(S^2,d)} \bigr)^2   ,                        
\end{align*}
where the suprema are taken over all $v_\b\in\Holder{\alpha} \bigl( \bigl(X^0_\b,d \bigr),\C \bigr)$, $v_\w\in\Holder{\alpha} \bigl( \bigl(X^0_\w,d \bigr),\C \bigr)$, and $\c'\in\{\b, \, \w\}$ with $\norm{v_\b}_{\CCC^0(X^0_\b)} \norm{v_\w}_{\CCC^0(X^0_\w)} \neq 0$. Here the constant $C_{7} = C_{7}(f,\CC,d,\alpha,T,K)$, with $T\coloneqq 2 s_0$ and $K\coloneqq \Hseminorm{\alpha,\, (S^2,d)}{\phi} > 0$, is defined in (\ref{eqConst_lmBound_aPhi1}) in Lemma~\ref{lmBound_aPhi} and depends only on $f$, $\CC$, $d$, $\alpha$, and $\Hseminorm{\alpha,\, (S^2,d)}{\phi}$ in our context.

Combining the above inequality with (\ref{eqSplitRuelleTildeOpHnormBound}), (\ref{eqDefwtB_eps}), (\ref{eqPfthmOpHolderNorm_PressureBound}), and (\ref{eqBound_Uaphi_Hnorm}) in Lemma~\ref{lmBound_aPhi}, we get that if $a\in (s_0 - \delta_\epsilon, s_0 + \delta_\epsilon)$ and $\abs{b} \geq \wt{b}_\epsilon$, then
\begin{equation*}
        \OpHnormD{\alpha}{\RRR_{\minus s\phi}^n} 
\leq   4A_0 \abs{b}^{1+\epsilon} \rho_\epsilon^{2n} \rho_\epsilon^{-n}  e^{2 C_{7}} 
      \bigl(   1+ \Hnorm{\alpha}{u_{\minus a\phi}}{(S^2,d)} \bigr)^2  
\leq D'   \abs{b}^{1+\epsilon}    \rho_\epsilon^n,
\end{equation*}
where $D' \coloneqq 4 A_0 e^{2 C_{7}}  \bigl( 8 \frac{ s_0  \Hseminorm{\alpha,\, (S^2,d)}{\phi} C_0 }{1-\Lambda^{-\alpha}} L   +2 \bigr)^2 \bigl(e^{2C_{7}}\bigr)^2>1$,
which depends only on $f$, $\CC$, $d$, $\alpha$, and $\phi$.
\end{proof}

\subsection{Bound the symbolic zeta function}   \label{subsctProofthmZetaAnalExt_SFT}

Using Proposition~\ref{propTelescoping} and Theorem~\ref{thmOpHolderNorm}, we can get the following bound for the zeta function $\zeta_{\sigma_{A_{\ti}}, \, \minus \phi \circsmall \pi_{\ti} } $ (see also (\ref{eqDefZetaFn})).

\begin{prop}   \label{propHighFreq}
Let $f$, $\CC$, $d$, $\Lambda$, $\alpha$, $\phi$, $s_0$ satisfy the Assumptions. We assume, in addition, that $\phi$ satisfies the $\alpha$-strong non-integrability condition, and that $f(\CC) \subseteq \CC$ and no $1$-tile in $\X^1(f,\CC)$ joins opposite sides of $\CC$. Then for each $\epsilon>0$ there exist constants $\wt{C}_\epsilon > 0$ and $\wt{a}_\epsilon \in (0, s_0)$ such that
\begin{equation}   \label{eqLogZetaFnBdd}
     \Absbigg{  \sum_{n=1}^{+\infty} \frac{1}{n}   Z_{\sigma_{A_{\ti}},\,\minus \phi \circsmall \pi_{\ti}}^{(n)} (s)  }
\leq  \wt{C}_\epsilon  \abs{ \Im(s) }^{2+\epsilon}
\end{equation}
for all $s\in\C$ with $\abs{\Re(s) - s_0}  < \wt{a}_\epsilon$ and $\abs{\Im(s)} \geq \wt{b}_\epsilon$, where $\wt{b}_\epsilon \geq 2s_0 + 1$ is the constant depending only on $f$, $\CC$, $d$, $\alpha$, $\phi$, and $\epsilon$ defined in Theorem~\ref{thmOpHolderNorm}.
\end{prop}

Recall $Z_{\sigma_{A_{\ti}},\,\minus \phi \circsmall \pi_{\ti}}^{(n)} (s)$ defined in (\ref{eqDefZn}).

\begin{proof}
Let $\delta \coloneqq \frac{1}{3} \log (\Lambda^\alpha) > 0.$ 

Since $t\mapsto P(f, -t\phi)$ is continuous on $\R$ (see for example, \cite[Theorem~3.6.1]{PrU10}), we fix $\wt{a}_\epsilon \in (0,\delta_\epsilon) \subseteq (0, s_0)$ such that $\abs{P(f, -t\phi)} < \frac{1}{3} \log(\Lambda^\alpha)$ for each $t\in\R$ with $\abs{t-s_0} < \wt{a}_\epsilon$, where $\delta_\epsilon \in (0, s_0)$ is the constant defined in Theorem~\ref{thmOpHolderNorm} depending only on $f$, $\CC$, $d$, $\alpha$, $\phi$, and $\epsilon$.

Fix an arbitrary point $x_{X^1} \in \inte (X^1)$ for each $X^1 \in \X^1$. By Lemmas~\ref{lmLDiscontProperties},~\ref{lmSplitRuelleCoordinateFormula}, and (\ref{eqSumHnormSplitRuelleOn1}) in Proposition~\ref{propTelescoping}, for each integer $n\geq 2$ and each $s\in \C$ with $\abs{ \Re(s) - s_0 } < \wt{a}_\epsilon$,  we have
\begin{align}   \label{eqPfpropHighFreq_SumIneq}
&                          \Absbigg{  \sum_{\c\in\{\b, \, \w\}}  \sum\limits_{\substack{X^1\in\X^1\\X^1 \subseteq X^0_\c}}     \RR_{\minus s\phi,\c,X^1}^{(n)} (\mathbbm{1}_{X^1}) (x_{X^1})  }  \notag \\
&\qquad   \leq    \sum_{\c\in\{\b, \, \w\}}  \sum\limits_{\substack{X^1\in\X^1\\X^1 \subseteq X^0_\c}}    
                                 \Absbigg{   \sum_{\c'\in\{\b, \, \w\}}   \RR_{\minus s\phi,\c,\c'}^{(n-1)} \Bigl( \RR_{\minus s\phi,\c',X^1}^{(1)}  (\mathbbm{1}_{X^1} ) \Bigr) (x_{X^1})  } \\
&\qquad   \leq    \OpHnormD{\alpha}{\RRR_{\minus s \phi}^{n-1}}    
                              \sum_{\c\in\{\b, \, \w\}}  \sum\limits_{\substack{X^1\in\X^1\\X^1 \subseteq X^0_\c}}  
                                       \max_{\c'\in\{\b, \, \w\}} \Hnormbig{\alpha}{\RR_{\minus s\phi,\c',X^1}^{(1)} (\mathbbm{1}_{X^1})}{(X^0_{\c'},d)}         \notag      \\
&\qquad   \leq    \OpHnormD{\alpha}{\RRR_{\minus s \phi}^{n-1}}    D_\delta  \abs{\Im(s)}  \Lambda^{-\alpha} \exp( \delta + P(f, - \Re(s) \phi) ),                        \notag                            
\end{align}
where $D_\delta>0$ is the constant depending only on  $f$, $\CC$, $d$, $\alpha$, $\phi$, and $\delta$ from Proposition~\ref{propTelescoping}.

Hence, by (\ref{eqDefZn}), Proposition~\ref{propTelescoping}, (\ref{eqPfpropHighFreq_SumIneq}), Theorem~\ref{thmOpHolderNorm}, and the choices of $\delta$ and $\wt{a}_\epsilon$ above, we get that for each $s\in\C$ with $\abs{ \Re(s) - s_0 } < \wt{a}_\epsilon$ and $\abs{ \Im(s) } \geq \wt{b}_\epsilon$, 
\begin{align*}
&                       \sum_{n=2}^{+\infty} \frac{1}{n} \AbsBig{  Z_{\sigma_{A_{\ti}},\,\minus \phi \circsmall \pi_{\ti}}^{(n)} (s)   }  \\
&\qquad\leq  \sum_{n=2}^{+\infty} \frac{1}{n} \biggl(                  
                                                             \Absbigg{               \sum_{\c\in\{\b, \, \w\}}  \sum\limits_{\substack{X^1\in\X^1\\X^1 \subseteq X^0_\c}}     \RR_{\minus s\phi,\c,X^1}^{(n)} (\mathbbm{1}_{X^1}) (x_{X^1})  }  \\
&\qquad\qquad \qquad\quad    + \Absbigg{ Z_{\sigma_{A_{\ti}},\,\minus \phi \circsmall \pi_{\ti}}^{(n)} (s)
                                               - \sum_{\c\in\{\b, \, \w\}}  \sum\limits_{\substack{X^1\in\X^1\\X^1 \subseteq X^0_\c}}     \RR_{\minus s\phi,\c,X^1}^{(n)} (\mathbbm{1}_{X^1}) (x_{X^1})  }  \biggr) \\
&\qquad\leq   \sum_{n=2}^{+\infty} \frac{1}{n} \biggl(     \OpHnormD{\alpha}{\RRR_{\minus s \phi}^{n-1}}    D_\delta \abs{ \Im(s) } \Lambda^{-\frac{\alpha}{3}} 
                                                                                                        +   D_\delta \abs{\Im(s)} \sum_{m=2}^{n} \OpHnormD{\alpha}{\RRR_{\minus s\phi}^{n-m}} \Lambda^{-\frac{m\alpha}{3}}   \biggr)   \\
&\qquad\leq  \abs{\Im(s)}^{2+\epsilon} \sum_{n=2}^{+\infty} \frac{D'}{n}  D_\delta \sum_{m=1}^{n} \rho_\epsilon^{n-m} \Lambda^{-\frac{m\alpha}{3}}   \\  
&\qquad\leq  D' D_\delta \abs{\Im(s)}^{2+\epsilon} \sum_{n=2}^{+\infty}  \wt{\rho}_\epsilon^n \\
&\qquad\leq         \frac{    D' D_\delta  } { 1 -  \wt{\rho}_\epsilon }                 \abs{\Im(s)}^{2+\epsilon},                                                                                                                                                  
\end{align*}
where the constant $\wt{\rho}_\epsilon \coloneqq \max \bigl\{ \rho_\epsilon, \, \Lambda^{ - \alpha / 3 }  \bigr\} < 1$ depends only on $f$, $\CC$, $d$, $\alpha$, $\phi$, and $\epsilon$. Here constants $D'\in (0,s_0)$ and $\rho_\epsilon\in(0,1)$ are from Theorem~\ref{thmOpHolderNorm} depending only on $f$, $\CC$, $d$, $\alpha$, $\phi$, and $\epsilon$.

Therefore, by Proposition~A.1~(i) in \cite{LZ24a} and (\ref{eqDefZn}), we have
\begin{equation*}
          \Absbigg{  \sum_{n=1}^{+\infty} \frac{1}{n}   Z_{\sigma_{A_{\ti}},\,\minus \phi \circsmall \pi_{\ti}}^{(n)} (s)      }        
\leq  \AbsBig{  Z_{\sigma_{A_{\ti}},\,\minus \phi \circsmall \pi_{\ti}}^{(1)} (s)  }    +   \sum_{n=2}^{+\infty} \frac{1}{n} \AbsBig{  Z_{\sigma_{A_{\ti}},\,\minus \phi \circsmall \pi_{\ti}}^{(n)} (s)   }    
\leq   \wt{C}_\epsilon  \abs{ \Im(s) }^{2+\epsilon}
\end{equation*}
for all $s\in\C$ with $\abs{\Re(s) - s_0}  < \wt{a}_\epsilon$ and $\abs{\Im(s)} \geq \wt{b}_\epsilon$, where the constant
\begin{equation*}
\wt{C}_\epsilon \coloneqq  D' D_\delta ( 1 -  \wt{\rho}_\epsilon )^{-1}    + 2 \deg f \exp \bigl(2 s_0 \norm{\phi}_{\CCC^0(S^2)} \bigr)
\end{equation*}
depends only on $f$, $\CC$, $d$, $\alpha$, $\phi$, and $\epsilon$.
\end{proof}

It follows immediately from the above proposition that $\zeta_{\sigma_{A_{\ti}}, \, \minus \phi \circsmall \pi_{\ti} } (s)$ has a non-vanishing holomorphic extension across the vertical line $\Re(s) = s_0$ for high frequency. In order to get a similar theorem for $\zeta_{\sigma_{A_{\ti}}, \, \minus \phi \circsmall \pi_{\ti} } (s)$ as Theorem~\ref{thmZetaAnalExt_InvC}, we just need to establish its holomorphic extension for low frequency.

\begin{proof}[Proof of Theorem~\ref{thmZetaAnalExt_SFT}]
Statement~(i) of Theorem~\ref{thmZetaAnalExt_SFT} has been established in \cite[Theorem~E]{LZ24a}.

To verify statement~(ii) in Theorem~\ref{thmZetaAnalExt_SFT}, we assume, in addition, that $\phi$ satisfies the $\alpha$-strong non-integrability condition.

Fix an arbitrary $\epsilon > 0$. Let $\wt{C}_\epsilon > 0$ and $\wt{a}_\epsilon \in (0,s_0)$ be the constants from Proposition~\ref{propHighFreq}, and $\wt{b}_\epsilon \geq 2 s_0 + 1$ be the constant from Theorem~\ref{thmOpHolderNorm}, all of which depend only on $f$, $\CC$, $d$, $\alpha$, $\phi$, and $\epsilon$. The inequality (\ref{eqZetaBound_SFT}) follows immediately from (\ref{eqLogZetaFnBdd}) in Proposition~\ref{propHighFreq}.

Therefore, by the compactness of $\bigl[ - \wt{b}_\epsilon, \wt{b}_\epsilon \bigr]$, we can choose $\wt{\epsilon}_0 \in (0, \wt{a}_\epsilon) \subseteq(0, s_0)$ small enough such that $\zeta_{\sigma_{A_{\ti}}, \, \minus \phi \circsmall \pi_{\ti} } (s)$ extends to a non-vanishing holomorphic function on the closed half-plane $\{s\in \C  :  \Re(s) \geq s_0 - \wt{\epsilon}_0 \}$ except for a simple pole at $s= s_0$.
\end{proof}

\subsection{Proof of Theorem~\ref{thmZetaAnalExt_InvC}}  \label{subsctDynOnC_Reduction}

In this subsection, we give a proof of Theorem~\ref{thmZetaAnalExt_InvC} assuming Theorem~\ref{thmZetaAnalExt_SFT}.

\begin{proof}[Proof of Theorem~\ref{thmZetaAnalExt_InvC}]
Statement~(i) is established in \cite[Theorem~D]{LZ24a}.

To verify statement~(ii), we continue with the proof of \cite[Theorem~D]{LZ24a} and assume in addition that $\phi$ satisfies the $\alpha$-strong non-integrability condition. By statement~(ii) in Theorem~\ref{thmZetaAnalExt_SFT} and the proof of Claim in the proof of \cite[Theorem~D]{LZ24a} in \cite[Section~8]{LZ24a}, $\DS_{f,\,\minus\phi,\,\deg_f}$ extends to a non-vanishing holomorphic function on $\overline{\H}_{s_0 - \epsilon'_0 }$ except for the simple pole at $s=s_0$. Here $\epsilon '_{0} > 0$ is the constant from the proof of Proposition~8.1 in \cite{LZ24a}.  Moreover, for each $\epsilon >0$, there exists a constant $C'_\epsilon >0$ such that
\begin{equation*}  
     \exp\left( - C'_\epsilon \abs{\Im(s)}^{2+\epsilon} \right) 
\leq \Absbig{ \DS_{ f,\,\minus\phi,\,\deg_f} (s) }
\leq \exp\left(   C'_\epsilon \abs{\Im(s)}^{2+\epsilon} \right) 
\end{equation*}
for all $s\in\C$ with $\abs{\Re(s) - s_0} < \epsilon'_0$ and $\abs{\Im(s)}  \geq b_\epsilon$, where $b_\epsilon \coloneqq \wt{b}_\epsilon>0$ is the constant from Theorem~\ref{thmZetaAnalExt_SFT} depending only on $f$, $\CC$, $d$, $\phi$, and $\epsilon$.

Therefore, statement~(ii) in Theorem~\ref{thmZetaAnalExt_InvC} holds for  $a_\epsilon \coloneqq \min \{ \epsilon_0, \, \wt{a}_\epsilon \}>0$, $b_\epsilon = \wt{b}_\epsilon>0$, and some constant $C_{\epsilon} > C'_{\epsilon}>0$ depending only on $f$, $\CC$, $d$, $\phi$, and $\epsilon$. 
\end{proof}

\subsection{Proof of Theorem~\ref{thmPrimeOrbitTheorem}}    \label{subsctProofthmLogDerivative}
We first state the following theorem on the logarithmic derivative of the zeta function, which will be proved at the end of this subsection.

\begin{theorem}  \label{thmLogDerivative}
Let $f\: S^2 \rightarrow S^2$ be an expanding Thurston map, and $d$ be a visual metric on $S^2$ for $f$. Let $\phi \in \Holder{\alpha}(S^2,d)$ be an eventually positive real-valued H\"{o}lder continuous function with an exponent $\alpha\in (0,1]$ that satisfies the $\alpha$-strong non-integrability condition. Denote by $s_0$ the unique positive number with $P(f,-s_0 \phi) = 0$. 

Then there exists $N_f\in\N$ depending only on $f$ such that for each $n\in \N$ with $n\geq N_f$, the following statement holds for $F\coloneqq f^n$ and $\Phi\coloneqq \sum_{i=0}^{n-1} \phi \circ f^i$:

There exist constants $a \in (0, s_0)$, $b\geq 2 s_0 + 1$, and $D>0$ such that 
\begin{equation}   \label{eqLogDerivative}
          \Absbigg{  \frac{ \zeta'_{F,\, \minus\Phi}(s) }{ \zeta_{F,\, \minus\Phi}(s) } }
\leq  D \abs{\Im(s)}^{\frac12}
\end{equation}
for all $s\in\C$ with $\abs{\Re(s) - s_0} < a$ and $\abs{\Im(s)}  \geq b$.
\end{theorem}

Statement~(i) in Theorem~\ref{thmPrimeOrbitTheorem} is established in \cite[Theorem~C]{LZ24a}. Once Theorem~\ref{thmZetaAnalExt_InvC} and Theorem~\ref{thmLogDerivative} are established, statement~(ii) in Theorem~\ref{thmPrimeOrbitTheorem} follows from standard number-theoretic arguments. More precisely, a proof of statement~(ii) in Theorem~\ref{thmPrimeOrbitTheorem}, relying on Proposition~\ref{propZetaFnConv_s0}, Theorem~\ref{thmLogDerivative}, and statement~(ii) in Theorem~\ref{thmZetaAnalExt_InvC}, is verbatim the same as that of \cite[Theorem~1]{PS98} presented in \cite[Section~3]{PS98}. We omit this proof here and direct the interested readers to the references cited above.

To prove Theorem~\ref{thmLogDerivative}, following the ideas from \cite{PS98}, we convert the bounds of the zeta function for an expanding Thurston map from Theorem~\ref{thmZetaAnalExt_InvC} to a bound of its logarithmic derivative.

We first record a standard result from complex analysis (see \cite[Theorem~4.2]{EE85}) as in \cite[Section~2]{PS98}.

\begin{lemma}    \label{lmLogDerivFromEE}
Consider $z\in\C$, $R>0$, and $\delta>0$. Let $F\: \Delta \rightarrow \C$ be a holomorphic function on the closed disk $\Delta\coloneqq \bigl\{ s\in\C  :  \abs{s-z} \leq R(1+\delta)^3 \bigr\}$. Assume that $F$ satisfies the following two conditions:
\begin{enumerate}
\smallskip
\item[(i)] $F(s)$ has no zeros on the subset 
\begin{equation*}
\bigl\{ s\in \C  :   \abs{s-z} \leq R (1+\delta)^2, \Re(s) > \Re(z) - R(1+\delta) \bigr\} \subseteq \Delta.
\end{equation*}
\item[(ii)] There exists a constant $U\geq 0$ depending only on $z$, $R$, $\delta$, and $F$ such that
\begin{equation*}
\log \abs{F(s)} \leq U + \log \abs{F(z)}
\end{equation*}
for all $s\in\Delta$ with $\abs{s-z} \leq R (1+\delta)^3$.
\end{enumerate}
\smallskip
Then for each $s\in\Delta$ with $\abs{s-z} \leq R$, we have
\begin{equation*}
           \Absbigg{ \frac{ F'(s) }{ F(s) } }
\leq  \frac{ 2 + \delta }{ \delta }   \biggl(  \Absbigg{ \frac{ F'(z) }{ F(z) } }   +  \frac{ \bigl(2 + (1+\delta)^{-2} \bigr) (1+\delta) }{  R \delta^2 } U  \biggr).
\end{equation*}
\end{lemma}

We will also need a version of the well-known Phragm\'en--Lindel\"of theorem recorded below. See \cite[Section~5.65]{Ti39} for the statement and proof of this theorem.

\begin{theorem}[The Phragm\'en--Lindel\"of Theorem]    \label{thmPhragmenLindelof}
Consider real numbers $\delta_1<\delta_2$. Let $h(s)$ be a holomorphic function on the strip $\{s\in\C  :  \delta_1\leq \Re(s) \leq \delta_2 \}$. Assume that the following conditions are satisfied:
\begin{enumerate}
\smallskip
\item[(i)] For each $\sigma > 0$, there exist real numbers $C_\sigma >0$ and $T_\sigma >0$ such that
\begin{equation*}
\abs{ h(\delta + \I t) } \leq C_{\sigma} e^{\sigma\abs{t}}
\end{equation*}
for all $\delta, \,  t\in\R$ with $\delta_1 \leq \delta \leq \delta_2$ and $\abs{t} \geq T_\sigma$.

\smallskip
\item[(ii)] There exist real numbers $C_0>0$, $T_0 >0$, and $k_1, \, k_2\in \R$ such that
\begin{equation*}
\abs{ h(\delta_1 + \I t) } \leq C_0 \abs{t}^{k_1}       \qquad \text{and} \qquad 
\abs{ h(\delta_2 + \I t) } \leq C_0 \abs{t}^{k_2} 
\end{equation*}
for all $t\in\R$ with $\abs{t} \geq T_0$.
\end{enumerate}
\smallskip
Then there exist real numbers $D>0$ and $T>0$ such that
\begin{equation*}
           \abs{ h(\delta + \I t) }   \leq C\abs{t}^{k(\delta)}
\end{equation*}
for all $\delta, \,  t\in\R$ with $\delta_1 \leq \delta \leq \delta_2$ and $\abs{t} \geq T$, where $k(\delta)$ is the linear function of $\delta$ that takes values $k_1$, $k_2$ for $\delta=\delta_1$, $\delta_2$, respectively.
\end{theorem}

Assuming Theorem~\ref{thmZetaAnalExt_InvC}, we establish Theorem~\ref{thmLogDerivative} as follows.

\begin{proof}[Proof of Theorem~\ref{thmLogDerivative}]
We choose $N_f\in\N$ as in Remark~\ref{rmNf}. Note that $P \bigl( f^i, - s_0 S_i^f \phi \bigr) = i P(f, - s_0 \phi) = 0$ for each $i\in\N$ (see for example, \cite[Theorem~9.8]{Wal82}). We observe that by Lemma~\ref{lmCexistsL}, it suffices to prove the case $n=N_f = 1$. In this case, $F=f$, $\Phi=\phi$, and there exists a Jordan curve $\CC\subseteq S^2$ satisfying $f(\CC)\subseteq \CC$, $\post f\subseteq \CC$, and no $1$-tile in $\X^1(f,\CC)$ joins opposite sides of $\CC$.

Let  $C_\epsilon$, $a_\epsilon\in(0,s_0)$, and $b_\epsilon\geq 2s_0 +1$ be the constants from Theorem~\ref{thmZetaAnalExt_InvC} depending only on $f$, $\CC$, $d$, $\alpha$, $\phi$, and $\epsilon$. We fix $\epsilon\coloneqq 1$ throughout this proof.

Define $R\coloneqq \frac{ a_\epsilon }{ 3 }$,  $\beta\coloneqq b_\epsilon + \frac{ a_\epsilon }{2}$, and $\delta \coloneqq  \bigl(\frac{3}{2} \bigr)^{1/3} - 1$. Note that $R(1+\delta)^3 = \frac{ a_\epsilon }{2}$.

Fix an arbitrary $z\in \C$ with $\Re(z) = s_0 + \frac{a_\epsilon}{4}$ and $\abs{ \Im(z) } \geq \beta$. The closed disk 
\begin{equation*}
\Delta \coloneqq \bigl\{ s\in\C  :  \abs{s-z} \leq R(1+\delta)^3 \bigr\}  =  \{ s\in\C  :  \abs{s-z} \leq  a_\epsilon / 2   \}
\end{equation*}
is a subset of $\{ s\in\C  :  \abs{ \Re(s) - s_0 } < a_\epsilon, \, \abs{ \Im(s) } \geq b_\epsilon \}$. Thus, by Theorem~\ref{thmZetaAnalExt_InvC}, inequality~(\ref{eqZetaBound}) holds for all $s\in\Delta$, and the zeta function $\zeta_{f,\,\minus\phi}$ has no zeros in $\Delta$.

For each $s\in \Delta$, by (\ref{eqZetaBound}) in Theorem~\ref{thmZetaAnalExt_InvC} and the fact that $\abs{ \Im(z) } \geq \beta = b_\epsilon + \frac{a_\epsilon}{2}$,
\begin{equation*}
           \Absbig{  \log \Absbig{ \zeta_{f,\,\minus\phi} (s) } -  \log \Absbig{ \zeta_{f,\,\minus\phi} (z) }  } 
 \leq  2 C_\epsilon \bigl( \abs{ \Im(z) } +2^{-1} a_\epsilon  \bigr)^3
\leq 2^4 C_\epsilon \abs{ \Im(z) }^3
\eqqcolon U.
\end{equation*}

\smallskip

\emph{Claim.} For each $a\in\R$ with $a> s_0$, there exists a real number $\mathcal{K}(a)>0$ depending only on $f$, $\CC$, $d$, $\phi$, and $a$ such that
$
\abs{  \zeta'_{f,\,\minus \phi} ( a+ \I t )  /    \zeta_{f,\,\minus \phi} ( a+ \I t ) }      \leq \mathcal{K}(a)
$
for all $t\in\R$.

\smallskip

To establish the claim, we first fix an arbitrary $a\in \R$ with $a> s_0$. By Corollary~\ref{corS0unique}, the topological pressure $P( f, -a \phi) < 0$. It follows from \cite[Proposition~6.8]{Li15} that there exist numbers $N_a\in\N$ and $\tau_a \in (0,1)$ such that for each integer $n\in\N$ with $n\geq N_a$,
\begin{equation*}
\sum_{ x \in P_{1,f^n} } \exp ( - a S_n\phi (x) ) \leq \tau_a^n.
\end{equation*}
Since the zeta function $\zeta_{f,\,\minus\phi}$ converges uniformly and absolutely to a non-vanishing holomorphic function on $\bigl\{ s\in\C  :  \Re(s) \geq \frac{ a+ s_0}{2}  \bigr\}$ (see Proposition~\ref{propZetaFnConv_s0}), we get from (\ref{eqDefZetaFn}), Theorem~3.20~(ii) in \cite{LZ24a}, and (\ref{eqDegreeProduct}) that
\begin{align*}
           \Absbigg{  \frac{ \zeta'_{f,\,\minus \phi} ( a+ \I t ) } {    \zeta_{f,\,\minus \phi} ( a+ \I t ) } }   
&  =    \Absbigg{ \sum_{n=1}^{+\infty}     \frac{1}{n}    \sum_{ x \in P_{1,f^n} }  ( S_n\phi (x)  )    \exp ( - (a + \I t) S_n\phi (x) )  }  \\
&\leq \norm{\phi}_{\CCC^0(S^2)}    \sum_{n=1}^{+\infty}     \sum_{ x \in P_{1,f^n} }  \exp ( - a S_n\phi (x) )  \\
&\leq \norm{\phi}_{\CCC^0(S^2)}   \biggl(    \sum_{n=N_a + 1}^{ + \infty }  \tau_a^n    +     \sum_{n=1}^{N_a} \card P_{1,f^n}     \biggr) \\
&\leq  \mathcal{K}(a),
\end{align*}
for all $t\in\R$, where $\mathcal{K}(a) \coloneqq \norm{\phi}_{\CCC^0(S^2)}   \bigl(    \frac{1}{1-\tau_a} + N_a + \sum_{n=1}^{N_a} (\deg f)^n   \bigr)$ is a constant depending only on $f$, $\CC$, $d$, $\phi$, and $a$. This establishes the claim.

\smallskip

Hence, by Lemma~\ref{lmLogDerivFromEE}, the claim with $a\coloneqq s_0 + \frac{a_\epsilon}{4}$, and the choices of $U$, $R$, and $\delta$ above, we get that for all $s\in\Delta$ with $\Im(s) = \Im(z)$ and $\Absbig{ \Re(s) -  \bigl( s_0 + \frac{a_\epsilon}{4} \bigr) }   \leq R= \frac{a_\epsilon}{3}$, we have
\begin{equation}   \label{eqPfthmLogDerivative_1}
          \Absbigg{  \frac{ \zeta'_{f,\,\minus \phi} ( s ) } {    \zeta_{f,\,\minus \phi} ( s ) } } 
 \leq   \frac{ 2 + \delta }{ \delta }   \biggl(   \mathcal{K}\Bigl(s_0 + \frac{a_\epsilon}{4}\Bigr)    +  \frac{ 2^4 C_\epsilon \bigl(2 + (1+\delta)^{-2} \bigr) (1+\delta) }{  R \delta^2 }  \abs{ \Im(z) }^3  \biggr)   
 \leq  C_{13} \abs{ \Im(s) }^3,
\end{equation}
where
$
C_{13}  \coloneqq \frac{ 2 + \delta }{ \delta }   \bigl(   \mathcal{K}\bigl(s_0 + \frac{a_\epsilon}{4}\bigr)    +  \frac{ 2^4 C_\epsilon  ( 2 + (1+\delta)^{-2}  ) (1+\delta) }{  R \delta^2 }   \bigr)
$
is a constant depending only on $f$, $\CC$, $d$, $\alpha$, and $\phi$. Recall that the only restriction on $\Im(z)$ is that $\abs{\Im(z)} \geq \beta$. Thus, (\ref{eqPfthmLogDerivative_1}) holds for all $s\in\C$ with $\Absbig{ \Re(s) -  \bigl( s_0 + \frac{a_\epsilon}{4} \bigr) }   \leq  \frac{a_\epsilon}{3}$ and $\abs{\Im(s)} \geq \beta$.

By Theorem~\ref{thmZetaAnalExt_InvC}, $h(s)  \coloneqq   \frac{ \zeta'_{f,\,\minus \phi} ( s ) } {    \zeta_{f,\,\minus \phi} ( s ) }  + \frac{ 1 } { s - s_0 }$ is holomorphic on $\{ s\in\C  :  \abs{ \Re(s) - s_0 } < a_\epsilon \}$. Applying the Phragm\'en--Lindel\"of theorem (Theorem~\ref{thmPhragmenLindelof}) to $h(s)$ on the strip $\{ s\in\C  :  \delta_1 \leq \Re(s) \leq \delta_2  \}$ with $\delta_1 \coloneqq  s_0 - \frac{a_\epsilon}{12}$ and $\delta_2  \coloneqq  s_0 + \frac{a_\epsilon}{200}$. It follows from (\ref{eqPfthmLogDerivative_1}) that condition~(i) of Theorem~\ref{thmPhragmenLindelof} holds. On the other hand, (\ref{eqPfthmLogDerivative_1}) and the claim above guarantees condition~(ii) of Theorem~\ref{thmPhragmenLindelof} with $k_1 \coloneqq 3$ and $k_2 \coloneqq 0$. Hence, by Theorem~\ref{thmPhragmenLindelof}, there exist constants $\wt{D} >0$ and $b\geq 2 s_0 +1$ depending only on $f$, $\CC$, $d$, $\alpha$, and $\phi$ such that
$
\abs{h(s)}   \leq \wt{D} \abs{ \Im(s) }^{ 1 / 2 }
$
for all $s\in\C$ with $\abs{ \Re(s) - s_0 } \leq \frac{ a_\epsilon }{ 200 }$ and $\abs{ \Im(s) } \geq b$.

Therefore, inequality (\ref{eqLogDerivative}) holds for all $s\in\C$ with $\abs{ \Re(s) - s_0 } \leq \frac{ a_\epsilon }{ 200 } \eqqcolon a$ and $\abs{ \Im(s) } \geq b$, where $a\in(0,s_0)$, $b\geq 2s_0 + 1$, and $D\coloneqq \wt{D} + 1$ are constants depending only on $f$, $\CC$, $d$, $\alpha$, and $\phi$.
\end{proof}

\section{The Dolgopyat cancellation estimate}   \label{sctDolgopyat}

We adapt the arguments of D.~Dolgopyat \cite{Do98} in our metric-topological setting, aiming to prove Theorem~\ref{thmL2Shrinking} at the end of this section. In Subsection~\ref{subsctSNI}, we first give a formulation of the \emph{$\alpha$-strong non-integrability condition}, $\alpha\in(0,1]$, for our setting and then show its independence on the choice of the Jordan curve $\CC$. In Subsection~\ref{subsctDolgopyatOperator}, a consequence of the $\alpha$-strong non-integrability condition that we will use in the remaining part of this section is formulated in Proposition~\ref{propSNI}. We remark that it is crucial for the arguments in Subsection~\ref{subsctCancellation} to have the same exponent $\alpha\in(0,1]$ in both the lower bound and the upper bound in (\ref{eqSNIBounds}). The definition of the Dolgopyat operator $\MM_{J,s,\phi}$ in our context is given in Definition~\ref{defDolgopyatOperator} after important constants in the construction are carefully chosen. In Subsection~\ref{subsctCancellation}, we adapt the cancellation arguments of D.~Dolgopyat to establish the $l^2$-bound in Theorem~\ref{thmL2Shrinking}.

\subsection{Strong non-integrability}    \label{subsctSNI}

\begin{definition}[Strong non-integrability condition]   \label{defStrongNonIntegrability}
Let $f\: S^2 \rightarrow S^2$ be an expanding Thurston map and $d$ be a visual metric on $S^2$ for $f$. Fix $\alpha \in (0,1]$. Let $\phi\in \Holder{\alpha}(S^2,d)$ be a real-valued H\"{o}lder continuous function with an exponent $\alpha$. 

\begin{enumerate}

\smallskip
\item[(1)] We say that $\phi$ satisfies the \defn{$(\CC, \alpha)$-strong non-integrability condition} (with respect to $f$ and $d$), for a Jordan curve $\CC\subseteq S^2$ with $\post f \subseteq \CC$, if there exist
\begin{enumerate}
	\smallskip
	\item[(a)] numbers $N_0,\, M_0\in\N$, $\varepsilon \in (0,1)$, and 
	
	\smallskip
	\item[(b)] $M_0$-tiles $Y^{M_0}_\b\in \X^{M_0}_\b(f,\CC)$,  $Y^{M_0}_\w\in \X^{M_0}_\w(f,\CC)$ 
\end{enumerate}
 such that for each $\c\in\{\b, \, \w\}$, each integer $M \geq M_0$, and each $M$-tile $X\in \X^M(f,\CC)$ with $X \subseteq Y^{M_0}_\c$, there exist two points $x_1(X),\, x_2(X) \in X$ with the following properties:
\begin{enumerate}
\smallskip
\item[(i)] $\min \{ d(x_1(X), S^2 \setminus X), \, d(x_2(X), S^2 \setminus X), d(x_1(X), x_2(X)) \} \geq \varepsilon \diam_d(X)$, and

\smallskip
\item[(ii)] for each integer $N \geq N_0$, there exist two  $(N+M_0)$-tiles $X^{N+M_0}_{\c,1}, \, X^{N+M_0}_{\c,2} \in \X^{N+M_0}(f,\CC)$ such that $Y^{M_0}_\c = f^N\bigl(  X^{N+M_0}_{\c,1}  \bigr) =   f^N\bigl(  X^{N+M_0}_{\c,2}  \bigr)$, and that
\begin{equation}   \label{eqSNIBoundsDefn}
\frac{ \Abs{  S_{N }\phi ( \varsigma_1 (x_1(X)) ) -  S_{N }\phi ( \varsigma_2 (x_1(X)) )  -S_{N }\phi ( \varsigma_1 (x_2(X)) ) +  S_{N }\phi ( \varsigma_2 (x_2(X)) )   }  } {d(x_1(X),x_2(X))^\alpha}
\geq \varepsilon,
\end{equation}
where we write $\varsigma_1 \coloneqq \bigl(f^{N}\big|_{X^{N+M_0}_{\c,1}} \bigr)^{-1}$ and $\varsigma_2 \coloneqq \bigl(f^{N}\big|_{X^{N+M_0}_{\c,2}} \bigr)^{-1}$. 
\end{enumerate}

\smallskip
\item[(2)] We say that $\phi$ satisfies the \defn{$\alpha$-strong non-integrability condition} (with respect to $f$ and $d$) if $\phi$ satisfies the $(\CC, \alpha)$-strong non-integrability condition with respect to $f$ and $d$ for some  Jordan curve $\CC\subseteq S^2$ with $\post f \subseteq \CC$.

\smallskip
\item[(3)] We say that $\phi$ satisfies the \defn{strong non-integrability condition} (with respect to $f$ and $d$) if $\phi$ satisfies the $\alpha'$-strong non-integrability condition with respect to $f$ and $d$ for some  $\alpha' \in (0, \alpha]$.
\end{enumerate}
\end{definition}

For given $f$, $d$, and $\alpha$ as in Definition~\ref{defStrongNonIntegrability}, if $\phi\in \Holder{\alpha} (S^2, d)$ satisfies the $(\CC,\alpha)$-strong non-integrability condition for some Jordan curve $\CC\subseteq S^2$ with $\post f \subseteq \CC$, then we fix the choices of $N_0$, $M_0$, $\varepsilon$, $Y^{M_0}_\b$, $Y^{M_0}_\w$, $x_1(X)$, $x_2(X)$, $X^{N+M_0}_{\b,1}$, $X^{N+M_0}_{\w,1}$ as in Definition~\ref{defStrongNonIntegrability}, and say that something depends only on $f$, $d$, $\alpha$, and $\phi$ even if it also depends on some of these choices.

We will see in the following lemma that the strong non-integrability condition is independent of the Jordan curve $\CC$.

\begin{lemma}   \label{lmSNIwoC}
Let $f$, $d$, $\alpha$ satisfies the Assumptions. Let $\CC$ and $\widehat\CC$ be Jordan curves on $S^2$ with $\post f \subseteq \CC \cap \widehat\CC$. Let $\phi \in \Holder{\alpha}(S^2,d)$ be a real-valued H\"{o}lder continuous function with an exponent $\alpha$. Fix arbitrary integers $n,\, \widehat{n} \in \N$. Let $F\coloneqq f^n$ and $\widehat{F} \coloneqq f^{\widehat{n}}$ be iterates of $f$. Then $\Phi \coloneqq S_n^f \phi$ satisfies the $(\CC, \alpha)$-strong non-integrability condition with respect to $F$ and $d$ if and only if $\widehat\Phi \coloneqq S_{\widehat{n}}^f \phi$ satisfies the $( \widehat\CC, \alpha )$-strong non-integrability condition with respect to $\widehat{F}$ and $d$.

In particular, if $\phi$ satisfies the $\alpha$-strong non-integrability condition with respect to $f$ and $d$, then it satisfies the $(\CC, \alpha)$-strong non-integrability condition with respect to $f$ and $d$.
\end{lemma}

\begin{proof}
Let $\Lambda>1$ be the expansion factor of the visual metric $d$ for $f$. Note that $\post f = \post F = \post \widehat{F}$, and that it follows immediately from Lemma~\ref{lmCellBoundsBM} that $d$ is a visual metric for both $F$ and $\widehat{F}$.

By Lemma~\ref{lmCellBoundsBM}~(ii) and (v), there exist numbers $C_{14} \in (0,1)$ and $l\in \N$ such that for each $\widehat{m}\in\N_0$, each $\widehat{X} \in \X^{\widehat{m}}  ( \widehat{F}, \widehat\CC  )$, there exists $X\in \X^{  \lceil   \widehat{m} \widehat{n} / n  \rceil + l}  (F, \CC)$ such that $X\subseteq \widehat{X}$ and $\diam_d (X) \geq C_{14}  \diam_d  ( \widehat{X}  )$.

By symmetry, it suffices to show the forward implication in the first statement of Lemma~\ref{lmSNIwoC}.

We assume that $\Phi$ satisfies the $(\CC,\alpha)$-strong non-integrability condition with respect to $F$ and $d$. We use the choices of numbers $N_0$, $M_0$, $\varepsilon$, tiles $Y^{M_0}_\b\in \X^{M_0}_\b(F,\CC)$,  $Y^{M_0}_\w\in \X^{M_0}_\w(F,\CC)$, $X^{N+M_0}_{\c,1}, X^{N+M_0}_{\c,2} \in \X^{N+M_0}(F,\CC)$, points $x_1(X)$, $x_2(X)$, and functions $\varsigma_1$, $\varsigma_2$ as in Definition~\ref{defStrongNonIntegrability} (with $f$ and $\phi$ replaced by $F$ and $\Phi$, respectively).

It follows from Lemma~\ref{lmCellBoundsBM}~(ii) and (v) again that we can choose an integer $\widehat{M}_0 \in \N$ large enough such that the following statements hold:
\begin{enumerate}
\smallskip
\item[(1)] $\bigl\lceil  \widehat{M}_0 \widehat{n} / n \bigr\rceil + l \geq M_0$.

\smallskip
\item[(2)] There exist $\widehat{M}_0$-tiles $\widehat{Y}^{\widehat{M}_0}_\b\in \X^{\widehat{M}_0}_\b( \widehat{F}, \widehat\CC )$ and $\widehat{Y}^{\widehat{M}_0}_\w\in \X^{\widehat{M}_0}_\w( \widehat{F}, \widehat\CC )$ such that $\widehat{Y}^{\widehat{M}_0}_\b \subseteq \inte \bigl( Y^{M_0}_\b \bigr)$ and $\widehat{Y}^{\widehat{M}_0}_\w \subseteq \inte \bigl( Y^{M_0}_\w \bigr)$.
\end{enumerate}

We define the following constants:
\begin{align}
\widehat{N}_0         &  \coloneqq    \biggl\lceil \frac{1}{ \alpha \widehat{n} }  \log_\Lambda   \frac{  2 \Hseminorm{\alpha,\, (S^2,d)}{\phi} C_0 C^{2\alpha} }{ ( 1 - \Lambda^{-\alpha} )  \varepsilon^{1+\alpha}  (1 - C_{14})  }   \biggr\rceil.  \label{eqPflmSNIwoC_wt_N0} \\
\widehat\varepsilon &\coloneqq  \varepsilon C_{14}  \in (0, \varepsilon) .   \label{eqPflmSNIwoC_wt_varepsilon}
\end{align}

For each $\c\in\{\b, \, \w\}$, each integer $\widehat{M} \geq \widehat{M}_0$, and each $\widehat{M}$-tile $\widehat{X} \in \X^{\widehat{M}} ( \widehat{F}, \widehat\CC )$ with $\widehat{X} \subseteq \widehat{Y}^{\widehat{M}_0}_\c$, we denote $M\coloneqq \bigl\lceil   \widehat{M} \widehat{n} / n \bigr\rceil + l \geq M_0$, and choose an $M$-tile $X\in\X^M(F,\CC)$ with
\begin{equation}  \label{eqPflmSNIwoC_X}
X \subseteq \widehat{X}  \qquad\text{and}\qquad \diam_d(X) \geq C_{14} \diam_d ( \widehat{X} ).
\end{equation}
Define, for each $i\in\{1, \, 2\}$,
\begin{equation}  \label{eqPflmSNIwoC_wt_x_i}
\widehat{x}_i ( \widehat{X} ) \coloneqq x_i(X).
\end{equation}

We need to verify Properties~(i) and (ii) in Definition~\ref{defStrongNonIntegrability} for the $(\widehat\CC, \alpha)$-strong non-integrability condition of $\widehat\Phi$ with respect to $\widehat{F}$ and $d$.

Fix arbitrary $\c\in\{\b, \, \w\}$, $\widehat{M}\in\N$, and $\widehat{X} \in \X^{\widehat{M}} ( \widehat{F}, \widehat\CC )$ with $\widehat{M} \geq \widehat{M}_0$ and $\widehat{X} \subseteq \widehat{Y}^{\widehat{M}_0}_\c$.

\smallskip

\emph{Property~(i).} By (\ref{eqPflmSNIwoC_X}), (\ref{eqPflmSNIwoC_wt_x_i}), (\ref{eqPflmSNIwoC_wt_varepsilon}), and Property~(i) for the $(\CC, \alpha)$-strong non-integrability condition of $\Phi$ with respect to $F$ and $d$, we get
\begin{equation*}
              d( \widehat{x}_1 ( \widehat{X} ),  \widehat{x}_2 ( \widehat{X} ) )  / \diam_d ( \widehat{X} )   
\geq       d(         x_1 (         X  ),  x_2         (         X  ) )  / \bigl( C_{14}^{-1}  \diam_d ( X )  \bigr)  
\geq   \varepsilon C_{14} 
=        \widehat\varepsilon,
\end{equation*}
and for each $i\in\{1, \, 2\}$,
\begin{equation*}
               d( \widehat{x}_i ( \widehat{X} ),  S^2 \setminus \widehat{X} )  / \diam_d ( \widehat{X} )  
\geq       d(         x_i (         X  ),  S^2 \setminus         X )   / \bigl( C_{14}^{-1}  \diam_d ( X ) \bigr)  
\geq   \varepsilon C_{14} 
=        \widehat\varepsilon.
\end{equation*}

\smallskip

\emph{Property~(ii).}  Fix an arbitrary integer $\widehat{N} \geq \widehat{N}_0$. Choose an integer $N \geq N_0$ large enough so that 
$
Nn > \widehat{N} \widehat{n}.
$

By Proposition~\ref{propCellDecomp}~(i) and (vii), for each $i\in\{1, \, 2\}$, since $F^N$ maps $X^{N+M_0}_{\c,i}$ injectively onto $Y^{M_0}_\c$ and $\widehat{Y}^{\widehat{M}_0}_\c \subseteq \inte \bigl( Y^{M_0}_\c \bigr)$, we have
\begin{equation*}
\varsigma_i \bigl(  \widehat{Y}^{\widehat{M}_0}_\c   \bigr) \in \X^{ \widehat{M}_0 \widehat{n} + Nn } (f, \widehat\CC),
\end{equation*}
where $\varsigma_i = \bigl(F^{N}\big|_{X^{N+M_0}_{\c,i}} \bigr)^{-1}$. Define, for each $i\in\{1, \, 2\}$,
\begin{equation*}
\widehat{X}^{ \widehat{N} + \widehat{M}_0 }_{\c,i} \coloneqq f^{Nn - \widehat{N} \widehat{n} }  \bigl( \varsigma_i \bigl(  \widehat{Y}^{\widehat{M}_0}_\c   \bigr) \bigr) 
\in \X^{ \widehat{N} \widehat{n} + \widehat{M}_0 \widehat{n} } (f, \widehat\CC) 
   = \X^{ \widehat{N} + \widehat{M}_0 } (\widehat{F}, \widehat\CC),
\end{equation*}
and write $\widehat\varsigma_i \coloneqq \bigl( \widehat{F}^{\widehat{N}}\big|_{\widehat{X}^{\widehat{N}+\widehat{M}_0}_{\c,i}} \bigr)^{-1}   =   \bigl( f^{\widehat{N} \widehat{n} }\big|_{\widehat{X}^{\widehat{N}+\widehat{M}_0}_{\c,i}} \bigr)^{-1} $. Note that $f^{ Nn - \widehat{N} \widehat{n} } \circ \varsigma_i = \widehat\varsigma_i$.

By (\ref{eqPflmSNIwoC_X}), (\ref{eqPflmSNIwoC_wt_x_i}), Properties~(i) and (ii) for the $(\CC, \alpha)$-strong non-integrability condition of $\Phi$ with respect to $F$ and $d$, Lemmas~\ref{lmMetricDistortion},~\ref{lmCellBoundsBM}~(ii), (\ref{eqPflmSNIwoC_wt_N0}), and (\ref{eqPflmSNIwoC_wt_varepsilon}), we have
\begin{align*}
&                        \frac{ \AbsBig{    S_{\widehat{N}}^{\widehat{F}} \widehat\Phi ( \widehat\varsigma_1 ( \widehat{x}_1( \widehat{X} )) )  -   S_{\widehat{N}}^{\widehat{F}} \widehat\Phi ( \widehat\varsigma_2 ( \widehat{x}_1( \widehat{X} )) ) 
                                                     -  S_{\widehat{N}}^{\widehat{F}} \widehat\Phi ( \widehat\varsigma_1 ( \widehat{x}_2( \widehat{X} )) ) +   S_{\widehat{N}}^{\widehat{F}} \widehat\Phi ( \widehat\varsigma_2 ( \widehat{x}_2( \widehat{X} )) )  }  }   
                                   {d(  \widehat{x}_1( \widehat{X} ), \widehat{x}_2( \widehat{X} ))^\alpha}      \\
&\qquad =        \frac{ \AbsBig{    S_{\widehat{N} \widehat{n}}^{f} \phi ( \widehat\varsigma_1 ( x_1( X )) )  -   S_{\widehat{N} \widehat{n}}^{f} \phi ( \widehat\varsigma_2 ( x_1( X )) )
                                                     -  S_{\widehat{N} \widehat{n}}^{f} \phi ( \widehat\varsigma_1 ( x_2( X )) ) +   S_{\widehat{N} \widehat{n}}^{f} \phi ( \widehat\varsigma_2 ( x_2( X )) )  }  }  
                                  {d(  x_1( X ), x_2( X ))^\alpha}    \\  
&\qquad \geq  \frac{ \AbsBig{    S_{Nn}^{f} \phi ( \varsigma_1 ( x_1( X )) )  -   S_{Nn}^{f} \phi ( \varsigma_2 ( x_1( X )) )
                                                     -  S_{Nn}^{f} \phi ( \varsigma_1 ( x_2( X )) ) +   S_{Nn}^{f} \phi ( \varsigma_2 ( x_2( X )) )  }  }  
                                  {d(  x_1( X ), x_2( X ))^\alpha}    \\
&\qquad\qquad -  \sum_{i\in\{1, \, 2\}}         \frac{ \AbsBig{    S_{ Nn - \widehat{N} \widehat{n} }^{f} \phi ( \varsigma_i ( x_1( X )) )  -   S_{ Nn - \widehat{N} \widehat{n} }^{f} \phi ( \varsigma_i ( x_2( X )) )  }  }      {d(  x_1( X ), x_2( X ))^\alpha} \\
&\qquad \geq  \frac{ \AbsBig{    S_{N}^{F} \Phi ( \varsigma_1 ( x_1( X )) )  -   S_{N}^{F} \Phi ( \varsigma_2 ( x_1( X )) )
                                                     -  S_{N}^{F} \Phi ( \varsigma_1 ( x_2( X )) ) +   S_{N}^{F} \Phi ( \varsigma_2 ( x_2( X )) )  }  }  
                                  {d(  x_1( X ), x_2( X ))^\alpha}    \\
&\qquad\qquad -  \sum_{i\in\{1, \, 2\}}       \frac{ \Hseminorm{\alpha,\, (S^2,d)}{\phi} C_0 }{ 1 - \Lambda^{-\alpha}  } 
                                                                       \cdot  \frac{    d \bigl(  \bigl( f^{ Nn - \widehat{N} \widehat{n} } \circ \varsigma_i \bigr) ( x_1( X ) )  ,  \bigl( f^{ Nn - \widehat{N} \widehat{n} } \circ \varsigma_i \bigr) ( x_2( X ) ) \bigr)    } 
                                                                                          { \varepsilon^\alpha ( \diam_d (X) )^\alpha }   \\
&\qquad \geq      \varepsilon  -   \sum_{i\in\{1, \, 2\}}       \frac{ \Hseminorm{\alpha,\, (S^2,d)}{\phi} C_0 }{ 1 - \Lambda^{-\alpha}  } 
                                                                                               \cdot  \frac{    \diam_d \bigl(  \bigl( f^{ Nn - \widehat{N} \widehat{n} } \circ \varsigma_i \bigr) ( X )   \bigr)    }   { \varepsilon^\alpha ( \diam_d (X) )^\alpha }   \\           
&\qquad \geq      \varepsilon  -    \frac{  2 \Hseminorm{\alpha,\, (S^2,d)}{\phi} C_0 }{ 1 - \Lambda^{-\alpha}  }       
                            \cdot  \frac{   C^\alpha \Lambda^{ - \alpha ( Mn+Nn - ( Nn - \widehat{N}\widehat{n} ) ) }  }   { \varepsilon^\alpha  C^{-\alpha} \Lambda^{ - \alpha  Mn } }   \\
&\qquad \geq      \varepsilon  -    \frac{  2 \Hseminorm{\alpha,\, (S^2,d)}{\phi} C_0 C^{2 \alpha} }{ ( 1 - \Lambda^{-\alpha} )  \varepsilon^\alpha  }    \Lambda^{ - \alpha\widehat{N}_0\widehat{n} } \\
&\qquad \geq      \varepsilon -  \varepsilon (1 - C_{14}) \\
&\qquad   =        \widehat\varepsilon,                       
\end{align*}
where $C\geq 1$ is the constant from Lemma~\ref{lmCellBoundsBM} and $C_0 > 1$ is the constant from Lemma~\ref{lmMetricDistortion}, both of which depend only on $f$, $\CC$, and $d$.

The first statement of Lemma~\ref{lmSNIwoC} is now established. The second statement is a special case of the first statement. 
\end{proof}

\begin{prop}  \label{propSNI2NLI}
Let $f$, $d$, $\alpha$ satisfy the Assumptions. Fix $\phi \in \Holder{\alpha} (S^2,d)$. If $\phi$ satisfies the $\alpha$-strong non-integrability condition (in the sense of Definition~\ref{defStrongNonIntegrability}), then $\phi$ is non-locally integrable (in the sense of Definition~\ref{defLI}). 
\end{prop}

\begin{proof}
We argue by contradiction and assume that $\phi$ is locally integrable and satisfies the $\alpha$-strong non-integrability condition.

Let $\Lambda>1$ be the expansion factor of $d$ for $f$. We first fix a Jordan curve $\CC \subseteq S^2$ containing $\post f$. Then we fix $N_0$, $M_0$, $Y^{M_0}_\b$, and $Y^{M_0}_\w$ as in Definition~\ref{defStrongNonIntegrability}. We choose $M\coloneqq M_0$ and consider an arbitrary $M$-tile $X\in \X^M(f,\CC)$ with $X\subseteq Y^{M_0}_\b$. We fix $x_1(X),  \, x_2(X)\in X$ satisfying Properties~(i) and (ii) in Definition~\ref{defStrongNonIntegrability}~(1). By Theorem~F in \cite{LZ24a}, $\phi = K + \beta \circ f - \beta$ for some constant $K\in\C$ and some H\"older continuous function $\beta \in \Holder{\alpha}((S^2,d),\C)$.

Then by Property~(ii) in Definition~\ref{defStrongNonIntegrability}~(1), for each $N\geq N_0$,
\begin{equation*}
 \frac {\abs{   \beta ( \varsigma_1 (x_1(X)) ) -  \beta ( \varsigma_2 (x_1(X)) )  -  \beta ( \varsigma_1 (x_2(X)) ) +  \beta ( \varsigma_2 (x_2(X)) )   }}{  d(x_1(X),x_2(X))^\alpha }
\geq \varepsilon > 0,
\end{equation*}
where $\varsigma_1 \coloneqq \bigl(f^{N}\big|_{X^{N+M_0}_{\c,1}} \bigr)^{-1}$ and $\varsigma_2 \coloneqq \bigl(f^{N}\big|_{X^{N+M_0}_{\c,2}} \bigr)^{-1}$. Combining the above with Property~(i) in Definition~\ref{defStrongNonIntegrability} and Proposition~\ref{propCellDecomp}~(i), we get
\begin{equation*}
\frac{ 2 \Hseminorm{\alpha,\, (S^2,d)}{\beta}  \bigl( \max \bigl\{ \diam_d \bigl(Y^{N+M_0} \bigr)  :  Y^{N+M_0} \in \X^{N+M_0}(f,\CC) \bigr\} \bigr)^\alpha } { \varepsilon^\alpha (\diam_d(X))^\alpha }  \geq \varepsilon > 0.
\end{equation*}
Thus, by Lemma~\ref{lmCellBoundsBM}~(ii), $2 \Hseminorm{\alpha,\, (S^2,d)}{\beta}  \frac{ C^\alpha \Lambda^{-\alpha N-\alpha M_0 } } { C^{-\alpha} \Lambda^{-\alpha M_0} }  \geq \varepsilon^{1+\alpha} > 0$, where $C\geq 1$ is the constant from Lemma~\ref{lmCellBoundsBM} depending only on $f$, $\CC$, and $d$. This is impossible since $N\geq N_0$ is arbitrary.
\end{proof}

\subsection{Dolgopyat operator}    \label{subsctDolgopyatOperator}

We now fix an expanding Thurston map $f\: S^2 \rightarrow S^2$, a visual metric $d$ on $S^2$ for $f$ with expansion factor $\Lambda>1$, a Jordan curve $\CC\subseteq S^2$ with $f(\CC)\subseteq \CC$ and $\post f\subseteq \CC$, and an eventually positive real-valued H\"{o}lder continuous function $\phi\in \Holder{\alpha}(S^2,d)$ that satisfies the $(\CC, \alpha)$-strong non-integrability condition. We use the notations from Definition~\ref{defStrongNonIntegrability} below.

We set the following constants that will be repeatedly used in this section. We will see that all these constants defined from (\ref{eqDef_m0}) to (\ref{eqDef_eta}) below depend only on $f$, $\CC$, $d$, $\alpha$, and $\phi$.

\begin{align}
  m_0       & \coloneqq      \max \bigl\{    \bigl\lceil \alpha^{-1}  \log_\Lambda  \bigl( 8C_1 \varepsilon^{\alpha - 1}  \bigr) \bigr\rceil   ,\,
                                                              \bigl\lceil \log_\Lambda \bigl( 10 \varepsilon^{-1} C^2 \bigr) \bigr\rceil   \bigr\}                    \geq 1.   \label{eqDef_m0}  \\
\delta_0  & \coloneqq     \min \bigl\{ (2C_1)^{-1},\, \varepsilon^2 C^{-2} / 20 \bigr\} \in (0,1).   \label{eqDef_delta0}  \\  
  b_0        & \coloneqq \max \bigl\{ 2 s_0  +1,\, 
                                C_0 T_0 / ( 1-\Lambda^{-\alpha}),\,
                               2 A_0 \Hseminormbig{\alpha,\, (S^2,d)}{\wt{- s_0\phi}} \big/ ( 1-\Lambda^{-\alpha})   \bigr\}.    \label{eqDef_b0} \\
   A         & \coloneqq \max \{ 3C_{4}T_0,\, 4 A_0 \}.   \label{eqDef_A} \\
\epsilon_1 & \coloneqq \min \bigl\{  \pi \delta_0 / 16 ,\,  (4A)^{-1} \Lambda^{-M_0}  \bigr\} \in(0,1).  \label{eqDef_epsilon1}  \\
  N_1      & \coloneqq \max \bigl\{   N_0,\,  \bigl\lceil  \alpha^{-1} \log_\Lambda \bigl( \max \bigl\{ 2^{10} A, \, 
                                                                 1280 A \Lambda C^2 / (\varepsilon \delta_0) ,\,4A_0, \, 4C_{4} \bigr\} \bigr) \bigr\rceil  \bigr\} .   \label{eqDef_N1} \\
 \eta       & \coloneqq \min \biggl\{ 2^{-12},\, \biggl( \frac{ \varepsilon \delta_0 \epsilon_1 }{ 1280 \Lambda C^2 } \biggr)^2, \,
                             \frac{A\epsilon_1 \varepsilon^\alpha }{240 C_{4} C^2} \Lambda^{-2\alpha m_0 - 1} \bigl(\LIP_d(f)\bigr)^{-\alpha N_1} \biggr\}. \label{eqDef_eta}
\end{align}   
Here the constants $M_0\in \N$, $N_0\in\N$, and $\varepsilon\in (0,1)$ depending only on $f$, $d$, $\CC$, and $\phi$ are from Definition~\ref{defStrongNonIntegrability}; the constant $s_0$ is the unique positive real number satisfying $P(f,-s_0\phi)=0$; the constant $C\geq 1$ depending only on $f$, $d$, and $\CC$ is from Lemma~\ref{lmCellBoundsBM}; the constant $C_0>1$ depending only on $f$, $d$, and $\CC$ is from Lemma~\ref{lmMetricDistortion}; the constant $C_1>0$ depending only on $f$, $d$, $\CC$, $\phi$, and $\alpha$ is from Lemma~\ref{lmSnPhiBound}; the constant $A_0>2$ depending only on $f$, $\CC$, $d$, $\Hseminorm{\alpha,\, (S^2,d)}{\phi}$, and $\alpha$ is from Lemma~\ref{lmBasicIneq}; the constant $C_{4}=C_{4}(f,\CC,d,\alpha,T_0)>1$ depending only on $f$, $\CC$, $d$, $\alpha$, and $\phi$ is defined in (\ref{eqDefC10}) from Lemma~\ref{lmExpToLiniear}; and the constant $T_0>0$ depending only on $f$, $\CC$, $d$, $\phi$, and $\alpha$ is defined in (\ref{eqDefT0}), and according to Lemma~\ref{lmBound_aPhi} satisfies
\begin{equation}  \label{eqT0Bound_sctDolgopyat}
\sup\bigl\{ \Hseminormbig{\alpha,\, (S^2,d)}{\wt{a\phi}}  :  a\in\R, \, \abs{a}\leq 2 s_0     \bigr\} \leq T_0.
\end{equation}

We denote for each $b\in\R$ with $\abs{b}\geq 1$,
\begin{equation}   \label{eqDefCb}
\mathfrak{C}_b \coloneqq \bigl\{  X\in \X^{m(b)}(f,\CC)  :  X\subseteq Y^{M_0}_\b \cup Y^{M_0}_\w \bigr\},
\end{equation}
where we write 
\begin{equation}   \label{eqDefmb}
m(b) \coloneqq  \bigl\lceil  \alpha^{-1}  \log_\Lambda (  C\abs{b} / \epsilon_1 ) \bigr\rceil    .
\end{equation}

Note that by (\ref{eqDef_epsilon1}),
\begin{equation*}
m(b) \geq \log_{\Lambda}( 1 / \epsilon_1 ) \geq M_0,
\end{equation*}
and if $X\in \mathfrak{C}_b$, then $\diam_d(X) \leq \bigl(  \frac{\epsilon_1}{\abs{b}}  \bigr)^{ 1 / \alpha }$ by Lemma~\ref{lmCellBoundsBM}~(ii).

For each $X\in\mathfrak{C}_b$, we now fix choices of tiles $\mathfrak{X}_1 (X),  \, \mathfrak{X}_2 (X) \in \X^{m(b)+m_0}(f,\CC)$ and $\mathfrak{X}'_1 (X),  \, \mathfrak{X}'_2 (X) \in \X^{m(b)+ 2 m_0}(f,\CC)$ in such a way that for each $i\in\{1, \, 2\}$,
\begin{equation} \label{eqXXXsubsetXX}
 x_i(X) \in \mathfrak{X}'_i (X)       \subseteq \mathfrak{X}_i(X).
\end{equation}
By Property~(i) in Definition~\ref{defStrongNonIntegrability}, (\ref{eqDef_m0}), and Lemma~\ref{lmCellBoundsBM}~(ii) and (v), it is easy to see that the constant $m_0$ we defined in (\ref{eqDef_m0}) is large enough so that the following inequalities hold:
\begin{align}
  d( \mathfrak{X}_i (X) , S^2\setminus X ) & \geq  \frac{\varepsilon}{10} C^{-1} \Lambda^{-m(b)}, \label{eqXXtoCX_LB}\\
 \diam_d (\mathfrak{X}_i (X))                     & \leq  \frac{\varepsilon}{10} C^{-1} \Lambda^{-m(b)}, \label{eqDiamXX_UB} \\
 d( \mathfrak{X}'_i(X), S^2 \setminus \mathfrak{X}_i(X) )  & \geq  \frac{\varepsilon}{10} C^{-1} \Lambda^{-m(b)-m_0}, \label{eqXXXtoCXX_LB}\\
 \diam_d (\mathfrak{X}'_i (X))                    & \leq  \frac{\varepsilon}{10} C^{-1} \Lambda^{-m(b)-m_0}  \label{eqDiamXXX_UB}
\end{align}
for $i\in\{1, \, 2\}$, and that
\begin{equation} \label{eqXX1toXX2_LB}
     d(\mathfrak{X}_1 (X),\mathfrak{X}_2 (X)  )  \geq \frac{\varepsilon}{10} C^{-1} \Lambda^{-m(b)}.  
\end{equation}

For each $X\in \mathfrak{C}_b$ and each $i\in\{1, \, 2\}$, we define a function $\psi_{i,X} \: S^2\rightarrow \R$ by
\begin{equation}    \label{eqDefCharactFn}
\psi_{i,X}(x) \coloneqq \frac{  d(x, S^2 \setminus \mathfrak{X}_i(X))^\alpha  }
                     {  d(x,\mathfrak{X}'_i(X))^\alpha + d(x, S^2 \setminus \mathfrak{X}_i(X))^\alpha  }
\end{equation}
for $x\in S^2$. Note that 
\begin{equation}  \label{eqPsi_iX}
\psi_{i,X}(x)=1  \text{ if }x\in \mathfrak{X}'_i(X), \quad \text{ and } \quad  \psi_{i,X}(x)=0 \text{ if } x\notin \mathfrak{X}_i(X).
\end{equation} 

\begin{definition}    \label{defFull}
We say that a subset $J\subseteq \{1, \, 2\} \times \{1, \, 2\} \times \mathfrak{C}_b$ has a \defn{full projection} if $\pi_3(J) = \mathfrak{C}_b$, where $\pi_3\: \{1, \, 2\} \times \{1, \, 2\} \times \mathfrak{C}_b\rightarrow \mathfrak{C}_b$ is the projection $\pi_3(j,i,X)=X$. We write $\mathcal{F}$ for the collection of all subsets of $\{1, \, 2\} \times \{1, \, 2\} \times \mathfrak{C}_b$ that have full projections.
\end{definition}

For a subset $J \subseteq \{1, \, 2\} \times \{1, \, 2\} \times \mathfrak{C}_b$, we define a function $\beta_J \: S^2 \rightarrow \R$ as
\begin{equation}    \label{eqDefBetaJ}
\beta_J(x) \coloneqq  \begin{cases} 1- \frac{\eta}{4} \sum_{i\in\{1, \, 2\}} \sum\limits_{\substack{X\in\mathfrak{C}_b\\ (1,i,X)\in J}}  \psi_{i,X} \bigl(f^{N_1}(x)\bigr) & \text{if } x\in \inte\bigl(X_{\b,1}^{N_1+M_0}\bigr) \cup \inte\bigl( X_{\w,1}^{N_1+M_0} \bigr), \\  
       1-\frac{\eta}{4} \sum_{i\in\{1, \, 2\}} \sum\limits_{\substack{X\in\mathfrak{C}_b\\ (2,i,X)\in J}}  \psi_{i,X} \bigl(f^{N_1}(x)\bigr) & \text{if } x\in \inte\bigl(X_{\b,2}^{N_1+M_0}\bigr) \cup \inte\bigl(X_{\w,2}^{N_1+M_0}\bigr), \\
       1  & \text{otherwise}, \end{cases}
\end{equation}
for $x\in S^2$.

The only properties of potentials that satisfy $\alpha$-strong non-integrability used in this section are summarized in the following proposition.

\begin{prop}   \label{propSNI}
Let $f$, $\CC$, $d$, $\alpha$, $\phi$ satisfy the Assumptions. We assume, in addition, that $f(\CC)\subseteq\CC$ and that $\phi$ satisfies the $\alpha$-strong non-integrability condition. Let $b\in\R$ with $\abs{b}\geq 1$. Using the notation above, the following statement holds:

\smallskip

For each $\c\in\{\b, \, \w\}$, each $X \in \mathfrak{C}_b$, each $x\in \mathfrak{X}'_1(X)$, and each $y\in \mathfrak{X}'_2(X)$,
\begin{equation}   \label{eqSNIBounds}
     \delta_0 
\leq \frac {\abs{  S_{N_1 }\phi ( \tau_1 (x) ) -  S_{N_1 }\phi ( \tau_2 (x) )  -S_{N_1 }\phi ( \tau_1 (y) ) +  S_{N_1 }\phi ( \tau_2 (y) )   } }{d(x,y)^\alpha}
\leq \delta_0^{-1},
\end{equation}
where we write $\tau_1 \coloneqq \bigl(f^{N_1}\big|_{X^{N_1+M_0}_{\c,1}} \bigr)^{-1}$ and $\tau_2 \coloneqq \bigl(f^{N_1}\big|_{X^{N_1+M_0}_{\c,2}} \bigr)^{-1}$.
\end{prop}

\begin{proof}
We first observe that the second inequality in (\ref{eqSNIBounds}) follows immediately from the triangle inequality, Lemma~\ref{lmSnPhiBound}, and (\ref{eqDef_delta0}).

It suffices to prove the first inequality in (\ref{eqSNIBounds}). Fix arbitrary $\c\in\{\b, \, \w\}$, $X \in \mathfrak{C}_b$, $x\in \mathfrak{X}'_1(X)$, and $y\in \mathfrak{X}'_2(X)$. By (\ref{eqXXXsubsetXX}), (\ref{eqXX1toXX2_LB}), Lemmas~\ref{lmCellBoundsBM}~(ii),~\ref{lmSnPhiBound}, and (\ref{eqXXXtoCXX_LB}),
\begin{align*}
&                          \abs{  S_{N_1 }\phi ( \tau_1 (x) ) -  S_{N_1 }\phi ( \tau_2 (x) )  -S_{N_1 }\phi ( \tau_1 (y) ) +  S_{N_1 }\phi ( \tau_2 (y) )   }   / d(x,y)^\alpha   \\
&\qquad\geq    \frac{  \abs{  S_{N_1 }\phi ( \tau_1 (x) ) -  S_{N_1 }\phi ( \tau_2 (x) )  -S_{N_1 }\phi ( \tau_1 (y) ) +  S_{N_1 }\phi ( \tau_2 (y) )   }   } { d(x_1(X),x_2(X))^\alpha }   \cdot
                            \frac{ d( \mathfrak{X}_1(X),\mathfrak{X}_2(X))^\alpha }{ (\diam_d(X))^\alpha}  \\
&\qquad\geq   \biggl( \varepsilon  -  \frac{ 2 C_1 ( \diam_d ( \mathfrak{X}'_1(X))^\alpha + 2 C_1 ( \diam_d ( \mathfrak{X}'_2(X))^\alpha }  {(\diam_d(X))^\alpha}    \biggr) 
                           \frac{ 10^{-\alpha} \varepsilon^\alpha C^{-\alpha} \Lambda^{-\alpha m(b)} }{ (\diam_d(X))^\alpha}  \\
&\qquad\geq   \biggl( \varepsilon  -  \frac{ 4  C_1 10^{-\alpha}  \varepsilon^\alpha C^{-\alpha} \Lambda^{-\alpha m(b) - \alpha m_0}   }  { C^{-\alpha} \Lambda^{-\alpha m(b)} }    \biggr)        
                           \frac{ 10^{-\alpha} \varepsilon^\alpha C^{-\alpha} \Lambda^{-\alpha m(b)} }{ C^{\alpha} \Lambda^{-\alpha m(b)} }  \\
&\qquad\geq   \varepsilon^{1+\alpha} \big/ \bigl(2C^{2\alpha} 10^\alpha \bigr) \\
&\qquad\geq \delta_0,
\end{align*}
where the last two inequalities follow from (\ref{eqDef_m0}) and (\ref{eqDef_delta0}).
\end{proof}

\begin{lemma}   \label{lmBetaJ}
Let $f$, $\CC$, $d$, $\Lambda$, $\alpha$, $\phi$, $s_0$ satisfy the Assumptions. We assume, in addition, that $f(\CC)\subseteq \CC$ and that $\phi$ satisfies the $\alpha$-strong non-integrability condition. We use the notation in this section.

Fix $b\in\R$ with $\abs{b}\geq 2 s_0 +1$. Then for each $X\in \mathfrak{C}_b$ and each $i\in\{1, \, 2\}$, the function $\psi_{i,X} \: S^2\rightarrow \R$ defined in (\ref{eqDefCharactFn}) is H\"{o}lder with an exponent $\alpha$ and
\begin{equation}  \label{eqCharactFnHolderConst}
\Hseminorm{\alpha,\, (S^2,d)}{\psi_{i,X}}  \leq 20 \varepsilon^{-\alpha} C \Lambda^{\alpha ( m(b) + 2m_0 ) }.
\end{equation} 
Moreover, for each subset $J\subseteq \{1, \, 2\}\times\{1, \, 2\}\times\mathfrak{C}_b$, the function $\beta_J\:S^2\rightarrow\R$ defined in (\ref{eqDefBetaJ}) satisfies
\begin{equation} \label{eqBetaJBounds}
1\geq \beta_J(x) \geq 1-\eta > 1 / 2 
\end{equation}
for $x\in S^2$. In addition, $\beta_J \in \Holder{\alpha}(S^2,d)$ with $\Hseminorm{\alpha,\, (S^2,d)}{\beta_J} \leq L_\beta$, where
\begin{equation}   \label{eqDefLipBoundBetaJ}
L_{\beta} \coloneqq 40 \varepsilon^{-\alpha} C \Lambda^{ \alpha ( m(b)+2m_0 ) }   (\LIP_d(f))^{\alpha N_1} \eta
\end{equation}
is a constant depending only on $f$, $\CC$, $d$, $\alpha$, $\phi$, and $b$. Here $C\geq 1$ is the constant from Lemma~\ref{lmCellBoundsBM} depending only on $f$, $\CC$, and $d$.
\end{lemma}

\begin{proof}
We will first establish (\ref{eqCharactFnHolderConst}). Consider distinct points $x, \, y\in S^2$.

If $x, \, y\in S^2 \setminus \mathfrak{X}_i(X)$, then $(\psi_{i,X}(x)-\psi_{i,X}(y) ) / d(x,y)^\alpha  = 0$.

If $x \in S^2 \setminus \mathfrak{X}_i(X)$ and $y \in \mathfrak{X}_i(X)$, then by (\ref{eqXXXtoCXX_LB}),
\begin{align*}
&         \abs{ \psi_{i,X}(x)-\psi_{i,X}(y) }  / d(x,y)^\alpha  \\
&\qquad =     d \bigl(y, S^2\setminus \mathfrak{X}_i(X) \bigr)  \cdot d(x,y)^{-\alpha }  \cdot
       \bigl( d \bigl(y,\mathfrak{X}'_i(X) \bigr)^\alpha + d\bigl(y, S^2\setminus \mathfrak{X}_i(X) \bigr)^\alpha \bigr)^{-1} \\
&\qquad \leq     d \bigl(\mathfrak{X}'_i(X), S^2\setminus \mathfrak{X}_i(X) \bigr)^{-\alpha } \\
&\qquad \leq     10^\alpha \varepsilon^{-\alpha} C^\alpha \Lambda^{ \alpha ( m(b)+m_0 ) }  \\
&\qquad \leq  20 \varepsilon^{-\alpha} C \Lambda^{ \alpha ( m(b)+2m_0 ) }.
\end{align*}

Similarly, if $y \in S^2 \setminus \mathfrak{X}_i(X)$ and $x \in \mathfrak{X}_i(X)$, then $\abs{ \psi_{i,X}(x)-\psi_{i,X}(y) } / d(x,y)^\alpha   \leq 20 \varepsilon^{-\alpha} C \Lambda^{ \alpha ( m(b)+2m_0 ) }$.

If $x, \, y\in \mathfrak{X}_i(X)$, then by (\ref{eqDiamXX_UB}), (\ref{eqXXXsubsetXX}), and (\ref{eqXXXtoCXX_LB}),
\begin{align*}    
     &   \abs{ \psi_{i,X}(x)-\psi_{i,X}(y) } / d(x,y)^\alpha   \\
&\qquad\leq  \frac{ d(x, S^2 \setminus \mathfrak{X}_i(X))^\alpha \abs{d(x,\mathfrak{X}'_i(X))^\alpha - d(y,\mathfrak{X}'_i(X))^\alpha}   }
            { d(x,y)^\alpha  (d(x,\mathfrak{X}'_i(X))^\alpha + d(x, S^2 \setminus \mathfrak{X}_i(X))^\alpha ) (d(y,\mathfrak{X}'_i(X))^\alpha + d(y, S^2 \setminus \mathfrak{X}_i(X))^\alpha )}    \\
&\qquad \; +\frac{\abs{d(x, S^2 \setminus \mathfrak{X}_i(X))^\alpha - d(y, S^2 \setminus \mathfrak{X}_i(X))^\alpha} d(x,\mathfrak{X}'_i(X))^\alpha}
     {  d(x,y)^\alpha   (d(x,\mathfrak{X}'_i(X))^\alpha + d(x, S^2 \setminus \mathfrak{X}_i(X))^\alpha ) (d(y,\mathfrak{X}'_i(X))^\alpha + d(y, S^2 \setminus \mathfrak{X}_i(X))^\alpha )}   \\ 
&\qquad\leq  \frac{ d(x, S^2 \setminus \mathfrak{X}_i(X))^\alpha d(x,y)^\alpha  + d(x,y)^\alpha d(x,\mathfrak{X}'_i(X))^\alpha  }
            {  d(x,y)^\alpha  d( \mathfrak{X}'_i(X), S^2 \setminus \mathfrak{X}_i(X)) ^{2\alpha}} \\
&\qquad\leq  \bigl(  10^{-\alpha}  \varepsilon^\alpha C^{-\alpha}\Lambda^{-\alpha m(b)}
               + 10^{-\alpha} \varepsilon^\alpha C^{-\alpha}\Lambda^{-\alpha m(b)}  \bigr)
              \bigl(  10 \varepsilon^{-1} C \Lambda^{ m(b) +m_0 }  \bigr)^{2 \alpha}    \\
&\qquad\leq  20 \varepsilon^{-\alpha} C \Lambda^{ \alpha ( m(b)+2m_0) } . 
\end{align*}

Hence, $\Hseminorm{\alpha,\, (S^2,d)}{\psi_{i,X}} \leq 20 \varepsilon^{-\alpha} C \Lambda^{ \alpha ( m(b)+2m_0) }$, establishing (\ref{eqCharactFnHolderConst}).

In order to establish (\ref{eqBetaJBounds}), we only need to observe that for each $j\in\{1, \, 2\}$, and each $x\in \inte\bigl(X_{\b,j}^{N_1+M_0}\bigr) \cup \inte\bigl( X_{\w,j}^{N_1+M_0} \bigr)$, at most one term in the summations in (\ref{eqDefBetaJ}) is nonzero. Indeed, we note that for each pair of distinct tiles $X_1, \, X_2\in\mathfrak{C}_b$, $\mathfrak{X}_{i_1}(X_1)\cap\mathfrak{X}_{i_2}(X_2) = \emptyset$ for all $i_1, \, i_2\in\{1, \, 2\}$ by (\ref{eqXXtoCX_LB}), and $\mathfrak{X}_1(X_1) \cap \mathfrak{X}_2(X_1) = \emptyset$ by (\ref{eqXX1toXX2_LB}). Hence, by (\ref{eqPsi_iX}), at most one term in the summations in (\ref{eqDefBetaJ}) is nonzero, and (\ref{eqBetaJBounds}) follows from (\ref{eqDef_eta}).

We now show the continuity of $\beta_J$. Note that for each $i\in \{1, \, 2\}$ and each $X\in\mathfrak{C}_b$, by (\ref{eqXXtoCX_LB}), (\ref{eqPsi_iX}), and the continuity of $\psi_{i,X}$, we have 
\begin{equation*}
\psi_{i,X} \bigl( f^{N_1} \bigl( \partial X_{\c,j}^{N_1+M_0} \bigr) \bigr) = \psi_{i,X}\bigl( Y_\c^{M_0} \bigr) = \{0\}
\end{equation*}
for $\c\in\{\b, \, \w\}$ and $j\in\{1, \, 2\}$. It follows immediately from (\ref{eqDefBetaJ}) that $\beta_J$ is continuous.

Finally, for arbitrary $x, \, y\in S^2$ with $x\neq y$, we will establish $\frac{\abs{\beta_J(x) - \beta_J(y)}}{ d(x,y)^\alpha } \leq L_\beta$ by considering the following two cases.

\smallskip

\emph{Case 1.} $x, \, y\in X^{N_1 + m(b)}$ for some $X^{N_1+m(b)} \in \X^{N_1+m(b)}$. If 
\begin{equation*}
X^{N_1+m(b)}  \nsubseteq \bigcup \bigl\{ X_{\c,j}^{N_1+M_0}  :  \c\in\{\b, \, \w\},\, j\in\{1, \, 2\} \bigr\},
\end{equation*}
then $\beta_J(x)-\beta_J(y) = 1-1 = 0$. If
\begin{equation*}
X^{N_1+m(b)}  \subseteq \bigcup \bigl\{ X_{\c,j}^{N_1+M_0}  :  \c\in\{\b, \, \w\},\, j\in\{1, \, 2\} \bigr\},
\end{equation*}
then by (\ref{eqPsi_iX}), 
\begin{align*}
          \frac{\abs{\beta_J(x) - \beta_J(y)}}{ d(x,y)^\alpha }
 &  =     \frac{ \bigl( 1- \frac{\eta}{4} \sum_{i\in\{1, \, 2\}} \psi_{i,X_*} \bigl( f^{N_1} (x) \bigr) \bigr)
                      -\bigl( 1- \frac{\eta}{4}  \sum_{i\in\{1, \, 2\}} \psi_{i,X_*} \bigl( f^{N_1} (y) \bigr) \bigr) } { d(x,y)^\alpha }\\
&\leq       \eta \Hseminorm{\alpha,\, (S^2,d)}{\psi_{i,X_*}}  \bigl( \LIP_d(f) \bigr)^{\alpha N_1}  \leq L_\beta,                    
\end{align*}
where we denote $X_* \coloneqq  f^{N_1} \bigl(X^{N_1+m(b)} \bigr)$.

\smallskip
\emph{Case 2.} $\card \bigl( \{x, \, y\} \cap X^{N_1+m(b)} \bigr) \leq 1$ for all $X^{N_1+m(b)} \in \X^{N_1+m(b)}$. We assume, without loss of generality, that $\beta_J(x)-\beta_J(y) \neq 0$. Then by (\ref{eqPsi_iX}) and (\ref{eqXXtoCX_LB}), $d\bigl( f^{N_1}(x),f^{N_1}(y)\bigr) \geq \frac{\varepsilon}{10} C^{-1}\Lambda^{-m(b)}$. Thus, $d(x,y) \geq \frac{\varepsilon}{10} C^{-1}\Lambda^{-m(b)}  ( \LIP_d(f) )^{-N_1}$. Hence, by (\ref{eqBetaJBounds}), 
$
\frac{\abs{\beta_J(x) - \beta_J(y)}}{ d(x,y)^\alpha } \leq  10 \varepsilon^{-\alpha} C \Lambda^{\alpha m(b)} ( \LIP_d(f) )^{\alpha N_1} \eta \leq L_\beta.
$ 
\end{proof}

\begin{definition}   \label{defDolgopyatOperator}
Let $f$, $\CC$, $d$, $\alpha$, $\phi$ satisfy the Assumptions. We assume, in addition, that $f(\CC)\subseteq \CC$ and that $\phi$ satisfies the $\alpha$-strong non-integrability condition. Let $a, \, b\in\R$ satisfy $\abs{b}\geq 1$. Denote $s \coloneqq a+\I b$. For each subset $J\subseteq \{1, \, 2\}\times\{1, \, 2\}\times \mathfrak{C}_b$, the \defn{Dolgopyat operator} $\MM_{J,s,\phi}$ on $\Holder{\alpha}\bigl(\bigl(X^0_\b,d\bigr),\C\bigr) \times \Holder{\alpha}\bigl(\bigl(X^0_\w,d \bigr),\C\bigr)$ is defined by
\begin{equation}   \label{eqDefDolgopyatOperator}
\MM_{J,s,\phi} (u_\b,u_\w) \coloneqq \RRR_{ \wt{a\phi}}^{N_1+M_0}  \bigl( u_\b \beta_J|_{X^0_\b} ,  u_\w  \beta_J|_{X^0_\w} \bigr)
\end{equation}
for $u_\b \in \Holder{\alpha}\bigl(\bigl(X^0_\b,d\bigr),\C\bigr)$ and $u_\w\in\Holder{\alpha}\bigl(\bigl(X^0_\w,d\bigr),\C\bigr)$.
\end{definition}

Here $\mathfrak{C}_b$ is defined in (\ref{eqDefCb}), $\beta_J$ is defined in (\ref{eqDefBetaJ}), $M_0\in\N$ is the constant from Definition~\ref{defStrongNonIntegrability}, and $N_1$ is given in (\ref{eqDef_N1}). Note that in (\ref{eqDefDolgopyatOperator}), since $\beta_J \in \Holder{\alpha}(S^2,d)$ (see Lemma~\ref{lmBetaJ}), we have $u_\c \beta_J|_{X^0_\c}  \in \Holder{\alpha}\bigl(\bigl(X^0_\c,d\bigr),\C\bigr)$ for $\c\in\{\b, \, \w\}$.

\subsection{Cancellation argument}    \label{subsctCancellation}

\begin{lemma}    \label{lmGibbsDoubling}
Let $f$, $\CC$, $d$ satisfy the Assumptions. Let $\varphi\in \Holder{\alpha}(S^2,d)$ be a real-valued H\"{o}lder continuous function with an exponent $\alpha\in(0,1]$. Then there exists a constant $C_{\mu_\varphi}\geq 1$ depending only on $f$, $d$, and $\varphi$ such that for all integers $m, \, n\in\N_0$, and tiles $X^n\in \X^n(f,\CC)$, $X^{m+n}\in \X^{m+n}(f,\CC)$ satisfying $X^{m+n} \subseteq X^n$, we have
\begin{equation}    \label{eqGibbsDoubling}
  \mu_\varphi(X^n) / \mu_\varphi(X^{m+n})  \leq C_{\mu_\varphi}^2 \exp  ( m  ( \norm{\varphi}_{\CCC^0(S^2) } +P(f,\varphi) )  ),
\end{equation}
where $\mu_\varphi$ is the unique equilibrium state for the map $f$ and the potential $\varphi$, and $P(f,\varphi)$ denotes the topological pressure for $f$ and $\varphi$.
\end{lemma}

\begin{proof}
By \cite[Theorems~5.16,~1.1, and Corollary~5.18]{Li18}, the unique equilibrium state $\mu_\varphi$ is a \emph{Gibbs state} with respect to $f$, $\CC$, and $\varphi$ as defined in Definition~5.3 in \cite{Li18}. More precisely, there exist constants $P_{\mu_\varphi} \in \R$ and $C_{\mu_\varphi} \geq 1$ such that for each $n\in\N_0$, each $n$-tile $X^n \in\X^n$, and each $x\in X^n$, we have
$             C_{\mu_\varphi}^{-1}
\leq    \frac{ \mu_\varphi(X^n)}  { \exp ( S_n\varphi (x) - n P_{\mu_\varphi}  ) }
\leq   C_{\mu_\varphi}$.

We fix arbitrary integers $m, \, n\in\N_0$, and tiles $X^n\in \X^n$, $X^{m+n}\in \X^{m+n}$ satisfying $X^{m+n} \subseteq X^n$. Choose an arbitrary point $x \in X^{m+n}$. Then
\begin{equation*}
                 \frac{  \mu_\varphi(X^n) }{ \mu_\varphi(X^{m+n}) } 
\leq   C_{\mu_\varphi}^2  \frac{  \exp \bigl( S_n\varphi (x) - n P_{\mu_\varphi}  \bigr) }{ \exp \bigl( S_{n+m}\varphi (x) - (n+m) P_{\mu_\varphi}  \bigr) }
\leq   C_{\mu_\varphi}^2 \exp  ( m  ( \norm{\varphi}_{\CCC^0(S^2) } +P(f,\varphi)  )   ).
\end{equation*}

Inequality~(\ref{eqGibbsDoubling}) follows immediately from the fact that $P_{\mu_\varphi} = P(f,\varphi)$ (see \cite[Theorem~5.16 and Proposition~5.17]{Li18}).
\end{proof}

\begin{lemma}   \label{lmArgumentIneqalities}
For all $z_1, \, z_2\in \C\setminus\{0\}$, the following inequalities hold:
\begin{align} 
\abs{\Arg (z_1 z_2)}            & \leq \abs{\Arg(z_1)} + \abs{\Arg(z_2)}, \label{eqArgumentTriangleIneq}  \\
\abs{z_1 + z_2}                   & \leq \abs{z_1} + \abs{z_2} -    (\Arg( z_1 / z_2 ) )^2  \min\{\abs{z_1}, \, \abs{z_2}\} / 16, \label{eqTriangleIneqWithAngle} \\
\AbsBig{ \Arg(  z_1 / z_2 ) } & \leq  2 \abs{z_1 - z_2} / \min\{\abs{z_1}, \, \abs{z_2} \}   . \label{eqArgDist}
\end{align}
\end{lemma}

\begin{proof}
Inequality (\ref{eqArgumentTriangleIneq}) follows immediately from the definition of $\Arg$ (see Section~\ref{sctNotation}).

We then verify (\ref{eqTriangleIneqWithAngle}). Without loss of generality, we assume that $\abs{z_1}\leq\abs{z_2}$ and $\theta \coloneqq \Arg\bigl(\frac{z_1}{z_2}\bigr) \geq 0$. Using the labeling in Figure~\ref{figTriangleIneqWithAngle}, we let $\overrightarrow{OQ} = z_2$ and $\overrightarrow{QC} = z_1$. Then
\begin{align*}
\abs{z_1+z_2}  &=   \abs{OA} + \abs{AC} \leq \abs{z_2} + \abs{BC} = \abs{z_2} + \abs{z_1}\cos(  \theta / 2 ) \\
             &\leq  \abs{z_2} + \abs{z_1} \biggl( 1- \frac{\theta^2}{8} + \frac{\theta^4}{4! 2^4} \biggr) 
             \leq   \abs{z_2} + \biggl( 1- \frac{\theta^2}{16} \biggr) \abs{z_1} .
\end{align*}

\begin{figure}
    \centering
    \begin{overpic}
    [width=8cm, 
    tics=20]{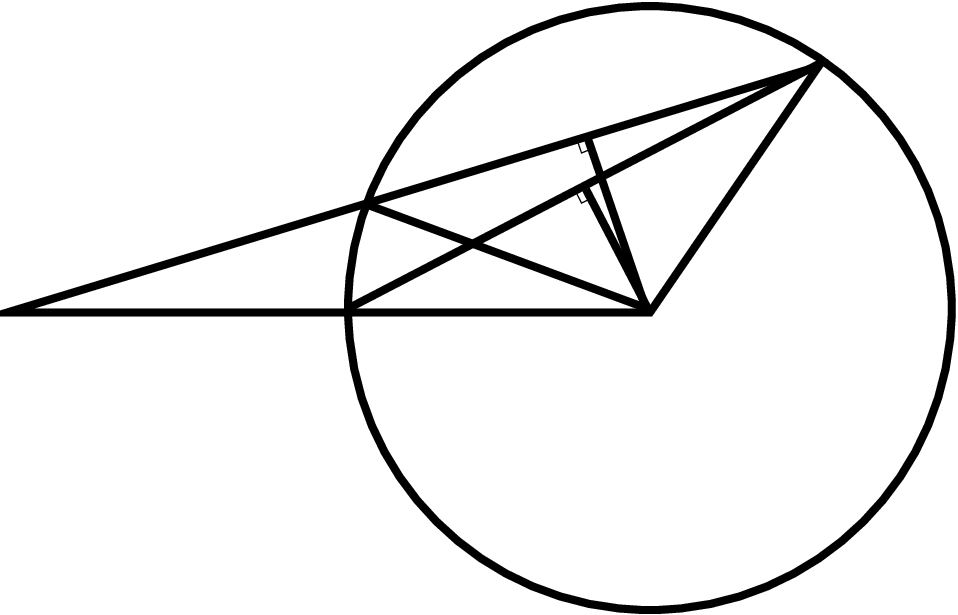}
    \put(-13,67){$O$}
    \put(77,100){$D$}
    \put(154,61){$Q$}
    \put(136,118){$A$}
    \put(131,88){$B$}
    \put(197,133){$C$}    
    \put(175,93){$z_1$}   
    \put(110,63){$z_2$}     
    \end{overpic}
    \caption{Proof of (\ref{eqTriangleIneqWithAngle}) of Lemma~\ref{lmArgumentIneqalities}.}
    \label{figTriangleIneqWithAngle}
\end{figure}

\begin{figure}
    \centering
    \begin{overpic}
    [width=5cm, 
    tics=20]{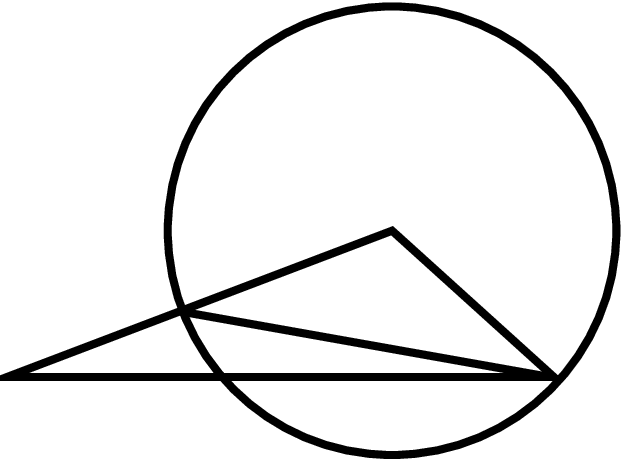}
    \put(-13,15){$B$}
    \put(25,35){$C$}
    \put(87,57){$O$}    
    \put(130,12){$A$}
    \end{overpic}
    \caption{Proof of (\ref{eqArgDist}) of Lemma~\ref{lmArgumentIneqalities}.}
    \label{figArgDist}
\end{figure}

Inequality (\ref{eqArgDist}) follows immediately from the following observation in elementary Euclidean plane geometry. As seen in Figure~\ref{figArgDist}, assume $A=z_1$ and $B=z_2$. Then $\abs{z_1-z_2}=\abs{AB}\geq \abs{AC} \geq \frac12 \abs{OA} \measuredangle AOC = \frac12 \abs{z_1} \abs{ \Arg(  z_1 / z_2 ) }$.    
\end{proof}

\begin{lemma}  \label{lmSupLeq2Inf}
Let $f$, $\CC$, $d$, $\alpha$, $\phi$, $s_0$ satisfy the Assumptions. We assume, in addition, that $f(\CC)\subseteq\CC$ and that $\phi$ satisfies the $\alpha$-strong non-integrability condition. Fix $b\in\R$ with $\abs{b} \geq 2 s_0 +1$. Fix $\c\in\{\b, \, \w\}$ and $h_\c\in K_{A\abs{b}} \bigl(X^0_\c,d\bigr)$. For each $m\geq  m(b) - M_0$ and each $m$-tile $X^m\in\X^m(f,\CC)$ with $X^m\subseteq X^0_\c$, we have
\begin{equation*}
 \sup \{ h_\c(x)  :  x\in X^m \}  \leq 2  \inf \{ h_\c(x)  :  x\in X^m \}.
\end{equation*}
\end{lemma}

Recall that the cone $K_{A\abs{b}} \bigl(X^0_\c,d\bigr)$ is defined in Definition~\ref{defCone}.

\begin{proof}
Fix arbitrary $x, \, x' \in X^m$. By Definition~\ref{defCone}, Lemma~\ref{lmCellBoundsBM}~(ii), (\ref{eqDefmb}), and (\ref{eqDef_epsilon1}),
\begin{align*}
        \abs{h_\c(x) - h_\c(x')}   
&\leq  A\abs{b} (h_\c(x) + h_\c(x')) d(x,x')^\alpha  \\
&\leq    A\abs{b} (h_\c(x) + h_\c(x'))  ( \diam_d(X^m)  )^\alpha  \\
&\leq  A\abs{b} (h_\c(x) + h_\c(x')) C\Lambda^{\alpha M_0 - \alpha m(b)} \\
&\leq  A\abs{b} ( \epsilon_1 / \abs{b} )  \Lambda^{ \alpha M_0}(h_\c(x) + h_\c(x'))  \\
&\leq    (h_\c(x) + h_\c(x')) / 4,
\end{align*}
where $C\geq 1$ is the constant from Lemma~\ref{lmCellBoundsBM} depending only on $f$, $\CC$, and $d$. The lemma follows immediately.
\end{proof}

\begin{lemma}   \label{lmUHDichotomy}
Let $f$, $\CC$, $d$, $\alpha$, $\phi$, $s_0$ satisfy the Assumptions. We assume, in addition, that $f(\CC)\subseteq \CC$ and that $\phi$ satisfies the $\alpha$-strong non-integrability condition. Fix $b\in\R$, $m\in\N$, $\c\in\{\b, \, \w\}$, $u_\c\in \Holder{\alpha}\bigl( \bigl( X^0_\c, d \bigr),\C \bigr)$, and $h_\c \in K_{A\abs{b}} \bigl( X^0_\c, d \bigr)$ such that $\abs{b} \geq 2 s_0 +1$, $m\geq N_1 + m(b)$, $\abs{u_\c(y)} \leq h_\c(y)$, and $\abs{u_\c(y) - u_\c(y')} \leq A\abs{b} (h_\c(y) + h_\c(y')) d(y,y')^\alpha$ whenever $y, \, y'\in X^0_\c$. Then for each $X^m \in \X^m(f,\CC)$ with $X^m\subseteq X^0_\c$, at least one of the following statements holds:

\begin{enumerate}
\smallskip
\item[(1)] $\abs{u_\c(x)} \leq \frac{3}{4} h_\c(x)$ for all $x\in X^m$.

\smallskip
\item[(2)] $\abs{u_\c(x)} \geq \frac{1}{4} h_\c(x)$ for all $x\in X^m$. 
\end{enumerate}
\end{lemma}

\begin{proof}
Assume that $\abs{u_\c(x_0)} < \frac{1}{4} h_\c(x_0)$ for some $x_0 \in X^m$. Then by Lemmas~\ref{lmCellBoundsBM}~(ii),~\ref{lmSupLeq2Inf}, and (\ref{eqDefmb}), for each $x\in X^m$, 
\begin{align*}
\abs{u_\c(x)}  
& < \abs{u_\c(x) - u_\c(x_0)} + 4^{-1} h_\c(x_0)  \\
& \leq  A\abs{b} (h_\c(x) + h_\c(x_0))  ( \diam_d(X^m)  )^\alpha  + 4^{-1} h_\c(x_0) \\
& \leq \bigl(2 A\abs{b} C \Lambda^{-\alpha N_1 - \alpha m(b)}   +  4^{-1} \bigr)   \sup \{ h_\c(y)  :  y\in X^m \} \\
& \leq \bigl(4 A \epsilon_1 \Lambda^{-\alpha N_1} + 2^{-1} \bigr) h_\c(x) \\
& \leq \frac{3}{4} h_\c(x),
\end{align*}
where $C\geq 1$ is the constant from Lemma~\ref{lmCellBoundsBM}. The last inequality follows from (\ref{eqDef_N1}) and the fact that $\epsilon_1\in (0,1)$ (see (\ref{eqDef_epsilon1})).
\end{proof}

\begin{lemma}   \label{lmUHDichotomySum}
Let $f$, $\CC$, $d$, $\alpha$, $\phi$, $s_0$ satisfy the Assumptions. We assume, in addition, that $f(\CC)\subseteq \CC$ and that $\phi$ satisfies the $\alpha$-strong non-integrability condition. Fix arbitrary $s \coloneqq a+\I b$ with $a, \, b\in \R$ satisfying $\abs{a-s_0}\leq s_0$ and $\abs{b}\geq b_0$. Given arbitrary $h_\b\in K_{A\abs{b}}\bigl( X^0_\b,d\bigr)$,  $h_\w\in K_{A\abs{b}}\bigl( X^0_\w,d\bigr)$, $u_\b\in\Holder{\alpha} \bigl( \bigl(X^0_\b,d\bigr), \C\bigr)$, and $u_\w\in\Holder{\alpha} \bigl( \bigl(X^0_\w,d\bigr), \C \bigr)$ satisfying the property that for each $\c\in\{\b, \, \w\}$, we have $\abs{u_\c (y)}\leq h_\c (y)$ and $\abs{ u_\c(y) - u_\c(y')} \leq A\abs{b} (h_\c(y) + h_\c(y')) d(y,y')^\alpha$ whenever $y, \, y'\in X^0_\c$.

Define the functions $Q_{\c,j}\: Y_\c^{M_0} \rightarrow \R$ for $j\in\{1, \, 2\}$ and $\c\in\{\b, \, \w\}$ by
\begin{equation*}
    Q_{\c,j}(x) 
\coloneqq \frac{ \Absbig{   \sum_{k\in\{1, \, 2\}} u_{ \varsigma (\c,k)} (\tau_k(x)) e^{S_{N_1} \wt{\minus s\phi} (\tau_k(x))} }   }
       { - \frac{1}{2}\eta h_{\varsigma(\c,j)} (\tau_j(x)) e^{S_{N_1} \wt{\minus a\phi} (\tau_j(x))}
         + \sum_{k\in\{1, \, 2\}} h_{\varsigma(\c,k)} (\tau_k(x)) e^{S_{N_1} \wt{\minus a\phi} (\tau_k(x))}    }  ,
\end{equation*}
for $x\in Y_\c^{M_0}$, where we write $\tau_k \coloneqq \bigl( f^{N_1} \big|_{X^{ N_1+M_0}_{\c,k} } \bigr)^{-1}$ for $k\in\{1, \, 2\}$, and we set $\varsigma(\c,j)\in\{\b, \, \w\}$ in such a way that $\tau_j \bigl( Y_{\c}^{M_0} \bigr) \subseteq X^0_{\varsigma(\c,j)}$ for $j\in \{1, \, 2\}$.

Then for each $\c\in\{\b, \, \w\}$ and each $X\in \mathfrak{C}_b$ with $X\subseteq Y_\c^{M_0}$, we have
\begin{equation*}
\min\bigl\{     \Norm{Q_{\c,j}}_{\CCC^0(\mathfrak{X}_i(X))}   :   i, \, j\in\{1, \, 2\} \bigr\}   \leq 1.
\end{equation*}
\end{lemma}

\begin{proof}
Fix arbitrary $\c\in\{ \b,  \, \w \}$ and $X\in \mathfrak{C}_b$ with $X\subseteq Y_\c^{M_0}$. For typographic reasons, we denote in this proof
\begin{equation}  \label{eqPflmUHDichotomySum_uhe}
u_{i,x} \coloneqq u_{\varsigma(\c,i)} ( \tau_i(x) ), \qquad 
h_{i,x} \coloneqq h_{\varsigma(\c,i)} ( \tau_i(x) ), \qquad
e_{i,x} \coloneqq e^{S_{N_1} \wt{\minus s\phi} (\tau_i(x))}
\end{equation}
for $i\in\{1, \, 2\}$ and $x \in X$.

If $\abs{ u_{j,\cdot} } \leq \frac{3}{4} h_{j,\cdot}$ on $X$, for some $j\in\{1, \, 2\}$, then $ \Norm{Q_{\c,j}}_{\CCC^0(\mathfrak{X}_i(X))} \leq 1$ for all $i\in\{1, \, 2\}$. Thus, by Lemma~\ref{lmUHDichotomy}, we can assume that 
\begin{equation}  \label{eqPflmUHDichotomySum_U14h}
\abs{ u_{k,x} } \geq   h_{k,x} / 4  \qquad \text{for all } x\in  X \text{ and }k\in\{1, \, 2\}.
\end{equation}

We define a function $\Theta\: X \rightarrow (-\pi,\pi]$ by setting
\begin{equation}  \label{eqPflmUHDichotomySum_Theta}
\Theta(x) \coloneqq \Arg  \biggl(   \frac{  u_{1,x}  e_{1,x} }{  u_{2,x}  e_{2,x}  }   \biggr)   
\qquad \text{ for } x\in X.
\end{equation}

\smallskip

We first claim that for all $x, \, y\in X$, we have
\begin{align}   
\Absbigg{    \Arg \biggl(  \frac{   u_{1,x}  /  u_{2,x}    }{    u_{1,y}  /  u_{2,y}  }  \biggr) } 
\leq 16 A \epsilon_1  \Lambda^{-\alpha N_1}
& \leq  \pi / 16
\qquad \text{ and}  \label{eqPflmUHDichotomySum_ArgUxy} \\
  \abs{b} \abs{ -S_{N_1} \phi (\tau_1(x)) + S_{N_1} \phi (\tau_2(x)) + S_{N_1} \phi (\tau_1(y)) - S_{N_1} \phi (\tau_2(y))  } 
& \leq  \pi / 16. \label{eqPflmUHDichotomySum_ArgEXPxy}
\end{align}

Indeed, by (\ref{eqArgumentTriangleIneq}) and (\ref{eqArgDist}) in Lemma~\ref{lmArgumentIneqalities}, (\ref{eqPflmUHDichotomySum_uhe}), (\ref{eqPflmUHDichotomySum_U14h}), Lemmas~\ref{lmCellBoundsBM}~(ii),~\ref{lmSupLeq2Inf}, (\ref{eqDefCb}), and (\ref{eqDefmb}),
\begin{align*}
           \Absbigg{    \Arg \biggl(  \frac{   u_{1,x}  /  u_{2,x}    }{    u_{1,y}  /  u_{2,y}  }  \biggr) }   
& \leq  \Absbigg{ \Arg \biggl(   \frac{ u_{1,x} }{ u_{1,y} } \biggr)    }   +\Absbigg{ \Arg \biggl(   \frac{ u_{2,x} }{ u_{2,y} } \biggr)    }   \\
& \leq  \sum_{j\in\{1, \, 2\}} \frac{ 2 \abs{ u_{j,x}  - u_{j,y}   } }   { \inf \{ \abs{ u_{j,z} }   :  z\in X \}  }  \\
& \leq  \sum_{j\in\{1, \, 2\}} \frac{ 8 A\abs{b}  ( h_{j,x}  + h_{j,y} )  } { \inf \{  h_{j,z}     :  z\in X \}  }     d(\tau_j(x),\tau_j(y))^\alpha  \\ 
& \leq  16 A\abs{b} \sum_{j\in\{1, \, 2\}} \frac{ \sup \{  h_{j,z}     :  z\in X \}  } { \inf \{  h_{j,z}    :  z\in X \}  }   C\Lambda^{-\alpha N_1 - \alpha m(b)}  \\
& \leq   64 A\abs{b} ( \epsilon_1 / \abs{b} ) \Lambda^{-\alpha N_1}  \\
& \leq     \pi / 16,
\end{align*} 
where $C\geq 1$ is the constant from Lemma~\ref{lmCellBoundsBM}. The last inequality follows from the fact that $N_1\geq \bigl\lceil \frac{1}{\alpha} \log_\Lambda \bigl(   2^{10} A  \bigr) \bigr\rceil $ (see (\ref{eqDef_N1})) and the fact that $\epsilon_1\in (0,1)$ (see (\ref{eqDef_epsilon1})). We have now verified (\ref{eqPflmUHDichotomySum_ArgUxy}). To show (\ref{eqPflmUHDichotomySum_ArgEXPxy}), we note that by Lemma~\ref{lmCellBoundsBM}~(ii), (\ref{eqDefCb}), (\ref{eqDefmb}), and (\ref{eqDef_epsilon1}),
\begin{align*}
     & \abs{b} \abs{ -S_{N_1} \phi (\tau_1(x)) + S_{N_1} \phi (\tau_2(x)) + S_{N_1} \phi (\tau_1(y)) - S_{N_1} \phi (\tau_2(y))  }  \\
&\qquad \leq  \abs{b} \delta_0^{-1} d(x,y)^\alpha 
\leq   \abs{b} \delta_0^{-1} ( \diam_d(X) )^\alpha
\leq   \abs{b} \delta_0^{-1}  C^\alpha \Lambda^{-\alpha m(b)}    
\leq   \delta_0^{-1} \epsilon_1
\leq   \pi / 16.
\end{align*}

The claim is now verified.

\smallskip

We will choose $i_0\in\{1, \, 2\}$, by separate discussions in the following two cases, in such a way that
\begin{equation} \label{eqPflmUHDichotomySum_ThetaLB}
\abs{\Theta(x)} \geq 16 \eta^{1/2} \qquad \text{for all } x\in \mathfrak{X}_{i_0}(X).
\end{equation}

\smallskip

\emph{Case 1.} $\abs{\Theta(y)} \geq  \pi / 4$ for some $y\in X$. Then by (\ref{eqArgumentTriangleIneq}) in Lemma~\ref{lmArgumentIneqalities}, (\ref{eqPflmUHDichotomySum_uhe}), (\ref{eqPflmUHDichotomySum_Theta}), (\ref{eqPflmUHDichotomySum_ArgUxy}), (\ref{eqPflmUHDichotomySum_ArgEXPxy}), and the fact that $\eta \in \bigl( 0, 2^{-12} \bigr)$ (see (\ref{eqDef_eta})), for each $x\in X$,
\begin{align*}
\abs{\Theta(x)}   &\geq    \abs{\Theta(y)}    -  \Absbigg{ \Arg \biggl(  \frac{ ( u_{1,y} e_{1,y} ) / ( u_{2,y} e_{2,y} ) }{    ( u_{1,x} e_{1,x} ) / ( u_{2,x} e_{2,x} ) }  \biggr)} \\
                  &\geq   \frac{\pi}{4}  -  \Absbigg{ \Arg \biggl(  \frac{   u_{1,y}  /  u_{2,y}  }{    u_{1,x}   /   u_{2,x}  }  \biggr)} 
                                                  -  \Absbigg{ \Arg \biggl(  \frac{   e_{1,y}  /  e_{2,y}  }{    e_{1,x}   /   e_{2,x}  }  \biggr)}  
                  \geq    \frac{\pi}{4} - \frac{\pi}{16} - \frac{\pi}{16}
                  \geq    \frac{\pi}{8}
                  \geq    16\eta^{1/2}.
\end{align*}
We can choose $i_0=1$ in this case.

\smallskip

\emph{Case 2.} $\abs{\Theta(z)} <  \pi / 4$ for all $z\in X$. Then by (\ref{eqArgumentTriangleIneq}) in Lemma~\ref{lmArgumentIneqalities}, (\ref{eqPflmUHDichotomySum_uhe}), (\ref{eqPflmUHDichotomySum_Theta}), (\ref{eqPflmUHDichotomySum_ArgUxy}), (\ref{eqPflmUHDichotomySum_ArgEXPxy}), $\abs{b} \geq b_0 \geq 1$ (see (\ref{eqDef_b0})), (\ref{eqSNIBounds}), (\ref{eqXX1toXX2_LB}), and (\ref{eqDefmb}), for each $x\in \mathfrak{X}_1(X)$ and each $y\in \mathfrak{X}_2(X)$,
\begin{align*}
&        \abs{ \Theta (x) - \Theta(y) } \\
&\qquad   = \Absbigg{ \Arg \biggl(  \frac{ ( u_{1,x} e_{1,x} ) / ( u_{2,x} e_{2,x} ) }{    ( u_{1,y} e_{1,y} ) / ( u_{2,y} e_{2,y} ) }  \biggr)} \\
&\qquad \geq   \Absbigg{ \Arg \biggl(  \frac{   e_{1,x}  /  e_{2,x}  }{    e_{1,y}   /   e_{2,y}  }  \biggr)}  -  \Absbigg{ \Arg \biggl(  \frac{   u_{2,y}  /  u_{1,y}  }{    u_{2,x}   /   u_{1,x}  }  \biggr)}  \\ 
&\qquad\geq      \abs{b} \abs{ - S_{N_1} \phi (\tau_1(x)) + S_{N_1} \phi (\tau_2(x)) + S_{N_1} \phi (\tau_1(y)) - S_{N_1} \phi (\tau_2(y))  } 
       -   16 A \epsilon_1  \Lambda^{-\alpha N_1}   \\
&\qquad\geq      \abs{b} \delta_0 d(x,y)^\alpha -  16  A \epsilon_1  \Lambda^{-\alpha N_1} \\
&\qquad\geq       \abs{b} \delta_0 \bigl(10^{-1} \varepsilon C^{-1}  \Lambda^{-m(b)} \bigr)^\alpha  -   16 A \epsilon_1  \Lambda^{-\alpha N_1}   \\
&\qquad\geq      \varepsilon \delta_0  ( 10 \Lambda )^{-1}  C^{-2} \epsilon_1  -   16 A \epsilon_1  \Lambda^{-\alpha N_1}   \\
&\qquad\geq       \varepsilon \delta_0 \epsilon_1 \big/ \bigl( 20\Lambda C^2 \bigr),
\end{align*}
where the last inequality follows from the observation that $ 16 A \Lambda^{-\alpha N_1} \leq \frac{\varepsilon \delta_0 }{20\Lambda C^2} $ since $N_1 \geq \bigl\lceil \frac{1}{\alpha} \log_\Lambda \bigl(  \frac{320 A \Lambda C^2}{ \varepsilon \delta_0} \bigr) \bigr\rceil$ (see (\ref{eqDef_N1})).

\smallskip

We now claim that at least one of the following statements holds:
\begin{enumerate}
\smallskip
\item[(1)] $\abs{\Theta(x)} \geq \frac{\varepsilon \delta_0 \epsilon_1}{80\Lambda C^2}$ for all $x\in \mathfrak{X}_1(X)$.

\smallskip
\item[(2)] $\abs{\Theta(y)} \geq \frac{\varepsilon \delta_0 \epsilon_1}{80\Lambda C^2}$ for all $y\in \mathfrak{X}_2(X)$.
\end{enumerate}

Indeed, assume that statement (1) fails, then there exists $x_0\in\mathfrak{X}_1(X)$ such that $ \abs{\Theta(x_0)} \leq \frac{ \varepsilon \delta_0 \epsilon_1}{80\Lambda C^2}$. Hence, for all $y\in \mathfrak{X}_2(X)$,
\begin{equation*}
     \abs{\Theta(y)}  
\geq \abs{\Theta(y)-\Theta(x_0)} - \abs{\Theta(x_0)} 
\geq \frac{\varepsilon \delta_0 \epsilon_1}{20\Lambda C^2} - \frac{\varepsilon \delta_0 \epsilon_1}{80\Lambda C^2}
\geq \frac{\varepsilon \delta_0 \epsilon_1}{80\Lambda C^2}.
\end{equation*}
The claim is now verified.

\smallskip

Thus, we can fix $i_0\in\{1, \, 2\}$ such that $\abs{\Theta(x)} \geq \frac{\varepsilon \delta_0 \epsilon_1}{80\Lambda C^2} \geq 16\eta^{1/2}$ (see (\ref{eqDef_eta})) for all $x\in \mathfrak{X}_{i_0} (X)$ in this case.

\smallskip

By (\ref{eqPflmUHDichotomySum_uhe}), Lemmas~\ref{lmSnPhiBound},~\ref{lmSupLeq2Inf},~\ref{lmCellBoundsBM}~(ii), (\ref{eqDefCb}), and (\ref{eqDefmb}), for arbitrary $x, \, y\in\mathfrak{X}_{i_0}(X)$ and $j\in\{1, \, 2\}$,
\begin{align}   \label{eqPflmUHDichotomySum_QuotHBound}
               \AbsBigg{ \frac{   h_{j,x}  \exp\bigl( S_{N_1} \wt{- a\phi} \bigl(\tau_j(x)\bigr)  \bigr) }
                                         {   h_{j,y}  \exp\bigl( S_{N_1} \wt{- a\phi} \bigl(\tau_j(y)\bigr)  \bigr) }   }  
& \leq  \Absbigg{ \frac{   h_{j,x}  } {   h_{j,y} }  }  e^{  \abs{ S_{N_1} \wt{ \minus a\phi}  (\tau_j(x) ) - S_{N_1} \wt{ \minus a\phi}  (\tau_j(y) ) } }  \notag  \\
& \leq  2 \exp \bigl(  C_0 \Hseminormbig{\alpha,\, (S^2,d)}{\wt{- a\phi}}  d(x,y)^\alpha  /  (1- \Lambda^{-\alpha} )    \bigr)  \notag  \\
& \leq  2 \exp \bigl(  C_0 \Hseminormbig{\alpha,\, (S^2,d)}{\wt{- a\phi}}   C^\alpha \Lambda^{-\alpha m(b)}  /  (1- \Lambda^{-\alpha} ) \bigr) \\
& \leq  2 \exp \bigl(  \epsilon_1  \abs{b}^{-1}  C_0 \Hseminormbig{\alpha,\, (S^2,d)}{\wt{- a\phi}} / ( 1- \Lambda^{-\alpha} )  \bigr) \notag  \\
& \leq   8,   \notag
\end{align}
where the last inequality follows from (\ref{eqDef_b0}), (\ref{eqT0Bound_sctDolgopyat}), the condition that $\abs{b} \geq b_0$, and the fact that $\epsilon_1\in(0,1)$ (see (\ref{eqDef_epsilon1})).

We fix $k_0\in\{1, \, 2\}$ such that
\begin{equation}  \label{eqPflmUHDichotomySum_k0}
  \inf \{  h_{j,x} \abs{ e_{j,x} }               :  x\in\mathfrak{X}_{i_0}(X),\, j\in\{1, \, 2\} \} 
=\inf \{  h_{k_0,x} \abs{ e_{k_0,x} }     :  x\in\mathfrak{X}_{i_0}(X) \} .  
\end{equation}   

Hence, by (\ref{eqTriangleIneqWithAngle}) in Lemma~\ref{lmArgumentIneqalities}, (\ref{eqPflmUHDichotomySum_U14h}), (\ref{eqPflmUHDichotomySum_Theta}), (\ref{eqPflmUHDichotomySum_uhe}), (\ref{eqPflmUHDichotomySum_ThetaLB}), (\ref{eqPflmUHDichotomySum_k0}), (\ref{eqDefWideTildePsiComplex}), and (\ref{eqPflmUHDichotomySum_QuotHBound}), for each $x\in\mathfrak{X}_{i_0} (X)$, we have
\begin{align*}
&                    \abs{  u_{1,x} e_{1,x}  +  u_{2,x} e_{2,x}  } \\
&\qquad \leq   - \frac{\Theta^2(x)}{16}  \min_{k\in\{1, \, 2\}}   \{  \abs{  u_{k,x} e_{k,x} } \}    + \sum_{j\in\{1, \, 2\}}   \abs{  u_{j,x} e_{j,x} }     \\
&\qquad \leq   - \frac{\Theta^2(x)}{64}  \min_{k\in\{1, \, 2\}}  \{  h_{k,x}  \abs{ e_{k,x} }    \}
                          + \sum_{j\in\{1, \, 2\}}          h_{j,x}   \abs{ e_{j,x} }          \\     
&\qquad \leq  -4\eta   \inf  \Bigl\{  h_{k_0,y}   e^{ S_{N_1} \wt{\minus a\phi}  (\tau_{k_0}(y) ) }   : y\in\mathfrak{X}_{i_0}(X) \Bigr\}
                          + \sum_{j\in\{1, \, 2\}}          h_{j,x}  e^{ S_{N_1} \wt{\minus a\phi}  (\tau_j(x) )   }     \\   
&\qquad \leq  -\frac{1}{2} \eta   h_{k_0,x}   e^{ S_{N_1} \wt{\minus a\phi}  (\tau_{k_0}(x) ) }  + \sum_{j\in\{1, \, 2\}}  h_{j,x}  e^{ S_{N_1} \wt{\minus a\phi}  (\tau_j(x) )   }    .     
\end{align*}
Therefore, we conclude that $\norm{Q_{\c,k_0}}_{\CCC^0  ( \mathfrak{X}_{i_0} (X)  ) }  \leq 1$. 
\end{proof}

\begin{prop}    \label{propDolgopyatOperator}
Let $f$, $\CC$, $d$, $\alpha$, $\phi$, $s_0$ satisfy the Assumptions. We assume, in addition, that $f(\CC)\subseteq \CC$ and that $\phi$ satisfies the $\alpha$-strong non-integrability condition. We use the notation in this section.

There exist numbers $a_0\in (0, s_0 )$ and $\rho\in (0,1)$ such that for all $s \coloneqq a+\I b$ with $a, \, b\in\R$ satisfying $\abs{a-s_0}\leq a_0$ and $\abs{b}\geq b_0$, there exists a subset $\mathcal{E}_s\subseteq \mathcal{F}$ of the set $\mathcal{F}$ of all subsets of $\{1, \, 2\}\times\{1, \, 2\}\times\mathfrak{C}_b$ with a full projection such that the following statements hold:

\begin{enumerate}
\smallskip
\item[(i)] The cone $K_{A\abs{b}}\bigl(X^0_\b,d\bigr)\times K_{A\abs{b}}\bigl(X^0_\w,d\bigr)$ is invariant under $\MM_{J,\minus s,\phi}$ for all $J\in \mathcal{F}$, i.e.,
\begin{equation*}
	\MM_{J,\minus s,\phi} \bigl( K_{A\abs{b}}\bigl(X^0_\b,d\bigr)\times K_{A\abs{b}}\bigl(X^0_\w,d\bigr) \bigr) 
           \subseteq K_{A\abs{b}}\bigl(X^0_\b,d\bigr)\times K_{A\abs{b}}\bigl(X^0_\w,d\bigr).
\end{equation*}

\smallskip
\item[(ii)] For all $J\in\mathcal{F}$, $h_\b \in K_{A\abs{b}}\bigl(X^0_\b,d\bigr)$, and $h_\w \in  K_{A\abs{b}}\bigl(X^0_\w,d\bigr)$, we have
\begin{equation}  \label{eqDolgopyatOpL2Shrinking}
                   \sum_{ \c\in\{\b, \, \w\} }    \int_{X^0_\c} \! \abs{ \pi_\c ( \MM_{J,\minus s,\phi} (h_\b,h_\w) ) }^2 \,\mathrm{d}\mu_{\minus s_0\phi}   
\leq   \rho \sum_{ \c\in\{\b, \, \w\} }    \int_{X^0_\c} \! \abs{ h_\c }^2 \,\mathrm{d}\mu_{\minus s_0\phi}.
\end{equation}

\smallskip
\item[(iii)] Given arbitrary $h_\b\in K_{A\abs{b}}\bigl( X^0_\b,d\bigr)$,  $h_\w\in K_{A\abs{b}}\bigl( X^0_\w,d\bigr)$, $u_\b\in\Holder{\alpha} \bigl( \bigl(X^0_\b,d\bigr), \C\bigr)$, and $u_\w\in\Holder{\alpha} \bigl( \bigl(X^0_\w,d\bigr), \C \bigr)$ satisfying the property that for each $\c\in\{\b, \, \w\}$, we have $\abs{u_\c (y)}\leq h_\c (y)$ and $\abs{ u_\c(y) - u_\c(y')} \leq A\abs{b} (h_\c(y) + h_\c(y')) d(y,y')^\alpha$ whenever $y,\,y'\in X^0_\c$. Then the following statement is true:

\smallskip

There exists $J\in\mathcal{E}_s$ such that
\begin{align}   
&\AbsBig{ \pi_\c \Bigl(\RRR_{\wt{\minus s\phi}}^{N_1+M_0} (u_\b,u_\w) \Bigr)  (x) }    \leq \pi_\c (\MM_{J,\minus s,\phi} (h_\b,h_\w) )(x) \qquad\text{ and} \label{eqDolgopyatOpPtwiseBound}  \\
      & \AbsBig{   \pi_\c \Bigl(\RRR_{\wt{\minus s\phi}}^{N_1+M_0} (u_\b,u_\w) \Bigr)  (x)
                 - \pi_\c \Bigl(\RRR_{\wt{\minus s\phi}}^{N_1+M_0} (u_\b,u_\w) \Bigr)  (x') }     \label{eqDolgopyatOpLipBound} \\
&\qquad \leq  A\abs{b} (  \pi_\c (\MM_{J,\minus s,\phi} (h_\b,h_\w) ) (x) + \pi_\c (\MM_{J,\minus s,\phi} (h_\b,h_\w) ) (x') ) d(x,x')^\alpha\notag 
\end{align}
for each $\c\in\{\b, \, \w\}$ and all $x,\,x'\in X^0_\c$.
\end{enumerate}

\end{prop}

\begin{proof}
For typographical convenience, we write $\iota \coloneqq N_1+M_0$ in this proof. 

We fix an arbitrary number $s=a+\I b$ with $a, \, b\in\R$ satisfying $\abs{a-s_0}\leq  s_0$ and $\abs{b}\geq b_0$.

\smallskip

(i) Without loss of generality, it suffices to show that for each $J\in \mathcal{F}$,
\begin{equation*}
\pi_\b \bigl( \MM_{J,\minus s,\phi} \bigl( K_{A\abs{b}}\bigl(X^0_\b,d\bigr)\times K_{A\abs{b}}\bigl(X^0_\w,d\bigr) \bigr)\bigr) 
           \subseteq K_{A\abs{b}}\bigl(X^0_\b,d\bigr).
\end{equation*}

Fix $J\in \mathcal{F}$, functions $h_\b \in K_{A\abs{b}}\bigl(X^0_\b,d\bigr)$, $h_\w \in K_{A\abs{b}}\bigl(X^0_\w,d\bigr)$, and points $x, \, x'\in X^0_\b$ with $x\neq x'$. For each $X^\iota \in \X^\iota_\b$, denote $y_{X^\iota} \coloneqq (f^\iota|_{X^\iota})^{-1}(x)$ and $y'_{X^\iota} \coloneqq (f^\iota|_{X^\iota})^{-1}(x')$.

Then by Definition~\ref{defDolgopyatOperator}, (\ref{eqSplitRuelleCoordinateFormula}) in Lemma~\ref{lmSplitRuelleCoordinateFormula}, Definition~\ref{defSplitRuelle}, and (\ref{eqDefWideTildePsiComplex}),
\begin{align*}
&             \abs{ \pi_\b ( \MM_{J,\minus s,\phi} (h_\b,h_\w) ) (x) - \pi_\b ( \MM_{J,\minus s,\phi} (h_\b,h_\w) ) (x') } \\
&\qquad =     \Absbigg{   \sum_{ \c \in \{ \b, \, \w \} } \RR_{\wt{\minus a\phi},\b,\c}^{(\iota)} \bigl(h_\c \beta_J|_{X^0_\c} \bigr) (x) 
                                     -  \sum_{ \c \in \{ \b, \, \w \} } \RR_{\wt{\minus a\phi},\b,\c}^{(\iota)} \bigl(h_\c \beta_J|_{X^0_\c} \bigr) (x') } \\
&\qquad\leq   \sum_{\c\in\{\b, \, \w\}} \sum\limits_{\substack{X^\iota\in\X^\iota_\b\\X^\iota\subseteq X^0_\c}}
              \AbsBig{   h_\c( y_{X^\iota} )  \beta_J ( y_{X^\iota} )  e^{S_\iota\wt{\minus a\phi}(y_{X^\iota})}  
                      -  h_\c( y'_{X^\iota} ) \beta_J ( y'_{X^\iota} ) e^{S_\iota\wt{\minus a\phi}(y'_{X^\iota})}   } \\
&\qquad\leq  \sum_{\c\in\{\b, \, \w\}} \sum\limits_{\substack{X^\iota\in\X^\iota_\b\\X^\iota\subseteq X^0_\c}}
                 \Absbig{  h_\c( y_{X^\iota} )   \beta_J ( y_{X^\iota} ) - h_\c( y'_{X^\iota} )  \beta_J ( y'_{X^\iota} ) }
                  e^{S_\iota\wt{\minus a\phi}(y'_{X^\iota})}  \\
&\qquad\quad + \sum_{\c\in\{\b, \, \w\}} \sum\limits_{\substack{X^\iota\in\X^\iota_\b\\X^\iota\subseteq X^0_\c}}
                 h_\c( y_{X^\iota} )   \beta_J ( y_{X^\iota} )  
                 \AbsBig{ e^{S_\iota\wt{\minus a\phi}(y_{X^\iota})} - e^{S_\iota\wt{\minus a\phi}(y'_{X^\iota})} }   \\
&\qquad\leq    \sum_{\c\in\{\b, \, \w\}} \sum\limits_{\substack{X^\iota\in\X^\iota_\b\\X^\iota\subseteq X^0_\c}}
                 h_\c( y_{X^\iota} ) \Absbig{    \beta_J ( y_{X^\iota} ) - \beta_J ( y'_{X^\iota} ) }
                 e^{S_\iota\wt{\minus a\phi}(y_{X^\iota})} e^{\Absbig{ S_\iota\wt{\minus a\phi}(y'_{X^\iota}) - S_\iota\wt{\minus a\phi}(y_{X^\iota})}}\\
&\qquad\quad    + \sum_{\c\in\{\b, \, \w\}} \sum\limits_{\substack{X^\iota\in\X^\iota_\b\\X^\iota\subseteq X^0_\c}}  
                  \Absbig{  h_\c( y_{X^\iota} ) - h_\c( y'_{X^\iota} ) }   \beta_J ( y'_{X^\iota} ) 
                  e^{S_\iota\wt{\minus a\phi}(y'_{X^\iota})}\\
&\qquad\quad    + \sum_{\c\in\{\b, \, \w\}} \sum\limits_{\substack{X^\iota\in\X^\iota_\b\\X^\iota\subseteq X^0_\c}}
                  h_\c( y_{X^\iota} )  \beta_J ( y_{X^\iota} )  e^{S_\iota\wt{\minus a\phi}(y_{X^\iota})} 
                  \AbsBig{1 - e^{ S_\iota\wt{\minus a\phi}(y'_{X^\iota}) - S_\iota\wt{\minus a\phi}(y_{X^\iota}) } }  .
\end{align*}

By Lemmas~\ref{lmSnPhiBound},~\ref{lmBetaJ},~\ref{lmMetricDistortion}, and~\ref{lmExpToLiniear}, the right-hand side of the last inequality is
\begin{align*}
&\leq    \exp \biggl( \frac{ T_0 C_0 \bigl(\diam_d(S^2)\bigr)^\alpha}{1-\Lambda^{-\alpha}}  \biggr)  
         \Biggl(  \sum_{\c\in\{\b, \, \w\}} \sum\limits_{\substack{X^\iota\in\X^\iota_\b\\X^\iota\subseteq X^0_\c}}
                      h_\c(y_{X^\iota}) L_\beta C_0^\alpha \Lambda^{-\iota \alpha} d(x,x')^\alpha  e^{S_\iota\wt{\minus a\phi}(y_{X^\iota})}   \\
      &    \qquad     + \sum_{\c\in\{\b, \, \w\}} \sum\limits_{\substack{X^\iota\in\X^\iota_\b\\X^\iota\subseteq X^0_\c}}
                      A\abs{b}  \Bigl(    h_\c(y_{X^\iota})   e^{S_\iota\wt{\minus a\phi}(y_{X^\iota})}  
                                       +  h_\c(y'_{X^\iota})  e^{S_\iota\wt{\minus a\phi}(y'_{X^\iota})}   \Bigr) 
           C_0^\alpha \Lambda^{-\alpha\iota} d(x,x')^\alpha    \Biggr)  \\
      & \qquad +C_{4} T_0 d(x,x')^\alpha   \sum_{\c\in\{\b, \, \w\}} 
           \RR_{\wt{\minus a\phi},\b,\c}^{(\iota)}  \bigl(h_\c \beta_J|_{X^0_\c} \bigr) (x),
\end{align*}
where $C_0>1$ is the constant from Lemma~\ref{lmMetricDistortion} depending only on $f$, $\CC$, and $d$; $L_\beta$ is the constant defined in (\ref{eqDefLipBoundBetaJ}) in Lemma~\ref{lmBetaJ}; $T_0>0$ is the constant defined in (\ref{eqDefT0}) giving an upper bound for $\Hseminormbig{\alpha,\, (S^2,d)}{\wt{ - a\phi}}$ by Lemma~\ref{lmBound_aPhi} (see also (\ref{eqT0Bound})); and $C_{4} \coloneqq C_{4}(f,\CC,d,\alpha,T_0)>1$ is the constant defined in (\ref{eqDefC10}) in Lemma~\ref{lmExpToLiniear}. Both $T_0$ and $C_{4}$ depend only on $f$, $\CC$, $d$, $\phi$, and $\alpha$. Thus, by (\ref{eqDefC10}), (\ref{eqBetaJBounds}) and (\ref{eqDefLipBoundBetaJ}) in Lemma~\ref{lmBetaJ}, Definition~\ref{defDolgopyatOperator}, (\ref{eqDefmb}), and the calculation above, we get
\begin{align*}
&              \frac{ \abs{ \pi_\b ( \MM_{J,\minus s,\phi} (h_\b,h_\w) ) (x) - \pi_\b ( \MM_{J,\minus s,\phi} (h_\b,h_\w) ) (x') }   }
                     { A\abs{b}  ( \pi_\b ( \MM_{J,\minus s,\phi} (h_\b,h_\w) ) (x) + \pi_\b ( \MM_{J,\minus s,\phi} (h_\b,h_\w) ) (x') )  
                       d(x,x')^\alpha  } \\
&\qquad\leq    \frac{C_{4}}{A\abs{b} (1-\eta)} ( L_\beta   + A\abs{b} ) \Lambda^{-\alpha\iota}
                 + \frac{C_{4} T_0}{ A\abs{b} }  \\
&\qquad\leq    \frac{C_{4}}{1-\eta} \biggl(
                   \frac{40 \varepsilon^{-\alpha} }{A\abs{b}} C  \Lambda^{2\alpha m_0 +1} \frac{ C \abs{b} }{\epsilon_1}
                           (\LIP_d(f))^{\alpha N_1} \eta   + 1 \biggr) \Lambda^{-\alpha(N_1+M_0)} + \frac{C_{4} T_0}{ A\abs{b} } \\
&\qquad\leq    1.      
\end{align*}
The last inequality follows from the observations that $\frac{C_{4}T_0}{A}\leq \frac13$ (see (\ref{eqDef_A})), that 
\begin{equation*}  
 40 \varepsilon^{-\alpha} C_{4} C^2 \eta  \Lambda^{-\alpha N_1 - \alpha M_0 + 2 \alpha m_0 +1} (\LIP_d(f))^{\alpha N_1}  / ( A \epsilon_1 (1-\eta) ) \leq 1 / 3
\end{equation*}
(by (\ref{eqDef_eta})), and that by (\ref{eqDef_N1}) and (\ref{eqDef_eta}), 
$\Lambda^{-\alpha (N_1 + M_0)} \leq \frac{1}{4 C_{4}} \leq \frac{1}{3}  \frac{1-\eta}{C_{4}}$.

\smallskip

(ii) Fix $J\in \mathcal{F}$ and two functions $h_\b \in K_{A\abs{b}} \bigl(X^0_\b, d \bigr)$, $h_\w \in K_{A\abs{b}} \bigl(X^0_\w, d \bigr)$.

We first establish that 
\begin{align}   \label{eqPfpropDolgopyatOperator_HolderIneq}
&      \bigl( \pi_\c \bigl( \MM_{J,\minus s,\phi} (h_\b,h_\w) \bigr) (x) \bigr)^2 \notag \\
&\qquad \leq         \pi_\c \Bigl( \RRR^\iota_{\wt{\minus a\phi}} \bigl( h_\b^2,h_\w^2\bigr) \Bigr) (x) 
     \cdot   \pi_\c \Bigl( \RRR^\iota_{\wt{\minus a\phi}} \Bigl( \bigl( \beta_J|_{X^0_\b} \bigr)^2, \bigl( \beta_J|_{X^0_\w} \bigr)^2 \Bigr) \Bigr) (x) 
\end{align} 
for $\c \in \{ \b,  \, \w \}$ and $x\in X^0_\c$. Indeed, fix arbitrary $\c\in\{\b, \, \w\}$ and $x\in X^0_\c$. For each $X^\iota \in \X_\c^\iota$, denote $y_{X^\iota} \coloneqq (f^\iota|_{X^\iota})^{-1}(x)$. Then by Definition~\ref{defDolgopyatOperator}, (\ref{eqSplitRuelleCoordinateFormula}) in Lemma~\ref{lmSplitRuelleCoordinateFormula}, and the Cauchy--Schwarz inequality, we have
\begin{align*}
&             \bigl( \pi_\c \bigl( \MM_{J,\minus s,\phi} (h_\b,h_\w) \bigr) (x) \bigr)^2 \\
&\qquad =     \biggl(   \sum_{\c'\in\{\b, \, \w\}} \RR_{\wt{\minus a\phi},\c,\c'}^{(\iota)} \bigl(h_{\c'} \beta_J|_{X^0_{\c'}} \bigr)(x) \biggr)^2 \\
&\qquad =     \biggl(   \sum_{\c'\in\{\b, \, \w\}} \sum\limits_{\substack{X^\iota\in\X^\iota_\c\\X^\iota\subseteq X^0_{\c'}}}
                      \bigl( h_{\c'} \beta_J \exp\bigl( S_\iota \wt{- a\phi} \bigr)  \bigr) ( y_{X^\iota} )  \biggr)^2  \\
&\qquad \leq  \biggl(   \sum_{\c'\in\{\b, \, \w\}} \sum\limits_{\substack{X^\iota\in\X^\iota_\c\\X^\iota\subseteq X^0_{\c'}}}
                      \bigl( h_{\c'}^2    \exp\bigl( S_\iota \wt{- a\phi} \bigr)  \bigr) ( y_{X^\iota} )  \biggr)
        \biggl(   \sum_{\c'\in\{\b, \, \w\}} \sum\limits_{\substack{X^\iota\in\X^\iota_\c\\X^\iota\subseteq X^0_{\c'}}}
                      \bigl( \beta_J^2 \exp\bigl( S_\iota \wt{- a\phi} \bigr)  \bigr) ( y_{X^\iota} )  \biggr)  \\
&\qquad =          \pi_\c \Bigl( \RRR^\iota_{\wt{\minus a\phi}} \bigl( h_\b^2,h_\w^2\bigr) \Bigr) (x) 
       \cdot \pi_\c \Bigl( \RRR^\iota_{\wt{\minus a\phi}} \Bigl( \bigl( \beta_J|_{X^0_\b} \bigr)^2, \bigl( \beta_J|_{X^0_\w} \bigr)^2 \Bigr) \Bigr) (x).                   
\end{align*}

\smallskip

We will focus on the case where the potential is $\wt{ - s_0\phi}$ for now, and only consider the general case at the end of the proof of statement~(ii).

Next, we define a set
\begin{equation}     \label{eqDefWJ}
W_J \coloneqq \bigcup_{(j,i,X)\in J}  f^{M_0} ( \mathfrak{X}'_i(X) ).
\end{equation}
We claim that for each $\c\in\{\b, \, \w\}$ and each $x\in W_J \cap X^0_\c$, we have
\begin{equation}   \label{eqPfpropDolgopyatOperator_1-eta}
      \pi_\c \Bigl( \RRR^\iota_{\wt{\minus s_0\phi}} \Bigl( \bigl( \beta_J|_{X^0_\b} \bigr)^2, \bigl( \beta_J|_{X^0_\w} \bigr)^2 \Bigr) \Bigr) (x) 
\leq  1 - \frac{1}{4} \eta \exp \Bigl( - \iota \Normbig{ \wt{ - s_0 \phi} }_{\CCC^0(S^2)} \Bigr).
\end{equation}
Indeed, we first fix arbitrary $\c\in\{\b, \, \w\}$ and $x\in W_J \cap X^0_\c$. Let $X\in \mathfrak{C}_b$ denote the unique $m(b)$-tile in $\mathfrak{C}_b$ with $x\in f^{M_0} (X)$. By (\ref{eqSplitRuelleCoordinateFormula}) in Lemma~\ref{lmSplitRuelleCoordinateFormula}, Definition~\ref{defSplitRuelle}, and (\ref{eqDefBetaJ}), we have
\begin{align*}
&            \pi_\c \Bigl( \RRR^\iota_{\wt{\minus s_0\phi}} \Bigl( \bigl( \beta_J|_{X^0_\b} \bigr)^2, \bigl( \beta_J|_{X^0_\w} \bigr)^2 \Bigr) \Bigr) (x) \\  
&\qquad =    \sum_{\c'\in\{\b, \, \w\}} \RR_{\wt{\minus s_0\phi}, \c, \c'}^{(\iota)} \Bigl( \bigl( \beta_J|_{X^0_{\c'}} \bigr)^2 \Bigr)(x) \\
&\qquad =    \sum_{\c'\in\{\b, \, \w\}} \sum\limits_{\substack{X^\iota\in\X^\iota_\c\\X^\iota\subseteq X^0_{\c'}}}
                 \beta_J^2 ( y_{X^\iota}) \exp \bigl( S_\iota \wt{ - s_0\phi} (y_{X^\iota}) \bigr)   \\
&\qquad\leq  \sum_{\c'\in\{\b, \, \w\}} \RR_{\wt{\minus s_0\phi},\c,\c'}^{(\iota)} \bigl(\mathbbm{1}_{X^0_{\c'}} \bigr) (x) 
            - \frac{1}{4} \eta \psi_{i_X,X} \bigl( f^{N_1}  ( y_*   )  \bigr) \exp \bigl( S_\iota \wt{ - s_0\phi} (y_*) \bigr)  \\
&\qquad\leq  1 - \frac{1}{4} \eta \exp \Bigl( - \iota \Normbig{ \wt{ - s_0 \phi} }_{\CCC^0(S^2)} \Bigr),
\end{align*}
where $i_X, \, j_X\in\{1, \, 2\}$ are chosen in such a way that $(j_X,i_X,X)\in J$ (due to the fact that $J\in\mathcal{F}$ has a full projection (see Definition~\ref{defFull})), and we denote $y_{X^\iota} \coloneqq (f^\iota|_{X^\iota})^{-1}(x)$ for $X^\iota \in \X^\iota_\c$, and write $y_* \coloneqq y_{X_{\c,j_X}^{N_1+M_0}}$. The last inequality follows from (\ref{eqSplitRuelleTildeSupDecreasing}) in Lemma~\ref{lmRtildeNorm=1}, (\ref{eqPsi_iX}), and (\ref{eqDefWJ}). The claim is now verified.

\smallskip

Next, we claim that for each $\c\in\{\b, \, \w\}$,
\begin{equation}   \label{eqPfpropDolgopyatOperator_RuelleH2inCone}
        \pi_\c \Bigl( \RRR^\iota_{\wt{\minus s_0\phi}} \bigl(h_\b^2,h_\w^2 \bigr) \Bigr) \in K_{A\abs{b}} (X^0_\c,d).
\end{equation}
Indeed, by (\ref{eqSplitRuelleCoordinateFormula}) in Lemma~\ref{lmSplitRuelleCoordinateFormula}, Lemmas~\ref{lmConePower}, and~\ref{lmBasicIneq}~(i), for all $x, \, y\in X^0_\c$,
\begin{align*}
&            \AbsBig{   \pi_\c \Bigl( \RRR^\iota_{\wt{\minus s_0\phi}} \bigl( h_\b^2,h_\w^2 \bigr) \Bigr) (x) 
                              -  \pi_\c \Bigl( \RRR^\iota_{\wt{\minus s_0\phi}} \bigl( h_\b^2,h_\w^2 \bigr) \Bigr) (y)   } \\
&\qquad\leq  \sum_{\c'\in\{\b, \, \w\}}  \AbsBig{   \RR_{\wt{\minus s_0\phi},\c,\c'}^{(\iota)} \bigl( h_{\c'}^2 \bigr)(x)
                                                                                            - \RR_{\wt{\minus s_0\phi},\c,\c'}^{(\iota)} \bigl( h_{\c'}^2 \bigr)(y)  }\\
&\qquad\leq  A_0 \biggl(  \frac{ 2A\abs{b} }{ \Lambda^{\alpha\iota} }
                        + \frac{ \Hseminormbig{\alpha,\,(S^2,d)}{\wt{ - s_0\phi}} }{ 1 - \Lambda^{-\alpha} } \biggr)  d(x,y)^\alpha 
             \sum_{\c'\in\{\b, \, \w\}}   \sum_{z\in\{x, \, y\}}   \RR_{\wt{\minus s_0\phi},\c,\c'}^{(\iota)} \bigl( h_{\c'}^2 \bigr)(z)  \\
&\qquad\leq  A\abs{b}  d(x,y)^\alpha   \sum_{z\in\{x, \, y\}}   \pi_\c \Bigl( \RRR^\iota_{\wt{\minus s_0\phi}} \bigl( h_\b^2,h_\w^2 \bigr) \Bigr) (z) ,   
\end{align*}
where $A_0=A_0\bigl(f,\CC,d,\Hseminorm{\alpha,\,(S^2,d)}{\phi}, \alpha\bigr)>2$ is the constant from Lemma~\ref{lmBasicIneq} depending only on $f$, $\CC$, $d$, $\Hseminorm{\alpha,\,(S^2,d)}{\phi}$, and $\alpha$; and $C\geq 1$ is the constant from Lemma~\ref{lmCellBoundsBM} depending only on $f$, $\CC$, and $d$. The last inequality follows from $\frac{A_0}{\Lambda^{\alpha(N_1 + M_0)}} \leq \frac14$ (see (\ref{eqDef_N1})) and 
$ 
\frac{A_0 \Hseminorm{\alpha,\,(S^2,d)}{\wt{ \minus s_0\phi}} }{ 1 - \Lambda^{-\alpha} } \leq \frac12  b_0 \leq \frac12 A b_0 \leq \frac12 A \abs{b}
$
(see (\ref{eqDef_b0}) and (\ref{eqDef_A})). The claim now follows immediately.

\smallskip

We now combine (\ref{eqPfpropDolgopyatOperator_RuelleH2inCone}), Lemmas~\ref{lmSupLeq2Inf},~\ref{lmGibbsDoubling}, (\ref{eqDefWJ}), and $\abs{b} \geq b_0 \geq 2 s_0 + 1$ (see (\ref{eqDef_b0})) to deduce that for each $\c\in\{\b, \, \w\}$, we have
\begin{align}    \label{eqPfpropDolgopyatOperator_BoundIntW_J}
&             \int_{X^0_\c} \!  \pi_\c \Bigl(  \RRR_{\wt{\minus s_0\phi}}^\iota \bigl(h_\b^2,h_\w^2\bigr) \Bigr) \,\mathrm{d}\mu_{\minus s_0\phi}    \\
&\qquad\leq   \sum\limits_{\substack{X \in\mathfrak{C}_b \\X \subseteq Y^{M_0}_\c}} 
                  \int_{f^{M_0}(X)} \, \pi_\c \Bigl(  \RRR_{\wt{\minus s_0\phi}}^\iota \bigl(h_\b^2,h_\w^2\bigr) \Bigr) \,\mathrm{d}\mu_{\minus s_0\phi}   \notag   \\
&\qquad\leq   \sum\limits_{\substack{X \in\mathfrak{C}_b \\X \subseteq Y^{M_0}_\c}}        \mu_{\minus s_0\phi} \bigl( f^{M_0}(X) \bigr)
                  \sup_{ x\in f^{M_0}(X) }  \Bigl\{ \pi_\c \Bigl(  \RRR_{\wt{\minus s_0\phi}}^\iota \bigl(h_\b^2,h_\w^2\bigr) \Bigr) (x)  \Bigr\}    \notag \\
&\qquad\leq   \sum\limits_{\substack{X \in\mathfrak{C}_b \\X \subseteq Y^{M_0}_\c}}   \mu_{\minus s_0\phi} \bigl( f^{M_0}(X) \bigr) \cdot
                   2 \inf_{  x\in f^{M_0}(X) }   \Bigl\{ \pi_\c \Bigl(  \RRR_{\wt{\minus s_0\phi}}^\iota \bigl(h_\b^2,h_\w^2\bigr) \Bigr) (x)   \Bigr\}     \notag \\  
&\qquad\leq   C_{12} \sum\limits_{\substack{X \in\mathfrak{C}_b \\X \subseteq Y^{M_0}_\c}}    \mu_{\minus s_0\phi} \bigl(  f^{M_0} \bigl( \mathfrak{X}'_{i_{J,X}}(X) \bigr) \bigr)
                   \inf_{ x\in f^{M_0} \bigl( \mathfrak{X}'_{i_{J,X}}(X) \bigr) }   \Bigl\{ \pi_\c \Bigl(  \RRR_{\wt{\minus s_0\phi}}^\iota \bigl(h_\b^2,h_\w^2\bigr) \Bigr) (x)   \Bigr\}  \notag \\     
&\qquad\leq   C_{12} \sum\limits_{\substack{X \in\mathfrak{C}_b \\X \subseteq Y^{M_0}_\c}} 
                  \int_{f^{M_0}\bigl(\mathfrak{X}'_{i_{J,X}} (X) \bigr)} \, \pi_\c \Bigl(  \RRR_{\wt{\minus s_0\phi}}^\iota \bigl(h_\b^2,h_\w^2\bigr) \Bigr) \,\mathrm{d}\mu_{\minus s_0\phi}    \notag \\    
&\qquad\leq     C_{12} \int_{W_J \cap X^0_\c} \!  \pi_\c \Bigl(  \RRR_{\wt{\minus s_0\phi}}^\iota \bigl(h_\b^2,h_\w^2\bigr) \Bigr) \,\mathrm{d}\mu_{\minus s_0\phi},      \notag                              
\end{align}
where $i_{J,X}\in \{1, \, 2\}$ can be set in such a way that either $(1,i_{J,X},X)\in J$ or $(2,i_{J,X},X)\in J$ due to the assumption that $J\in\mathcal{F}$ has a full projection, and the constant $C_{12}$ can be chosen as
$C_{12} \coloneqq 2 C_{\mu_{\minus s_0\phi}}^2 \exp \bigl( 2 m_0 \bigl( \norm{ {-}s_0\phi}_{\CCC^0(S^2)} + P(f,-s_0\phi) \bigr) \bigr) >1$,
which depends only on $f$, $\CC$, $d$, and $\phi$. Here the constant $C_{\mu_{\minus s_0\phi}}\geq 1$ is from Lemma~\ref{lmGibbsDoubling}, depending only on $f$, $d$, and $\phi$.

We now observe that by (\ref{eqSplitRuelleCoordinateFormula}) in Lemma~\ref{lmSplitRuelleCoordinateFormula} and Lemma~\ref{lmRDiscontTildeDual},
\begin{equation}   \label{eqPfpropDolgopyatOperator_IterateSplitRuelle}
   \sum_{\c\in\{\b, \, \w\}}   
       \int_{X^0_\c} \!  \pi_\c \Bigl(  \RRR_{\wt{\minus s_0\phi}}^\iota \bigl(h_\b^2,h_\w^2\bigr) \Bigr) \,\mathrm{d}\mu_{\minus s_0\phi}
=  \sum_{\c\in\{\b, \, \w\}}   \int_{X^0_\c} \! h_\c^2     \,\mathrm{d}\mu_{\minus s_0\phi}.
\end{equation}

Combining (\ref{eqPfpropDolgopyatOperator_IterateSplitRuelle}), (\ref{eqPfpropDolgopyatOperator_HolderIneq}), Lemma~\ref{lmRtildeNorm=1}, (\ref{eqBetaJBounds}) in Lemma~\ref{lmBetaJ}, (\ref{eqPfpropDolgopyatOperator_1-eta}), and (\ref{eqPfpropDolgopyatOperator_BoundIntW_J}), we get
\begin{align}  \label{eqPfpropDolgopyatOperator_Difference}
&               \sum_{\c\in\{\b, \, \w\}}   \int_{X^0_\c} \! h_\c^2     \,\mathrm{d}\mu_{\minus s_0\phi}
             -  \sum_{\c\in\{\b, \, \w\}}  \int_{X^0_\c} \Absbig{ \pi_\c \bigl( \MM_{J,\minus s_0,\phi} (h_\b,h_\w) \bigr) }^2 \, \mathrm{d}\mu_{\minus s_0\phi}  \\
&\qquad =        \sum_{\c\in\{\b, \, \w\}}   
                     \int_{X^0_\c} \!  \pi_\c \Bigl(  \RRR_{\wt{\minus s_0\phi}}^\iota \bigl(h_\b^2,h_\w^2\bigr) \Bigr) \,\mathrm{d}\mu_{\minus s_0\phi}
             -  \sum_{\c\in\{\b, \, \w\}}  \int_{X^0_\c} \Absbig{ \pi_\c \bigl( \MM_{J,\minus s_0,\phi} (h_\b,h_\w) \bigr) }^2 \, \mathrm{d}\mu_{\minus s_0\phi}  \notag \\
&\qquad\geq     \sum_{\c\in\{\b, \, \w\}}   
                 \int_{X^0_\c} \!  \pi_\c \Bigl(  \RRR_{\wt{\minus s_0\phi}}^\iota \bigl(h_\b^2,h_\w^2\bigr) \Bigr) \cdot
                 \Bigl( 1 - \pi_\c \Bigl( \RRR^\iota_{\wt{\minus s_0\phi}} \Bigl( \bigl( \beta_J|_{X^0_\b} \bigr)^2, \bigl( \beta_J|_{X^0_\w} \bigr)^2 \Bigr)\Bigr)\Bigr) \, \mathrm{d}\mu_{\minus s_0\phi}  \notag \\
&\qquad\geq     \sum_{\c\in\{\b, \, \w\}}   
                 \int_{W_J \cap X^0_\c} \!  \pi_\c \Bigl(  \RRR_{\wt{\minus s_0\phi}}^\iota \bigl(h_\b^2,h_\w^2\bigr) \Bigr) \cdot
                 \Bigl( 1 - \pi_\c \Bigl( \RRR^\iota_{\wt{\minus s_0\phi}} \Bigl( \bigl( \beta_J|_{X^0_\b} \bigr)^2, \bigl( \beta_J|_{X^0_\w} \bigr)^2 \Bigr)\Bigr)\Bigr) \, \mathrm{d}\mu_{\minus s_0\phi}   \notag\\ 
&\qquad\geq   \frac{\eta }{4} \exp \Bigl( -\iota \Normbig{\wt{ - s_0\phi}}_{\CCC^0(S^2)} \Bigr)  \sum_{\c\in\{\b, \, \w\}}   
                     \int_{W_J \cap X^0_\c} \!  \pi_\c \Bigl(  \RRR_{\wt{\minus s_0\phi}}^\iota \bigl(h_\b^2,h_\w^2\bigr) \Bigr) \,\mathrm{d}\mu_{\minus s_0\phi}  \notag \\
&\qquad\geq   \frac{\eta}{4 C_{12}}    \exp \Bigl( -\iota \Normbig{\wt{ - s_0\phi}}_{\CCC^0(S^2)} \Bigr)  \sum_{\c\in\{\b, \, \w\}}   
                     \int_{X^0_\c} \!  \pi_\c \Bigl(  \RRR_{\wt{\minus s_0\phi}}^\iota \bigl(h_\b^2,h_\w^2\bigr) \Bigr) \,\mathrm{d}\mu_{\minus s_0\phi}  \notag \\
&\qquad\geq   \frac{\eta}{4 C_{12}}    \exp \Bigl( -\iota \Normbig{\wt{ - s_0\phi}}_{\CCC^0(S^2)} \Bigr)     
              \sum_{\c\in\{\b, \, \w\}}   \int_{X^0_\c} \! h_\c^2     \,\mathrm{d}\mu_{\minus s_0\phi}.    \notag  
\end{align}

\smallskip

We now consider the general case where the potential is $\wt{- s\phi}$. Fix $\c'\in\{\b, \, \w\}$ and an arbitrary point $x\in X^0_{\c'}$. For each $X^\iota \in \X^\iota_{\c'}$, denote $y_{X^\iota} \coloneqq (f^\iota|_{X^\iota})^{-1}(x)$. Then by Definition~\ref{defDolgopyatOperator} and (\ref{eqSplitRuelleCoordinateFormula}) in Lemma~\ref{lmSplitRuelleCoordinateFormula},
\begin{align*}
&      \pi_{\c'} \bigl( \MM_{J,\minus s,\phi} (h_\b,h_\w) \bigr) (x) \\
&\quad =    \sum_{\c\in\{\b, \, \w\}} \sum\limits_{\substack{X^\iota\in\X^\iota_{\c'}\\X^\iota\subseteq X^0_\c}}
                      h_\c ( y_{X^\iota} ) \beta_J ( y_{X^\iota} ) \exp\bigl( S_\iota \wt{ - a\phi} ( y_{X^\iota} ) \bigr)   \\
&\quad\leq \sum_{\c\in\{\b, \, \w\}} \sum\limits_{\substack{X^\iota\in\X^\iota_{\c'}\\X^\iota\subseteq X^0_\c}} 
                      h_\c ( y_{X^\iota} ) \beta_J ( y_{X^\iota} ) \exp\bigl( S_\iota \wt{ - s_0\phi} ( y_{X^\iota} ) \bigr)
                      \exp \bigl(  \Absbig{ S_\iota \wt{ - a\phi} ( y_{X^\iota} ) - S_\iota \wt{ - s_0\phi} ( y_{X^\iota} ) }  \bigr)  \\
&\quad\leq \pi_{\c'} \bigl( \MM_{J,\minus s_0,\phi} (h_\b,h_\w) \bigr) (x) 
       e^{ \iota \bigl(   \abs{a-s_0} \norm{\phi}_{\CCC^0(S^2)} + \abs{P(f,-a\phi) - P(f,-s_0\phi)}
                                      + 2 \norm{ \log u_{\minus a\phi} - \log u_{\minus s_0\phi}}_{\CCC^0(S^2)}   \bigr)   } .
\end{align*}

Since the function $t\mapsto P(f,t\phi)$ is continuous (see for example, \cite[Theorem~3.6.1]{PrU10}) and the map $t\mapsto u_{t\phi}$ is continuous on $\Holder{\alpha}(S^2,d)$ equipped with the uniform norm $\norm{\cdot}_{\CCC^0(S^2)}$ by Lemma~\ref{lmUphiCountinuous}, we can choose $a_0\in(0, s_0 )$ depending only on $f$, $\CC$, $d$, $\alpha$, and $\phi$ such that if $s=a+ \I b$ with $a, \, b\in\R$ satisfies $\abs{a-s_0} \leq a_0$ and $\abs{b} \geq 2 s_0 + 1$, then
\begin{align*}
     &  \exp \bigl( \iota \bigl(   \abs{a-s_0} \norm{\phi}_{\CCC^0(S^2)} + \abs{P(f,-a\phi) - P(f,-s_0\phi)}
                                        + 2 \norm{ \log u_{\minus a\phi} - \log u_{\minus s_0\phi}}_{\CCC^0(S^2)}   \bigr)   \bigr)\\
&\qquad\leq  \bigl( 1+  (4 C_{12})^{ - 1 }  \eta \exp\bigl( -\iota \norm{\wt{ - s_0\phi}}_{\CCC^0(S^2)} \bigr)    \bigr)^{ 1 / 2 } ,
\end{align*}
and consequently,
\begin{equation}  \label{eqPfpropDolgopyatOperator_s_s0}
      \pi_{\c'} \bigl( \MM_{J,\minus s,\phi} (h_\b,h_\w) \bigr) (x)  
\leq  \biggl( 1+ \frac{ \exp\bigl( -\iota \norm{\wt{ - s_0\phi}}_{\CCC^0(S^2)} \bigr) } { 4  C_{12} }  \biggr)^{ 1 / 2 }
      \pi_{\c'}  ( \MM_{J,\minus s_0,\phi} (h_\b,h_\w)) (x).
\end{equation}

Therefore, if $s=a+\I b$ with $a, \, b\in\R$ satisfies $\abs{a-s_0} \leq a_0$ and $\abs{b} \geq b_0 \geq 2 s_0 + 1$ (see (\ref{eqDef_b0})), we get from (\ref{eqPfpropDolgopyatOperator_s_s0}) and (\ref{eqPfpropDolgopyatOperator_Difference}) that
\begin{align*}
&            \sum_{\c\in\{\b, \, \w\}}   \int_{X^0_\c} \! \abs{ \pi_\c ( \MM_{J,\minus s,\phi} (h_\b,h_\w) ) }^2 \,\mathrm{d}\mu_{\minus s_0\phi}   \\
&\qquad\leq  \biggl( 1 + \frac{ \eta \exp\bigl( -\iota \norm{\wt{ - s_0\phi}}_{\CCC^0(S^2)} \bigr) } { 4 C_{12} }  \biggr) 
             \sum_{\c\in\{\b, \, \w\}}   \int_{X^0_\c} \! \abs{ \pi_\c ( \MM_{J,\minus s_0,\phi} (h_\b,h_\w) ) }^2 \,\mathrm{d}\mu_{\minus s_0\phi}   \\
&\qquad\leq  \biggl( 1 - \frac{ \eta^2 \exp\bigl( -2\iota \norm{\wt{ - s_0\phi}}_{\CCC^0(S^2)} \bigr) } { 16 C_{12}^2 }  \biggr) 
             \sum_{\c\in\{\b, \, \w\}}  \int_{X^0_\c} \! \abs{h_\c}^2 \,\mathrm{d}\mu_{\minus s_0\phi}.   
\end{align*}
We finish the proof of (ii) by choosing 
\begin{equation*} 
\rho \coloneqq 1 -  16^{-1} C_{12}^{ - 2 }\eta^2 \exp\bigl( -2\iota \norm{\wt{ - s_0\phi}}_{\CCC^0(S^2)} \bigr)   \in (0,1),
\end{equation*}
which depends only on $f$, $\CC$, $d$, $\alpha$, and $\phi$.

\smallskip

(iii) Given arbitrary $h_\b$, $h_\w$, $u_\b$, and $u_\w$ satisfying the hypotheses in (iii), we construct a subset $J\subseteq \{1, \, 2\}\times \{1, \, 2\} \times \mathfrak{C}_b$ as follows: For each $X\in\mathfrak{C}_b$, 
\begin{enumerate}
\smallskip
\item[(1)] if $\Normbig{Q_{\c_X, 1}}_{\CCC^0(\mathfrak{X}_1(X))} \leq 1$, then include $(1,1,X)$ in $J$, otherwise

\smallskip
\item[(2)] if $\Normbig{Q_{\c_X, 2}}_{\CCC^0(\mathfrak{X}_1(X))} \leq 1$, then include $(2,1,X)$ in $J$, otherwise

\smallskip
\item[(3)] if $\Normbig{Q_{\c_X, 1}}_{\CCC^0(\mathfrak{X}_2(X))} \leq 1$, then include $(1,2,X)$ in $J$, otherwise

\smallskip
\item[(4)] if $\Normbig{Q_{\c_X, 2}}_{\CCC^0(\mathfrak{X}_2(X))} \leq 1$, then include $(2,2,X)$ in $J$,
\end{enumerate}
where we denote $\c_X\in\{\b, \, \w\}$ with the property that $X\subseteq Y_{\c_X}^{M_0}$. Here functions $Q_{\c,j} \: Y_\c^{M_0} \rightarrow \R$, $\c\in\{\b, \, \w\}$ and $j\in\{1, \, 2\}$, are defined in Lemma~\ref{lmUHDichotomySum}.

By Lemma~\ref{lmUHDichotomySum}, at least one of the four cases above occurs for each $X\in \mathfrak{C}_b$. Thus, the set $J$ constructed above has a full projection (see Definition~\ref{defFull}). 

We finally set $\mathcal{E}_s \coloneqq \bigcup \{ J \}$, where the union ranges over all $h_\b$, $h_\w$, $u_\b$, and $u_\w$ satisfying the hypotheses in (iii).

We now fix such $h_\b$, $h_\w$, $u_\b$, $u_\w$, and the corresponding $J$ constructed above. Then for each $\c\in\{\b, \, \w\}$ and each $x\in X_\c^0$, we will establish (\ref{eqDolgopyatOpPtwiseBound}) as follows:

\begin{enumerate}
\smallskip
\item[(1)] If $x \notin \bigcup_{X\in\mathfrak{C}_b}  f^{M_0} (\mathfrak{X}_1(X) \cup \mathfrak{X}_2(X))$, then by (\ref{eqPsi_iX}) and (\ref{eqDefBetaJ}), $\beta_J(y)=1$ for all $y\in f^{-(N_1+M_0)}(x)$. Thus, (\ref{eqDolgopyatOpPtwiseBound}) holds for $x$ by Definition~\ref{defDolgopyatOperator}, (\ref{eqSplitRuelleCoordinateFormula}) in Lemma~\ref{lmSplitRuelleCoordinateFormula}, and Definition~\ref{defSplitRuelle}.

\smallskip
\item[(2)] If $x\in f^{M_0} (\mathfrak{X}_i(X))$ for some $X\in\mathfrak{C}_b$ and $i\in\{1, \, 2\}$, then one of the following two cases occurs:

\begin{enumerate}
\smallskip
\item[(a)] $(1,i,X)\notin J$ and $(2,i,X)\notin J$. Then by (\ref{eqDefBetaJ}), $\beta_J(y)=1$ for all $y\in f^{-(N_1+M_0)}(x)$. Thus, (\ref{eqDolgopyatOpPtwiseBound}) holds for $x$  by Definition~\ref{defDolgopyatOperator}, (\ref{eqSplitRuelleCoordinateFormula}) in Lemma~\ref{lmSplitRuelleCoordinateFormula}, and Definition~\ref{defSplitRuelle}.

\smallskip
\item[(b)] $(j,i,X)\in J$ for some $j\in\{1, \, 2\}$. Then by the construction of $J$, we have $(j',i',X)\in J$ if and only if $(j',i')=(j,i)$. We denote the inverse branches $\tau_k \coloneqq \Bigl( f^{N_1} \big|_{X^{ N_1+M_0}_{\c,k} } \Bigr)^{-1}$ for $k\in\{1, \, 2\}$. Write $z \coloneqq \bigl( f^{N_1+M_0} \big|_{X^{ N_1+M_0}_{\c,j} } \bigr)^{-1} (x)$. Then $\beta_J(y)=1$ for each $y \in f^{-(N_1+M_0)}(x) \setminus \tau_j(\mathfrak{X}_i(X)) = f^{-(N_1+M_0)}(x) \setminus \{z\}$. In particular, $\beta_J\bigl(\tau_{j_*}\bigl(f^{N_1}(z) \bigr)\bigr) = 1$, where $j_*\in\{1, \, 2\}$ and $j_*\neq j$. By the construction of $J$, we get $Q_{\c,j} \bigl( f^{N_1}(z) \bigr) \leq 1$, i.e.,
\begin{align*}
&              \Absbigg{ \sum_{k\in\{1, \, 2\}}  \Bigl( u_{\varsigma(\c,k)} e^{S_{N_1} \wt{-s\phi} } \Bigr) \bigl( \tau_k \bigl( f^{N_1} (z) \bigr)\bigr)  }  \\
&\qquad \leq   -\frac{1}{2} \eta h_{\varsigma(\c,j)} (z) e^{S_{N_1} \wt{ \minus a\phi} (z) }
                     +  \sum_{k\in\{1, \, 2\}}  \Bigl( h_{\varsigma(\c,k)} e^{S_{N_1} \wt{ \minus a\phi} } \Bigr) \bigl( \tau_k \bigl( f^{N_1} (z) \bigr)\bigr)   \\
&\qquad \leq    \Bigl( \beta_J h_{\varsigma(\c,j  )} e^{S_{N_1} \wt{ \minus a\phi} } \Bigr) (z) 
              + \Bigl( \beta_J h_{\varsigma(\c,j_*)} e^{S_{N_1} \wt{ \minus a\phi} } \Bigr) \bigl( \tau_{j_*} \bigl( f^{N_1} (z) \bigr)\bigr),
\end{align*}
where $\varsigma(\c,k)$ is defined as in the statement of Lemma~\ref{lmUHDichotomySum}. Hence, (\ref{eqDolgopyatOpPtwiseBound}) holds for $x$ by Definition~\ref{defDolgopyatOperator}, (\ref{eqSplitRuelleCoordinateFormula}) in Lemma~\ref{lmSplitRuelleCoordinateFormula}, and Definition~\ref{defSplitRuelle}.
\end{enumerate}
\end{enumerate}

We are going to establish (\ref{eqDolgopyatOpLipBound}) now. By (\ref{eqSplitRuelleCoordinateFormula}) in Lemma~\ref{lmSplitRuelleCoordinateFormula}, (\ref{eqBasicIneqC}) in Lemma~\ref{lmBasicIneq}, Definition~\ref{defSplitRuelle}, and (\ref{eqBetaJBounds}), for all $\c\in\{\b, \, \w\}$ and $x, \, x'\in X^0_\c$ with $x\neq x'$,
\begin{align*}
&  \frac{1}{ d(x,x')^\alpha}        \AbsBig{   \pi_\c \Bigl(\RRR_{\wt{\minus s\phi}}^{N_1+M_0} (u_\b,u_\w) \Bigr)  (x)
                                                                     - \pi_\c \Bigl(\RRR_{\wt{\minus s\phi}}^{N_1+M_0} (u_\b,u_\w) \Bigr)  (x') }   \\
&\qquad\leq   \frac{1}{ d(x,x')^\alpha}  \sum_{\c'\in\{\b, \, \w\}}  \AbsBig{   \RR_{\wt{\minus s\phi},\c,\c'}^{(\iota)} (u_{\c'})   (x)
                                                                                                                                        - \RR_{\wt{\minus s\phi},\c,\c'}^{(\iota)} (u_{\c'})   (x') }   \\
&\qquad\leq  A_0  \sum_{\c'\in\{\b, \, \w\}} \biggl( \biggl(     
                       \frac{A\abs{b}}{\Lambda^{\alpha\iota}}    \sum_{z\in\{x, \, x'\}}   \RR_{\wt{\minus a\phi},\c,\c'}^{(\iota)} (h_{\c'})   (z)    \biggr) 
                                  + \abs{b}  \RR_{\wt{\minus a\phi},\c,\c'}^{(\iota)} (h_{\c'})   (x) \biggr) \\
&\qquad\leq  \biggl( \frac{ A_0 A}{\Lambda^{\alpha\iota}}  + A_0 \biggr)  \abs{b}   \sum_{\c'\in\{\b, \, \w\}}  \sum_{z\in\{x, \, x'\}}
                                \RR_{\wt{\minus a\phi},\c,\c'}^{(\iota)} \bigl(2 h_{\c'} \beta_J|_{X^0_{\c'}}\bigr)   (z)  \\
&\qquad\leq   \biggl( \frac{ 2 A_0 A}{\Lambda^{\alpha\iota}}  + 2 A_0 \biggr)
               \abs{b}  \sum_{z\in\{x, \, x'\}}  \pi_\c (\MM_{J,\minus s,\phi} (h_\b,h_\w) ) (z)  \\
&\qquad\leq   A\abs{b}  \sum_{z\in\{x, \, x'\}}  \pi_\c (\MM_{J,\minus s,\phi} (h_\b,h_\w) ) (z) ,
\end{align*}
where the last inequality follows from $\frac{ 2 A_0}{\Lambda^{\alpha\iota}} \leq \frac12$ (see (\ref{eqDef_N1})) and  $A\geq 4 A_0$ (see (\ref{eqDef_A})).
\end{proof}

\begin{proof}[Proof of Theorem~\ref{thmL2Shrinking}]
We set $\iota \coloneqq N_1 + M_0$, where $N_1\in\Z$ is defined in (\ref{eqDef_N1}) and $M_0\in\N$ is the constant from Definition~\ref{defStrongNonIntegrability}. We take the constants $a_0\in (0,  s_0 )$ and $\rho\in(0,1)$ from Proposition~\ref{propDolgopyatOperator}, and $b_0$ as defined in (\ref{eqDef_b0}).

Fix arbitrary $s \coloneqq a+\I b$ with $a, \, b\in\R$ satisfying $\abs{a-s_0} \leq a_0$ and $\abs{b}\geq b_0$. Fix arbitrary $u_\b\in\Holder{\alpha}\bigl(\bigl(X^0_\b,d\bigr),\C\bigr)$ and $u_\w\in\Holder{\alpha}\bigl(\bigl(X^0_\w,d\bigr),\C\bigr)$ satisfying 
\begin{equation}  \label{eqPfthmL2Shrinking_UNormBound}
\NHnorm{\alpha}{\Im(s)}{u_\b}{(X^0_\b,d)} \leq 1 \qquad \text{and}\qquad  \NHnorm{\alpha}{\Im(s)}{u_\w}{(X^0_\w,d)} \leq 1.
\end{equation}
We recall the constant $A\in\R$ defined in (\ref{eqDef_A}) and the subset $\mathcal{E}_s \subseteq \mathcal{F}$ constructed in Proposition~\ref{propDolgopyatOperator}. 

We will construct sequences $\{h_{\b,k}\}_{k=-1}^{+\infty}$ in $K_{A\abs{b}}\bigl( X^0_\b,d \bigr)$, $\{h_{\w,k}\}_{k=-1}^{+\infty}$ in $K_{A\abs{b}}\bigl( X^0_\w,d \bigr)$, $\{u_{\b,k}\}_{k=0}^{+\infty}$ in $\Holder{\alpha}\bigl(\bigl( X^0_\b,d \bigr),\C \bigr)$, $\{u_{\w,k}\}_{k=0}^{+\infty}$ in $\Holder{\alpha}\bigl(\bigl( X^0_\w,d \bigr),\C \bigr)$, and $\{ J_k \}_{k=0}^{+\infty}$ in $\mathcal{E}_s$ recursively so that the following properties hold for each $k\in\N_0$, each $\c\in\{\b, \, \w\}$, and all $x, \, x'\in X^0_\c$:

\begin{enumerate}
\smallskip
\item[(1)] $u_{\c,k} = \pi_\c \Bigl(  \RRR_{\wt{\minus s\phi}}^{k\iota} (u_\b,u_\w) \Bigr)$.

\smallskip
\item[(2)] $\abs{u_{\c,k}(x)} \leq h_{\c,k}(x)$ and $\abs{u_{\c,k}(x) - u_{\c,k}(x')} \leq A\abs{b} (h_{\c,k}(x) + h_{\c,k}(x')) d(x,x')^\alpha$.

\smallskip
\item[(3)] $\sum_{ \c'\in\{\b, \, \w\} } \int_{ X^0_{\c'} } \!   h_{\c',k}^2 \,\mathrm{d}\mu_{\minus s_0\phi} 
    \leq \rho \sum_{ \c'\in\{\b, \, \w\} }    \int_{ X^0_{\c'} } \!   h_{\c',k-1}^2 \,\mathrm{d}\mu_{\minus s_0\phi}$.
            
\smallskip
\item[(4)] $\pi_\c \Bigl(  \RRR_{\wt{\minus s\phi}}^{\iota} (u_{\b,k},u_{\w,k}) \Bigr) (x) \leq \pi_\c \bigl(  \MM_{J_k,\minus s,\phi} (h_{\b,k},h_{\w,k}) \bigr) (x)$ and
\begin{align*}
     & \AbsBig{    \pi_\c \Bigl(  \RRR_{\wt{\minus s\phi}}^{\iota} (u_{\b,k},u_{\w,k}) \Bigr) (x) - \pi_\c \Bigl(  \RRR_{\wt{\minus s\phi}}^{\iota} (u_{\b,k},u_{\w,k}) \Bigr) (x')  }\\ 
&\quad \leq  A\abs{b} \bigl( \pi_\c \bigl(  \MM_{J_k,\minus s,\phi} (h_{\b,k},h_{\w,k}) \bigr) (x) + \pi_\c \bigl(  \MM_{J_k,\minus s,\phi} (h_{\b,k},h_{\w,k}) \bigr) (x')  \bigr) d(x,x')^\alpha.
\end{align*}
\end{enumerate}

We first set $h_{\c,-1}  \coloneqq 1 / \rho$, $h_{\c,0} \coloneqq \NHnormD{\Holder{\alpha}(X^0_\c,d)}{b}{u_\c}\in [0,1]$, and $u_{\c,0} \coloneqq u_\c$ for each $\c\in\{\b, \, \w\}$. Then clearly, Properties~(1), (2), and (3) hold for $k=0$. By Property~(2) for $k=0$, we can choose $J_0\in\mathcal{E}_s$ according to Proposition~\ref{propDolgopyatOperator}~(iii) such that Property~(4) holds for $k=0$.

We continue our construction recursively as follows. Assume that we have chosen $u_{\b,i} \in \Holder{\alpha}\bigl( \bigl(X^0_\b,d\bigr),\C \bigr)$, $u_{\w,i} \in \Holder{\alpha}\bigl( \bigl(X^0_\w,d\bigr),\C \bigr)$, $h_{\b,i} \in K_{A\abs{b}} \bigl( X^0_\b,d\bigr)$, $h_{\w,i} \in K_{A\abs{b}} \bigl( X^0_\w,d\bigr)$, and $J_i\in\mathcal{E}_s$ for some $i\in\N_0$. Then we define, for each $\c\in\{\b, \, \w\}$,
\begin{equation*}
u_{\c,i+1}   \coloneqq   \pi_\c \Bigl(\RRR_{\wt{\minus s\phi}}^\iota(u_{\b,i}, u_{\w,i}) \Bigr)  \qquad \text{and} \qquad 
h_{\c,i+1}   \coloneqq   \pi_\c  (\MM_{J_i,\minus s,\phi}(h_{\b,i}, h_{\w,i})  ).
\end{equation*}
Then for each $\c\in\{\b, \, \w\}$, by (\ref{eqSplitRuelleRestrictHolder}) we get $u_{\c,i+1} \in \Holder{\alpha}\bigl( \bigl(X^0_\c,d \bigr), \C \bigr)$, and by Proposition~\ref{propDolgopyatOperator}~(i) we have $h_{\c,i+1} \in K_{A\abs{b}} \bigl(X^0_\c,d \bigr)$. Property~(1) for $k=i+1$ follows from Property~(1) for $k=i$. Property~(2) for $k=i+1$ follows from Property~(4) for $k=i$. Property~(3) for $k=i+1$ follows from Proposition~\ref{propDolgopyatOperator}~(ii). By Property~(2) for $k=i+1$ and Proposition~\ref{propDolgopyatOperator}~(iii), we can choose $J_{i+1}\in\mathcal{E}_s$ such that Property~(4) for $k=i+1$ holds. This completes the recursive construction and the verification of Properties~(1) through (4) for all $k\in\N_0$.

By (\ref{eqSplitRuelleCoordinateFormula}) in Lemma~\ref{lmSplitRuelleCoordinateFormula}, Properties~(1), (2), (3), and Theorem~\ref{thmEquilibriumState}~(iii), we have
\begin{align*}
            \int_{X^0_\c}\! \AbsBig{  \RR_{\wt{\minus s\phi}, \c, \b}^{(n\iota)} (u_\b) 
                                                        + \RR_{\wt{\minus s\phi}, \c, \w}^{(n\iota)} (u_\w) }^2 \,\mathrm{d}\mu_{\minus s_0\phi}
&  =    \int_{X^0_\c}\! \AbsBig{ \pi_\c \Bigl( \RRR_{\wt{\minus s\phi}}^{n\iota} (u_\b, u_\w) \Bigr) }^2 \, \mathrm{d}\mu_{\minus s_0\phi} \\
&  =    \int_{X^0_\c}\! \abs{ u_{\c,n} }^2 \, \mathrm{d}\mu_{\minus s_0\phi} \\
&\leq  \int_{X^0_\c}\! h_{\c,n} ^2 \, \mathrm{d}\mu_{\minus s_0\phi} \\
&\leq  \rho^n \biggl( \int_{X^0_\b} \!   h_{\b,0}^2 \,\mathrm{d}\mu_{\minus s_0\phi} + \int_{X^0_\w} \!  h_{\w,0}^2 \,\mathrm{d}\mu_{\minus s_0\phi} \biggr) \\
&\leq  \rho^n,
\end{align*}
for all $\c\in\{\b, \, \w\}$ and $n\in\N$.
\end{proof}

\section{Latt\`{e}s maps and smooth potentials}   \label{sctLattes}

\subsection{Non-local integrability}   \label{subsctNLI_Def}

We briefly recall the notion of non-local integrability discussed in \cite[Subsection~7.1]{LZ24a}.

Let $f\: S^2\rightarrow S^2$ be an expanding Thurston map, $d$ be a visual metric on $S^2$ for $f$, and $\CC\subseteq S^2$ be a Jordan curve satisfying $f(\CC)\subseteq \CC$ and $\post f\subseteq \CC$. We define
\begin{equation}   \label{eqDefSigma-}
\Sigma_{f,\,\CC}^- \coloneqq \bigl\{\{X_{\minus i}\}_{i\in\N_0} : X_{\minus i}\in \X^1(f,\CC) \text{ and } f \bigl(X_{\minus (i+1)}\bigr) \supseteq X_{\minus i},\text{ for } i\in\N_0 \bigr\}.
\end{equation}
For each $X\in\X^1(f,\CC)$, since $f$ is injective on $X$ (see Proposition~\ref{propCellDecomp}~(i)), we denote the inverse branch of $f$ restricted on $X$ by $f_X^{-1}\: f(X) \rightarrow X$, i.e., $f_X^{-1} \coloneqq (f|_X)^{-1}$.

Let $\psi\in \Holder{\alpha}((S^2,d),\C)$ be a complex-valued H\"{o}lder continuous function with an exponent $\alpha\in (0,1]$. For each $\xi=\{ \xi_{\minus i} \}_{i\in\N_0} \in  \Sigma_{f,\,\CC}^-$, we define the function
\begin{equation}   \label{eqDelta}
\Delta^{f,\,\CC}_{\psi,\,\xi} (x,y) \coloneqq \sum_{i=0}^{+\infty} \bigl( \bigl(\psi \circ f^{-1}_{\xi_{\minus i}} \circ \cdots \circ f^{-1}_{\xi_{0}}\bigr) (x) -  \bigl( \psi \circ f^{-1}_{\xi_{\minus i}} \circ \cdots \circ f^{-1}_{\xi_{0}}\bigr) (y)   \bigr)
\end{equation}
for each $(x,y)\in \bigcup\limits_{\substack{X\in\X^1(f,\CC) \\ X\subseteq f(\xi_0)}}X \times X$.  

The following lemma is verified in \cite[Subsection~7.1]{LZ24a}.

\begin{lemma}  \label{lmDeltaHolder}
Let $f$, $\CC$, $d$, $\psi$, $\alpha$ satisfy the Assumptions in Section~\ref{sctAssumptions}. We assume, in addition, that $f(\CC)\subseteq \CC$. Let $\xi=\{ \xi_{\minus i} \}_{i\in\N_0} \in  \Sigma_{f,\,\CC}^-$. Then for each $X\in\X^1(f,\CC)$ with $X\subseteq f(\xi_0)$, we get that $\Delta^{f,\,\CC}_{\psi,\,\xi} (x,y)$ as a series defined in (\ref{eqDelta}) converges absolutely and uniformly in $x, \, y\in X$, and moreover, for each triple of $x, \, y,z\in X$, the identity 
\begin{equation}   \label{eqDeltaDifference}
\Delta^{f,\,\CC}_{\psi,\,\xi} (x,y)   = \Delta^{f,\,\CC}_{\psi,\,\xi} (z,y) - \Delta^{f,\,\CC}_{\psi,\,\xi} (z,x)
\end{equation}
holds with
$
\Absbig{ \Delta^{f,\,\CC}_{\psi,\,\xi} (x,y) }   \leq C_1 d(x,y)^\alpha,
$
where $C_1=C_1(f,\CC,d,\psi,\alpha)$ is the constant depending on $f$, $\CC$, $d$, $\psi$, and $\alpha$ from Lemma~\ref{lmSnPhiBound}.
\end{lemma}

\begin{definition}[Temporal distance]  \label{defTemporalDist}
Let $f$, $\CC$, $d$, $\psi$, $\alpha$ satisfy the Assumptions in Section~\ref{sctAssumptions}. We assume, in addition, that $f(\CC)\subseteq \CC$. For $\xi=\{ \xi_{\minus i} \}_{i\in\N_0} \in  \Sigma_{f,\,\CC}^-$ and $\eta=\{ \eta_{\minus i} \}_{i\in\N_0} \in  \Sigma_{f,\,\CC}^-$ with $f(\xi_0) = f(\eta_0)$, we define the \defn{temporal distance} $\psi^{f,\,\CC}_{\xi,\,\eta}$ as 
\begin{equation*}
\psi^{f,\,\CC}_{\xi,\,\eta}(x,y) \coloneqq \Delta^{f,\,\CC}_{\psi,\,\xi} (x,y) - \Delta^{f,\,\CC}_{\psi,\,\eta} (x,y)
\end{equation*}
for each
$
(x,y)\in \bigcup\limits_{\substack{X\in\X^1(f,\CC) \\ X\subseteq f(\xi_0)}}X \times X.
$
\end{definition}

Recall that $f^n$ is an expanding Thurston map with $\post f^n = \post f$ for each expanding Thurston map $f\: S^2\rightarrow S^2$ and each $n\in\N$.

\begin{definition}[Local integrability]   \label{defLI}
Let $f\: S^2\rightarrow S^2$ be an expanding Thurston map and $d$ a visual metric on $S^2$ for $f$. A complex-valued H\"{o}lder continuous function $\psi \in \Holder{\alpha}((S^2,d),\C)$ is \defn{locally integrable} (with respect to $f$ and $d$) if for each natural number $n\in\N$, and each Jordan curve $\CC\subseteq S^2$ satisfying $f^n(\CC)\subseteq \CC$ and $\post f \subseteq \CC$, we have 
\begin{equation*}
\bigl(S_n^f \psi\bigr)^{f^n,\,\CC}_{\xi,\,\eta}(x,y)=0
\end{equation*}
for all $\xi=\{ \xi_{\minus i} \}_{i\in\N_0} \in  \Sigma_{f^n,\,\CC}^-$ and $\eta=\{ \eta_{\minus i} \}_{i\in\N_0} \in  \Sigma_{f^n,\,\CC}^-$ satisfying $f^n(\xi_0) = f^n(\eta_0)$, and all
$
(x,y)\in \bigcup\limits_{\substack{X\in\X^1(f^n,\CC) \\ X\subseteq f^n(\xi_0)}}X \times X.
$

The function $\psi$ is \defn{non-locally integrable} if it is not locally integrable.
\end{definition}

\subsection{Characterizations} 

In this subsection, we show that for Latt\`{e}s maps, in the class of continuously differentiable real-valued potentials, the weaker condition of non-local integrability implies the (stronger) $1$-strong non-integrability for some visual metric $d$ for $f$. This leads to a characterization in the Prime Orbit Theorem in this context (Theorem~\ref{thmLattesPOT}). The proof relies on the geometric properties of various metrics in this setting and does not generalize to other rational expanding Thurston maps. However, we are able to show the genericity of the $\alpha$-strong non-integrability condition in $\Holder{\alpha}(S^2,d)$ in the next paper \cite{LZ23c} of this series.

In order to carry out the cancellation argument in Section~\ref{sctDolgopyat}, it is crucial to have both the lower bound and the upper bound in (\ref{eqSNIBounds}). As seen in the proof of Proposition~\ref{propSNI}, the upper bound in (\ref{eqSNIBounds}) is guaranteed automatically by the H\"{o}lder continuity of the potential $\phi$ with the right exponent $\alpha$. If we could assume in addition that the identity map on $S^2$ is a bi-Lipschitz equivalence (or more generally, snowflake equivalence) from a visual metric $d$ to the Euclidean metric on $S^2$, and the temporal distance $\phi^{f,\,\CC}_{\xi,\,\xi'}$ is nonconstant and continuously differentiable, then we could expect a lower bound with the same exponent as that in the upper bound in (\ref{eqSNIBounds}) near the same point. 

However, for a rational expanding Thurston map $f\: \widehat{\C}  \rightarrow \widehat{\C}$, the chordal metric $\sigma$ (see Remark~\ref{rmChordalVisualQSEquiv} for the definition), which is bi-Lipschitz equivalent to the Euclidean metric away from the infinity, is never a visual metric for $f$ (see \cite[Lemma~8.12]{BM17}). In fact, $(S^2,d)$ is snowflake equivalent to $\bigl(\widehat{\C}, \sigma \bigr)$ if and only if $f$ is topologically conjugate to a Latt\`{e}s map (see \cite[Theorem~18.1~(iii)]{BM17} and Definition~\ref{defLattesMap} below). 

Recall that we call two metric spaces $(X_1,d_1)$ and $(X_2,d_2)$ are \defn{bi-Lipschitz}, \defn{snowflake}, or \defn{quasisymmetrically equivalent} if there exists a homeomorphism from $(X_1,d_1)$ to $(X_2,d_2)$ with the corresponding property (see Definition~\ref{defQuasiSymmetry}).

We recall a version of the definition of Latt\`{e}s maps.

\begin{definition}   \label{defLattesMap}
Let $f\: \widehat{\C}  \rightarrow \widehat{\C}$ be a rational Thurston map on the Riemann sphere $\widehat{\C}$. If $f$ is expanding and the orbifold $\mathcal{O}_f = (S^2, \alpha_f)$ associated to $f$ is parabolic, then it is called a \defn{Latt\`{e}s map}.
\end{definition}

See \cite[Chapter~3]{BM17} and \cite{Mi06} for other equivalent definitions and more properties of Latt\`{e}s maps.

The special phenomenon mentioned above is not common in the study of Prime Orbit Theorems for smooth dynamical systems, as we are endeavoring out of Riemannian settings into general self-similar metric spaces. We content ourselves with the smooth examples of strongly non-integrable potentials for Latt\`{e}s maps in Proposition~\ref{propLattes} below.

\begin{rem}  \label{rmCanonicalOrbifoldMetric}
For a Latt\`{e}s map $f\: \widehat{\C} \rightarrow \widehat{\C}$, the universal orbifold covering map $\Theta \: \C \rightarrow \widehat{\C}$ of the orbifold $\mathcal{O}_f = \bigl( \widehat{\C}, \alpha_f \bigr)$ associated to $f$ is holomorphic (see \cite[Theorem~A.26, Definition~A.27, and Corollary~A.29]{BM17}). Let $d_0$ be the Euclidean metric on $\C$. Then the \defn{canonical orbifold metric} $\omega_f$ of $f$ is the pushforward of $d_0$ by $\Theta$, more precisely,
\begin{equation*}
\omega_f(p,q)  \coloneqq \inf \bigl\{ d_0(z,w)  :  z\in \Theta^{-1}(p), \, w\in \Theta^{-1}(q)  \bigr\} 
\end{equation*}
for $p, \, q\in\widehat{\C}$ (see Section~2.5 and Appendices~A.9 and~A.10 in \cite{BM17} for more details on the canonical orbifold metric). Let $\sigma$ be the chordal metric on $\widehat{\C}$ as recalled in Remark~\ref{rmChordalVisualQSEquiv}. By \cite[Proposition~8.5]{BM17}, $\omega_f$ is a visual metric for $f$. By \cite[Lemma~A.34]{BM17}, $\bigl( \widehat{\C}, \omega_f \bigr)$ and $\bigl( \widehat{\C}, \sigma \bigr)$ are bi-Lipschitz equivalent, i.e., there exists a bi-Lipschitz homeomorphism $h \:  \widehat{\C}  \rightarrow  \widehat{\C}$ from $\bigl( \widehat{\C}, \omega_f \bigr)$ to $\bigl( \widehat{\C}, \sigma \bigr)$. Moreover, by the discussion in \cite[Appendix~A.10]{BM17}, $h$ cannot be the identity map.
\end{rem}

\begin{prop}   \label{propLattes}
Let $f\: \widehat{\C} \rightarrow \widehat{\C}$ be a Latt\`{e}s map, and $d\coloneqq \omega_f$ be the canonical orbifold metric of $f$ on $\widehat{\C}$ (which is, in particular, a visual metric for $f$, as recalled in Remark~\ref{rmCanonicalOrbifoldMetric}). Let $\phi \: \widehat{\C} \rightarrow \R$ be a continuously differentiable real-valued function on the Riemann sphere $\widehat{\C}$. Then $\phi \in \Holder{1} \bigl( \widehat{\C}, d \bigr)$, and the following statements are equivalent:
\begin{enumerate}
\smallskip
\item[(i)] $\phi$ is not cohomologous to a constant in $\CCC\bigl(\widehat{\C},\C \bigr)$.

\smallskip
\item[(ii)] $\phi$ is non-locally integrable with respect to $f$ and $d$ (in the sense of Definition~\ref{defLI}).

\smallskip
\item[(iii)] $\phi$ satisfies the $1$-strong non-integrability condition with respect to $f$ and $d$ (in the sense of Definition~\ref{defStrongNonIntegrability}).
\end{enumerate} 
\end{prop}

See Definition~\ref{defCohomologous} for the notion of cohomologous functions.

\begin{proof}
We denote the Euclidean metric on $\C$ by $d_0$. Let $\sigma$ be the chordal metric on $\C$ as recalled in Remark~\ref{rmChordalVisualQSEquiv}. By \cite[Proposition~8.5]{BM17}, the canonical orbifold metric $d=\omega_f$ is a visual metric for $f$. Let $\Lambda>1$ be the expansion factor of $d$ for $f$. 

Let $\mathcal{O}_f=(S^2,\alpha_f)$ be the orbifold associated to $f$ (see Subsection~7.2 in \cite{LZ24a}). Since $f$ has no periodic critical points, the ramification function $\alpha_f(z) < +\infty$ for all $z\in\widehat{\C}$ (see Definition~7.4 in \cite{LZ24a}).

By inequality (A.43) in \cite[Appendix~A.10]{BM17}, 
\begin{equation}   \label{eqPfrmCanonicalOrbifoldMetric_Chordal<Orbifold}
\sup \bigl\{  \sigma (z_1,z_2) / d (z_1,z_2)  :  z_1, \, z_2 \in \widehat{\C}, \, z_1\neq z_2 \bigr\} < +\infty.
\end{equation}
By (\ref{eqPfrmCanonicalOrbifoldMetric_Chordal<Orbifold}) and the assumption that $\phi$ is continuously differentiable, we get $\phi \in \Holder{1} \bigl( \widehat{\C}, \sigma \bigr) \subseteq \Holder{1} \bigl( \widehat{\C}, d \bigr)$.

\smallskip

We establish the equivalence of statements~(i) through (iii) as follows.

\smallskip

(i) $\Longleftrightarrow$ (ii): The equivalence follows immediately from Theorem~F in \cite{LZ24a}.

\smallskip

(ii) $\Longleftrightarrow$ (iii). The backward implication follows from Proposition~\ref{propSNI2NLI}. To show the forward implication, we assume that $\phi$ is non-locally integrable. We observe from Lemma~\ref{lmCexistsL}, Theorem~F in \cite{LZ24a}, and Lemma~\ref{lmSNIwoC} that by replacing $f$ with an iterate of $f$ if necessary, we can assume without loss of generality that there exists a Jordan curve $\CC \subseteq S^2$ such that $\post f \subseteq \CC$, $f(\CC) \subseteq \CC$, and that there exist $\xi=\{ \xi_{\minus i} \}_{i\in\N_0} \in  \Sigma_{f,\,\CC}^-$ and $\eta=\{ \eta_{\minus i} \}_{i\in\N_0} \in  \Sigma_{f,\,\CC}^-$, $X^1\in\X^1 (f,\CC)$, and $u_0, \, v_0 \in X^1$ with $X^1 \subseteq f (\xi_0) = f(\eta_0)$, and
\begin{equation}   \label{eqPfrmCanonicalOrbifoldMetric_NLI}
\phi^{f,\,\CC}_{\xi,\,\eta} (u_0, v_0) \neq 0.
\end{equation}
By the continuity of $\phi^{f,\,\CC}_{\xi,\,\eta}$ (see Lemma~\ref{lmDeltaHolder} and Definition~\ref{defTemporalDist}), we can assume that $u_0, \, v_0\in \inte(X^1)$. Without loss of generality, we can assume that $\infty \notin  X^1$. We use the usual coordinate $z = (x,y) \in \R^2$ on $X^1$. We fix a constant $C_{15} \geq 1$ depending only on $f$ and $\CC$ such that 
\begin{equation}   \label{eqPfrmCanonicalOrbifoldMetric_ChordalEuclLocalBiLip}
 C_{15}^{-1} \sigma(z_1,z_2)   \leq d_0 (z_1,z_2)  \leq C_{15} d(z_1,z_2) \qquad \text{ for all } z_1, \, z_2 \in X^1.
\end{equation}

Note that $\alpha_f (z) = 1$ for all $z\in\widehat{\C}\setminus \post f$ (see Definition~7.4 in \cite{LZ24a}). Recall the notion of \emph{singular conformal metrics} from \cite[Appendix~A.1]{BM17}. By Proposition~A.33 and the discussion proceeding it in \cite[Appendix~A.10]{BM17}, the following statements hold:
\begin{enumerate}
\smallskip
\item[(1)] The canonical orbifold metric $d$ is a singular conformal metric with a conformal factor $\rho$ that is continuous and positive everywhere except at the points in $\supp(\alpha_f) \subseteq \post f$.

\smallskip
\item[(2)] $d(z_1,z_2) = \inf\limits_{\gamma} \int_\gamma \! \rho \,\mathrm{d} \sigma$, where the infimum is taken over all $\sigma$-rectifiable paths $\gamma$ in $\widehat{\C}$ joining $z_1$ and $z_2$.  

\smallskip
\item[(3)] For each $z\in\widehat{\C} \setminus \supp(\alpha_f)$, there exists a neighborhood $U_z \subseteq \widehat{\C}$ containing $z$ and a constant $C_z \geq 1$ such that $C_z^{-1} \leq \rho(u) \leq C_z$ for all $u \in U_z$. 
\end{enumerate}

Choose connected open sets $V$ and $U$ such that $u_0, \, v_0 \in V \subseteq \overline{V} \subseteq U \subseteq \overline{U} \subseteq \inte(X^1)$. By compactness and statement~(3) above, there exists a constant $C_{16} \geq 1$ such that
\begin{equation} \label{eqPfrmCanonicalOrbifoldMetric_ConformalDensity}
C_{16}^{-1} \leq \rho(z) \leq C_{16} \qquad \text{ for all } z\in\overline{U}. 
\end{equation}
Thus, by (\ref{eqPfrmCanonicalOrbifoldMetric_ChordalEuclLocalBiLip}), (\ref{eqPfrmCanonicalOrbifoldMetric_Chordal<Orbifold}), and a simple covering argument using statement~(2) above, inequality~(\ref{eqPfrmCanonicalOrbifoldMetric_ConformalDensity}), and the fact that $\overline{V}\subseteq U$, there exists a constant $C_{17} \geq 1$ depending only on $f$, $\CC$, $d$, $\phi$, and the choices of $U$ and $V$ such that
\begin{equation} \label{eqPfrmCanonicalOrbifoldMetric_BiLipMetrics}
 C_{17}^{-1} d(z_1,z_2)   \leq d_0 (z_1,z_2)  \leq C_{17} d(z_1,z_2)  \qquad \text{ for all } z_1, \, z_2\in \overline{V}.
\end{equation}

We denote, for each $i\in\N$,
\begin{equation}  \label{eqPfrmCanonicalOrbifoldMetric_tau}
\tau_i \coloneqq  ( f|_{\xi_{1-i}}  )^{-1} \circ \cdots \circ ( f|_{\xi_{ \minus 1}} )^{-1} \circ ( f|_{\xi_{0}} )^{-1} \text{ and }
\tau'_i \coloneqq  ( f|_{\eta_{1-i}}  )^{-1} \circ \cdots \circ ( f|_{\eta_{ \minus 1}} )^{-1} \circ ( f|_{\eta_{0}}  )^{-1}.
\end{equation}
We define a function $\Phi \: X^1 \rightarrow \R$ by $\Phi(z)\coloneqq \phi^{f,\,\CC}_{\xi,\,\eta} (u_0, z)$ for $z\in X^1$ (see Definition~\ref{defTemporalDist} and Lemma~\ref{lmDeltaHolder}).

\smallskip

\emph{Claim.} $\Phi$ is continuously differentiable on $V$.

\smallskip

By Definition~\ref{defTemporalDist}, it suffices to show that the function $\mathcal{D}(\cdot) \coloneqq \Delta^{f,\,\CC}_{\phi,\, \xi} (u_0, \cdot)$ is continuously differentiable on $V$. By Lemma~\ref{lmDeltaHolder}, the function $\mathcal{D}(z) = \sum_{i=0}^{+\infty} ( (\phi \circ \tau_i) (u_0) - (\phi \circ \tau_i) (z) )$ is the uniform limit of a series of continuous functions on $V$. Since $V \subseteq \inte(X^1)$, by (\ref{eqPfrmCanonicalOrbifoldMetric_tau}) and Proposition~\ref{propCellDecomp}~(i), the function $\phi \circ \tau_i$ is differentiable on $V$ for each $i\in\N$.

We fix an arbitrary integer $i\in\N$. For each pair of distinct points $z_1, \, z_2 \in \inte(X^1)$, we choose the maximal integer $m\in\N$ with the property that there exist two $m$-tiles $X^m_1,\, X^m_2 \in \X^m(f,\CC)$ such that $z_1 \in X^m_1$, $z_2 \in X^m_2$, and $X^m_1 \cap X^m_2 \neq \emptyset$. Then by Proposition~\ref{propCellDecomp}~(i) and Lemma~\ref{lmCellBoundsBM}~(i) and (ii),
\begin{align*}
            \frac{  \abs{ (\phi\circ \tau_i) (z_1) - (\phi\circ \tau_i) (z_2)  }  }  { d(z_1,z_2) }
& \leq \frac{  \Hnorm{1}{\phi}{ ( \widehat{\C}, d ) }  \diam_d ( \tau_i ( X^m_1 \cup X^m_2 ) ) }  { C^{-1} \Lambda^{-(m+1)} }  \\
& \leq \Hnorm{1}{\phi}{ ( \widehat{\C}, d ) }   \frac{ 2 C \Lambda^{- (m+i)} }  { C^{-1} \Lambda^{-(m+1)} } 
\leq  2 C^2\Hnorm{1}{\phi}{ ( \widehat{\C}, d ) }   \Lambda^{1-i},
\end{align*}
where $C\geq 1$ is the constant from Lemma~\ref{lmCellBoundsBM} depending only on $f$, $\CC$, and $d$. Thus, by (\ref{eqPfrmCanonicalOrbifoldMetric_BiLipMetrics}),
\begin{align*}
&                        \sup \biggl\{    \Absbigg{  \frac{\partial}{\partial x}  (\phi\circ\tau_i) (z)  }     :  z\in V  \biggr\} \\
&\qquad \leq  \sup \{                        \abs{ (\phi\circ \tau_i) (z_1) - (\phi\circ \tau_i) (z_2)  }  /  d_0(z_1,z_2)        :  z_1,  \, z_2 \in V ,\, z_1 \neq z_2  \} \\    
&\qquad \leq  C_{17} \sup \{          \abs{ (\phi\circ \tau_i) (z_1) - (\phi\circ \tau_i) (z_2)  }  / d     (z_1,z_2)         :  z_1,  \, z_2 \in V ,\, z_1 \neq z_2  \}      \\                
&\qquad \leq   2 C_{17} C^2\Hnorm{1}{\phi}{ ( \widehat{\C}, d ) }   \Lambda^{1-i}.
\end{align*}

Hence, $\frac{\partial}{\partial x }  \mathcal{D}$ exists and is continuous on $V$. Similarly, $\frac{\partial}{\partial y }  \mathcal{D}$ exists and is continuous on $V$. Therefore, $\mathcal{D}$ is continuously differentiable on $V$, establishing the claim.

\smallskip

By the claim, (\ref{eqPfrmCanonicalOrbifoldMetric_NLI}), and the simple observation that $\phi^{f,\,\CC}_{\xi,\,\eta} (u_0, u_0) =0$, there exist numbers $M_0 \in \N$, $\varepsilon \in (0,1)$, and $C_{18} > 1$, and $M_0$-tiles $Y^{M_0}_\b \in \X^{M_0}_\b(f,\CC)$ and $Y^{M_0}_\w \in \X^{M_0}_\w(f,\CC)$ such that $C_{18} \geq C_{17}$, $Y^{M_0}_\b \cup Y^{M_0}_\w \subseteq V \subseteq \inte(X^1)$, and at least one of the following two inequalities holds:
\begin{enumerate}
\smallskip
\item[(a)] $\inf \bigl\{ \Absbig{\frac{\partial}{\partial x}  \Phi(z) }    :   z\in h^{-1} \bigl( Y^{M_0}_\b \cup Y^{M_0}_\w \bigr)   \bigr\}  \geq 2 C_{18} \varepsilon$,

\smallskip
\item[(b)] $\inf \bigl\{ \Absbig{\frac{\partial}{\partial y}  \Phi(z) }    :   z\in h^{-1} \bigl( Y^{M_0}_\b \cup Y^{M_0}_\w \bigr)   \bigr\}  \geq 2 C_{18} \varepsilon$.
\end{enumerate}

We assume now that inequality (a) holds and remark that the proof in the other case is similar.

Without loss of generality, we can assume that $\varepsilon \in \bigl(0,(2C_{18}C)^{-2} \bigr)$.

Then by Lemma~\ref{lmCellBoundsBM}~(v), for each $\c\in\{\b, \, \w\}$, each integer $M \geq M_0$, and each $M$-tile $X\in \X^M(f,\CC)$ with $X\subseteq  Y^{M_0}_\c$, there exists a point $u_1(X) = (x_1(X), y_0(X)) \in X$ such that $B_d \bigl( u_1(X), C^{-1}\Lambda^{-M} \bigr) \subseteq X$. We choose $x_2(X)\in\R$ such that $\abs{x_1(X) - x_2(X) } = (4C_{18}C)^{-1} \Lambda^{-M}$. Then by (\ref{eqPfrmCanonicalOrbifoldMetric_BiLipMetrics}) and $C_{18} \geq C_{17}$, we get
\begin{align}   \label{eqPfrmCanonicalOrbifoldMetric_u1u2Def}
u_2(X) \coloneqq (x_2(X), y_0(X)) \in & B_{d_0} \bigl(u_1(X), (2C_{18}C)^{-1} \Lambda^{-M} \bigr)  \notag \\
                                                                 &     \subseteq B_d  \bigl(u_1(X), (  2    C)^{-1} \Lambda^{-M} \bigr)  
                                                                      \subseteq B_d  \bigl(u_1(X),  C^{-1} \Lambda^{-M} \bigr)   \subseteq X. 
\end{align}
In particular, the entire horizontal line segment connecting $u_1(X)$ and $u_2(X)$ is contained in $\inte (X)$. By (\ref{eqPfrmCanonicalOrbifoldMetric_u1u2Def}), Lemma~\ref{lmCellBoundsBM}~(ii), (\ref{eqPfrmCanonicalOrbifoldMetric_BiLipMetrics}), and $C_{18} \geq C_{17}$, we get
\begin{align}  \label{eqPfrmCanonicalOrbifoldMetric_u1u2}
&                       \min \bigl\{  d \bigl( u_1(X), \widehat{\C} \setminus X \bigr), \, d \bigl( u_2(X),  \widehat{\C} \setminus X \bigr), \, d( u_1(X), u_2(X) ) \bigr\}  \\
&\qquad \geq \min \bigl\{ (  2    C)^{-1} \Lambda^{-M} ,  \,  C_{18}^{-1} (4C_{18}C)^{-1} \Lambda^{-M}   \bigr\}
                 \geq \varepsilon \diam_d(X).   \notag
\end{align}
On the other hand, by (\ref{eqPfrmCanonicalOrbifoldMetric_BiLipMetrics}), $C_{18} \geq C_{17}$, Definition~\ref{defTemporalDist}, inequality~(a) above, and the mean value theorem, 
\begin{equation*}
           \frac{ \Absbig{  \phi^{f,\,\CC}_{\xi,\,\eta} ( u_1(X), u_2(X) ) } } { d ( u_1(X), u_2(X) )}
\geq   \frac{ \Absbig{  \phi^{f,\,\CC}_{\xi,\,\eta} ( u_1(X), u_2(X) ) } } { C_{18} d_0( u_1(X), u_2(X) )}
 =        \frac{ \abs{ \Phi( u_1(X) ) - \Phi( u_2(X) )}  }   {  C_{18}  \abs{ x_1(X) - x_2(X) }} 
 \geq   2 \varepsilon.
\end{equation*}

We choose
\begin{equation}  \label{eqPfrmCanonicalOrbifoldMetric_N0}  
N_0    \coloneqq 
 \bigl\lceil  \log_\Lambda  \bigl( 2 C^2 \varepsilon^{-2}\Hseminorm{1, (\widehat{\C},d)}{\phi} C_0 \big/ \bigl(1-\Lambda^{-1} \bigr) \bigr)   \bigr\rceil,  
\end{equation}
where $C_0 > 1$ is the constant depending only on $f$, $\CC$, and $d$ from Lemma~\ref{lmMetricDistortion}.

Fix arbitrary $N\geq N_0$. Define $X^{N+M_0}_{\c,1} \coloneqq \tau_{N} \bigl( Y^{M_0}_\c \bigr)$ and $X^{N+M_0}_{\c,2} \coloneqq \tau'_{N} \bigl( Y^{M_0}_\c \bigr)$ (see (\ref{eqPfrmCanonicalOrbifoldMetric_tau})). Note that $\varsigma_1 = \tau_{N}|_{ Y^{M_0}_\c }$ and $\varsigma_2 = \tau'_{N}|_{ Y^{M_0}_\c }$.

Then by Definition~\ref{defTemporalDist}, (\ref{eqPfrmCanonicalOrbifoldMetric_u1u2}), Lemmas~\ref{lmDeltaHolder},~\ref{lmSnPhiBound},~\ref{lmCellBoundsBM}~(i) and~(ii), and Proposition~\ref{propCellDecomp}~(i), 
\begin{align*}
&                       \frac{ \abs{  S_{N }\phi ( \varsigma_1 (u_1(X)) ) -  S_{N }\phi ( \varsigma_2 (u_1(X)) )  -S_{N }\phi ( \varsigma_1 (u_2(X)) ) +  S_{N }\phi ( \varsigma_2 (u_2(X)) )   }  } 
                                  {  d(u_1(X),u_2(X))  }  \\
&\qquad \geq  \frac{ \Absbig{ \phi^{f,\,\CC}_{\xi,\,\eta} (u_1(X)) , u_2(X))  }   } { d(u_1(X),u_2(X)) }
                                 -  \limsup_{n\to+\infty}       \frac{ \abs{  S_{n - N }\phi ( \tau_n (u_1(X)) ) -  S_{n - N }\phi ( \tau_n  (u_2(X)) )  }  }   {  \varepsilon \diam_d (X) }  \\
&\qquad\qquad    -  \limsup_{n\to+\infty}       \frac{ \abs{  S_{n - N }\phi ( \tau'_n (u_1(X)) ) -  S_{n - N }\phi ( \tau'_n (u_2(X)) )  }  }   {  \varepsilon \diam_d (X) }              \\
&\qquad \geq 2 \varepsilon -  \frac{\Hseminorm{1, (\widehat{\C},d)}{\phi} C_0}{1-\Lambda^{-1}} \cdot
                                                     \frac{    d ( \tau_{N} (u_1(X)) , \tau_{N} (u_2(X)))  +  d ( \tau'_{N} (u_1(X)) , \tau'_{N} (u_2(X))) }   {  \varepsilon   \diam_d (X) }   \\
&\qquad \geq 2 \varepsilon -  \frac{\Hseminorm{1, (\widehat{\C},d)}{\phi} C_0}{1-\Lambda^{-1}} \cdot
                                                     \frac{  \diam_d ( \tau_{N} (X) )   +    \diam_d ( \tau'_{N} (X) ) }   {  \varepsilon   \diam_d (X) }           \\
&\qquad \geq 2 \varepsilon -   \frac{\Hseminorm{1, (\widehat{\C},d)}{\phi} C_0}{1-\Lambda^{-1}} \cdot
                                                     \frac{  2 C  \Lambda^{ -  (M  + N )}  }   {  \varepsilon  C^{-1} \Lambda^{ - M} }           \\
&\qquad \geq 2 \varepsilon -    2 C^2 \varepsilon^{-1}\Hseminorm{1, (\widehat{\C},d)}{\phi} C_0     \Lambda^{ - N_0}  \bigl(1-\Lambda^{-1} \bigr)^{-1}\\
&\qquad \geq \varepsilon,                                                                                                 
\end{align*}
where the last inequality follows from (\ref{eqPfrmCanonicalOrbifoldMetric_N0}). 

Therefore, $\phi$ satisfies the $1$-strong non-integrability condition with respect to $f$ and $d$.
\end{proof}

\begin{proof}[Proof of Theorem~\ref{thmLattesPOT}]
By Proposition~\ref{propLattes}, $\phi \in \Holder{\alpha}\bigl( \widehat{\C},d \bigr)$. So the existence and uniqueness of $s_0>0$ follows from Corollary~\ref{corS0unique}.

The implication (i)$\implies$(iii) follows from Proposition~\ref{propLattes} and Theorem~\ref{thmPrimeOrbitTheorem}. The implication (iii)$\implies$(ii) is trivial. The implication (ii)$\implies$(i) follows immediately from \cite[Theorem~B]{LZ24a}.
\end{proof}

\end{document}